\documentclass[reqno,12pt,letterpaper]{amsart}
\usepackage[proof]{sdmacros}

\DeclareMathOperator{\Refl}{Ref}
\DeclareMathOperator{\sing}{sing}

\title{Mathematics of internal waves in a 2D aquarium}
\author{Semyon Dyatlov}
\email{dyatlov@math.mit.edu}
\address{Department of Mathematics, Massachusetts Institute of Technology, Cambridge, MA 02139}
\author{Jian Wang}
\email{wangjian@email.unc.edu}
\address{Department of Mathematics, University of North Carolina, Chapel Hill, NC 27514}
\author{Maciej Zworski}
\email{zworski@math.berkeley.edu}
\address{Department of Mathematics, University of California, Berkeley, CA 94720}

\begin{document}

\begin{abstract}
Following theoretical and experimental work of Maas et al \cite {Maas-Nature} we consider a linearized model for internal waves in effectively two dimensional
aquaria. We provide a precise description of singular profiles appearing in long time 
wave evolution and 
associate them to classical attractors. That is done by microlocal analysis of the spectral
Poincar\'e problem,  leading in particular to a limiting absorption principle. 
Some aspects of the paper (for instance \S \ref{s:microp}) can be considered as a natural microlocal 
continuation of the work of John \cite{John-Dirichlet} on the Dirichlet problem for 
hyperbolic equations in two dimensions.
\end{abstract}

\maketitle


\section{Introduction}

Internal waves  are a central topic in oceanography and the theory of rotating fluids
-- see  \cite{MaasLNL} and \cite{Siber} for reviews and references. 
\begin{figure}
\includegraphics[width=12.3cm]{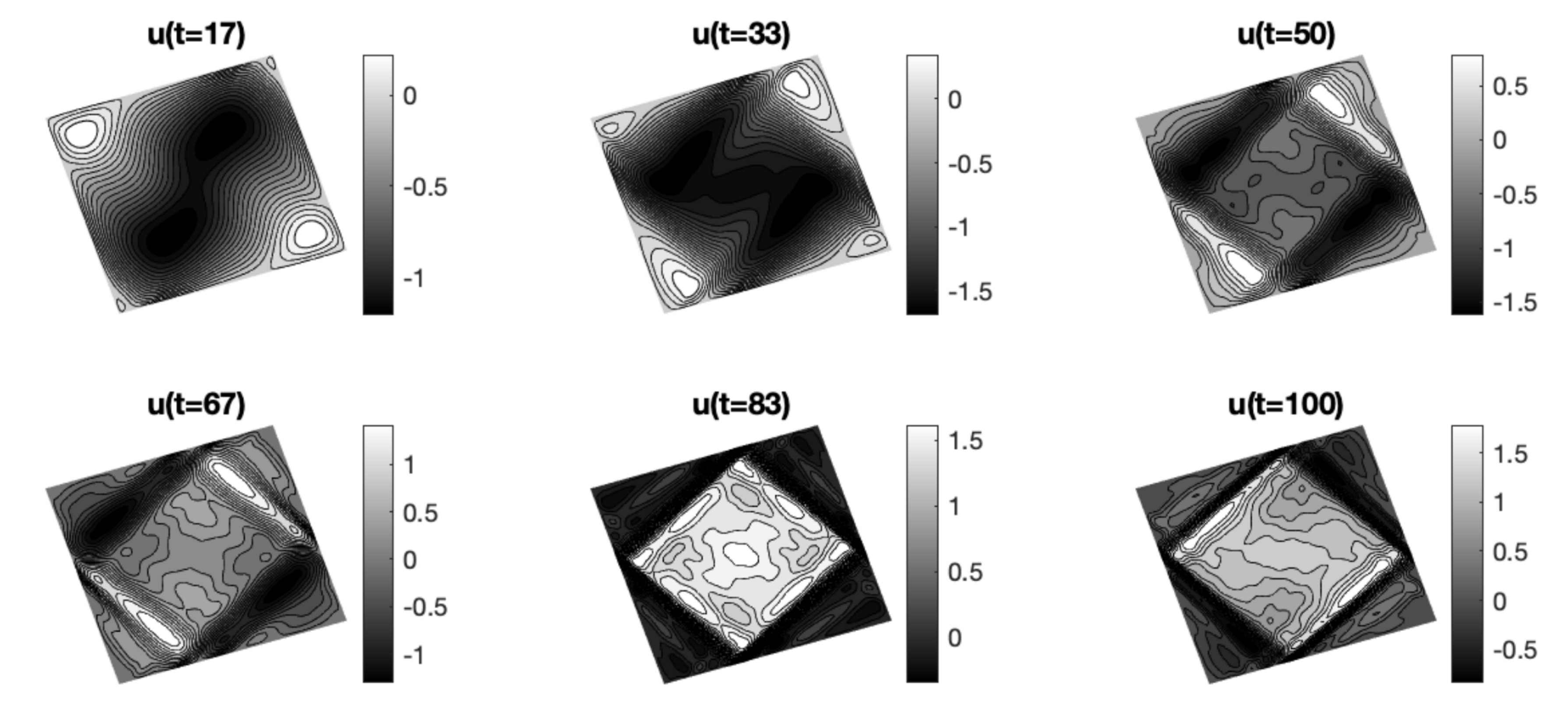}
\caption{Contour plots of  a numerical solution to \eqref{eq:PDE} for $ \Omega $ given by a unit square  tilted by $ \pi/10 $ (see 
\S \ref{s:exms}), with $ f ( x ) = e^{-5\pi^2(x-x^0)^2 } $, where $ x^0 $ is the 
center  and $\lambda = 0.8 $. In that case the rotation number of the 
billiard ball map is $ \frac12 $ (see Figure \ref{f:examples-2}) and the classical attractor is given by 
a parallelogram on which $ u $ develops a singularity -- see Theorem \ref{t:1}.}
\label{f:tilted}
\end{figure}
They can be described by linear perturbations of the initial state of rest of a stable-stratified fluid (dense fluid lies everywhere below less-dense fluid and the isodensity surfaces are all horizontal). Forcing can take place at linear level by pushing fluid away from this equilibrium state either mechanically, by wind, a piston, a moving boundary, or thermodynamically, by spatially differential heating or evaporation/rain. 

The mechanism behind formation of internal waves comes from ray dynamics of the classical system which underlies wave equations -- see \S \ref{s:ass} for the case of nonlinear ray dynamics 
relevant to the case we consider. When parameters
of the system produce hyperbolic dynamics, attractors are observed in wave evolution -- see Figure \ref{f:tilted}.  This phenomenon is both physically and theoretically more accessible in dimension two. The analysis in the physics literature, see \cite{MaasLNL}, \cite{Trotsky}, has focused on constructions of
standing and propagating waves and did not address the evolution problem analytically. (See however
\cite{Bajars} for an analysis of  a numerical approach to the evolution problem.) In this paper we 
prove the emergence of singular profiles in the long time evolution of linear waves for two dimensional domains.

\begin{figure}
\includegraphics[width=7.25cm]{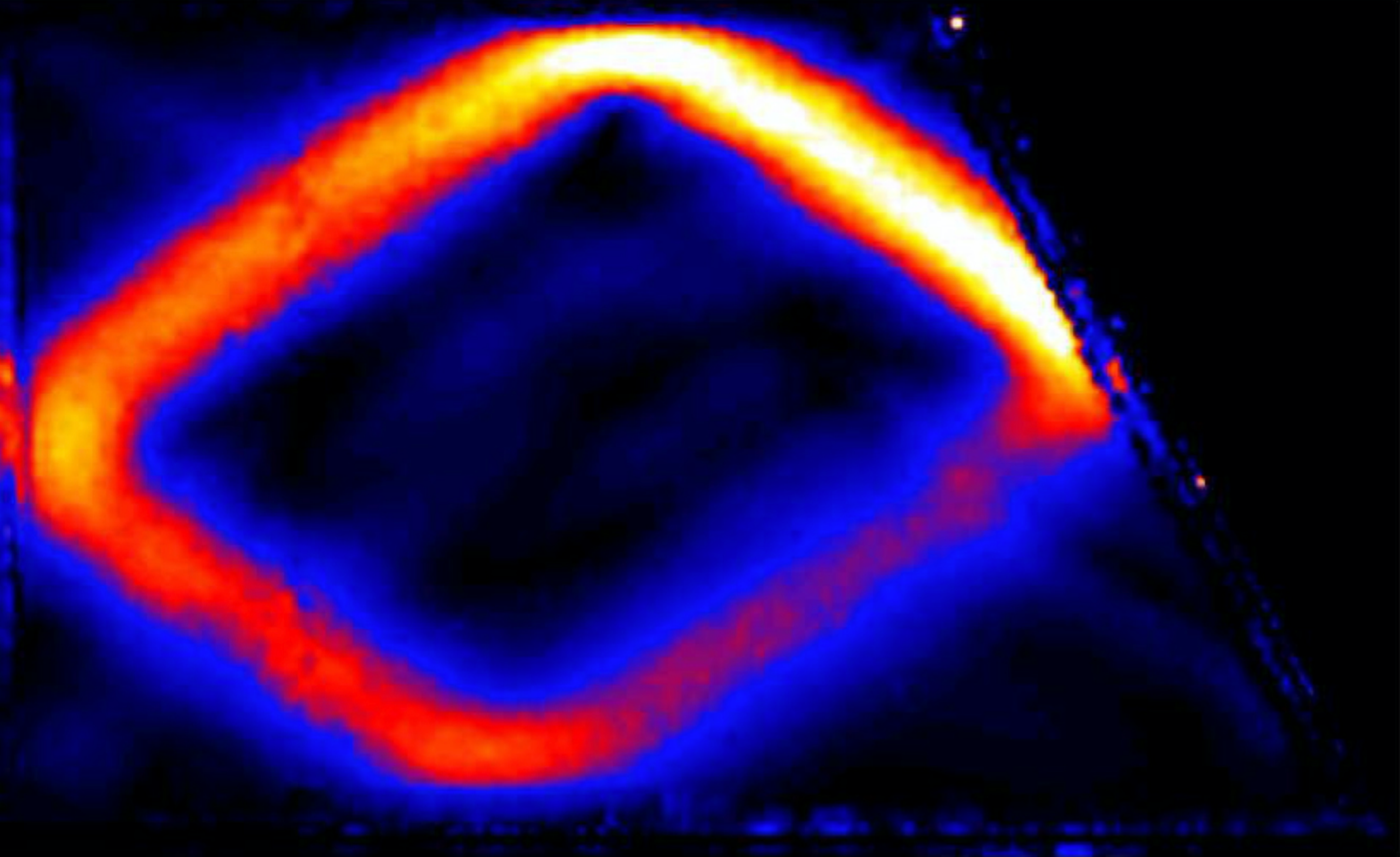}
\qquad
\includegraphics[width=7.25cm]{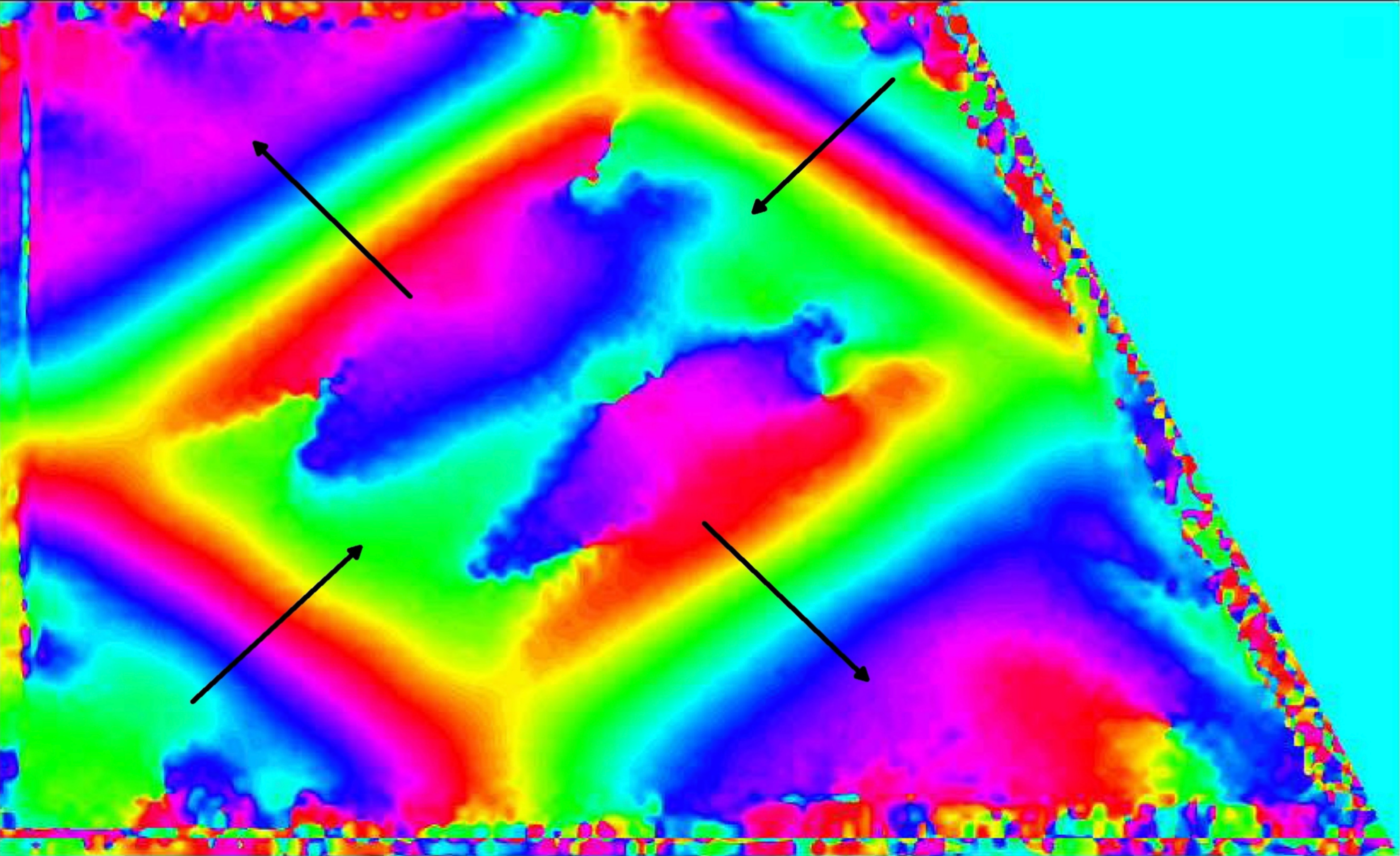}
\caption{Experimental results of
Hazewinkel et al \cite{Hazewinkel}: 
horizontal component of the observed perturbation
buoyancy gradient projected onto a field
that oscillates at the forcing frequency,
thus reducing the
time series to an amplitude field (left) and a phase field (right). 
In terms of our Theorem \ref{t:1} this corresponds to amplitude
and phase of $ u^+ $. 
 The arrows indicate directions of phase propagation 
in agreement with our analysis, shown on Figure \ref{f:Lamp}.}
\label{f:psycho}
\end{figure}

The model we consider is described as follows.
Let $ \Omega \subset \mathbb R^2 = \{ x = (x_1, x_2 ) : x_j \in  \mathbb R \} $ be a bounded simply connected open set with $C^\infty$ boundary $\partial\Omega$. Following the fluid mechanics literature we consider the following evolution problem, sometimes referred to as the Poincar\'e problem:
\begin{equation}
\label{eq:PDE}
( \partial_t^2 \Delta + \partial_{x_2}^2 ) u = f(x) \cos \lambda t , \ \ u|_{t=0} = \partial_t u|_{t=0} = 0 , \ \ 
u |_{\partial \Omega } = 0 ,
\end{equation}
where $ \lambda \in (0, 1) $ and $ \Delta := \partial_{x_1}^2 + \partial_{x_2}^2 $, see
Sobolev~\cite[equation~(48)]{Sobolev-fluid}, 
Ralston~\cite[p.374]{RalstonR}, Maas et al~\cite{Maas-Nature}, Brouzet~\cite[\S\S 1.1.2--3]{Brouzet-thesis}, Dauxois et al~\cite{Dauxois-internal}, Colin de Verdi\`ere--Saint-Raymond~\cite{CdV-LSR}, Sibgatullin--Ermanyuk~\cite{Sibgatullin}, and references given there.
It models internal waves in a stratified fluid in an effectively two-dimensional aquarium~$ \Omega $
with an oscillatory forcing term (here we follow~\cite{CdV-LSR} rather than change the boundary condition). 
The geometry of $ \Omega $ and the forcing frequency $ \lambda $ can produce concentration 
of the fluid velocity $ \mathbf v = ( \partial_{x_2} u , -\partial_{x_1} u)  $ on attractors. This phenomenon was predicted by Maas--Lam~\cite{Maas-Lam} and was then observed experimentally by 
Maas--Benielli--Sommeria--Lam~\cite{Maas-Nature}, see Figure~\ref{f:psycho}
for experimental data from the more recent~\cite{Hazewinkel}.
(See also the earlier work of Wunsch~\cite{Wunsch1968} which studied the case of an internal wave converging to a corner,
along a trajectory of the type pictured on Figure~\ref{f:corners}.)
In this paper we provide a mathematical 
explanation: as mentioned above 
the physics papers concentrated on the analysis of modes and classical dynamics
rather than on the long time behaviour of solutions to~\eqref{eq:PDE}.

\subsection{\texorpdfstring{Assumptions on $ \Omega $ and $ \lambda $}{Assumptions on \unichar{"03A9} and \unichar{"03BB}}} 
\label{s:ass}

The assumptions on~$ \Omega $ and~$ \lambda$ which guarantee existence of singular profiles (internal waves) in long time evolution of \eqref{eq:PDE} are formulated using a ``chess billiard''~-- see~\cite{Nogueira-Troubetzkoy}, \cite{Lenci-chess} for recent studies and references. It was first considered in similar context by John~\cite{John-Dirichlet} (see
also the later work of Aleksandrjan~\cite{Aleksandrian}) and was the basis of the analysis in~\cite{Maas-Lam}. 
It is defined as the reflected bicharacteristic flow for 
 $( 1 - \lambda^2 ) \xi_2^2 - \lambda^2 \xi_1^2 $,  which is  the Hamiltonian for the $1+1$ wave equation with $ x_2 $ corresponding to time and the speed given by $ c = \lambda/\sqrt { 1 - \lambda^2 } $~-- see Figure \ref{f:1} and~\S \ref{s:chess-basic}. 
This flow has a simple reduction to the boundary which we describe using a factorization of the quadratic form dual to $ ( 1 - \lambda^2 ) \xi_2^2 - \lambda^2 \xi_1^2 $:
\begin{equation}
\label{eq:dual2} 
-\frac{x_1^2} { \lambda^2 } + \frac{x_2^2}{1 - \lambda^2 } = \ell^+ ( x, \lambda ) \ell^-
( x, \lambda ) , \quad
 \ell^\pm(x,\lambda): =\pm{x_1\over\lambda}+{x_2\over\sqrt{1-\lambda^2}} .
\end{equation}
We often suppress the dependence on $\lambda$, writing simply $\ell^\pm(x)$.
Same applies to other $\lambda$-dependent objects introduced below.
\begin{defi}
\label{d:1}
Let $0<\lambda<1$.
We say that $\Omega$ is \emph{$\lambda$-simple} if each of the functions $ \partial\Omega \ni x \mapsto\ell^\pm(x,\lambda)$ has only two critical points,
which are both nondegenerate. We denote these minimum/maximum points by~$x^\pm_{\min}(\lambda),x^\pm_{\max}(\lambda)$.
\end{defi}
Under the assumption of $\lambda$-simplicity we define the following two smooth orientation reversing involutions on the boundary (see~\S\ref{s:chess-basic} for more details):
\begin{equation}
\label{e:gamma-pm-def} 
\gamma^\pm(\bullet,\lambda):\partial\Omega\to\partial\Omega,\quad
\ell^\pm(x)=\ell^\pm(\gamma^\pm(x)).
\end{equation} 
These  maps correspond to interchanging intersections of the boundary with lines
with slopes $ \mp 1/c $, respectively~-- see Figure~\ref{f:1}.
The \emph{chess billiard map} $b(\bullet,\lambda)$ is defined as the composition
\begin{equation}
  \label{e:b-def}
b:=\gamma^+\circ \gamma^-
\end{equation}
and is a $C^\infty$ orientation preserving diffeomorphism of $\partial\Omega$.

\begin{figure}
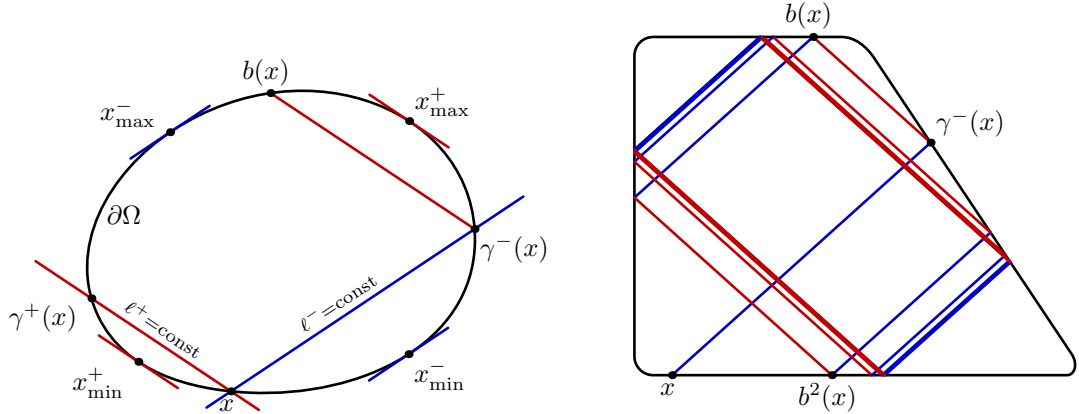

\includegraphics{bflop.1}
\qquad
\includegraphics{bflop.2}
\caption{Left: the involutions~$\gamma^\pm$ and the chess billiard map~$b$.
Right: a forward trajectory of the map~$b$ on a trapezium with rounded
corners, converging to a periodic trajectory. We remark that the effect of smoothed
corner on classical dynamics was investigated by Manders--Duistermaat--Maas~\cite{MaD}, 
see also \S \ref{s:domc}.}
\label{f:1}
\end{figure}
Denoting by $b^n$ the $n$-th iterate of $b$, we consider the set of periodic points
\begin{equation}
  \label{e:periodic-points}
\Sigma_\lambda:=\{x\in\partial\Omega\mid b^n(x,\lambda) = x\text{ for some }n\geq 1\}.
\end{equation}
If $\Sigma_\lambda\neq\emptyset$, then all the periodic points in~$\Sigma_\lambda$
have the same minimal period, see~\S\ref{s:chess-basic}.

We are now ready to state the dynamical assumptions on the chess billiard:
\begin{defi}
\label{d:2}
Let $0<\lambda<1$. We say that $\lambda$ satisfies the \emph{Morse--Smale conditions} if:
\begin{enumerate}
\item $\Omega$ is $\lambda$-simple;
\item the map $b$ has periodic points, that is $\Sigma_\lambda\neq\emptyset$;
\item the periodic points are hyperbolic, that is
$\partial_x b^n(x,\lambda)\neq 1$ for all $x\in \Sigma_\lambda$
where $n$ is the minimal period.
\end{enumerate}
\end{defi}
Under the Morse--Smale conditions we have
$\Sigma_\lambda=\Sigma_\lambda^+\sqcup \Sigma_\lambda^-$ where
$\Sigma_\lambda^+,\Sigma_\lambda^-$ are the sets of attractive, respectively repulsive, periodic points of~$b$:
\begin{equation}
  \label{e:Sigma-pm-def}
\Sigma_\lambda^+:=\{x\in \Sigma_\lambda\mid \partial_x b^n(x,\lambda)<1\},\quad
\Sigma_\lambda^-:=\{x\in \Sigma_\lambda\mid \partial_x b^n(x,\lambda)>1\}.
\end{equation}
Moreover, each of the involutions $\gamma^\pm$ exchanges $\Sigma^+_\lambda$ with $\Sigma^-_\lambda$,
see~\eqref{e:b-inversor}.

For $y\in\partial\Omega$, let
\begin{equation}
  \label{e:Gamma-def}
\Gamma_\lambda^\pm(y):=\{x\in\Omega\mid \ell^\pm(x,\lambda)=\ell^\pm(y,\lambda)\}
\end{equation}
be the open line segment connecting $y$ with $\gamma^\pm(y,\lambda)$. Denote
$\Gamma_\lambda(y):=\Gamma_\lambda^+(y)\cup\Gamma_\lambda^-(y)$.
Then $\Gamma_\lambda(\Sigma_\lambda)$ gives the closed trajectories of the chess
billiard inside~$\Omega$.

\begin{figure}
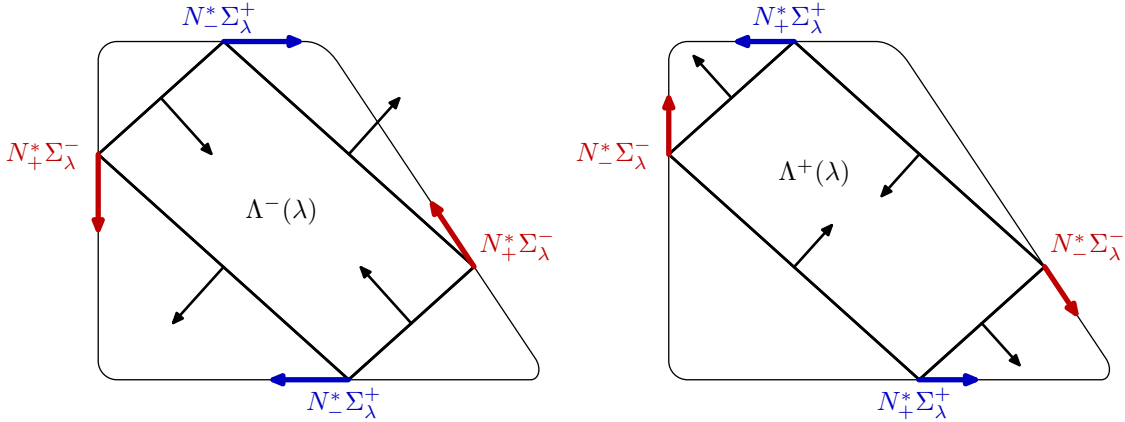

\includegraphics{bflop.10}
\includegraphics{bflop.11}
\caption{A visualization of the Lagrangian submanifolds  \eqref{eq:defLa} corresponding to 
attractive and repulsive cycles of $b$ given in \eqref{e:b-def}. The parallelogram
represents the projection of the attractive ($+$) and repulsive ($-$)  Lagrangians $ \Lambda^\pm ( \lambda ) $ and the arrows perpendicular to the sides represent the conormal directions distinguishing the
two Lagrangians. We also indicate the corresponding sets on the boundary: $ \Sigma_\lambda^\pm $ are the attractive ($+$) and repulsive ($-$) periodic points of $ b $ given by \eqref{e:b-def}
and the arrows indicate the sign of the conormal directions.}
\label{f:Lamp}
\end{figure}

For $y\in\partial\Omega$ which is not a critical point of~$\ell^+$, we split the conormal bundle $N^*\Gamma_\lambda^+(y)$ into the positive/negative directions:
\begin{equation}
  \label{e:n-pm-gamma-def}
\begin{gathered}
N^*\Gamma_\lambda^+(y)\setminus 0=N^*_+\Gamma_\lambda^+(y)\sqcup N^*_-\Gamma_\lambda^+(y),\\
N^*_\pm \Gamma_\lambda^+(y):=\{(x,\tau d\ell^+(x))\mid x\in\Gamma_\lambda^+(y),\
\pm (\partial_\theta\ell^+(y))\tau>0\}
\end{gathered}
\end{equation}
and similarly for $N^*\Gamma_\lambda^-(y)$. Here $\partial_\theta$ is the derivative with respect to a positively oriented
(that is, counterclockwise when $\Omega$ is convex) parametrization of the boundary $\partial\Omega$.
Note that the orientation depends
on the choice of $y$ and not just on~$\Gamma_\lambda^\pm(y)$: we have
$N^*_+\Gamma_\lambda^\pm(\gamma^\pm(y))=N^*_-\Gamma_\lambda^\pm(y)$.
 
We now define Lagrangian submanifolds $ \Lambda^\pm ( \lambda ) \subset 
T^* \Omega \setminus 0 $ by 
\begin{equation}
  \label{eq:defLa} 
\Lambda^\pm(\lambda):=N_+^*\Gamma^-_\lambda(\Sigma^\pm_\lambda)\sqcup N_-^*\Gamma^+_\lambda(\Sigma^\mp_\lambda),
\end{equation}
see Figure~\ref{f:Lamp}. We note that $ \pi ( \Lambda^\pm ( \lambda ) ) = \Gamma_\lambda ( \Sigma_\lambda)$ and $N_-^*\Gamma^\pm_\lambda(\Sigma^+_\lambda)=N_+^*\Gamma^\pm_\lambda(\Sigma^-_\lambda)$.

\subsection{Statement of results}
 
The main result of this paper is formulated using the concept of {\em wave front set}, see~\cite[\S8.1]{Hormander1} and~\cite[Theorem 18.1.27]{Hormander3}.
The wave front set of a distribution, $ \WF ( u ) $, is a closed subspace of 
the cotangent bundle of $ T^* \Omega \setminus 0 $ and it provides phase space information about singularities. Its projection to the base, $ \pi ( \WF ( u ) ) $,  is the singular support, 
$ \sing\supp u $.
\begin{theo}
\label{t:1} 
Suppose that $ \Omega $ and $ \lambda \in ( 0, 1 ) $ satisfy the Morse--Smale conditions of Definition \ref{d:2}.
Assume that $f\in\CIc(\Omega;\mathbb R)$.
Then the solution to \eqref{eq:PDE}
is decomposed as
\begin{equation}
\label{eq:prof}
\begin{gathered}
u ( t ) = \Re \big(e^{  i \lambda t } u^+\big) + r ( t) +e ( t ) , \quad
u^+ \in H^{\frac12-} ( \Omega ) ,\quad
\WF ( u^+ ) \subset  \Lambda^+( \lambda )  ,
 \\
r ( t ) \in H^1 ( \Omega ) ,  \quad
\| r ( t ) \|_{ H^1 ( \Omega ) } \leq C , \quad
\| e( t ) \|_{ H^{\frac12-} ( \Omega )  } \to 0 \quad\text{as } t \to \infty ,
\end{gathered}
\end{equation}
where $ \Lambda^+(\lambda )  \subset T^* \Omega \setminus 0 $ is the attracting
Lagrangian~-- see \eqref{eq:defLa} 
 and Figure~\ref{f:Lamp}. In particular,
$\sing\supp u^+ $ is contained in the union of closed orbits of the chess billiard flow.
In addition, $ u^+ $ is a Lagrangian distribution, $ u^+ \in I^{-1} ( \overline\Omega,\Lambda^+ ( \lambda ) ) $
(see \S \ref{s:con}) and $ u^+ |_{\partial \Omega } = 0 $ (well defined because
of the wave front set condition). 
\end{theo}
For a numerical illustration of \eqref{eq:prof}, see Figure~\ref{f:tilted}. We remark that
numerically it is easier to consider polygonal domains -- see~\S \ref{s:domc} for a discussion of the 
stability of our assumptions for smoothed out polygonal domains.

Theorem~\ref{t:1} is proved using spectral properties of a self-adjoint operator
associated to the evolution equation \eqref{eq:PDE}. To define it,
let $  \Delta_\Omega $ be the (negative definite) Dirichlet Laplacian of $ \Omega $ with the
inverse denoted by $  \Delta_\Omega^{-1} : H^{-1} ( \Omega ) \to H_0^1 ( \Omega ) $. 
Then
\begin{equation}
\label{eq:Rals}
P := \partial_{x_2}^2 \Delta_\Omega^{-1}  : H^{-1} ( \Omega ) \to H^{-1} ( \Omega ) , \ \ \
\langle u, w\rangle_{ H^{-1} ( \Omega ) } := \langle \nabla \Delta_\Omega^{-1} u , \nabla \Delta_\Omega^{-1} w
\rangle_{L^2 ( \Omega )}  , 
\end{equation}
is a bounded non-negative (hence self-adjoint) operator studied by Aleksandrjan~\cite{Aleksandrian} and Ralston~\cite{RalstonR} -- see~\S \ref{s:sap}. Studying the spectrum of $ P $ is referred to as a {\em Poincar\'e problem}.  

\begin{figure}
 \includegraphics[width=15cm]{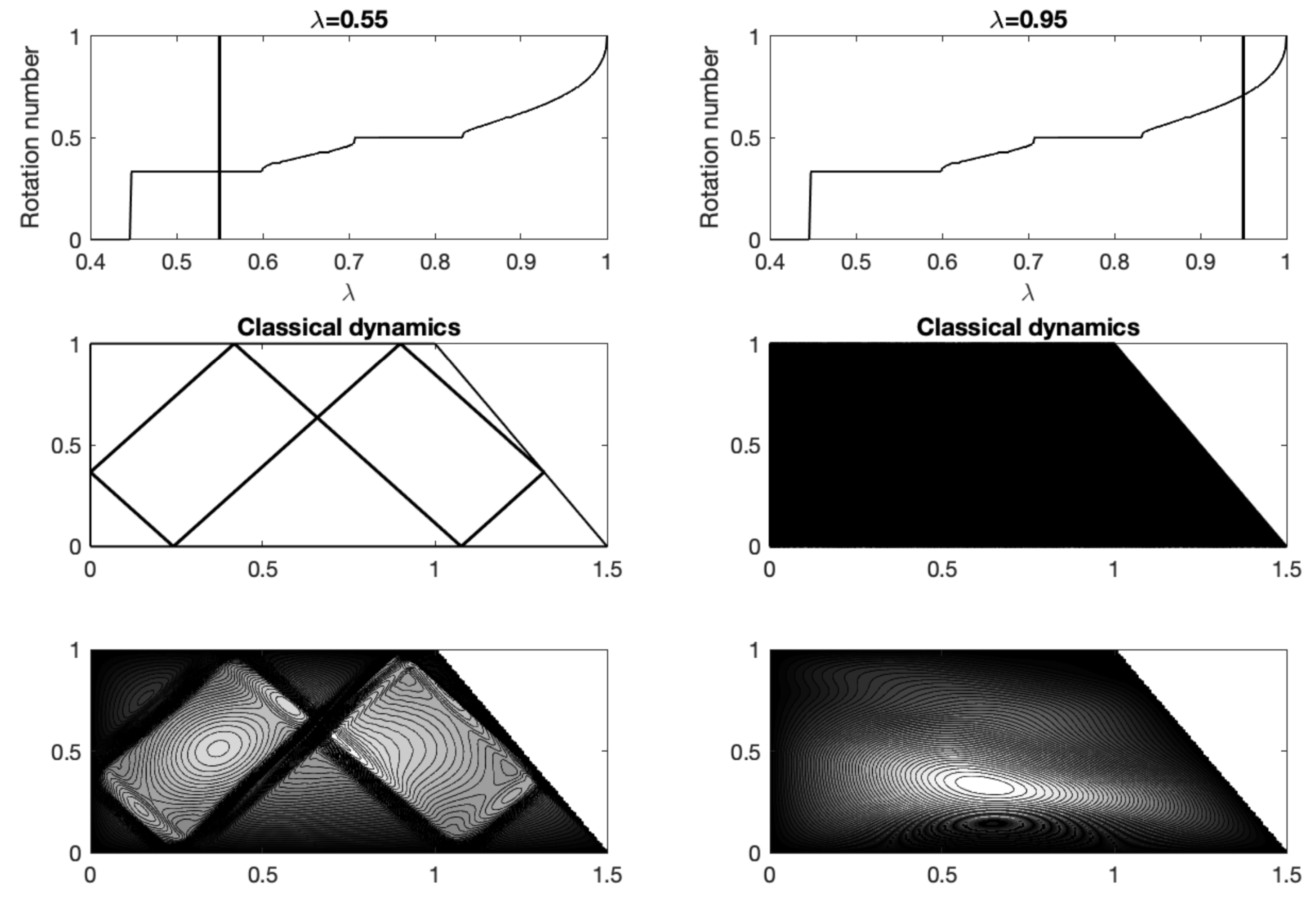} 
\caption{Numerical illustration of Theorem \ref{t:2}: contour plots of 
$ |u ( x ) |$ for $ u = ( \partial_{x_2}^2 - (\lambda^2 + i \varepsilon )\Delta_\Omega )^{-1} f $
where $ \varepsilon = 0.005 $ and $ f( x) =e^{ - 10 ( x-(\frac12, \frac12))^2} $ and
$ \Omega = \mathcal T_{0.5} $ (see \S \ref{s:exms}). On the left, the rotation number is
given by~$ \frac13 $ and we see concentration on an attractor; on the right,  the 
rotation number is (nearly) irrational and, as $ \varepsilon \to 0+ $, $ u $ is
expected to be uniformly distributed \cite{MaasLNL}. Morse--Smale
assumptions do not hold, at least not on scales relevant to numerical calculations and trajectories
are uniformly distributed in the trapezium.
In the contour plots of $ | u ( x) | $ black corresponds to $ 0 $. }
\label{f:trapez}
\end{figure}
The evolution equation 
\eqref{eq:PDE} is equivalent to 
\begin{equation}
\label{eq:PsDE}
( \partial^2_t + P ) w = f \cos \lambda t , \ \ \ w|_{t=0} = \partial_t w|_{t=0} = 0 , \ \ f \in C^\infty_{\rm{c}} ( \Omega; \mathbb R ) ,  \ \ \ 
u = \Delta_\Omega^{-1} w . \end{equation}
This equation is easily solved using the functional calculus of $ P $:
\begin{equation}
  \label{eq:sol}
\begin{gathered}
w ( t )  = \Re \big(e^{i\lambda t}\mathbf W_{t,\lambda}(P)f\big)\qquad\text{where}\\
\mathbf W_{t,\lambda}(z)  = 
 \int_0^t  {\sin \big(s\sqrt z\big)\over \sqrt z}e^{-i\lambda s}\,ds=\sum_{\pm} {1-e^{-it(\lambda\pm\sqrt z)}\over 2\sqrt z(\sqrt z\pm \lambda)}. 
\end{gathered}
\end{equation}
Using the Fourier transform of the Heaviside function (see~\eqref{e:x-i0-FT}),
we see that for any $\varphi\in \CIc(\mathbb R)$ we have
$$
\int_{\mathbb R}{1-e^{-it\zeta}\over \zeta}\varphi(\zeta)\,d\zeta
=i\int_{0}^t\widehat\varphi(\eta)\,d\eta\ \xrightarrow{t\to\infty}\ 
i\int_0^\infty\widehat\varphi(\eta)\,d\eta=\int_{\mathbb R}(\zeta-i0)^{-1}\varphi(\zeta)\,d\zeta
$$
and thus for any $\lambda\in (0,1)$ we have the distributional limit
\begin{equation}
  \label{eq:heur}
\mathbf W_{t,\lambda}(z)\to (z-\lambda^2+i0)^{-1}\quad\text{as }t\to\infty\quad\text{in }\mathcal D'_z((0,\infty)).
\end{equation}
This suggests that, as long as we only look at the spectrum of $P$ near~$\lambda^2$ (the rest of the spectrum
contributing the term $r(t)$ in Theorem~\ref{t:1}), if the spectral measure
of $P$ applied to $f$ is smooth in the spectral parameter $z$, then $\mathbf W_{t,\lambda}(P)f\to (P-\lambda^2+i0)^{-1}f$
as $t\to\infty$. By Stone's Formula, it suffices to establish the
{\em limiting absorption principle} for the operator $ P $ near $ \lambda^2 $
and that is the content of
\begin{theo}
\label{t:2}
Suppose that $\mathcal J\subset (0,1)$ is an open interval such that each $ \lambda \in \mathcal J $
satisfies the Morse--Smale conditions of Definition~\ref{d:2}.
Then for each $f\in \CIc(\Omega)$ and $\lambda\in\mathcal J$ the limits
\begin{equation}
\label{eq:Pla}
( P - \lambda^2 \pm i 0 )^{-1}f
=\lim_{\varepsilon\to 0+} (P-(\lambda\mp i\varepsilon)^2)^{-1}f \quad\text{in } \mathcal D' ( \Omega ) 
\end{equation} 
exist and the spectrum of $ P $ is purely absolutely continuous in $ \mathcal J^2:=\{\lambda^2\mid\lambda\in\mathcal J\} $:
\begin{equation}
\label{eq:Pla1} 
 \sigma ( P ) \cap \mathcal J^2  = \sigma_{\rm{ac}} ( P ) \cap\mathcal   J^2 . 
 \end{equation}
 Moreover, 
 \begin{equation}
 \label{eq:Pla2} 
 ( P - \lambda^2 \pm i 0 )^{-1} f \in I^{1} ( \overline\Omega,\Lambda^\pm (\lambda)  ) \ \subset\ H^{-\frac 32-}(\Omega) ,
 \end{equation}
where $ \Lambda^\pm (\lambda)  $ are given in \eqref{eq:defLa} and the definition of
the conormal spaces $ I^1 ( \overline\Omega,\Lambda^\pm ( \lambda ) ) $ is reviewed in \S \ref{s:con}.
 \end{theo}
\Remarks 1.   The proof provides a more precise statement based on a reduction to the boundary -- see
\S \ref{s:liap}. We also have smooth dependence on $ \lambda $ which plays a crucial role
in proving Theorem~\ref{t:1} as in \cite[\S 5]{DZ-FLOP}~-- see \S \ref{s:asy}. This precise information
is important in obtaining the $H^{\frac12-}$ remainder in~\eqref{eq:prof}. The singular profile in 
Theorem \ref{t:1} satisfies
\[    u^+ =  \Delta_\Omega^{-1} (P-\lambda^2+i0)^{-1}f , \]
which agrees with the heuristic argument following \eqref{eq:heur}.

\noindent 2. As noted in \cite{RalstonR}, $ \sigma ( P ) = [0,1]$ but as emphasized there and in numerous physics papers the structure of the spectrum of
$ P $ is far from clear. Here we only characterize the spectrum~\eqref{eq:Pla1} under the Morse--Smale 
assumptions of Definition~\ref{d:2}.

Rather  than working with $ P $, we consider
the closely related stationary {\em Poincar\'e problem}
\[
( \partial_{x_2}^2 - \omega^2 \Delta ) u_\omega = f \in C_{\rm c}^\infty , \quad
u_\omega|_{\partial \Omega } = 0 , \quad
\Re \omega \in ( 0, 1 ) , \quad
\Im \omega > 0 .
\]
Then $ u_{\lambda+i\varepsilon} \in C^\infty ( \overline \Omega ) $ has a limit in $ \mathcal D' (
\Omega ) $ which satisfies $ u_{\lambda + i 0 } \in I^{-1} ( \overline\Omega,\Lambda^- ( \lambda )) $,
and we have $ ( P - \lambda^2 - i 0 )^{-1} f = \Delta u_{\lambda+i0 } $. 
 
\subsection{Related mathematical work}

Motivated by the study of internal waves results similar to Theorems \ref{t:1} and \ref{t:2} were
obtained for self-adjoint 0th order pseudodifferential operators on 2D tori with dynamical
conditions in Definitions \ref{d:1} and \ref{d:2}  replaced by demanding that a naturally defined flow is Morse--Smale.
That was done first by Colin de Verdi\`ere--Saint Raymond \cite{CdV-LSR,CdV-floppy}, with different
proofs provided by Dyatlov--Zworski \cite{DZ-FLOP}. The question of modes of viscosity limits 
in such models (addressing physics questions formulated for domains with boundary -- 
see  Rieutord--Valdettaro \cite{Rie} and references given there) were investigated by 
Galkowski--Zworski \cite{gaz1} and Wang \cite{jw2}. Finer questions related to spectral theory 
were also answered in \cite{jw1}. Unlike in the situation considered in this paper, embedded eigenvalues are possible in the case of 0th order pseudodifferential operators \cite{ztao}. 

The dynamical system \eqref{e:b-def} 
was recently studied by Nogueira--Troubetzkoy \cite{Nogueira-Troubetzkoy} and 
by Lenci et al~\cite{Lenci-chess}. We refer to those papers for additional references and dynamical results. 

\subsection{Organization of the paper} 
In \S \ref{s:chess} we provide a self-contained analysis of the dynamical system given
by the diffeomorphism \eqref{e:b-def}. We emphasize properties needed in the analysis
of the operator \eqref{eq:Rals}: properties of pushforwards by $ \ell^\pm $ and 
existence of suitable escape/Lyapounov functions. \S\ref{s:microlocal-prelim} is devoted to 
a review of microlocal analysis used in this paper and in particular to definitions and properties
of conormal/Lagrangian spaces used in the formulations of Theorems \ref{t:1} and \ref{t:2}.
In \S \ref{s:blp} we describe reduction to the boundary using 1+1 Feynman propagators 
which arise naturally in the limiting absorption principles. Despite the presence of 
characteristic points, the restricted operator enjoys good microlocal properties -- see
Proposition \ref{p:summa}. Microlocal analysis of that operator is given in \S \ref{s:abo}
with the key estimate \eqref{eq:TG} motivated by Lasota--Yorke inequalities and
radial estimates. The self-contained \S \ref{s:microp} analyses wave front set properties
of distributions invariant under the diffeomorphisms \eqref{e:b-def}. These results are 
combined in \S \ref{s:liap} to give the proof of the limiting absorption principle of
Theorem \ref{t:2}. Finally, in \S \ref{s:asy} we follow the strategy of \cite{DZ-FLOP} 
to describe long time properties of solutions to \eqref{eq:PDE} -- see Theorem \ref{t:1}.

\section{Geometry and dynamics}
\label{s:chess}

In this section we assume that $\Omega\subset\mathbb R^2$
is an open bounded simply connected set with $C^\infty$
boundary $\partial\Omega$ and review the basic properties of the
involutions $\gamma^\pm$ and the chess billiard $b$ defined in~\eqref{e:gamma-pm-def},
\eqref{e:b-def}. We orient $\partial\Omega$ in the positive direction as the boundary
of~$\Omega$ (that is, counterclockwise if $\Omega$ is convex).

\subsection{Basic properties}
\label{s:chess-basic}

Fix $\lambda\in (0,1)$ such that $\Omega$ is $\lambda$-simple in the sense of
Definition~\ref{d:1}. We first show that the involutions
$\gamma^\pm$ defined in~\eqref{e:gamma-pm-def}
are smooth. Away from the critical set~$\{x^\pm_{\min},x^\pm_{\max}\}$ this is immediate.
Next, we write
\begin{equation}
  \label{e:factorize}
\begin{aligned}
\ell^\pm(x)&=\ell^\pm(x^\pm_{\min})+\theta^\pm_{\min}(x)^2\quad\text{for
$x$ near $x^\pm_{\min}$},\\
\ell^\pm(x)&=\ell^\pm(x^\pm_{\max})-\theta^\pm_{\max}(x)^2\quad\text{for
$x$ near $x^\pm_{\max}$}
\end{aligned}
\end{equation}
where $\theta^\pm_{\min},\theta^\pm_{\max}$ are local coordinate
functions on $\partial\Omega$ which map $x^\pm_{\min},x^\pm_{\max}$ to~0.
Then for $x$ near $x^\pm_{\min}$ the point $\gamma^\pm(x)$ satisfies the equation
$$
\theta^\pm_{\min}(\gamma^\pm(x))=-\theta^\pm_{\min}(x)
$$
and similarly near $x^\pm_{\max}$.
This shows the smoothness of $ \partial \Omega \ni x \mapsto \gamma^\pm ( x ) $ near the critical points.

Next, note that since $\gamma^\pm$ are involutions, $b$ is conjugate to its inverse:
\begin{equation}
  \label{e:b-inversor}
b^{-1}=\gamma^\pm\circ b\circ\gamma^\pm.
\end{equation}
Therefore $\Sigma_\lambda^+=\gamma^\pm(\Sigma_\lambda^-)$
where $\Sigma_\lambda^\pm$ are defined in~\eqref{e:Sigma-pm-def}. Since $x^\pm_{\min},x^\pm_{\max}$ are fixed points of $\gamma^\pm$,
the Morse--Smale conditions (see Definition~\ref{d:2}) implies that there are no characteristic periodic points:
\begin{equation}
  \label{e:no-characteristic}
\Sigma_\lambda\cap \mathscr C_\lambda=\emptyset\quad\text{where}\quad
\mathscr C_\lambda:=\mathscr C_\lambda^+\sqcup\mathscr C_\lambda^-,
\quad
\mathscr C_\lambda^\pm:=\{x^\pm_{\min}(\lambda),x^\pm_{\max}(\lambda)\}.
\end{equation}

\subsubsection{Useful identities}

For $x\in\partial\Omega$ and $\lambda\in (0,1)$ we define the signs
\begin{equation}
  \label{e:omega-pm-def}
\nu^\pm(x,\lambda):=\sgn\partial_\theta \ell^\pm(x,\lambda)
\end{equation}
where $\partial_\theta$ is the derivative along $\partial\Omega$
with respect to a positively oriented parametrization.
\begin{lemm}
  \label{l:identitor-1}
Assume that $\Omega$ is $\lambda$-simple. Then
for all $x\in\partial\Omega$
\begin{gather}
  \label{e:sign-identity}
\sgn \ell^\mp(\gamma^\pm(x)-x)=\pm\nu^\pm(x),\\
  \label{e:flip-sign}
\nu^\pm(\gamma^\pm(x))=-\nu^\pm(x),\\
  \label{e:l-pm-differential}
\partial_\lambda \ell^\pm(x,\lambda)={2\lambda^2-1\over 2\lambda(1-\lambda^2)}\ell^\pm(x,\lambda)+{1\over 2\lambda(1-\lambda^2)}\ell^\mp(x,\lambda).
\end{gather}
\end{lemm}
\begin{proof}
To see~\eqref{e:sign-identity}, we first notice that it holds when $x\in\{x^\pm_{\min},x^\pm_{\max}\}$, as then both sides are equal to~0. Now, assume that $\gamma^\pm(x)\neq x$ (that is,
$x\notin\{x^\pm_{\min},x^\pm_{\max}\}$). 
Denote by $v(x)\in\mathbb R^2$ the velocity vector of the parametrization at the point $x\in\partial\Omega$.
The vector $\gamma^\pm(x)-x\in\mathbb R^2$ is pointing into $\Omega$ at the point $x\in\partial\Omega$.
Since we use a positively oriented parametrization, the vectors $v(x),\gamma^\pm(x)-x$
form a positively oriented basis. We now note that $\ell^+,\ell^-$
form a positively oriented basis of the dual space to $\mathbb R^2$,
and hence
$$
\det\begin{pmatrix}
\ell^+(v(x)) & \ell^+(\gamma^\pm(x)-x) \\
\ell^-(v(x)) & \ell^-(\gamma^\pm(x)-x)
\end{pmatrix}>0. 
$$
Since $ \partial_\theta \mathcal \ell^\pm (x ) = \ell^\pm ( v ( x ) )$, this  gives~\eqref{e:sign-identity}.
The identity~\eqref{e:flip-sign} follows from~\eqref{e:sign-identity},
and~\eqref{e:l-pm-differential} is verified by a direct computation.
\end{proof}
The next statement is used in the proof of Lemma~\ref{l:con}.
\begin{lemm}
  \label{l:identitor-2}
Assume that $\Omega$ is $\lambda$-simple. Then for all $y\in\partial\Omega$
and $x\in\Omega$
\begin{equation}
\label{e:signs-inside}
\nu^+(y)\ell^-(x-y)>0\quad\text{or}\quad
\nu^-(y)\ell^+(x-y)<0\quad\text{(or both).}
\end{equation}
\end{lemm}
\begin{proof}
Let $\Gamma^\pm_\lambda(y)$ be the sets
defined in~\eqref{e:Gamma-def} and recall that they are open line segments
with endpoints $y,\gamma^\pm(y)$. Then by~\eqref{e:sign-identity},
$$
\Omega\cap R^\pm(y)=\emptyset\quad\text{where}\quad
R^\pm(y):=\{x\in\mathbb R^2\mid \ell^\pm(x-y)=0,\ \pm\nu^\pm(y)\ell^\mp(x-y)\leq 0\}.
$$
The sets $R^\pm(y)$ are closed rays starting at $y$ when $\nu^\pm(y)\neq 0$
and lines passing through~$y$ when $\nu^\pm(y)=0$. Any continuous curve starting
at the set of~$x \in \mathbb R^2$ satisfying~\eqref{e:signs-inside} and ending in the complement of this
set has to intersect $R^+(y)\cup R^-(y)$, as can be seen (in the case $\nu^\pm(y)\neq 0$)
by applying the Intermediate Value Theorem to the pullback to that
curve of the function
$x\mapsto\max(\nu^+(y)\ell^-(x-y),-\nu^-(y)\ell^+(x-y))$.
Thus, since
$\Omega$ is connected and contains at least one point $x$ satisfying~\eqref{e:signs-inside}
(for instance, take any point in $\Gamma^\pm_\lambda(y)$), all points
$x\in\Omega$ satisfy~\eqref{e:signs-inside}.
\end{proof}

\subsubsection{Properties of pushforwards}

We next show basic properties of pushforwards of smooth
functions by the maps $\partial\Omega\ni x\mapsto \ell^\pm(x,\lambda)$,
which are used in the proof of Lemma~\ref{l:singl-2}.
Fix $\lambda\in (0,1)$ such that $\Omega$ is $\lambda$-simple and define
\begin{equation}
  \label{e:ell-pm-min-max}
\ell^\pm_{\min}:=\ell^\pm(x^\pm_{\min}),\quad
\ell^\pm_{\max}:=\ell^\pm(x^\pm_{\max}),
\end{equation}
so that $\ell^\pm$ maps $\partial\Omega$ onto the interval
$[\ell^\pm_{\min},\ell^\pm_{\max}]$. We again fix a positively oriented coordinate $\theta$
on $\partial\Omega$.
\begin{lemm}
  \label{l:pushforwards}
1. Assume that $f\in C^\infty(\partial\Omega)$ and define $\Pi^\pm_\lambda f\in\mathcal E'(\mathbb R)$ by the formula
\begin{equation}
  \label{e:Pi-pm-def}
\int_{\mathbb R} \Pi_\lambda^\pm f(s)\varphi(s)\,ds
=\int_{\partial\Omega} f(x)\varphi(\ell^\pm(x))\,d\theta(x)\quad\text{for all}\quad
\varphi\in C^\infty(\mathbb R).
\end{equation}
Then $\supp\Pi_\lambda^\pm f\subset [\ell^\pm_{\min},\ell^\pm_{\max}]$ and
\begin{equation}
  \label{e:pushf-sqrt}
\sqrt{(s-\ell^\pm_{\min})(\ell^\pm_{\max}-s)}\,\Pi_\lambda^\pm f(s)\ \in\ C^\infty([\ell^\pm_{\min},\ell^\pm_{\max}]).
\end{equation}

\noindent 2. Assume that $f\in C^\infty(\partial\Omega)$ and define the functions
$\Upsilon^\pm_\lambda f$ on $(\ell^\pm_{\min},\ell^\pm_{\max})$ by
$$
\Upsilon^\pm_\lambda f(s):=\sum_{x\in\partial\Omega,\, \ell^\pm(x)=s}
f(x),\quad s\in (\ell^\pm_{\min},\ell^\pm_{\max}).
$$
Then $\Upsilon^\pm_\lambda f\in C^\infty([\ell^\pm_{\min},\ell^\pm_{\max}])$.
\end{lemm}
\begin{proof}
1. The support property follows immediately from the definition:
if $\supp\varphi\cap [\ell^\pm_{\min},\ell^\pm_{\max}]=\emptyset$,
then $\varphi\circ\ell^\pm=0$ on $\partial\Omega$ and thus
$\int (\Pi_\lambda^\pm f)\varphi=0$.

To show~\eqref{e:pushf-sqrt}, we compute
\begin{equation}
  \label{e:pushf-1}
\Pi_\lambda^\pm f(s)=\sum_{x\in\partial\Omega,\, \ell^\pm(x)=s}{f(x)\over |\partial_\theta\ell^\pm(x)|}\quad\text{for all}\quad
s\in (\ell^\pm_{\min},\ell^\pm_{\max}).
\end{equation}
It follows that $\Pi_\lambda^\pm f$ is smooth on the open interval
$(\ell^\pm_{\min},\ell^\pm_{\max})$. Next, note that~\eqref{e:pushf-sqrt} does not
depend on the choice of the parametrization~$\theta$ since changing
the parametrization amounts to multiplying~$f$ by a smooth positive function.
Thus we can use the local coordinate $\theta=\theta^\pm_{\min}$ near $x^\pm_{\min}$
introduced in~\eqref{e:factorize}. With this choice
we have $\ell^\pm(x)=\ell^\pm_{\min}+\theta^2$ and the formula~\eqref{e:pushf-1}
gives for $s$ near $\ell^\pm_{\min}$
$$
\Pi_\lambda^\pm f(s)={f(\sqrt{s-\ell^\pm_{\min}})+f(-\sqrt{s-\ell^\pm_{\min}})\over
2\sqrt{s-\ell^\pm_{\min}}}
$$
where we view $f$ as a function of~$\theta$. It follows that
$\sqrt{s-\ell^\pm_{\min}}\,\Pi_\lambda^\pm f(s)$ is smooth at
the left endpoint of the interval $(\ell^\pm_{\min},\ell^\pm_{\max})$.
Similar analysis shows that
$\sqrt{\ell^\pm_{\max}-s}\,\Pi_\lambda^\pm f(s)$ is smooth at the right endpoint
of this interval.

\noindent 2. This is proved similarly to part~1, where we
no longer have $|\partial_\theta \ell^\pm(x)|$ in the denominator in~\eqref{e:pushf-1}.\end{proof}

\subsubsection{Dynamics of the chess billiard}

We now give a description of the dynamics of the orientation preserving diffeomorphism~$b=\gamma^+\circ\gamma^-$ in the presense of periodic points.
\begin{lemm}
\label{l:global-dynamics}
Assume that $\Sigma_\lambda\neq\emptyset$ (see~\eqref{e:periodic-points}). Then:
\begin{enumerate}
\item all periodic points of~$b$ have the same minimal period;
\item for each $x\in\partial\Omega$, the trajectory
$b^k(x)$ converges to $\Sigma_\lambda$ as $k\to\pm \infty$;
\item if $\partial_x b^n\neq 1$ on $\Sigma_\lambda$
where~$n$ denotes the minimal period, then the set $\Sigma_\lambda$ is finite.
\end{enumerate}
\end{lemm}
\begin{proof}
See for example~\cite[\S1.1]{Melo-Strien} or~\cite{Walsh-circle}
for the proof of the first two claims. The last claim follows
from the fact that $\Sigma_\lambda$ is the set of solutions
to $b^n(x)=x$ and thus $ \partial_x b^n ( x ) \neq 0 $ implies that it consists of isolated points.
\end{proof}
We finally discuss the rotation number of~$b$.
Fix a positively oriented parametrization on~$\partial\Omega$ which identifies
it with the circle~$\mathbb S^1=\mathbb R/\mathbb Z$ and
denote by~$\pi:\mathbb R\to \partial\Omega$ the covering map.
Consider a lift of $b(\bullet,\lambda)$ to $\mathbb R$, that is, an orientation preserving
diffeomorphism $\mathbf b(\bullet, \lambda):\mathbb R\to \mathbb R$ such that
$$
\pi(\mathbf b(\theta,\lambda))=b(\pi(\theta),\lambda)\quad\text{for all}\quad \theta\in\mathbb R.
$$
Denote by $\mathbf b^k(\bullet,\lambda)$ the $k$-th iterate of $\mathbf b(\bullet,\lambda)$.
Define the \emph{rotation number} of $b(\bullet,\lambda)$ as
\begin{equation}
  \label{e:rotation-number}
\mathbf r(\lambda):=\lim_{k\to\infty} {\mathbf b^k(\theta,\lambda)-\theta\over k}\mod \mathbb Z\ \in\ \mathbb R/\mathbb Z.
\end{equation}
The limit exists and is independent of the choice of $\theta\in\mathbb R$
and of the lift $\mathbf b$. We refer to~\cite{Walsh-circle} for a proof of this fact
as well that of the following
\begin{lemm}
  \label{l:rotation-periodic}
The rotation number $\mathbf r(\lambda)$ is rational if and only if $\Sigma_\lambda\neq\emptyset$. In this case $\mathbf r(\lambda)={q\over n}\bmod \mathbb Z$ where
$n>0$ is the minimal period of the periodic points and $q\in\mathbb Z$ is coprime with~$n$.
\end{lemm}

We remark that $b(\bullet,\lambda)$ cannot have fixed points:
indeed, if $x\in\partial\Omega$ and $b(x)=x$, then
$\gamma^+(x)=\gamma^-(x)$ which is impossible.
We then fix the lift $\mathbf b$ for which
\begin{equation}
  \label{e:fixed-lift}
0<\mathbf b(0,\lambda)<1.
\end{equation}
With this choice 
we have $0<\mathbf b^k(0,\lambda)<k$ for all $k\geq 0$
and thus~\eqref{e:rotation-number}
defines the rotation number $\mathbf r(\lambda)$ which satisfies
$0<\mathbf r(\lambda)<1$.

\subsection{Dependence on \texorpdfstring{$\lambda$}{\unichar{"03BB}}}

We now discuss the dependence of the dynamics of the chess billiard map
$b(\bullet,\lambda)$ on~$\lambda$. We first give a stability result:
\begin{lemm}
  \label{l:perturb-MS}
The set of $\lambda\in (0,1)$ satisfying the Morse--Smale conditions
(see Definition~\ref{d:2}) is open.
Moreover, the maps $\gamma^\pm(x,\lambda)$ and $b(x,\lambda)$, as well as
the sets $\Sigma_\lambda$, depend smoothly on~$\lambda$
as long as $\lambda$ satisfies the Morse--Smale conditions.
\end{lemm}
\begin{proof}
Assume that $\lambda_0$ satisfies the Morse--Smale conditions.
We need to show that all~$\lambda$ close enough to~$\lambda_0$
satisfy this condition as well.
From~\eqref{eq:dual2} we see that the functions $\ell^\pm(x,\lambda)$
depend smoothly on $x\in\partial\Omega,\lambda\in (0,1)$. Therefore,
$\Omega$ is $\lambda$-simple for $\lambda$ close to $\lambda_0$.
Moreover, $\gamma^\pm(x,\lambda)$ and $b(x,\lambda)$ depend smoothly
on $\lambda$ as long as $\Omega$ is $\lambda$-simple.

Next, let $m>0$ be the number of points in~$\Sigma_{\lambda_0}$
and let $n$ be their minimal period under $b(\bullet,\lambda_0)$. 
Since $\partial_x b^n(x,\lambda_0)\neq 1$ on~$\Sigma_{\lambda_0}$,
by the Implicit Function Theorem
for $\lambda$ close to $\lambda_0$ the equation $b^n(x,\lambda)=x$
has exactly $m$ solutions, which depend smoothly on $\lambda$.
It follows that $\lambda$ satisfies the Morse--Smale conditions.
\end{proof}
Lemmas~\ref{l:rotation-periodic} and~\ref{l:perturb-MS} imply in particular that when $\lambda_0$ satisfies the Morse--Smale conditions,
the rotation number $\mathbf r$ is constant in a neighborhood of~$\lambda_0$:
indeed, the rotation number is determined by the combinatorial structure
of the map $b$ on each closed orbit (if the rotation
number is equal to $q/n$ with $q,n$ coprime, then each
closed orbit has period $n$ and the action of $b$ on this orbit is the shift by $q$ points), which varies continuously with~$\lambda$.
A partial converse to this fact is given by the second part of the following
\begin{lemm}
Assume that $\mathcal J\subset (0,1)$ is an open interval such that
$\Omega$ is $\lambda$-simple for each $\lambda\in \mathcal J$. Then:
\begin{enumerate}
\item $\mathbf r(\lambda)$ is a continuous increasing function of $\lambda\in \mathcal J$;
\item if $\mathbf r$ is constant on~$\mathcal J$, then
this constant is rational and
the Morse--Smale conditions hold for Lebesgue almost every $\lambda\in \mathcal J$.
\end{enumerate}
\end{lemm}
\begin{proof}
1. Fix a positively oriented coordinate~$\theta$
on $\partial\Omega$. Using~\eqref{e:gamma-pm-def}, \eqref{e:l-pm-differential} 
we compute 
$$
\partial_\lambda\gamma^\pm(x,\lambda)={\partial_\lambda\ell^\pm(x-\gamma^\pm(x,\lambda),\lambda)\over\partial_\theta \ell^\pm(\gamma^\pm(x,\lambda),\lambda)} = 
\frac{ \ell^\mp ( x - \gamma^\pm(x,\lambda) , \lambda)  }
{2 \lambda ( 1 - \lambda^2 )\partial_\theta\ell^\pm ( \gamma^\pm ( x , \lambda ),\lambda)  }. 
$$
By~\eqref{e:sign-identity} and \eqref{e:flip-sign} we have
$$
\partial_\lambda\gamma^+>0,\quad
\partial_\lambda\gamma^-<0.
$$
We then compute
$$
\partial_\lambda b(x,\lambda)=\partial_\lambda\gamma^+(\gamma^-(x,\lambda),\lambda)
+\partial_\theta\gamma^+(\gamma^-(x,\lambda),\lambda)\partial_\lambda\gamma^-(x,\lambda).
$$
Since $ \partial \Omega \ni x \mapsto \gamma^+ ( x, \lambda ) \in \partial \Omega  $ is orientation reversing this gives
\begin{equation}
  \label{e:b-diff}
\partial_\lambda b(x,\lambda)>0\quad\text{for all}\quad
x\in \partial\Omega,\
\lambda\in \mathcal J.
\end{equation}
Fix the lift $\mathbf b(\theta,\lambda)$ satisfying~\eqref{e:fixed-lift}.
Then \eqref{e:b-diff} gives $\partial_\lambda\mathbf b(\theta,\lambda)>0$. This implies that
for each two points $\lambda_1<\lambda_2$ in~$\mathcal J$ and every $k\geq 1$
$$
\mathbf b^k(\theta,\lambda_1)<\mathbf b^k(\theta,\lambda_2).
$$
Recalling the definition~\eqref{e:rotation-number}
of $\mathbf r(\lambda)$, we see that $\mathbf r(\lambda_1)\leq \mathbf r(\lambda_2)$,
that is $\mathbf r(\lambda)$ is an increasing function of~$\lambda\in \mathcal J$.

\noindent 2. We now show that $\mathbf r(\lambda)$ is a continuous function
of $\lambda\in \mathcal J$. Fix arbitrary $\lambda_0\in \mathcal J$ and~$\varepsilon>0$;
since $\mathbf r$ is an increasing function
it suffices to show that there exists $\delta>0$ such that
$$
\mathbf r(\lambda_0+\delta)<\mathbf r(\lambda_0)+\varepsilon,\quad
\mathbf r(\lambda_0-\delta)>\mathbf r(\lambda_0)-\varepsilon.
$$
We show the first statement, with the second one proved similarly.
Choose a rational number ${q\over n}\in (\mathbf r(\lambda_0),\mathbf r(\lambda_0)+\varepsilon)$
where $n\in\mathbb N$ and $q\in\mathbb Z$ are coprime. Since
$\mathbf r(\lambda_0)<{q\over n}$, the definition~\eqref{e:rotation-number} implies that there exists
$k_0> 0$ such that
$$
{\mathbf b^{k_0n}(0,\lambda_0)\over nk_0}<{q\over n},
$$
that is $\mathbf b^{k_0n}(0,\lambda_0)<k_0q$.
Since $\mathbf b^{k_0n}(0,\lambda)$ is continuous in~$\lambda$, 
we can choose $\delta>0$ small enough so that
\begin{equation}
\label{eq:ind1} 
\mathbf b^{k_0n}(0,\lambda_0+\delta)<k_0q.
\end{equation}
By induction on $ j $ we see that 
\begin{equation}
\label{eq:indj}
\mathbf b^{jk_0n}(0,\lambda_0+\delta)<jk_0q\quad\text{for all}\quad j\geq 1.
\end{equation}
Here the inductive step is proved as follows: using $ \mathbf b^{p} ( r ) = \mathbf  b^{p} ( 0 ) + r  $, $ r \in \mathbb Z $ 
($  \mathbf b^{p } $ is a lift of the orientation preserving diffeomorphism 
$  b^{p } $; we dropped $ \lambda_0 + \delta $ in the notation),
\[ \mathbf b^{ (j+1) k_0 n } ( 0  ) =
 \mathbf b^{ k_0 n } ( \mathbf  b^{ jk_0n } ( 0  )  ) 
<  \mathbf b^{k_0 n } ( j k_0 q ) = \mathbf b^{k_0 n } ( 0 ) + j k_0 q  < (j+1) k_0 q . \]
Now, the definition~\eqref{e:rotation-number} and \eqref{eq:indj} imply that
$\mathbf r(\lambda_0+\delta)\leq {q\over n}<\mathbf r(\lambda_0)+\varepsilon$
as needed.

\noindent 3. Assume now that $\mathbf r$ is constant on~$\mathcal J$. We first show
that this constant is a rational number. Assume the contrary and take
arbitrary $\lambda_0\in \mathcal J$. 
By~\eqref{e:b-diff} (shrinking $\mathcal J$ slightly if necessary) we may assume that
$\partial_\lambda b(x,\lambda)\geq c>0$ for some $c>0$ and all $x\in\partial\Omega$,
$\lambda\in \mathcal J$. Then $\partial_\lambda b^n(x,\lambda)\geq c$ for all $n\geq 0$ as well.
Fix $\varepsilon>0$ such that $\lambda_1:=\lambda_0+\varepsilon/c$ lies in~$\mathcal J$.
Then $\mathbf b^n(\theta,\lambda_1)\geq \mathbf b^n(\theta,\lambda_0)+\varepsilon$
for all $\theta\in\mathbb R$.

Fix arbitrary $x_0=\pi(\theta_0)\in \partial\Omega$, $\theta_0\in\mathbb R$.
Since $\mathbf r(\lambda_0)$ is irrational and $b(\bullet,\lambda_0)$ is smooth,
by Denjoy's Theorem~\cite[\S I.2]{Melo-Strien} every orbit of $b(\bullet,\lambda_0)$
is dense, in particular the orbit $\{b^n(x_0,\lambda_0)\}_{n\geq 1}$ intersects
the $\varepsilon$-sized interval on $\partial\Omega$ whose right endpoint
is~$x_0$. That is, there exist $n\in\mathbb N$, $m\in\mathbb Z$ such that
$$
\theta_0+m-\varepsilon\leq \mathbf b^n(\theta_0,\lambda_0)\leq \theta_0+m.
$$
It follows that
$$
\mathbf b^n(\theta_0,\lambda_0)\leq \theta_0+m\leq \mathbf b^n(\theta_0,\lambda_0)+\varepsilon
\leq \mathbf b^n(\theta_0,\lambda_1).
$$
By the Intermediate Value Theorem, there exists $\lambda\in [\lambda_0,\lambda_1]\subset\mathcal  J$
such that $\mathbf b^n(\theta_0,\lambda)=\theta_0+m$. Then
$x_0=\pi(\theta_0)$ is a periodic orbit of $b(\bullet,\lambda)$,
which contradicts our assumption that $\mathbf r(\lambda)$ is irrational
for all $r\in \mathcal J$.

\noindent 4. Under the assumption of Step~3, we now have
$\mathbf r(\lambda)={q\over n}\bmod\mathbb Z$ for some coprime
$q\in\mathbb Z$, $n\in\mathbb N$ and all $\lambda\in \mathcal J$. By Lemma~\ref{l:rotation-periodic},
for each $\lambda\in \mathcal J$ the set of periodic points $\Sigma_\lambda$ is nonempty
and each such point has minimal period~$n$. Define
$$
\Sigma_{\mathcal J}:=\{(x,\lambda)\mid \lambda\in \mathcal J,\ x\in\Sigma_\lambda\}
=\{(x,\lambda)\in \partial\Omega\times \mathcal J\mid b^n(x,\lambda)=x\}.
$$
From~\eqref{e:b-diff} we see that $\partial_\lambda b^n(x,\lambda)>0$
for all $x\in\partial\Omega$, $\lambda\in \mathcal J$. Shrinking $\mathcal J$ if needed, we may assume
that~$\Sigma_{\mathcal J}$ is a one-dimensional submanifold of $\mathcal J\times\partial\Omega$
projecting diffeomorphically onto the~$x$ variable, that is
$\Sigma_{\mathcal J}=\{(x,\psi(x))\mid x\in U\}$
for some open set~$U\subset\partial\Omega$ and smooth function $\psi:U\to \mathcal J$,
$ \partial_x \psi ( x ) = (1 - \partial_x b^n ( x, \lambda ) )/ \partial_\lambda b^n ( x, \lambda ) $,
$ \lambda = \psi ( x ) $.
Then $\lambda\in \mathcal J$ satisfies the Morse--Smale conditions if and only if $\lambda$
is a regular value of~$\psi$, which by the Morse--Sard theorem happens for Lebesgue
almost every $\lambda\in \mathcal J$.
\end{proof}

\subsection{Escape functions}
\label{s:escape-functions}

We now construct an adapted parametrization of $\partial\Omega$
and a family of escape functions, which are used  \S \ref{s:abo} below.
Throughout this section we assume that $\lambda\in (0,1)$
satisfies the Morse--Smale conditions of Definition~\ref{d:2}.
Recall the sets $\Sigma_\lambda^\pm$ of attractive/repulsive
periodic points of the map $b(\bullet,\lambda)$ defined in~\eqref{e:Sigma-pm-def}.
Let $n\in\mathbb N$ be the minimal period of the corresponding
trajectories of $b$.

We first construct a parametrization of $\partial\Omega$
with a bound on $\partial_x b|_{\Sigma_\lambda^\pm}$
rather than on the derivative of the $n$-th iterate $\partial_x b^n|_{\Sigma_\lambda^\pm}$:
\begin{lemm}
  \label{l:adapted-theta}
Let $ \Sigma_\lambda^\pm $ be given by~\eqref{e:Sigma-pm-def}. There exists a positively oriented
coordinate $\theta:\partial\Omega\to\mathbb S^1$ such that,
taking derivatives on $\partial\Omega$ with respect to~$\theta$,
\begin{equation}
  \label{e:adapted-theta}
\begin{aligned}
\partial_x b(x,\lambda)<1&\quad\text{for all}\quad x\in\Sigma^+_\lambda,\\
\partial_x b(x,\lambda)>1&\quad\text{for all}\quad x\in\Sigma^-_\lambda.
\end{aligned}
\end{equation}
\end{lemm}
\begin{proof}
Fix any Riemannian metric $g_0$ on $\partial\Omega$ and consider the 
metric $g$ on~$\partial\Omega$ given by
$$
|v|_{g(x)}:=\sum_{j=0}^{n-1} |\partial_x b^j(x)v|_{g_0(b^j(x))}\quad\text{for all}\quad
(x,v)\in T(\partial\Omega).
$$
We have for all $(x,v)\in T(\partial\Omega)$
$$
|\partial_x b(x)v|_{g(b(x))}-|v|_{g(x)}=|\partial_x b^n(x)v|_{g_0(b^n(x))}-|v|_{g_0(x)}.
$$
Thus by~\eqref{e:Sigma-pm-def} we have for $v\neq 0$
$$
\begin{aligned}
|\partial_x b(x)v|_{g(b(x))} < |v|_{g(x)}&\quad\text{when}\quad x\in \Sigma^+_\lambda;\\
|\partial_x b(x)v|_{g(b(x))} > |v|_{g(x)}&\quad\text{when}\quad x\in \Sigma^-_\lambda.\end{aligned}
$$
It remains to choose the coordinate $\theta$ so that
$|\partial_\theta|_{g}$ is constant.
\end{proof}
We next use the global dynamics of~$b(\bullet,\lambda)$ described in Lemma~\ref{l:global-dynamics} to construct an \emph{escape function} in Lemma~\ref{l:escape-function} below. Fix a parametrization on $\partial\Omega$ which satisfies~\eqref{e:adapted-theta}
and denote by
$$
\Sigma_\lambda^\pm(\delta)\ \subset\ \partial\Omega
$$
the open $\delta$-neighborhoods of the sets~$\Sigma_\lambda^\pm$ with respect to this parametrization.
Here $\delta>0$ is a constant small enough
so that the closures $\overline{\Sigma_\lambda^+(\delta)}$ and $\overline{\Sigma_\lambda^-(\delta)}$
do not intersect each other. We also choose $\delta$ small enough so that
\begin{equation}
  \label{e:neighborhoods-shrink}
b\big(\overline{\Sigma_\lambda^+(\delta)}\big)\subset \Sigma_\lambda^+(\delta),\quad
b^{-1}\big(\overline{\Sigma_\lambda^-(\delta)}\big)\subset \Sigma_\lambda^-(\delta);
\end{equation}
this is possible by~\eqref{e:adapted-theta} and since $\Sigma_\lambda^\pm$ are
$b$-invariant.
\begin{lemm}
\label{l:escape-function}
Let $\alpha_+<\alpha_-$ be two real numbers. Then there exists
a function $g\in C^\infty(\partial\Omega;\mathbb R)$ such that:
\begin{enumerate}
\item $g(b(x))\leq g(x)$ for all $x\in \partial\Omega$;
\item $g(b(x))<g(x)$ for all $x\in\partial\Omega\setminus (\Sigma_\lambda^+(\delta)\cup\Sigma_\lambda^-(\delta))$;
\item $g(x)\geq \alpha_+$ for all $x\in\partial\Omega$;
\item $g(x)\geq \alpha_-$ for all $x\in\partial\Omega\setminus\Sigma_\lambda^+(\delta)$;
\item $g=\alpha_+$ on some neighbourhood of~$\Sigma_\lambda^+$;
\item for $ M\gg 1 $, 
 $ M ( g ( b( x ) )  - g (  x) ) + g ( x )  \leq   \alpha_+  $ for all 
$ x \in \partial \Omega \setminus \Sigma^-_\lambda  ( \delta ) $.
\end{enumerate}
See Figure~\ref{f:escape}.
\end{lemm}
\begin{figure}
\includegraphics{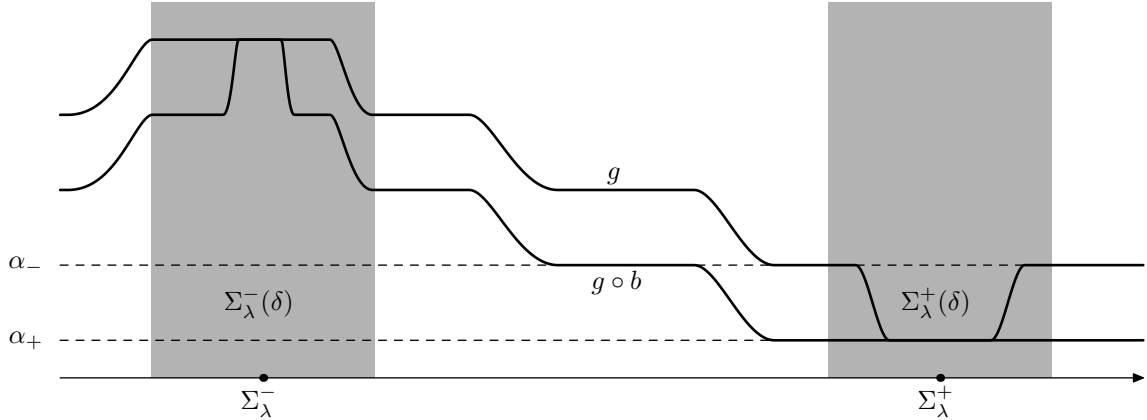}
\caption{The escape function $g$ constructed in Lemma~\ref{l:escape-function}
and the function $g\circ b$, where for simplicity we replace $\Sigma^\pm_\lambda$
by fixed points of the map $b$. The shaded regions correspond to
$\Sigma^\pm_\lambda(\delta)$ and the dashed lines correspond to $\alpha_\pm$.}
\label{f:escape}
\end{figure}
\Remark We note that the same construction works for $ b^{-1} $ with 
the roles of $ \Sigma^\pm_\lambda $ reversed.
 Hence for any real numbers $ \alpha_- < \alpha_+ $
we can find $ g \in C^\infty ( \partial \Omega ; \mathbb R ) $ such that
\begin{enumerate}
\item $g(x) \leq g(b(x))$ for all $x\in \partial\Omega$;
\item $g(x)<g(b(x))$ for all $x\in\partial\Omega\setminus (\Sigma_\lambda^+(\delta)\cup\Sigma_\lambda^-(\delta))$;
\item $g(x)\geq \alpha_-$ for all $x\in\partial\Omega$;
\item $g(x)\geq \alpha_+$ for all $x\in\partial\Omega\setminus\Sigma_\lambda^-(\delta)$;
\item $g=\alpha_-$ on some neighbourhood of~$\Sigma_\lambda^-$; 
\item for $ M\gg 1 $, 
 $ M ( g ( x ) - g ( b ( x ) )) + g ( b(x) )  \leq   \alpha_- $ for all 
$ x \in \partial \Omega \setminus \Sigma^+_\lambda ( \delta ) $.
\end{enumerate}
\begin{proof}
In view of \eqref{e:neighborhoods-shrink} there exists
$ 0 <  \delta_1 < \delta $ such that 
\begin{equation}
\label{eq:defd1}  b ( \Sigma_\lambda^+ ( \delta ) ) \subset \Sigma_\lambda^+ ( 
\delta_1 ) .
\end{equation}

\noindent 1. We first show that there exists $N\geq 0$ such that
\begin{equation}
  \label{e:more-contraction}
b^N\big(\partial\Omega\setminus\Sigma_\lambda^-(\delta)\big)\subset
\Sigma_\lambda^+(\delta_1).
\end{equation}
We argue by contradiction. Assume that~\eqref{e:more-contraction}
does not hold for any~$N$. Then there exist sequences
\begin{equation}
  \label{e:morecon-con}
x_j\in \partial\Omega\setminus\Sigma_\lambda^-(\delta),\quad
m_j\to\infty,\quad
b^{m_j}(x_j)\not\in \Sigma_\lambda^+(\delta_1).
\end{equation}
Passing to subsequences, we may assume that
$x_j\to x_\infty$ for some $x_\infty\in\partial\Omega$.
Since $x_j\not\in\Sigma_\lambda^-(\delta)$, we have $x_\infty\not\in\Sigma_\lambda^-(\delta)$
as well. Then by~\eqref{e:neighborhoods-shrink} the trajectory
$b^k(x_\infty)$, $k\geq 0$, does not intersect $\Sigma_\lambda^-(\delta)$.
On the other hand, by Lemma~\ref{l:global-dynamics}
this trajectory converges to $\Sigma_\lambda=\Sigma_\lambda^-\sqcup\Sigma_\lambda^+$
as $k\to\infty$. Thus this trajectory converges to $\Sigma_\lambda^+$, in particular
$$
\text{there exists}\quad k\geq 0\quad\text{such that}\quad
b^k(x_\infty)\in\Sigma_\lambda^+(\delta_1).
$$
Since $x_j\to x_\infty$, we have $b^k(x_j)\to b^k(x_\infty)$ as $j\to\infty$.
Since $\Sigma_\lambda^+(\delta_1)$ is an open set, there exists $j\geq 0$
such that $m_j\geq k$ and $b^k(x_j)\in \Sigma_\lambda^+(\delta_1)$.
But then by~\eqref{eq:defd1} we have
$b^{m_j}(x_j)\in\Sigma_\lambda^+(\delta_1)$ which contradicts~\eqref{e:morecon-con}.

\noindent 2. Choose $N$ such that~\eqref{e:more-contraction} holds and fix
a cutoff function
$$
\chi_+\in \CIc(\Sigma_\lambda^+(\delta);[0,1]),\quad
\chi_+=1\quad\text{on}\quad \Sigma_\lambda^+(\delta_1).
$$
Define the function $\tilde g\in C^\infty(\partial\Omega;\mathbb R)$ 
as an ergodic average of~$\chi_+$:
$$
\tilde g(x):={1\over N}\sum_{j=0}^{N-1} \chi_+(b^j(x))\quad\text{for all}\quad
x\in\partial\Omega.
$$
It follows from the definition and~\eqref{eq:defd1} that
\begin{equation}
  \label{e:tilde-g-works-1}
\begin{aligned}
0\leq \tilde g(x)\leq 1&\quad\text{for all}\quad x\in\partial\Omega,\\
\tilde g(x)=1&\quad\text{for all}\quad x\in \Sigma_\lambda^+(\delta_1),\\
\tilde g(x)\leq 1-\textstyle{1\over N}&\quad\text{for all}\quad x\in\partial\Omega\setminus\Sigma_\lambda^+(\delta).
\end{aligned}
\end{equation}
Next, we compute
$$
\tilde g(b(x))-\tilde g(x)=\textstyle{1\over N}\big(\chi_+(b^N(x))-\chi_+(x)\big).
$$
It follows that
\begin{equation}
  \label{e:tilde-g-works-2}
\begin{aligned}
\tilde g(b(x))\geq \tilde g(x)&\quad\text{for all}\quad x\in\partial\Omega,\\
\tilde g(b(x))=\tilde g(x)+\textstyle{1\over N}&\quad\text{for all}\quad x\in\partial\Omega\setminus (\Sigma^+_\lambda(\delta)\cup\Sigma^-_\lambda(\delta)).
\end{aligned}
\end{equation}
Indeed, take arbitrary $x\in\partial\Omega$. We have $\chi_+(x)=0$ unless
$x\in\Sigma^+_\lambda(\delta)$. By~\eqref{e:more-contraction},
we have $\chi_+(b^N(x))=1$ unless $x\in\Sigma^-_\lambda(\delta)$.
Recalling that $0\leq\chi_+\leq 1$ and $\Sigma^+_\lambda(\delta)\cap\Sigma^-_\lambda(\delta)=\emptyset$,
we get~\eqref{e:tilde-g-works-2}.

\noindent 3. Now put
\begin{equation}
\label{eq:defg}
g(x):=N\alpha_--(N-1)\alpha_+ - N(\alpha_--\alpha_+) \tilde g(x).
\end{equation}
Using~\eqref{e:tilde-g-works-1} and~\eqref{e:tilde-g-works-2}, we see
that the function~$g$ satisfies the first five properties, 
with the following quantitative versions of parts~(2) and~(5):
\begin{equation}
\label{eq:quant}
\begin{aligned}
g(b(x))-g(x)=\alpha_+-\alpha_-<0&\quad\text{for all}\quad
x\in \partial\Omega\setminus(\Sigma^+_\lambda(\delta)\cup\Sigma^-_\lambda(\delta)),\\
g(x)=\alpha_+&\quad\text{for all}\quad x\in \Sigma^+_\lambda(\delta_1).
\end{aligned}
\end{equation} 
To prove part~(6) 
we first use \eqref{eq:defg} and \eqref{eq:quant} to see that
for all $M\geq N$ and $ x \in \partial \Omega \setminus 
( \Sigma_\lambda^+ ( \delta ) \cup \Sigma_\lambda^- ( \delta )) $, 
\begin{equation}
\label{eq:Mgb}
   M \big( g ( b( x )) - g ( x ) \big) + g ( x ) \leq \alpha_+ . 
\end{equation}
To establish \eqref{eq:Mgb} for $ x \in \Sigma_\lambda^+ ( \delta ) $ we use 
\eqref{eq:defd1} and the fact that $ g |_{ \Sigma_\lambda^+ ( \delta_1 ) } = \alpha_+$ by~\eqref{eq:quant}.
Then, for $ M \geq 1 $ and $ x \in \Sigma^+_\lambda( \delta )  $,  property (1) gives
\[
M \big( g ( b( x )) - g ( x ) \big) + g ( x ) 
\leq g( b ( x ) ) = \alpha_+ ,
\]
which completes the proof of the lemma.
\end{proof}
\Remark
We discuss here the dependence
of the objects in this section on the parameter $\lambda$.
The parametrization~$\theta$ constructed in Lemma~\ref{l:adapted-theta} depends
smoothly on $\lambda$ as follows immediately from its construction
(recalling from the proof of Lemma~\ref{l:perturb-MS} that the period~$n$
is locally constant in~$\lambda$). Next, for each $\lambda_0\in (0,1)$
satisfying the Morse--Smale conditions there exists a neighborhood
$U(\lambda_0)$ such that we can construct a function~$g(x,\lambda)$ for each
$\lambda\in U(\lambda_0)$ satisfying the conclusions of Lemma~\ref{l:escape-function}
in such a way that it is smooth in~$\lambda$. 
Indeed,
the sets $\Sigma^\pm_\lambda$ depend smoothly on~$\lambda$ by Lemma~\ref{l:perturb-MS}, so the cutoff function $\chi_+$ can
be chosen $\lambda$-independent. The function $g(x,\lambda)$
is constructed explicitly using this cutoff, the map
$b(\bullet,\lambda)$, and the number~$N$. The latter can be chosen
$\lambda$-independent as well: if~\eqref{e:more-contraction} holds
for some $\lambda$, then it holds with the same $N$ and all nearby~$\lambda$.

\subsection{Domains with corners}
\label{s:domc}

We now discuss the case when the boundary of $\partial\Omega$
has corners. This includes the situation when $\partial\Omega$
is a convex polygon, which is the setting of the experiments.
Our results do not apply to such domains, however they apply to appropriate
`roundings' of these domains described below.

We first define domains with corners. Let $\Omega\subset\mathbb R^2$
be an open set of the form
$$
\Omega=\{x\in\mathbb R^2\mid F_1(x)> 0,\dots,F_k(x)> 0\}
$$
where $F_1,\dots,F_k:\mathbb R^2\to \mathbb R$ are $C^\infty$ functions such that:
\begin{enumerate}
\item the set $\overline\Omega:=\{F_1\geq 0,\dots,F_k\geq 0\}$
is compact and simply connected, and
\item for each $x\in\overline\Omega$, at most 2 of the functions
$F_1,\dots,F_k$ vanish at~$x$.
\end{enumerate}
If only one of the functions $F_1,\dots,F_k$ vanishes at
$x\in\overline\Omega$, then we call
$x$ a \emph{regular point} of the boundary~$\partial\Omega:=\overline\Omega\setminus\Omega$.
If two of the functions $F_1,\dots,F_k$ vanish at $x\in\overline\Omega$,
then we call $x$ a \emph{corner} of $\Omega$. We make the following natural
nondegeneracy assumptions:
\begin{enumerate}
\setcounter{enumi}{2}
\item if $x\in\partial\Omega$ is a regular point and $F_j(x)=0$,
then $dF_j(x)\neq 0$;
\item if $x\in\partial\Omega$ is a corner and $F_j(x)=F_{j'}(x)=0$
where $j\neq j'$, then $dF_j(x),dF_{j'}(x)$ are linearly independent.
\end{enumerate}
We call $\Omega$ a \emph{domain with corners} if it satisfies the assumptions~(1)--(4) above.

Since $\Omega$ is simply connected, the boundary $\partial\Omega$ is a Lipschitz continuous
piecewise smooth curve. We parametrize $\partial\Omega$ in the positively oriented direction
by a Lipschitz continuous map
\begin{equation}
  \label{e:corner-par}
\theta\in\mathbb S^1:=\mathbb R/\mathbb Z\quad\mapsto\quad
\mathbf x(\theta)\in\partial\Omega\ \subset\ \mathbb R^2
\end{equation}
where the corners are given by $\mathbf x(\theta_j)$ for some $\theta_1<\dots<\theta_m$
and the map~\eqref{e:corner-par} is smooth on each interval
$[\theta_j,\theta_{j+1}]$. See Figure~\ref{f:corners}.
\begin{figure}
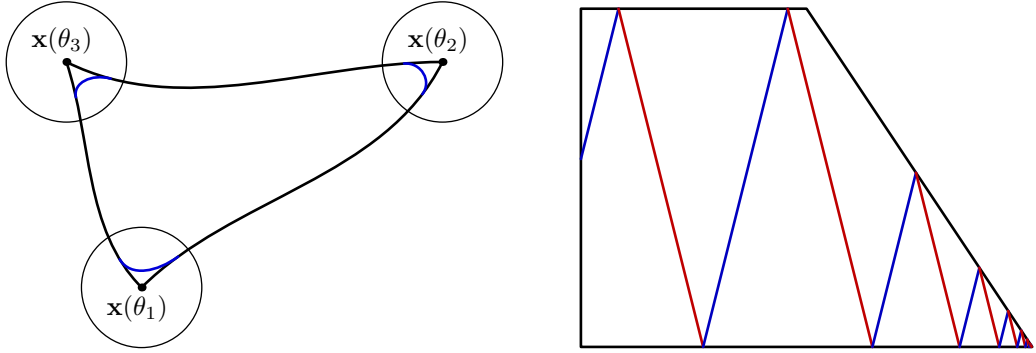

\includegraphics{bflop.3}
\qquad
\includegraphics{bflop.4}
\caption{Left: a domain with corners and its $\varepsilon$-rounding (in blue). The circles
have radius~$\varepsilon$.
Right: a trajectory on a trapezium which converges to a corner.}
\label{f:corners}
\end{figure}

We next extend the concept of $\lambda$-simplicity to domains with corners.
Let $\ell\in C^\infty(\mathbb R^2;\mathbb R)$ and $x=\mathbf x(\theta_j)$ be a corner of~$\Omega$.
Consider the one-sided derivatives $\partial_\theta (\ell\circ \mathbf x)(\theta_j\pm 0)$.
There are three possible cases:
\begin{enumerate}
\item Both derivatives are nonzero and have the same sign~--
then we call $x$ \emph{not a critical point} of~$\ell$;
\item Both derivatives are nonzero and have opposite signs~--
then we call $x$ a \emph{nondegenerate critical point} of~$\ell$.
\item At least one of the derivatives is zero~--
then we call $x$ a \emph{degenerate critical point} of~$\ell$. 
\end{enumerate}
If $x=\mathbf x(\theta)$ is instead a regular point of the boundary,
then we use the standard definition of critical points:
$x$ is a critical point of~$\ell$ if
$\partial_\theta (\ell\circ\mathbf x)(\theta)= 0$, and a critical point 
is nondegenerate if $\partial_\theta^2(\ell\circ\mathbf x)(\theta)\neq 0$.
With the above convention for critical points, we follow Definition~\ref{d:1}:
we say that a domain with corners~$\Omega$ is \emph{$\lambda$-simple}
if each of the functions $\ell^\pm(\bullet,\lambda)$ defined in~\eqref{eq:dual2}
has exactly 2 critical points on $\partial\Omega$, which are both nondegenerate.

If $\Omega$ is~$\lambda$-simple, then the involutions $\gamma^\pm(\bullet,\lambda):\partial\Omega\to\partial\Omega$
from~\eqref{e:gamma-pm-def} are well-defined and Lipschitz continuous.
Thus $b=\gamma^+\circ\gamma^-$ is an orientation preserving bi-Lipschitz
homeomorphism of $\partial\Omega$. We now revise the Morse--Smale conditions of Definition~\ref{d:2} as follows:
\begin{defi}
\label{d:2-corners}
Let $\Omega$ be a domain with corners. We say that $\lambda\in (0,1)$ satisfies
the \emph{Morse--Smale conditions} if:
\begin{enumerate}
\item $\Omega$ is $\lambda$-simple;
\item the set $\Sigma_\lambda$ of periodic points of the map $b(\bullet,\lambda)$
is nonempty;
\item the set $\Sigma_\lambda$ does not contain any corners of~$\Omega$;
\item for each $x\in\Sigma_\lambda$, $\partial_x b^n(x,\lambda)\neq 1$
where $n$ is the minimal period.
\end{enumerate}
\end{defi}
The new condition~(3) in Definition~\ref{d:2-corners} ensures that
$b$ is smooth near the $\gamma^\pm$-invariant set $\Sigma_\lambda$, so condition~(4) makes sense.
Without this condition we could have trajectories of $b$ converging to a corner,
see Figure~\ref{f:corners}.

We finally show that if $\Omega$ is a domain with corners satisfying
the Morse--Smale conditions of Definition~\ref{d:2-corners}
then an appropriate `rounding' of~$\Omega$ satisfies the Morse--Smale conditions
of Definition~\ref{d:2}:
\begin{prop}
  \label{l:rounder}
Let $\Omega$ be a domain with corners and $\lambda\in (0,1)$ satisfy the Morse--Smale conditions for~$\Omega$. Then there exists $\varepsilon>0$ such that for any open simply connected
$\widehat\Omega\subset\mathbb R^2$ with $C^\infty$ boundary and such that:
\begin{itemize}
\item $\widehat\Omega$ is an \emph{$\varepsilon$-rounding} of $\Omega$
in the sense that for each $x\in\mathbb R^2$ which lies distance
$\geq\varepsilon$ from all the corners of $\Omega$, we have
$x\in\Omega\iff x\in\widehat\Omega$; and
\item the domain $\widehat\Omega$ is $\lambda$-simple in the sense of Definition~\ref{d:1},
\end{itemize}
the Morse--Smale conditions is satisfied for $\lambda$ and $\widehat\Omega$.
\end{prop}
\begin{proof}
Fix a parametrization $\mathbf x(\theta)$ of~$\partial\Omega$ as in~\eqref{e:corner-par}.
Take a parametrization
$$
\theta\in \mathbb S^1\quad\mapsto\quad \hat{\mathbf x}(\theta)\in\partial\widehat\Omega
$$
which coincides with $\mathbf x(\theta)$ except $\varepsilon$-close
to the corners:
\begin{equation}
\label{e:rounder-int-0}
\hat{\mathbf x}(\theta)=\mathbf x(\theta)\quad\text{for all}\quad
\theta\notin\bigcup_{j=1}^m I_j(\varepsilon),\quad
I_j(\varepsilon):=[\theta_j-C\varepsilon,\theta_j+C\varepsilon].
\end{equation}
Here $C$ denotes a constant depending on $\Omega$ and the parametrization $\mathbf x(\theta)$,
but not on~$\widehat\Omega$ or~$\varepsilon$, whose precise value might change from place to place in the proof.

Denote by $\gamma^\pm,\hat\gamma^\pm$ the involutions~\eqref{e:gamma-pm-def}
corresponding to $\Omega,\widehat\Omega$, and consider them as homeomorphisms
of $\mathbb S^1$ using the parametrizations~$x,\hat x$. Then by~\eqref{e:rounder-int-0}
\begin{equation}
\label{e:rounder-int-1}
\gamma^\pm(\theta)=\hat\gamma^\pm(\theta)\quad\text{if}\quad
\theta,\gamma^\pm(\theta)\not\in \bigcup_{j=1}^m I_j(\varepsilon).
\end{equation}
Let $b=\gamma^+\circ\gamma^-$, $\hat b=\hat\gamma^+\circ\hat\gamma^-$
be the chess billiard maps of $\Omega,\widehat\Omega$
and $\Sigma_\lambda,\widehat\Sigma_\lambda$ be the corresponding sets
of periodic trajectories.
Choose $\varepsilon>0$ such that
the intervals $I_j(\varepsilon)$ do not intersect~$\Sigma_\lambda$;
this is possible since $\Sigma_\lambda$ does not contain any corners of~$\Omega$.
Since $\Sigma_\lambda$ is invariant under $\gamma^\pm$, we see
from~\eqref{e:rounder-int-1} that $b=\hat b$ in a neighborhood of $\Sigma_\lambda$
and thus $\Sigma_\lambda\subset\widehat\Sigma_\lambda$. That is, the periodic
points for the original domain $\Omega$ are also periodic points for the rounded domain $\widehat\Omega$, with the same period~$n$. It also follows that $\partial_x \hat b^n(x,\lambda)=\partial_x b^n(x,\lambda)\neq 1$ for all $x\in\Sigma_\lambda$.

It remains to show that $\widehat\Sigma_\lambda\subset\Sigma_\lambda$,
that is the rounding does not create any new periodic points for $\hat b$. Note that
all periodic points have the same period $n$,
and it is enough to show that
\begin{equation}
\label{e:rounder-int-2}
\hat b^n(\theta)\neq \theta\quad\text{for all}\quad \theta\in \bigcup_{j=1}^m I_j(\varepsilon).
\end{equation}
From~\eqref{e:rounder-int-1}, the monotonicity of $\gamma^\pm,\hat\gamma^\pm$, and the Lipschitz continuity of~$\gamma^\pm$ we have
$$
|\gamma^\pm(\theta)-\hat\gamma^\pm(\theta)|\leq C\varepsilon\quad\text{for all}\quad
\theta\in\mathbb S^1.
$$
Iterating this and using the Lipschitz continuity of~$\gamma^\pm$ again, we get
$$
|b^n(\theta)-\hat b^n(\theta)|\leq C\varepsilon\quad\text{for all}\quad\theta\in\mathbb S^1.
$$
Since $b^n(\theta_j)\neq\theta_j$ for all $j=1,\dots,m$,
taking $\varepsilon$ small enough we get~\eqref{e:rounder-int-2}, finishing the proof.
\end{proof}

\subsection{Examples of Morse--Smale chess billiards}
\label{s:exms}

Here we present two examples of Morse--Smale chess billiards.

\begin{figure}[hb]
\includegraphics{bflop.8}
\quad
\includegraphics[scale=0.32]{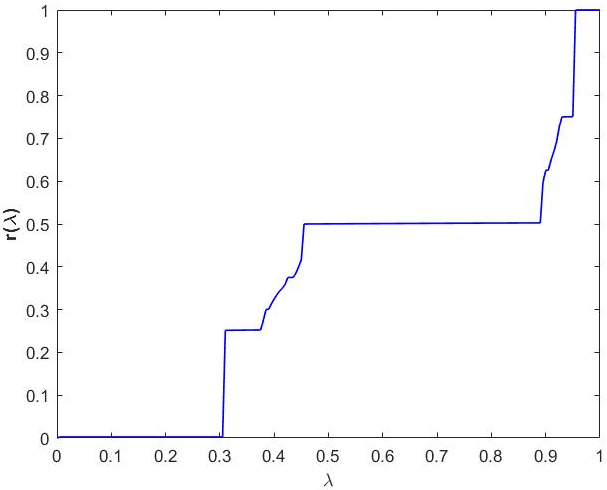}
\caption{Left: the chess billiard for $\Omega_{\alpha}$ in Example 1.
The numbers in bold
mark the values of the coordinate $\theta$ at the vertices.
Right: rotation numbers as functions of $ \lambda $ for $\Omega_{\pi/10}$.}
\label{f:examples-1}
\end{figure}

\noindent
\textbf{Example 1.}
For $\alpha\in (0,\frac{\pi}{2})$, let $\Omega_{\alpha}\subset \mathbb R^2$ be the open square with vertices 
$(0,0)$, $(\cos{\alpha}, \sin{\alpha})$, $\sqrt 2(\cos(\alpha+\frac{\pi}{4}), \sin(\alpha+\frac{\pi}{4}))$, $(\cos(\alpha+\frac{\pi}{2}), \sin(\alpha+\frac{\pi}{2}))$.
(See Figure~\ref{f:examples-1}.)
We parametrize $\partial \Omega_{\alpha}$ by $\theta\in \mathbb R/4\mathbb Z$ 
so that the parametrization $\mathbf x(\theta)$ is affine on each side of the square
and the vertices listed above correspond to $\theta=0,1,2,3$ respectively.
For $\lambda\in (0,1)$, we define
\[ 
\beta\in (0, {\pi}/{2}), \quad \tan{\beta}= {\sqrt{ 1-\lambda^2 }} / \lambda, \quad t_1:=\tan(\beta-\alpha), \quad t_2:=\tan(\beta+\alpha). 
\]
We will show that if
\begin{equation*}
0<\alpha<{\pi}/{8}, \ \ 
 {\pi}/{4}-\alpha <\beta <{\pi}/{4}+\alpha, 
\end{equation*}
or equivalently
\begin{equation} 
\label{lambda_omega}
    0<\alpha<{\pi}/{8}, \ \
    \cos( {\pi}/{4} + \alpha ) < \lambda < \cos( {\pi}/{4} - \alpha ),
\end{equation}
then $\lambda$ and $ \Omega_\alpha $ satisfy the Morse--Smale conditions (Definition 
\ref{d:2-corners}). Moreover, for $\alpha$, $\lambda$ satisfying \eqref{lambda_omega}, we have
(identifying $\theta$ with $\mathbf x(\theta)$)
\begin{equation}
\label{eq:sigma_square} 
\Sigma_{\lambda}=\left\{\tfrac{1-t_1}{t_2-t_1}, \ \ 1+\tfrac{t_1(t_2-1)}{t_2-t_1}, \ \ 2+\tfrac{1-t_1}{t_2-t_1},  \ \ 3+\tfrac{t_1(t_2-1)}{t_2-t_1} \right\}, 
\end{equation}
and the rotation number is $\mathbf r(\lambda)=\tfrac12$.

In fact, assume $ \alpha $, $\lambda$ satisfy \eqref{lambda_omega},
then $\ell^+(\bullet,\lambda)$ has exactly two nondegenerate critical points 
$\mathbf x(0)$, $\mathbf x(2)$ on $\partial\Omega_{\alpha}$; $\ell^-(\bullet,\lambda)$ also has two nondegenerate
critial points $\mathbf x(1), \mathbf x(3)$ on $\partial\Omega$. This shows that 
$\Omega_{\alpha}$ is $\lambda$-simple.

We have the following partial computation of the reflection maps~$\gamma^\pm$
(note that $0<t_1<1<t_2<\infty$ by~\eqref{lambda_omega}):
\begin{equation}
  \label{e:tilts-1}
\begin{aligned}
\gamma^+(\theta)&=\begin{cases}
t_2^{-1}(2-\theta)+2,& 1\leq\theta\leq 2,\\
t_2^{-1}(4-\theta),& 3\leq\theta\leq 4;
\end{cases}\\
\gamma^-(\theta)&=\begin{cases}
t_1(1-\theta)+1,& 0\leq \theta\leq 1,\\
t_1(3-\theta)+3,& 2\leq\theta\leq 3.
\end{cases}
\end{aligned}
\end{equation}
This in particular implies that we have the mapping properties
\begin{equation}
  \label{e:tilt-mapper}
[0,1]\xrightarrow{\gamma^-}
[1,2]\xrightarrow{\gamma^+}
[2,3]\xrightarrow{\gamma^-}
[3,4]\xrightarrow{\gamma^+}
[0,1].
\end{equation}
Recall that $b=\gamma^+\circ\gamma^-$.
We compute
\begin{equation}
  \label{e:tilts-2}
b^2(\theta)=\left({t_1}/{t_2}\right)^2 \theta +{(t_1+t_2)(1-t_1)}/{t_2^2}, \ \ \theta\in [0,1].
\end{equation}
By solving $b^2(\theta_0)=\theta_0$, $\theta_0\in [0,1]$, we find $\theta_0= ({1-t_1})/({t_2-t_1})$ and 
\[
\{\theta_0,\ \gamma^-(\theta_0),\ b(\theta_0),\ \gamma^+(\theta_0)\}\subset \Sigma_{\lambda}.
\]
This shows that the right-hand side of~\eqref{eq:sigma_square} lies in $\Sigma_\lambda$
and that the rotation number is $\mathbf r(\lambda)={1\over 2}$.
On the other hand, suppose $\theta_1\in \mathbb R/4\mathbb Z$ and $\theta_1\in \Sigma_{\lambda}$.
If $\theta_1\in [0,1]$, then $\theta_1=\theta_0$ by~\eqref{e:tilts-2}.
If $\theta_1\in [2,3]$, then $b(\theta_1)\in \Sigma_\lambda\cap [0,1]$
and thus $\theta_1=b(\theta_0)$. If $\theta_1\in [1,2]$, then $\gamma^+(\theta_1)\in \Sigma_\lambda\cap [2,3]$
and thus $\theta_1=\gamma^-(\theta_0)$. Finally, if $\theta_1\in[3,4]$, then $\gamma^+(\theta_1)\in \Sigma_\lambda\cap [0,1]$ and thus $\theta_1=\gamma^+(\theta_0)$. This shows~\eqref{eq:sigma_square}.

Using~\eqref{e:tilts-2} and the fact that $b^2$ commutes with $b$ and is conjugated
by $\gamma^\pm$ to $b^{-2}$ we compute
$$
\partial_\theta b^2(\theta)=\begin{cases}
t_1^2/t_2^2<1,& \theta\in \{\theta_0,b(\theta_0)\};\\
t_2^2/t_1^2>1,& \theta\in \{\gamma^-(\theta_0),\gamma^+(\theta_0)\}.
\end{cases}
$$
We have now checked that under the condition~\eqref{lambda_omega}, $\Omega_{\alpha}$ and $\lambda$ satisfy all conditions in Definition~\ref{d:2-corners}.

\begin{figure}
\includegraphics{bflop.9}
\includegraphics[scale=0.32]{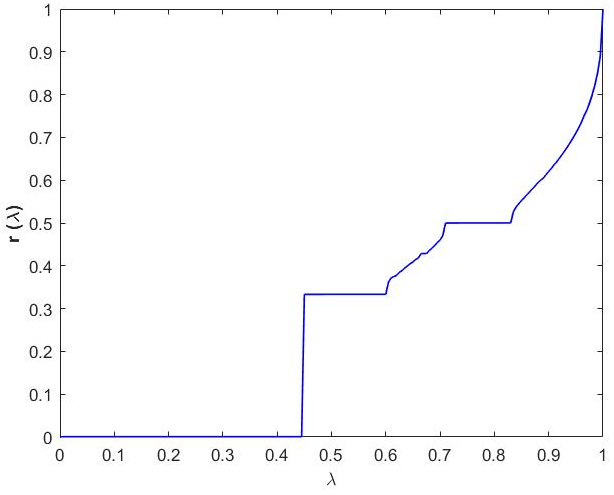}
\caption{Left: the chess billiard for 
$\mathcal T_{d}$ in Example 2.
Right: rotation numbers as functions of $ \lambda $ for 
$\mathcal T_{1/2}$.}
\label{f:examples-2}
\end{figure}

\noindent
\textbf{Example 2.}
Let $\mathcal T_d\subset \mathbb R^2$ be the open trapezium with vertices $(0,0)$, $(1+d,0)$, $(1,1)$, $(0,1)$, $d>0$. (See Figure~\ref{f:examples-2}.)
We parametrize $\partial \mathcal T_{d}$ by $\theta\in \mathbb R/4\mathbb Z$ 
so that the parametrization $\mathbf x(\theta)$ is affine on each side of the trapezium
and the vertices listed above correspond to $\theta=0,1,2,3$ respectively.

For $\lambda\in (0,1)$, we put $c=\lambda/\sqrt{1-\lambda^2}$. We assume that
\begin{equation}
\label{11att} 
    \max ( 1,d ) < c < d+1. 
\end{equation}
Under the condition \eqref{11att} we know $\ell^+(\bullet, \lambda)$ has exactly two 
nondegenerate critial points $\mathbf x(0)$, $\mathbf x(2)$; 
$\ell^-(\bullet, \lambda)$ also has two nondegenerate critical points $\mathbf x(1)$,
$\mathbf x(3)$. Hence $\mathcal T_d$ is $\lambda$-simple. 

We have the following partial computation of the reflection maps~$\gamma^\pm$:
\begin{equation}
  \label{e:trapz-1}
\begin{aligned}
\gamma^+(\theta)&=\begin{cases}
2+(c-d)(2-\theta),& 1\leq\theta\leq 2,\\
{c\over 1+d}(4-\theta),& 3\leq\theta\leq 4;
\end{cases}\\
\gamma^-(\theta)&=\begin{cases}
{1+d\over c+d}(1-\theta)+1,& 0\leq \theta\leq 1,\\
{1\over c}(3-\theta)+3,& 2\leq\theta\leq 3.
\end{cases}
\end{aligned}
\end{equation}
This in particular implies that we again have the mapping properties~\eqref{e:tilt-mapper}.
From here we compute
\[ b^2(\theta)=\frac{c-d}{c+d}\theta+\frac{ 2c(c-1) }{(1+d)(c+d)}, \quad \theta\in [0,1]. \]
The fixed point of this map is
\[ \theta_0=\frac{c(c-1)}{d(1+d)}. \]
Arguing as in Example~1, we see that $\mathcal T_d$, $\lambda$ satisfy the conditions
of Definition~\ref{d:2-corners}, with
$$
\Sigma_\lambda=\{\textstyle{c(c-1)\over d(1+d)},\ 2-{c-1\over d},\
3-c+{c(c-1)\over d},\
4-{c-1\over d}\}.
$$

\section{Microlocal preliminaries}
\label{s:microlocal-prelim}

In this section we present some general results needed in the proof.
Most of the microlocal analysis in this paper takes place on the one dimensional 
boundary $ \partial \Omega $; we review the basic notions in~\S\ref{s:micro}.
In~\S\ref{s:con} we review definitions and basic
properties of conormal distributions (needed in dimensions one and two).
These are used to prove and formulate Theorem \ref{t:1}:  
the singularities of $ ( P - \lambda^2 \mp i 0 )^{-1} f $  using conormal distributions. In our approach, 
this structure of $ ( P - \lambda^2 \mp i 0 )^{-1} f $ is essential for describing the long time
evolution profile in Theorem \ref{t:2}. Finally, \S\S\ref{s:conv-log}--\ref{s:micro-inv} contain technical
results needed in~\S\ref{s:blp}.

\subsection{Microlocal analysis on \texorpdfstring{$ \partial \Omega $}{\unichar{"2202}\unichar{"03A9}}}
\label{s:micro}

We first briefly discuss pseudodifferential operators on the circle $\mathbb S^1=\mathbb R/\mathbb Z$,
referring to~\cite[\S18.1]{Hormander3} for a detailed introduction to the theory of pseudodifferential operators.
Pseudodifferential operators on~$\mathbb S^1$
are given by quantizations of 1-periodic symbols. More precisely, if $0\leq \delta<{1\over 2}$
and $m\in\mathbb R$, then we say
that $a\in C^\infty(\mathbb R^2)$ lies in $S^m_\delta(T^*\mathbb S^1)$ if
(denoting $\langle\xi\rangle:=\sqrt{1+|\xi|^2}$)
\begin{equation}
\label{eq:defad}
a ( x+ 1 , \xi ) = a( x , \xi ) , \ \  
| \partial_x^\alpha \partial_\xi^\beta a ( x, \xi ) | \leq C_{\alpha\beta} 
\langle \xi \rangle^{ m + \delta  \alpha - ( 1 - \delta ) \beta  } .
\end{equation}
For brevity we just write $S^m_\delta:=S^m_\delta(T^*\mathbb S^1)$.
Each $a\in S^m_\delta$ is quantized by the operator $\Op(a):C^\infty(\mathbb S^1)\to C^\infty(\mathbb S^1)$,
$\mathcal D'(\mathbb S^1)\to\mathcal D'(\mathbb S^1)$
defined by
\begin{equation}
  \label{e:Op-a-def}
  \Op ( a ) u ( x ) = 
\frac{1}{2 \pi } \int_{\mathbb R^2} e^{ i ( x - y ) \xi  }a (x  , \xi )  u ( y ) \,dy d \xi ,
\end{equation}
where $ u\in C^\infty(\mathbb R) $ is 1-periodic and the integral is understood in the sense of 
oscillatory integrals \cite[\S 7.8]{Hormander1}. 
We introduce the following spaces of pseudodifferential operators:
\begin{gather*} 
\Psi_\delta^m := \{ 
\Op ( a ) : a \in S^m_\delta \},  \quad \Psi^{m+}_\delta = \bigcap_{m'>m }  \Psi^{m'}_\delta, \quad
\Psi^m_{\delta+} = \bigcap_{ \delta' > \delta} \Psi^m_{\delta'} , \\
   S^{m+}_\delta = \bigcap_{m' > m }  S^{m'}_\delta, \quad
S^m_{\delta+} = \bigcap_{ \delta' > \delta} S^m_{\delta'} .
\end{gather*}
We remark that $S^m_{\delta+}\subset S^{m+}_\delta$; moreover,
$a\in S^{m+}_{\delta}$ lies in $S^m_{\delta+}$ if and only
if $a(x,\xi)=\mathcal O(\langle\xi\rangle^m)$. 
We henceforth denote $\Psi^m := \Psi^m_0$. The
space $\Psi^{-\infty}:=\bigcap_m \Psi^m$ consists of smoothing operators.

In terms of Fourier series on $\mathbb S^1$, we have
\begin{equation}
\label{eq:Opa}
\begin{gathered} \Op ( a ) u ( x ) =  \sum_{k,n \in \mathbb Z} 
e^{2\pi i n x }  a_{n - k} ( k ) u_k  , \\
a_\ell (k ) :=  \int_{0}^{1} a ( x , 2\pi k ) e^{ - 2\pi i \ell x}\, dx , \ \ \
u_k  :=  \int_{0}^{1} u ( x ) e^{ - 2\pi i kx}\, dx . \end{gathered}
\end{equation}
This shows that $ \Op ( a ) $ does not determine $ a $ uniquely.
This representation also shows boundedness on Sobolev spaces 
$ \Op ( a ) : H^s ( \mathbb S^1 ) \to 
H^{s-m} ( \mathbb S^1 ) $, $ s \in \mathbb R $,
$a\in S^m_\delta$.
Indeed, smoothness of $ a $ in $ x $ shows that $ a_\ell ( k ) = 
\mathcal O ( \langle \ell \rangle^{-\infty } \langle k \rangle^{m} ) $ 
and the bound on the norm follows from the Schur criterion
\cite[(A.5.3)]{DZ-Book}.
Despite the fact that $A := \Op(a) $ does not determine
$ a $ uniquely, it does determine its {\em essential support}, which is the right hand side in
the definition of the wave front set of a pseudodifferential operator:
\[  \WF ( A ) := \complement \{ ( x, \xi )  : \xi \neq 0, \, 
\exists \, \rho>0:  a ( y , \eta) = \mathcal O ( \langle \eta \rangle^{-\infty} ) \text{ when } 
|x-y| < \rho,\
\tfrac{\eta}{\xi}>0 \}, \]
see \cite[\S E.2]{DZ-Book}. We refer to that section and \cite[\S 18.1]{Hormander3}
for a discussion of wave front sets. We also recall a definition of the wave front 
set of a distribution,
\begin{equation*}
\begin{gathered}
 \WF ( u ) :=  \bigcap_{ A u \in C^\infty, A \in \Psi^0 } {\rm{Char}} ( A) , \\ 
{\rm{Char}}  ( A ) := \complement\{ ( x, \xi ) : \xi \neq 0 , \, \exists \, \rho, c>0 : |a ( y , \eta) | > c , \ \
|x-y| < \rho , \ \tfrac{\eta}{\xi}>0 ,\ |\eta| > 1/\rho \} 
\end{gathered}
\end{equation*}

The symbol calculus on~$\mathbb S^1$ translates directly from the symbol calculus 
of pseudodifferential operators on $ \mathbb R $. We record in particular
the composition formula \cite[Theorem 18.1.8]{Hormander3}:
 for $ b_1 \in S^{m_1}_\delta $, $ b_2 \in S^{m_2}_\delta $, 
\begin{equation}
\label{eq:compo}
\begin{gathered} 
\Op ( b_1 ) \Op ( b_2 ) = \Op ( b ) , \ \ \ b \in S^{m_1 + m_2 }_\delta , \\
b ( x, \xi ) = 
\exp ( -i   \partial_y \partial_\eta  )
\left[ b_1 ( x, \eta ) b_2 ( y, \xi ) \right] |_{ ( y , \eta )=( x, \xi ) }, \\
b( x, \xi ) = \sum_{ 0\leq k<N} \frac{(-i)^{k}}{ k! } \partial_\xi^k b_1 ( x, \xi ) 
\partial_x^k b_2 ( x, \xi ) + b_N ( x, \xi ) , \ \ b_N \in S^{m-N(1-2 \delta) }_\delta . 
  \end{gathered} 
\end{equation}
where expanding the exponential gives an asymptotic expansion of~$ b$.

We record here a norm bound for pseudodifferential operators at high frequency:
\begin{lemm}
  \label{l:psi0-norm}
Assume that $a\in S^0_\delta$, $r\in S^{-1+}$, and $\sup|a|\leq R$. Then for all~$N$, $\nu>0$, and $u\in L^2(\mathbb S^1)$
we have
\begin{equation}
  \label{e:psi0-norm}
\|\Op(a+r)u\|_{L^2}\leq (R+\nu)\|u\|_{L^2}+C\|u\|_{H^{-N}}
\end{equation}
where the constant $C$ depends on $R$, $\nu$, $N$, and some seminorms of~$a$ and~$r$
but not on~$u$.
\end{lemm}
\begin{proof}
By~\cite[Lemma~4.6]{Grigis-Sjostrand} we can write
$$
(R+\nu)^2I=\Op(a+r)^*\Op(a+r)+\Op(b)^*\Op(b)+\Op(q)
$$
for some $b\in S^0_{\delta}$ and $q\in S^{-\infty}$. The bound~\eqref{e:psi0-norm} follows.
\end{proof}

Although $ a $ in \eqref{eq:Opa} is not unique, the  principal {\em symbol} of $ \Op ( a ) $
defined as 
\begin{equation}
\label{eq:Opas}
\sigma ( \Op ( a ) ) = [a ] \in S^{m}_\delta / S^{m-1+2\delta}_\delta 
\end{equation}
is, and we have a short exact sequence 
$ 0 \to \Psi_\delta^{m-1+2 \delta}  \to \Psi_\delta^{m} 
\xrightarrow{ \sigma } S^{m}_\delta / S^{m-1+2\delta}_\delta \to 0 $. Somewhat informally, 
we write $ \sigma(\Op ( a )) = b $ for any $ b $ satisfying  $ a - b \in S^{m-1+2 \delta}_\delta $. 

In our analysis, 
we also consider families $ \varepsilon \mapsto a_\varepsilon $, $ \varepsilon \geq 0 $, 
such that $ a_\varepsilon \in S^{-\infty} $ for $ \varepsilon > 0 $ and $ a_0 \in S^m_\delta $.
In that case, for $ A_\varepsilon = \Op ( a_\varepsilon ) $, 
\begin{equation}
\label{eq:Opasu}  \sigma ( A_\varepsilon ) = [ b_\varepsilon ] , \ \ \ b_\varepsilon - 
a_\varepsilon \in S^{m-1+2 \delta }_\delta \ \text{ uniformly for $ \varepsilon \geq 0 $.} 
\end{equation}
Again, we drop $ [ \bullet ] $ when writing $ \sigma ( A) $ for a specific operator.

We will crucially use mild exponential weights which result in 
pseudodifferential operators of varying order -- see \cite{Unter}, and in 
a related context, \cite{FRS}.
\begin{lemm}
\label{l:expwe1}
Suppose that (in the sense of~\eqref{eq:defad}) $ m_j \in S^0 $, $m_0$ is real-valued, and 
\begin{equation} 
\label{eq:formG} G ( x, \xi ) := m_0 ( x, \xi ) \log \langle \xi \rangle + m_1 ( x, \xi )  , \ \ \
m_0 ( x, t \xi ) = m_0 ( x, \xi) , \ \ t, |\xi| \geq 1 . \end{equation}
Then 
\begin{equation}
\label{eq:Gexp}  
e^G \in S^{M}_{0+} , \quad e^{-G} \in S^{-m}_{0+} , \ \ \ 
  M := \max_{|\xi|=1 } m_0 ( x, \xi ) , \ \ \  m := \min _{|\xi|=1 } m_0 ( x, \xi ), 
  \end{equation} 
and there exists $ r_G \in S^{-1+}  $ such that
\begin{equation}
\begin{gathered} 
\label{eq:OpbG}     \Op ( e^ G ) \Op ( e^{-G } ( 1 + r_G )  ) - I , \ \Op ( e^{-G} ( 1 + r_G)  )  \Op ( e^ G ) - I \in 
\Psi^{-\infty }  . 
\end{gathered}
\end{equation}
Also, if 
 $ G_j ( x, \xi ) $ are given by \eqref{eq:formG} with $ m_0 $ and $ m_1 $ replaced by 
$ m_{0j}$, $ m_{1j}$, respectively, then
for $a_j\in S^0$, $ r_j \in S^{-1+} $, $ j = 1,2 $, there exists $ r_3 \in S^{-1+ } $ such that
\begin{equation}
\label{eq:wecomp}
\Op ( e^{ G_1 } ( a_1 + r_1 ) ) \Op ( e^{ G_2} ( a_2 +r_2 ) )  = \Op ( e^{ G_1 + G_2 } ( a_1a_2 + r_3 ) ).
\end{equation}
\end{lemm}
\begin{proof}
Since $ \log \langle \xi \rangle = \mathcal O_\varepsilon  ( \langle \xi \rangle^\varepsilon )$ for
all $ \varepsilon > 0 $, \eqref{eq:Gexp} follows from \eqref{eq:defad}. In fact, we have the stronger bound
\begin{equation}
  \label{e:bounder}
|\partial^\alpha_x \partial^\beta_\xi (e^{\pm G(x,\xi)})|\leq C_{\alpha\beta\varepsilon}
e^{\pm G(x,\xi)}\langle\xi\rangle^{\varepsilon-|\beta|},\quad \varepsilon>0.
\end{equation}
This gives~\eqref{eq:wecomp}. Indeed, the remainder in the expansion~\eqref{eq:compo} is in $S^{M_1+M_2-N+}$
and the $k$-th term is in $e^{G_1+G_2}S^{-k+}$ by~\eqref{e:bounder};
it suffices to take $N\geq M+M_2+1$.

To obtain \eqref{eq:OpbG} we note that~\eqref{eq:wecomp} gives
$ \Op ( e^{\pm G} ) \Op ( e^{\mp G } ) = I - \Op (r_\pm )  $, $ r_\pm \in S^{-1+} $.
We then have parametrices for the operators $I-\Op(r_\pm)$ \cite[Theorem 18.1.9]{Hormander3},
$ I + \Op ( b_\pm ) $, which give left and right approximate inverses (in the sense
of \eqref{eq:OpbG})  $ (I + \Op ( b_- )  ) \Op ( e^{ - G} ) $, $ \Op ( e^{ - G } )( I + 
\Op ( b_+)  ) $. Those have the required form by~\eqref{eq:wecomp}
(where one of $G_1,G_2$ is equal to~$-G$ and the other one is equal to~0).
\end{proof}
We also record a change of variables formula. Suppose $ f : \mathbb R /  \mathbb Z 
\to  \mathbb R /\mathbb Z  $ is a diffeomorphism with a lift $ \mathbf f : \mathbb
R \to \mathbb R $, $ \mathbf f ( x + 1) = \mathbf f ( x) \pm 1 $
(with the $+$ sign for orientation preserving $f$ and the $-$ sign otherwise). For symbols
1-periodic in $ x $ we can use the standard formula given in 
\cite[Theorem 18.1.17]{Hormander3} and an argument similar to~\eqref{eq:wecomp}.
That gives, for $ G $ given by 
\eqref{eq:formG}, and $ r \in S^{-1+} $, 
\begin{equation}
\label{eq:chava}
\begin{gathered} 
f^* \circ \Op ( e^{ G } ( 1 + r ) ) = \Op ( e^{ G_f } ( 1 + r_f ) ) \circ f^* , \\
G_f ( x, \xi ) := G ( f ( x ) , f' ( x )^{-1} \xi ) , \ \ r_f \in S^{-1+} . 
\end{gathered}
\end{equation}

In~\S\ref{s:restricted-slp} below we will use pseudodifferential operators acting on
1-forms on~$\mathbb S^1$.
Using the canonical 1-form $dx$, $x\in\mathbb S^1=\mathbb R/\mathbb Z$,
we identify 1-forms with functions,
and this gives an identification of the class $\Psi^m_\delta(\mathbb S^1;T^*\mathbb S^1)$
(operators acting on 1-forms) with $\Psi^m_\delta(\mathbb S^1)$ (operators
acting on functions). This defines the principal symbol map, which
we still denote by~$\sigma$.

Fixing a positively oriented coordinate $\theta: \partial \Omega \to\mathbb S^1$, 
we can identify functions / distributions on 
$ \partial \Omega $ with functions / distributions on~$\mathbb S^1$.
The change of variables formula used for \eqref{eq:chava} also shows the
invariance of $ \sigma ( A )$ under changes of variables 
and allows pseudodifferential operators acting on section of bundles 
-- see \cite[Definition~18.1.32]{Hormander3}. 
In particular, we can define the class of pseudodifferential
operators $\Psi^m_\delta(\partial\Omega;T^*\partial\Omega)$ acting
on 1-forms on $\partial\Omega$ and the symbol map
\begin{equation}
\label{eq:Psim}
   \sigma : \Psi^m_\delta ( \partial\Omega ; T^*\partial\Omega ) \to 
   S^m_\delta ( T^* \partial\Omega   )/ S^{m-1+2 \delta}_\delta
   ( T^* \partial\Omega),
\end{equation}
with the class $\Psi^m_\delta$ and the map $\sigma$ independent
of the choice of coordinate on $\partial\Omega$.

\subsection{Conormal distributions}
\label{s:con}

We now review conormal distributions
associated to hypersurfaces, referring the reader to~\cite[\S18.2]{Hormander3}
for details. Although we consider the case of manifolds with boundaries, 
the hypersurfaces are assumed to be transversal to the boundaries and 
conormal distributions are defined as restrictions of conormal distributions in the
no-boundary case.

Let $M$ be a compact $m$-dimensional manifold with boundary and
$\Sigma\subset M$ be a compact hypersurface transversal 
to the boundary (that is,  $\Sigma$ is a compact codimension 1 submanifold of $M$
with boundary $\partial\Sigma=\Sigma\cap\partial M$ and $T_x\Sigma\neq T_x\partial M$
for all $x\in \partial\Sigma$). We should emphasize that in our case, 
the hypersurfaces $ \Sigma $ take a particularly simple form: we either have $ M = \overline \Omega $
and $ \Sigma $ given by straight lines transversal to $ \partial \Omega $ (see Theorems~\ref{t:1} and~\ref{t:2}) or $ M = \partial \Omega\simeq \mathbb S^1 $ and $ \Sigma $ is given by points (see Propositions~\ref{l:boundv} and~\ref{p:boundvder}).

The conormal bundle to $ \Sigma $ is given by 
$ N^*\Sigma:=\{(x,\xi)\in T^* M\colon x\in \Sigma,\ \xi|_{T_x \Sigma }=0\} $,
which is a Lagrangian submanifold of $T^*M$ and a one-dimensional vector bundle
over~$\Sigma$.
For $k\in\mathbb R$, define the symbol class $S^k(N^*\Sigma)$ consisting of
functions $a\in C^\infty(N^*\Sigma)$ satisfying the derivative bounds
\begin{equation}
\label{eq:symbk}
|\partial_x^\alpha \partial_\theta ^\beta a(x,\theta)|\leq C_{\alpha\beta} \langle\theta \rangle^{k-|\beta|}
\end{equation}
where we use local coordinates $ ( x , \theta ) \in \mathbb R^{m-1} \times \mathbb R  \simeq N^* \Sigma $. Here
$x$ is a coordinate on~$\Sigma$ and $\theta$ is a linear coordinate on the fibers
of $N^*\Sigma$; $\langle\theta\rangle:=\sqrt{1+|\theta|^2}$.
The estimates \eqref{eq:symbk} are supposed to be valid uniformly up to the boundary of $ \Sigma$. In other
words we can consider $ a $ as a restriction of a symbol defined on an extension of~$ \Sigma $. 

Denote by $I^s(M,N^*\Sigma) \subset \overline{\mathcal D' } ( M^\circ) $ the space of extendible distributions on the interior~$M^\circ$ (see \cite[\S B.2]{Hormander3}) 
which are conormal to $\Sigma$ of order~$s\in\mathbb R$ smoothly up to the boundary of~$M$. To describe the class $I^s$ we first consider two
model cases:
\begin{itemize}
\item if $M=\mathbb R^m$, we write points in $\mathbb R^m$ as $(x_1,x')\in \mathbb R\times\mathbb R^{m-1}$, and $\Sigma=\{x_1=0\}$, then a compactly
supported distribution $u\in\mathcal E'(\mathbb R^m)$ lies in $I^s(M,N^*\Sigma)$
if and only its Fourier transform in the $x_1$ variable, $\check u(\xi_1,x')$,
lies in $S^{{m\over 4}-{1\over 2}+s}(N^*\Sigma)$ where $N^*\Sigma=\{(0,x',\xi_1,0)\mid
x'\in\mathbb R^{m-1},\ \xi_1\in\mathbb R\}$.
\item if $M=\mathbb R_{x_1}\times [0,\infty)_{x_2}\times\mathbb R^{m-2}_{x''}$
and $\Sigma=\{x_1=0,\ x_2\geq 0\}$, then a distribution $u\in\mathcal D'(M^\circ)$ with bounded support lies in $I^s(M,N^*\Sigma)$ if and only if
$u=\tilde u|_{M^\circ}$ for some $\tilde u\in\mathcal E'(\mathbb R^m)$
which lies in $I^s(\mathbb R^m, N^*\widetilde\Sigma)$, with $\widetilde\Sigma:=\{x_1=0\}\subset \mathbb R^m$. Alternatively, $\check u(\xi_1,x')$ lies in $S^{{m\over 4}-{1\over 2}+s}(N^*\Sigma)$
where the derivative bounds are uniform up to the boundary.
\end{itemize}
In those model cases, elements of $ I^s ( M,N^* \Sigma ) $ are given by the oscillatory integrals
(where we use the prefactor from~\cite[Theorem~18.2.9]{Hormander3})
\begin{equation}
\label{eq:uosc}    u ( x ) = (2 \pi )^{-\frac m 4 - \frac12} \int_{\mathbb R} e^{ i x_1 \xi_1 } a ( x' , \xi_1 ) \,d\xi_1 , \ \ 
a \in S^{ \frac m 4 - \frac12 + s } ( N^* \{ x_1 = 0 \} ) .
\end{equation}
We note that in both of the above cases the distribution $u$ is in $C^\infty(M)$ (up to the boundary in the second case) outside of any neighborhood of $\Sigma$.

For the case of general compact manifold $M$ and hypersurface $\Sigma$ transversal to the boundary of~$M$,
we say that $u\in I^s(M,N^*\Sigma)$ if (see~\cite[Theorem~18.2.8]{Hormander3})
\begin{enumerate}
\item $u$ is in $C^\infty(M)$ (up to the boundary) outside of any neighborhood of $\Sigma$; and
\item the localizations of $u$ to the model cases using coordinates lie in $I^s$ as defined
above.
\end{enumerate}
Note that the wavefront set of any $u\in I^s(M,N^*\Sigma)$,
considered as a distribution on the interior $M^\circ$,
is contained in~$N^*\Sigma$.

In addition we define the space
$$
I^{s+}(M,N^*\Sigma):=\bigcap_{s'>s} I^{s'}(M,N^*\Sigma).
$$
Such spaces are characterized in terms of the Sobolev spaces (where for simplicity
assume that $M$ has no boundary, since this is the only case used in this paper),
$H^{s-}(M):=\bigcap_{s'<s}H^{s'}(M)$,
as follows:
\begin{equation}
\label{eq:iter}
u\in I^{s+}(M,N^*\Sigma) \ \Longleftrightarrow \  \left\{ \begin{array}{l} \text{ For any  
vector fields on~$M$ tangent to $\Sigma$,} \  X_1,\dots,X_\ell, \\
\quad X_1\dots X_\ell \, u\in H^{-{m\over 4}-s-}(M) , \end{array} \right. 
\end{equation}
see \cite[Definition~18.2.6 and Theorem~18.2.8]{Hormander3}. 

Assume now that the conormal bundle $N^*\Sigma$ is oriented;
for $(x,\xi)\in N^*\Sigma\setminus 0$ we say that
$\xi>0$ if $\xi$ is positively oriented and $\xi<0$ if $\xi$ is negatively oriented.
This gives the splitting
\begin{equation}
\label{eq:N2pm}
N^*\Sigma\setminus 0=N^*_+\Sigma\sqcup N^*_-\Sigma,\quad
N^*_\pm\Sigma:=\{(x,\xi)\in N^*\Sigma\mid \pm\xi>0\}.
\end{equation}
Denote by $I^s(M,N_\pm^*\Sigma)$ the space of distributions
$u\in I^s(M,N^*\Sigma)$ such that $\WF(u)\subset N_\pm^*\Sigma$,
up to the boundary. Since $ \Sigma $ is transversal to the boundary this means that 
an extension of $ u $ satisfies this condition. 
In the model case (and effectively in the cases considered in this paper) $M=\mathbb R^m$, $\Sigma=\{x_1=0\}$ they can be characterized as follows:
$\check u(\xi_1,x')$ lies in $S^{{m\over 4}-{1\over 2}+s}(N^*\Sigma)$
and $\check u(\xi_1,x')=\mathcal O(\langle\xi_1\rangle^{-\infty})$
as $\xi_1\to \mp\infty$.

In the present paper we will often study the case when $M=\partial\Omega$,
identified with $\mathbb S^1$ by a coordinate~$\theta$,
and we are given two finite sets $\Sigma^+,\Sigma^-\subset \partial\Omega$ with 
$\Sigma^+\cap\Sigma^-=\emptyset$. We denote
\begin{equation}
  \label{e:I-s-pm}
I^s(\partial\Omega,N^*_+\Sigma^-\sqcup N^*_-\Sigma^+):=
I^s(\partial\Omega,N^*_+\Sigma^-)+I^s(\partial\Omega,N^*_-\Sigma^+).
\end{equation}
Put $\Sigma:=\Sigma^+\sqcup\Sigma^-$, then $I^s(\partial\Omega,N^*_+\Sigma^-\sqcup N^*_-\Sigma^+)$
consists of the elements of $I^s(\partial\Omega;N^*\Sigma)$ with
wavefront set contained in $N^*_+\Sigma^-\sqcup N^*_-\Sigma^+$.

Assume that $\Sigma$ has an even number of points (which will be the case
in our application) and fix a defining function $\rho\in C^\infty(\partial\Omega;\mathbb R)$
of~$\Sigma$: that is, $\Sigma=\rho^{-1}(0)$ and $d\rho\neq 0$ on~$\Sigma$.
Fix also a pseudodifferential operator $A_\Sigma\in\Psi^0(\partial\Omega)$ such that
$\WF(A_\Sigma)\cap (N^*_+\Sigma^-\sqcup N^*_-\Sigma^+)=\emptyset$
and $A_\Sigma$ is elliptic on $N^*_-\Sigma^-\sqcup N^*_+\Sigma^+$.
Using~\eqref{eq:iter}, we see that $u\in \mathcal D'(\partial\Omega)$ lies in
$I^{s+}(\partial\Omega,N^*_+\Sigma^-\sqcup N^*_-\Sigma^+)$
if and only if the following seminorms are finite:
\begin{equation}
  \label{e:I-s-pm-seminorms}
\|(\rho\partial_\theta)^N u\|_{H^{-\frac14-s-\beta}},\quad
\|A_\Sigma u\|_{H^N}\quad\text{for all}\quad N\in\mathbb N_0,\
\beta>0.
\end{equation}
Choosing different $\rho$ and $A_\Sigma$ leads to an equivalent family of seminorms~\eqref{e:I-s-pm-seminorms}.
In particular, if $\rho,A_\Sigma$ are as above and $\widetilde A_\Sigma\in\Psi^0(\partial\Omega)$
satisfies $\WF(\widetilde A_\Sigma)\cap (N^*_+\Sigma^-\sqcup N^*_-\Sigma^+)=\emptyset$ then $\WF(\widetilde A_\Sigma)$
lies in the union of $\{\rho\neq 0\}$ and the elliptic set of $A_\Sigma$, thus by the elliptic estimate we have
for $N_0\geq N+\frac 14+s+\beta$
\begin{equation}
  \label{e:I-s-pm-sameor}
\|\widetilde A_\Sigma u\|_{H^N}\leq C\big(\|A_\Sigma u\|_{H^N}+ \|(\rho\partial_\theta)^{N_0}u\|_{H^{-\frac14-s-\beta}}
+\|u\|_{H^{-\frac14-s-\beta}}\big).
\end{equation}
Moreover, the operator $\rho\partial_\theta$ is bounded with respect to the seminorms~\eqref{e:I-s-pm-seminorms},
as are pseudodifferential operators in $\Psi^0(\partial\Omega)$~\cite[Theorem~18.2.7]{Hormander3}.

We will also need the notion of conormal distributions depending smoothly on a parameter~-- see \cite[Lemma 4.4]{DZ-FLOP} for a more general Lagrangian version. 
Here we restrict ourselves to the specific conormal distributions appearing in this paper
and define relevant smooth families of conormal distributions in Proposition~\ref{p:boundvder} and Lemma~\ref{l:conor}.

We will not discuss principal symbols of general conormal distributions to avoid introducing half-densities, however
we give here a special case of the way the principal symbol changes under pseudodifferential operators
and under pullbacks:
\begin{lemm}
  \label{l:conormal-symbol}
Assume that $u\in\mathcal E'(\mathbb R)$ lies in $I^s(\mathbb R,\{0\})$, that is $\widehat u\in S^{s-\frac14}(\mathbb R)$.
Then:
\begin{enumerate}
\item If $a(x,\xi)\in S^0(T^*\mathbb R)$ is compactly supported in the $x$ variable and
$\Op(a)$ is defined by~\eqref{e:Op-a-def}, then
\begin{equation}
  \label{e:conormal-symbol-1}
\widehat{\Op(a)u}(\xi)=a(0,\xi)\widehat u(\xi)+S^{s-\frac54}(\mathbb R).
\end{equation}
\item If $f$ is a diffeomorphism of open subsets of $\mathbb R$ such that $f(0)=0$ and the range of $f$ contains $\supp u$, then
\begin{equation}
  \label{e:conormal-symbol-2}
\widehat{f^*u}(\xi)={1\over |f'(0)|}\widehat u\Big({\xi\over f'(0)}\Big)+S^{s-\frac54}(\mathbb R).
\end{equation}
\end{enumerate}
\end{lemm}
\begin{proof}
Since these statements are standard, we only sketch the proofs, referring to~\cite[Theorems~18.2.9 and~18.2.12]{Hormander3} for details.
To see~\eqref{e:conormal-symbol-1} we use the formula
$$
\widehat{\Op(a)u}(\xi)={1\over 2\pi}\int_{\mathbb R^2} e^{ix(\eta-\xi)}a(x,\eta)\widehat u(\eta)\,d\eta dx.
$$
Assume that $|\xi|\geq 1$.
Using a smooth partition of unity, we split the integral above into 2 pieces: one where $|\eta-\xi|\geq \frac 14|\xi|$
and another one where $|\eta-\xi|\leq \frac12|\xi|$. The first piece is rapidly decaying in~$\xi$ by integration by parts in~$x$.
The second piece is equal to $a(0,\xi)\widehat u(\xi)+S^{s-\frac54}(\mathbb R)$ by the method of stationary phase.

To see~\eqref{e:conormal-symbol-2} we fix a cutoff $\chi\in \CIc(\mathbb R)$ such that $\supp \chi$ lies in the range of~$f$
and $\chi=1$ near $\supp u$. Using the Fourier inversion formula, we write
$$
\widehat{f^*u}(\xi)={1\over 2\pi}\int_{\mathbb R^2} e^{i(f(x)\eta-x\xi)}\chi(f(x))\widehat u(\eta)\,d\eta dx.
$$
Now~\eqref{e:conormal-symbol-2} is proved similarly to~\eqref{e:conormal-symbol-1}. Here in the application
of the method of stationary phase, the critical point is given by $x=0$, $\eta=\xi/f'(0)$ and the Hessian
of the phase at the critical point has signature~0 and determinant $-f'(0)^2$.
\end{proof}

\subsection{Convolution with logarithm}
\label{s:conv-log}

In \S \ref{s:blp} we need information about mapping properties between
spaces of conormal distributions on the boundary and conormal distributions in the interior.
In preparation for Lemma~\ref{l:singl-2} below we now prove the following
\begin{lemm}
  \label{l:log-conv}
Let $f\in \CIc(\mathbb R)$ and define
\begin{equation}
  \label{e:log-conv}
g(x):=\int_0^\infty \log|x-y|\,{f(y)\over \sqrt y}\,dy,\quad x>0.
\end{equation}
Then $g\in C^\infty([0,\infty))$.
\end{lemm}
\Remark In general $ g $ is {\em not} smooth on $ (-\infty,0]$. In fact, changing variables
 $ y = s^2 |x| $, we obtain (see \eqref{e:logconv-1} below)
\[
g' ( x ) = - 2|x|^{-\frac12} \int_0^\infty \frac {f ( s^2 |x| ) }{ 1 + s^2 }\, ds , \quad x < 0 ,
\]
which blows up as $ x \to 0 - $ if $ f(0)\neq 0 $. This, and the conclusion of
Lemma~\ref{l:log-conv}, can also be seen from analysis on the Fourier transform side. 
\begin{proof}
Denote $x_+^{-1/2}:=H(x)x^{-1/2}$ where $H(x)$ is the Heaviside function:
$H(x)=1$ for $x>0$ and $H(x)=0$ for $x<0$.
Since $g$ is the convolution of $\log|x|$ and $x_+^{-1/2}f(x)$, which are both smooth except at $x=0$, the function $g$ is smooth on $(0,\infty)$. Thus it suffices
to prove that $g$ is smooth on $[0,1]$ up to the boundary.

\noindent 1. Assume first that $f$ is real valued and extends holomorphically
to the disk $\{|z|<4\}$.
Making the change of variables $y:=t^2$, we write
\begin{equation}
  \label{e:logconv-1}
g(x)=2\int_0^\infty\log|t^2-x|f(t^2)\,dt=
\int_{\mathbb R}\log|t^2-x|f(t^2)\,dt.
\end{equation}
Assume that $x\in (0,1]$ and consider the holomorphic function
$$
\psi_x(z):=\log (z-\sqrt{x})+\log(z+\sqrt{x}),\quad
z\in\mathbb C\setminus \big((\sqrt x-i[0,\infty))\cup (-\sqrt x-i[0,\infty))\big),
$$
where we use the branch of the logarithm on $\mathbb C\setminus -i[0,\infty)$
which takes real values on~$(0,\infty)$. Then
$\Re\psi_x(t)=\log|t^2-x|$ for all~$t\in\mathbb R\setminus \{\sqrt x,-\sqrt x\}$.

\begin{figure}
\includegraphics{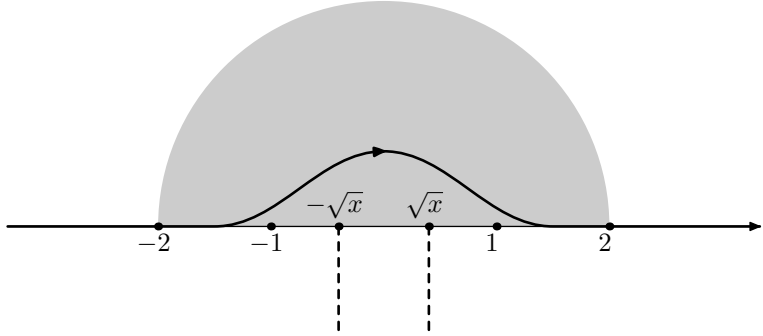}
\caption{The contour $\Gamma$ used in the proof of Lemma~\ref{l:log-conv}.
The dashed lines are the cuts needed to define the function $\psi_x(z)$.}
\label{f:contour}
\end{figure}

Fix an $x$-independent contour $\Gamma=\{t+iw(t)\mid t\in\mathbb R\}\subset\mathbb C$
such that $w(t)\geq 0$ everywhere, $w(t)=0$ for $|t|\geq {3\over 2}$,
$|t+iw(t)|<2$ for $|t|\leq {3\over 2}$, and
$w(t)>0$ for $|t|\leq 1$. (See Figure~\ref{f:contour}.)
Deforming the contour in~\eqref{e:logconv-1}, we get
$$
g(x)=\Re\int_\Gamma \psi_x(z)f(z^2)\,dz\quad\text{for all}\quad
x\in (0,1].
$$
Since $\partial_x\psi_x(z)=(x-z^2)^{-1}$, the function $\psi_x(z)$
and all its $x$-derivatives are bounded uniformly in $x\in (0,1]$
and locally uniformly in $z\in\Gamma$. It follows that
$g$ is smooth on the interval $[0,1]$.

\noindent 2. For the general case, fix a cutoff $\chi\in\CIc(\mathbb R)$
such that $\chi=1$ near $[-4,4]$. Take arbitrary $N\in\mathbb N$.
Using the Taylor expansion of~$f$ at~0, we write
$$
f(x)=f_1(x)+f_2(x),\quad
f_1(x)=p(x)\chi(x)
$$
where $p$ is a polynomial of degree at most~$N$ and $f_2\in \CIc(\mathbb R)$
satisfies $f_2(x)=\mathcal O(|x|^{N+1})$ as $x\to 0$.
We write $g=g_1+g_2$ where $g_j$ are constructed from~$f_j$ using~\eqref{e:log-conv}.

By Step~1 of the present proof,
we see that $g_1$ is smooth on~$[0,1]$.
On the other hand, $x_+^{-1/2}f_2(x)\in C^N(\mathbb R)$;
since $\log|x|$ is locally integrable we get $g_2\in C^N([0,1])$.
Since $N$ can be chosen arbitrary, this gives $g\in C^\infty([0,1])$
and finishes the proof.
\end{proof}
We also give the following general mapping property of convolution with logarithm
on conormal spaces used in  Lemma~\ref{l:con} below:
\begin{lemm}
  \label{l:log-conormal}
Let $\Sigma\subset\mathbb R$ be a finite set 
and put
$\log_\pm x:=\log(x\pm i0)$. 
Then 
\[
f\in  I^s(\mathbb R,N^*_\pm\Sigma) \cap \mathcal E'  ( \mathbb R )  \quad \Longrightarrow\quad  
\begin{cases}
\log_\pm*f\in I^{s-1}(\mathbb R,N^*_\pm\Sigma),\\
\log_\mp*f\in C^\infty(\mathbb R).
\end{cases}
\]
\end{lemm}
\begin{proof}
Assume that $f\in I^s(\mathbb R,N^*_+\Sigma) \cap \mathcal E'  ( \mathbb R )$
(the case of $N^*_-\Sigma$ is handled similarly). We may reduce to the case $\Sigma=\{0\}$.
Since $ \partial_x $ is an elliptic operator, the local definition \eqref{eq:uosc} (with no $ x' $ variable)
shows that it is enough to show that
$$
\partial_x \log_+ * f\in  I^{s}(\mathbb R,N^*_+ \{0\}),\quad
\partial_x \log_- * f\in  C^\infty(\mathbb R).
$$
It remains to use that $ \partial_x \log_\pm = 
( x \pm i 0 )^{-1} $ and we have the Fourier transform formula
(see \cite[Example 7.1.17]{Hormander1}; here $ H $ is the Heaviside function)
\begin{equation}
  \label{e:x-i0-FT}
\mathbf u_{\pm0}(x):=(x\pm i0)^{-1}\quad\Longrightarrow\quad
\widehat{\mathbf u}_{\pm0}(\xi)=\mp 2 \pi i H ( \pm \xi ).
\end{equation}
 \end{proof}

\subsection{Microlocal structure of \texorpdfstring{$(x\pm i\varepsilon\psi(x))^{-1}$}{the inverse of (x\unichar{"00B1}i\unichar{"03B5}\unichar{"03C8}(x))}}
\label{s:micro-inv}

In this section we study the behaviour as $\varepsilon\to 0+$ of functions of the form
\begin{equation}
  \label{e:our-inv}
\chi(x,\varepsilon)(x\pm i\varepsilon\psi(x,\varepsilon))^{-1}\ \in\ C^\infty(J),\quad
0<\varepsilon<\varepsilon_0
\end{equation}
where $J\subset\mathbb R$ is an open interval containing~0 and
\begin{equation}
  \label{e:inv-psi-ass}
\chi,\psi\in C^\infty(J\times [0,\varepsilon_0);\mathbb C),\quad
\Re\psi>0\quad\text{on}\quad J\times [0,\varepsilon_0).
\end{equation}
We first decompose~\eqref{e:our-inv}
into the sum of $r(x\pm i\varepsilon z)^{-1}$, where $r,z\in\mathbb C$ depend on~$\varepsilon$ but not on~$x$,
and a function which is smooth uniformly in~$\varepsilon$: 
\begin{lemm}
\label{l:prepare}
Under the conditions~\eqref{e:inv-psi-ass} we have for all $\varepsilon\in (0,\varepsilon_0)$
\begin{equation}
  \label{e:prepare}
\chi(x,\varepsilon)(x\pm i\varepsilon\psi(x,\varepsilon))^{-1}=r^\pm(\varepsilon)(x\pm i\varepsilon z^\pm(\varepsilon))^{-1}+q^\pm(x,\varepsilon)
\end{equation}
where $r^\pm,z^\pm\in C^\infty([0,\varepsilon_0))$ and $q^\pm\in C^\infty(J\times [0,\varepsilon_0))$
are complex valued, $\Re z^\pm>0$ on $[0,\varepsilon_0)$, $z^\pm(0)=\psi(0,0)$,
and $\chi(x,0)=xq^\pm(x,0)+r^\pm(0)$.
\end{lemm}
\begin{proof}
Since $(x\pm i\varepsilon\psi(x,\varepsilon))^{-1}$ is a smooth function of $(x,\varepsilon)\in J\times[0,\varepsilon_0)$ outside of $(0,0)$,
it is enough to show that~\eqref{e:prepare} holds for $|x|,\varepsilon$ small enough.

The complex valued function $F^\pm(x,\varepsilon):=x\pm i\varepsilon\psi(x,\varepsilon)$
is smooth in $(x,\varepsilon)\in J\times [0,\varepsilon_0)$ and satisfies $F^\pm(0,0)=0$
and $\partial_x F^\pm(0,0)=1$. Thus by the Malgrange Preparation Theorem~\cite[Theorem~7.5.6]{Hormander1},
we have for $(x,\varepsilon)$ in some neighbourhood of~$(0,0)$ in $J\times [0,\varepsilon_0)$
$$
x=q_1^\pm(x,\varepsilon)(x\pm i\varepsilon\psi(x,\varepsilon))+r^\pm_1(\varepsilon)
$$
where $q_1^\pm$, $r_1^\pm$ are smooth. Taking $\varepsilon=0$ we get
$r^\pm_1(0)=0$ and $q_1^\pm(x,0)=1$; differentiating in~$\varepsilon$ and then putting~$x=\varepsilon=0$ we get
$\partial_\varepsilon r_1^\pm(0)=\mp i\psi(0,0)$. We put
$z^\pm(\varepsilon):=\pm i\varepsilon^{-1}r_1^\pm(\varepsilon)$, so that when $\varepsilon>0$
\begin{equation}
  \label{e:prepare-1}
(x\pm i\varepsilon \psi(x,\varepsilon))^{-1}=q_1^\pm(x,\varepsilon)(x\pm i\varepsilon z^\pm(\varepsilon))^{-1}.
\end{equation}
Note that $z^\pm(0)=\psi(0,0)$ and thus $\Re z^\pm(\varepsilon)>0$ for small $\varepsilon$.

Now, we use the Malgrange Preparation Theorem again, this time for the function
$F^\pm(x,\varepsilon):=x\pm i\varepsilon z^\pm(\varepsilon)$, to get
for $(x,\varepsilon)$ in some neighbourhood of~$(0,0)$ in $J\times [0,\varepsilon_0)$
$$
\chi(x,\varepsilon)q_1^\pm(x,\varepsilon)=q^\pm(x,\varepsilon)(x\pm i\varepsilon z^\pm(\varepsilon))+r^\pm(\varepsilon)
$$
where $q^\pm$, $r^\pm$ are again smooth. Taking $\varepsilon=0$ we get
$\chi(x,0)=xq^\pm(x,0)+r^\pm(0)$.
Together with~\eqref{e:prepare-1} this gives the decomposition~\eqref{e:prepare}.
\end{proof}
As an application of Lemma~\ref{l:prepare}, we give
\begin{lemm}
\label{l:inverse-converger}
Assume that $\psi$ satisfies~\eqref{e:inv-psi-ass}. Then
we have for all $s<-{1\over 2}$
\begin{equation}
  \label{e:inverse-converger}
(x\pm i\psi(x,\varepsilon))^{-1}\ \to\ 
(x\pm i0)^{-1}\quad\text{as}\quad \varepsilon\to 0+\quad\text{in}\quad H^s_{\loc}(J).
\end{equation}
\end{lemm}
\begin{proof}
Put $\chi\equiv 1$ and let $z^\pm(\varepsilon)$, $r^\pm(\varepsilon)$, $q^\pm(x,\varepsilon)$ be given
by Lemma~\ref{l:prepare}; note that $1=xq^\pm(x,0)+r^\pm(0)$, thus $r^\pm(0)=1$ and $q^\pm(x,0)=0$.
We have
$$
(x\pm i\varepsilon z^\pm(\varepsilon))^{-1}\to (x\pm i0)^{-1}\quad\text{in}\quad H^s(\mathbb R).
$$
Indeed, the Fourier transform of the right-hand side is equal to~$\mp 2\pi iH(\pm \xi)$ by~\eqref{e:x-i0-FT}
and the Fourier transform of the left-hand side is equal to~$\mp 2\pi i H(\pm \xi)e^{-\varepsilon z^\pm(\varepsilon)|\xi|}$ by
\begin{equation}
  \label{e:x-ieps-FT}
\mathbf u_{\pm z}(x):=(x\pm iz)^{-1},\quad
\Re z>0\quad\Longrightarrow\quad
\widehat{\mathbf u}_{\pm z}(\xi)=\mp 2\pi i H(\pm \xi)e^{- z|\xi|}.
\end{equation}
We have convergence of these Fourier transforms in $L^2(\mathbb R;\langle\xi\rangle^{2s}\,d\xi)$
by the Dominated Convergence Theorem.

By~\eqref{e:prepare} this implies that the left-hand side of~\eqref{e:inverse-converger} converges in $H^s_{\loc}(J)$
to
$$
r^\pm(0)(x\pm i0)^{-1}+q^\pm(x,0)=(x\pm i0)^{-1}
$$
which finishes the proof.
\end{proof}
The functions $r^\pm,z^\pm,q^\pm$ in Lemma~\ref{l:prepare} are not uniquely determined
by $\chi,\psi$, however they are unique up to $\mathcal O(\varepsilon^\infty)$:
\begin{lemm}
  \label{l:prepare-unique}
Assume that $r^\pm_j,z^\pm_j\in C^\infty([0,\varepsilon_0))$, $j=1,2$, are complex valued functions
such that $\Re z^\pm_j>0$ on $[0,\varepsilon_0)$, $r^\pm_j(0)\neq 0$, and
$$
\tilde q^\pm(x,\varepsilon):=r^\pm_1(\varepsilon)(x\pm i\varepsilon z^\pm_1(\varepsilon))^{-1}
-r^\pm_2(\varepsilon)(x\pm i\varepsilon z^\pm_2(\varepsilon))^{-1}\ \in\ C^\infty(J\times [0,\varepsilon_0)).
$$
Then $r_1^\pm(\varepsilon)-r_2^\pm(\varepsilon)$, $z^\pm_1(\varepsilon)-z^\pm_2(\varepsilon)$,
and $\tilde q^\pm(x,\varepsilon)$
are $\mathcal O(\varepsilon^\infty)$, that is all their derivatives in~$\varepsilon$ vanish at~$\varepsilon=0$.
\end{lemm}
\begin{proof}
Differentiating $k-1$ times in~$x$ and then putting $x=0$ we see that for all $k\geq 1$
$$
{r_1^\pm(\varepsilon)\over \varepsilon^k z_1^\pm(\varepsilon)^k}-
{r_2^\pm(\varepsilon)\over \varepsilon^k z_2^\pm(\varepsilon)^k}\in C^\infty([0,\varepsilon_0)).
$$
Therefore
\begin{equation}
  \label{e:freakor}
\partial_{\varepsilon}^\ell|_{\varepsilon=0} \bigg({r_1^\pm(\varepsilon)\over z_1^\pm(\varepsilon)^k}
-{r_2^\pm(\varepsilon)\over z_2^\pm(\varepsilon)^k}\bigg)=0\quad\text{for all}\quad
0\leq \ell<k.
\end{equation}
Taking $\ell=0$ and $k=1,2$, we see that
$r_1^\pm(0)=r_2^\pm(0)$ and $z_1^\pm(0)=z_2^\pm(0)$. Arguing by induction on~$\ell$
and using $k=\ell+1,\ell+2$ in~\eqref{e:freakor}
we see that $\partial_\varepsilon^\ell r_1^\pm(0)=\partial_\varepsilon^\ell r_2^\pm(0)$
and $\partial_\varepsilon^\ell z_1^\pm(0)=\partial_\varepsilon^\ell z_2^\pm(0)$.
Thus $r_1^\pm(\varepsilon)-r_2^\pm(\varepsilon)$ and $z_1^\pm(\varepsilon)-z_2^\pm(\varepsilon)$
are $\mathcal O(\varepsilon^\infty)$, which implies that $\tilde q^\pm(x,\varepsilon)$
is $\mathcal O(\varepsilon^\infty)$ as well.
\end{proof}
\Remark  If $\chi$ and $\psi$ depend smoothly
on some additional parameter~$y$, then the proof of Lemma~\ref{l:prepare} shows that
$r^\pm$, $z^\pm$, $q^\pm$ can be chosen
to depend smoothly on~$y$ as well. 
Lemma~\ref{l:inverse-converger} also holds, with convergence locally uniform in~$y$,
as does Lemma~\ref{l:prepare-unique}.
In~\S\ref{s:restricted-slp} below we use this to study expressions of the form
\begin{equation}
  \label{e:prepare-remark}
\chi(\theta,\theta',\varepsilon)(\theta-\theta'\pm i\varepsilon\psi(\theta,\theta',\varepsilon))^{-1},
\end{equation}
where $(\theta,\theta')$ is in some neighbourhood of~0 and we put $x:=\theta-\theta'$,
$y:=\theta$.

For the use in~\S\ref{s:blp} we record the fact that operators with Schwartz kernels of the form~\eqref{e:prepare-remark} are pseudodifferential:
\begin{lemm}
  \label{l:inversor-psi0}
Assume that $c^\pm_\varepsilon(\theta')$
and $z^\pm_\varepsilon(\theta')$
are complex valued functions smooth in $\theta'\in\mathbb S^1:=\mathbb R/\mathbb Z$
and $\varepsilon\in [0,\varepsilon_0)$
and such that $\Re z^\pm_\varepsilon>0$.
Let $\chi\in C^\infty(\mathbb S^1\times\mathbb S^1)$ be supported in a neighbourhood
of the diagonal and equal to~1 on a smaller neighbourhood of the diagonal.
Consider the operator $A^\pm_\varepsilon$ on $C^\infty(\mathbb S^1)$
given by
\begin{equation}
\label{eq:K2A}
\begin{aligned}
A^\pm_\varepsilon f(\theta)&=\int_{\mathbb S^1}K^\pm_\varepsilon(\theta,\theta')f(\theta')\,d\theta',\\
K^\pm_\varepsilon(\theta,\theta')&=\begin{cases}
\chi(\theta,\theta')c^\pm_\varepsilon(\theta')(\theta-\theta'\pm i\varepsilon z^\pm_\varepsilon(\theta'))^{-1},& \varepsilon>0;\\
\chi(\theta,\theta')c^\pm_\varepsilon(\theta')(\theta-\theta'\pm i0)^{-1},& \varepsilon=0.
\end{cases}
\end{aligned}
\end{equation}
Then $A^\pm_\varepsilon\in\Psi^0(\mathbb S^1)$ uniformly in $\varepsilon$ and we have, uniformly in~$\varepsilon$,
\begin{equation}
\label{eq:WFAs}
\WF(A^\pm_\varepsilon)\subset \{\pm \xi>0\},\quad
\sigma(A^\pm_\varepsilon)(\theta,\xi)=\mp 2\pi i c^\pm_\varepsilon ( \theta) e^{ - \varepsilon z^\pm_\varepsilon (\theta) |\xi| } H (\pm  \xi ) ,
\end{equation}
where for $ \varepsilon > 0 $ the principal symbol is understood as in \eqref{eq:Opasu}
and $ H $ is the Heaviside function (with the symbol considered for $ |\xi| > 1 $).
\end{lemm}
\Remark We note that the definition \eqref{eq:Opasu} of the symbol of a family of operators
and Lemma \ref{l:prepare-unique} show that the principal symbol is independent of the (not unique) $ c^\pm_\varepsilon,z^\pm_\varepsilon $. 
\begin{proof}
Using the formulas~\eqref{e:x-i0-FT} and~\eqref{e:x-ieps-FT}, we write the kernel
$K_\varepsilon^\pm(\theta,\theta')$ as an oscillatory integral
(where $a_\varepsilon(\theta,\theta',\xi)$ is supported near $\{\theta=\theta'\}$
where $\theta-\theta'\in\mathbb R$ is well-defined)
$$
\begin{aligned}
K_\varepsilon^\pm(\theta,\theta')&={1\over 2\pi}\int_{\mathbb R}e^{i(\theta-\theta')\xi}a_\varepsilon(\theta,\theta',\xi)\,d\xi,\\
a_\varepsilon(\theta,\theta',\xi)&=\mp 2\pi i\chi(\theta,\theta')c^\pm_\varepsilon(\theta')e^{-\varepsilon z_\varepsilon^\pm(\theta')|\xi|}H(\pm\xi).
\end{aligned}
$$
Fix a cutoff function $\widetilde\chi\in\CIc(\mathbb R)$ such that $\widetilde\chi=1$ near~0
and split $a_\varepsilon=\widetilde\chi(\xi)a_\varepsilon+(1-\widetilde\chi(\xi))a_\varepsilon$.
The integral corresponding to $\widetilde\chi(\xi) a_\varepsilon$ gives a kernel which is in
$C^\infty$ in $\theta$, $\theta'$, and $\varepsilon\in [0,\varepsilon_0)$. Next,
$(1-\widetilde\chi(\xi))a_\varepsilon$ is a symbol of class $S^0$: for each
$\alpha,\beta$ there exists $C_{\alpha\beta}$
such that for all $\theta,\theta',\xi$ and $\varepsilon\in [0,\varepsilon_0)$, we have
$$
\big|\partial^\alpha_{(\theta,\theta')}\partial^\beta_\xi\big((1-\widetilde\chi(\xi))a_\varepsilon(\theta,\theta',\xi)\big)\big|
\leq C_{\alpha\beta}\langle\xi\rangle^{-\beta}.
$$
Therefore (see~\cite[Lemma~18.2.1]{Hormander3} or~\cite[Theorem~3.4]{Grigis-Sjostrand})
we see that $A^\pm_\varepsilon\in\Psi^0$ uniformly in~$\varepsilon$,
its wavefront set is contained in~$\{\pm\xi>0\}$, and
its principal symbol is the equivalence class of $a_\varepsilon(\theta,\theta,\xi)$.
\end{proof}
\Remark The fine analysis in~\S\ref{s:micro-inv} is strictly speaking not
necessary for our application in~\S\ref{s:restricted-slp}: indeed,
in the proof of Proposition~\ref{p:summa} one could instead use a version of Lemma~\ref{l:inversor-psi0}
which allows $c_\varepsilon^\pm$ and $z_\varepsilon^\pm$ to depend on both $\theta$ and $\theta'$.
Moreover, ultimately one just uses that the exponential in~\eqref{eq:WFAs} is
 bounded in absolute value by~1, see the proof of Lemma~\ref{l:TG}.
However, we feel that using the results of~\S\ref{s:micro-inv}
leads to nicer expressions for the kernels of the restricted single layer potentials
in~\S\S\ref{s:slp-ref}--\ref{s:slp-char} which could be useful elsewhere.

\section{Boundary layer potentials}
\label{s:blp}

In this section we describe microlocal properties of boundary layer potentials for 
the operator $ P - \omega^2 = \partial_{x_2}^2 \Delta_\Omega^{-1} - \omega^2 $,
or rather for the related partial differential operator $P(\omega)$ defined in~\eqref{e:P-lambda-def}. The
key issue is the transition from elliptic to hyperbolic behaviour as $ \Im \omega \to 0 $.
To motivate the results we explain the analogy with the standard boundary layer potentials
in \S \ref{s:mot}. In~\S\ref{s:fun} we compute fundamental solutions for $P(\omega)$ on $\mathbb R^2$
and in~\S\ref{s:redbone} we use these to study the Dirichlet problem for $P(\omega)$ on~$\Omega$.
This will lead us to single layer potentials: in~\S\ref{s:slp} we study their mapping properties
(in particular relating Lagrangian distributions on the 
boundary to Lagrangian distributions in the interior)
and in~\S\ref{s:restricted-slp} we give a microlocal description of their restriction to~$\partial\Omega$
uniformly as $ \Im \omega \to 0 $, which is crucially used in~\S\ref{s:abo}.

In \S\S\ref{s:blp}--\ref{s:liap} we generally use the letter $\lambda$ to denote the spectral parameter when it is
real and the letter $\omega$ for complex values of the spectral parameter, often taking the limit $\omega\to \lambda\pm i0$.

\subsection{Basic properties}

Consider the second order constant coefficient differential operator on $\mathbb R^2_{x_1,x_2}$
\begin{equation}
  \label{e:P-lambda-def}
P(\omega):=(1-\omega^2)\partial_{x_2}^2-\omega^2\partial_{x_1}^2\quad\text{where}\quad
\omega\in\mathbb C,\quad
0<\Re\omega<1.
\end{equation}
Formally,
\begin{equation}
\label{eq:form}  P ( \omega ) = ( P - \omega^2 ) \Delta_\Omega , \ \ \ 
P ( \omega )^{-1} = \Delta_\Omega^{-1} ( P - \omega^2 )^{-1} . \end{equation}
We note that $P(\omega)$ is hyperbolic when $\omega\in(0,1)$ and elliptic
otherwise.
We factorize $P(\omega)$ as follows:
\begin{equation}
  \label{e:L-pm-def}
P(\omega)=4L_\omega^+L_\omega^-,\quad
L^\pm_\omega:=\textstyle{1\over 2}(\pm\omega\,\partial_{x_1}+\sqrt{1-\omega^2}\,\partial_{x_2}).
\end{equation}
Here $\sqrt{1-\omega^2}$ is defined by taking the branch of the square root
on $\mathbb C\setminus (-\infty,0]$ which takes positive values on $(0,\infty)$.
We note that for $\lambda\in(0,1)$ the operators $L^\pm_\lambda$
are two linearly independent constant vector fields on $\mathbb R^2$.
For $\Im\omega\neq 0$, $L^\pm_\omega$ are Cauchy--Riemann type operators.

The definition~\eqref{eq:dual2} of the functions $\ell^\pm(x,\lambda)$ extends
to complex values of $\lambda$:
$$
\ell^\pm(x,\omega): =\pm{x_1\over\omega}+{x_2\over\sqrt{1-\omega^2}},\quad
x=(x_1,x_2)\in\mathbb R^2,\quad
\omega\in\mathbb C,\quad
0<\Re\omega<1.
$$
Then the linear functions $\ell^\pm(\bullet,\omega)$ are dual to the
operators $L^\pm_\omega$:
\begin{equation}
\label{eq:L2l}
L^\pm_\omega \ell^\pm(x,\omega)=1,\quad
L^\mp_\omega \ell^\pm(x,\omega)=0.
\end{equation}
We record here the following statement:
\begin{lemm}
  \label{l:which-rotate}
Assume that $\Im\omega>0$. Then the map $x\in \mathbb R^2\mapsto\ell^\pm(x,\omega)\in\mathbb C$
is orientation preserving in the case of $\ell^+$ and orientation reversing
in the case of $\ell^-$.
If $\Im\omega<0$ then a similar statement holds with the roles of $\ell^\pm$ switched.
\end{lemm}
\begin{proof}
This follows immediately from the sign identity
\begin{equation}
\sgn\Im{\omega\over\sqrt{1-\omega^2}}=\sgn\Im\omega,\quad
\omega\in\mathbb C,\quad
0<\Re\omega<1
\end{equation}
which can be verified by noting that $\Re\sqrt{1-\omega^2}>0$
and $\sgn\Im\sqrt{1-\omega^2}=-\sgn\Im\omega$.
\end{proof}

\subsection{Motivational discussion}
\label{s:mot}

When $ \Im \omega > 0 $ the decomposition \eqref{e:L-pm-def} is similar to the factorization 
of the Laplacian,
$$
\Delta = 4 \partial_z \partial_{\bar z },\quad
\partial_z = \tfrac12 (\partial_x - i \partial_y ),\quad
\partial_{\bar z} = \tfrac12 (\partial_x + i \partial_y ).
$$ 
The functions $z=x+iy$ and $\bar z$ play the role of $ \mathcal \ell^\pm ( x, \omega ) $ 
for $ \pm $, respectively (which matches the orientation in Lemma~\ref{l:which-rotate})
and $\partial_z,\partial_{\bar z}$ play the role of $L_\omega^+$, $L_\omega^-$.
Hence to explain the structure of the fundamental solution of $ P ( \omega ) $
and to motivate the restricted boundary layer potential in \S \ref{s:restricted-slp}
we review the basic
case when $\Omega=\{y>0\}$ and $P(\omega)$ is replaced by $\Delta$.
The fundamental solution is given by (see e.g.~\cite[Theorem~3.3.2]{Hormander1})
\[ 
\Delta E = \delta_0 , \ \   E := c \log (z \bar z ) , \ \  \partial_z E = c/z, \ \ \partial_{\bar z} E = c/\bar z , \ \  c = 1/{4 \pi}.
\]
We consider the {\em single layer} potential $S:\CIc(\mathbb R)\to \mathcal D'(\mathbb R^2)\cap C^\infty(\mathbb R^2\setminus \{y=0\})$,
\[ 
 S v ( x , y ) := \int_{\mathbb R}  E ( x - x' , y ) v ( x' ) \,dx' , \quad v \in C^\infty_{\rm{c}} ( {\mathbb R} ) . 
\]
We then have limits as $ y \to 0\pm$, 
\[
\mathcal C_\pm  v ( x ) = \mathcal C v ( x ) := \frac 1{ 2 \pi }  \int_{\mathbb R} \log | x - x ' | v ( x' ) dx' ,
\]
and we consider
\[ \partial_x \mathcal C v ( x ) =  \lim_{y \to 0\pm } \partial_x S v ( x, y ) = \lim_{y \to  0\pm } 
( \partial_z + \partial_{\bar z } ) S v (x, y) . \]
Then (where we recall $ c = 1/4 \pi $)
\[ \lim_{y \to 0 \pm } \partial_z S v (x,y ) = \lim_{y \to 0 \pm } 
\int_{{\mathbb R} }\frac{c} { x - x' +i y } v ( x' ) dx' = \int_{{\mathbb R} }\frac{c} { x - x' \pm i 0 } v ( x' ) dx,
\]
and similarly,
\[ \lim_{y \to 0 \pm } \bar \partial_z S v (x,y ) = \lim_{y \to 0 \pm } 
\int_{{\mathbb R} }\frac{c} { x - x' - i y } v ( x' ) dx' = \int_{{\mathbb R} }\frac{c} { x - x' \mp i 0 } v ( x' ) dx,
\]
where the right hand sides are understood as distributional pairings.
This gives
\begin{equation}
\label{eq:modelC} \partial_x \mathcal C v ( x ) = \frac{1}{4 \pi}  \int_{\mathbb R} \sum_\pm ( x - x'  \pm i 0 )^{-1} v(x' ) dx' 
\end{equation}
which is ${1\over 2}$ times the Hilbert transform, that is, the Fourier multiplier with symbol $ - i \sgn ( \xi ) $. (We note that
$ \sum_\pm ( x - x'  \pm i 0 )^{-1} = 2 \partial_x \log |x| $ which is the principal value of $ 2/x $.)


In \S \ref{s:restricted-slp} we describe the analogue of $ \partial_x \mathcal C $ in our 
case. It is similar to \eqref{eq:modelC} when $ \Im \omega > 0 $ but when $ \Im \omega \to 0 + $
it has additional singularities described using the chess billiard map $ b ( x, \lambda ) $,
$\lambda=\Re\omega$, 
or rather its building components $ \gamma^\pm $. The operator
becomes an elliptic operator of order $ 0 $ (just as is the case in \eqref{eq:modelC} if we restrict our
attention to compact sets) plus a Fourier integral operator -- see Proposition \ref{p:summa}.
 
\subsection{Fundamental solutions}
\label{s:fun}

We now construct a \emph{fundamental solution} of the operator $P(\omega)$
defined in~\eqref{e:P-lambda-def},
that is a distribution $E_\omega\in\mathcal D'(\mathbb R^2)$ such that
\begin{equation}
  \label{e:fund-sol-eqn}
P(\omega)E_\omega=\delta_0.
\end{equation}
For that we use the complex valued quadratic form
$$
A(x,\omega):=\ell^+(x,\omega)\ell^-(x,\omega)=-{x_1^2\over\omega^2}+{x_2^2\over 1-\omega^2}.
$$
Since $0<\Re\omega<1$ we have
$\sgn\Im (-\omega^{-2})=\sgn\Im ((1-\omega^2)^{-1})=\sgn\Im\omega$, thus
\begin{equation}
  \label{e:Im-A-sign}
\sgn \Im A(x,\omega)= 
\sgn\Im\omega\quad\text{for all}\quad
x\in\mathbb R^2\setminus \{0\}.
\end{equation}

\subsubsection{The non-real case}

We first consider the case $\Im\omega\neq 0$. In this case
our fundamental solution is the locally integrable function
\begin{equation}
  \label{e:fund-sol-imag}
E_\omega(x):=c_\omega\log A(x,\omega),\quad
x\in\mathbb R^2\setminus\{0\},\quad
c_\omega:={i\sgn\Im\omega\over 4\pi\omega\sqrt{1-\omega^2}}.
\end{equation}
Here we use the branch of logarithm on $\mathbb C\setminus (-\infty,0]$
which takes real values on $(0,\infty)$. Note that the function $E_\omega$
is smooth on $\mathbb R^2\setminus \{0\}$.
\begin{lemm}
The function $E_\omega$ defined in~\eqref{e:fund-sol-imag} solves~\eqref{e:fund-sol-eqn}.
\end{lemm}
\begin{proof}
We first check that $P(\omega)E_\omega=0$ on $\mathbb R^2\setminus \{0\}$:
this follows from~\eqref{e:L-pm-def}, \eqref{eq:L2l}, and the identities
\begin{equation}
\label{eq:LlogA}
L^\pm_\omega \log A(x,\omega)={1\over \ell^\pm(x,\omega)}\quad\text{for all}\quad
x\in\mathbb R^2\setminus \{0\}.
\end{equation}
Next, denote by $B_\varepsilon$ the ball of radius $\varepsilon>0$ centered at~0
and orient $\partial B_\varepsilon$ in the counterclockwise direction.
Using the Divergence Theorem twice, we compute for each~$\varphi\in \CIc(\mathbb R^2)$
$$
\begin{aligned}
\int_{\mathbb R^2} E_\omega(x)(P(\omega)\varphi(x))\,dx &=
4\lim_{\varepsilon\to 0+} \int_{\mathbb R^2\setminus B_\varepsilon}
E_\omega(x)(L^+_\omega L^-_\omega\varphi(x))\,dx\\
&=-4c_\omega\lim_{\varepsilon\to 0+}\int_{\mathbb R^2\setminus B_\varepsilon}
{L^-_\omega\varphi(x)\over \ell^+(x,\omega)}\,dx\\
&=-2c_\omega\lim_{\varepsilon\to 0+}\int_{\partial B_\varepsilon}
{\varphi(x)(\sqrt{1-\omega^2}\,dx_1+\omega\,dx_2)\over\ell^+(x,\omega)}\\
&=-2c_\omega \omega\sqrt{1-\omega^2}\,\varphi(0)\lim_{\varepsilon\to 0+}\int_{\partial B_\varepsilon}{d\ell^+(x,\omega)\over\ell^+(x,\omega)}=\varphi(0)
\end{aligned}
$$
which gives~\eqref{e:fund-sol-eqn}. Here in the last
equality we make the change of variables $z=\ell^+(x,\omega)$
and use Lemma~\ref{l:which-rotate}.
\end{proof}

\subsubsection{The real case}


We now discuss the case $\lambda\in (0,1)$. Define
the fundamental solutions $E_{\lambda\pm i0}\in L^1_{\loc}(\mathbb R^2)$
as follows:
\begin{equation}
\label{eq:Elapm}
\begin{gathered}
E_{\lambda\pm i0}(x):=\pm c_\lambda\log(A(x,\lambda)\pm i0),\quad
c_\lambda:={i\over 4\pi \lambda\sqrt{1-\lambda^2}},\\
\log(A(x,\lambda)\pm i0)=\begin{cases}
\log A(x,\lambda),& A(x,\lambda)>0;\\
\log(-A(x,\lambda))\pm i\pi,& A(x,\lambda)<0.
\end{cases}
\end{gathered}
\end{equation}
The next lemma shows that $E_{\lambda\pm i\varepsilon}\to E_{\lambda\pm i0}$
in $\mathcal D'(\mathbb R^2)$ as $\varepsilon\to 0+$. In fact it gives a stronger convergence statement with
derivatives in $\lambda$. To make this statement we introduce the following notation:
if $\mathcal J\subset (0,1)$ is an open interval then
\begin{equation}
  \label{e:O-analytic}
\mathscr O(\mathcal J+i [0,\infty))\ \subset\ C^\infty(\mathcal J+i[0,\infty))
\end{equation}
consists of $C^\infty$ functions on $\mathcal J+i[0,\infty)$ which are holomorphic in the interior
$\mathcal J+i(0,\infty)$.
Similarly one can define $\mathscr O(\mathcal J-i[0,\infty))$.
\begin{lemm}
  \label{l:loc}
The maps
\begin{equation}
  \label{e:loc-mapp}
\omega\in (0,1)\pm i[0,\infty)\ \mapsto\ \begin{cases}
E_{\omega},&\Im\omega\neq 0,\\
E_{\lambda\pm i0},&\omega=\lambda\in (0,1)
\end{cases}
\end{equation}
lie in $\mathscr O((0,1)\pm i[0,\infty);\mathcal D'(\mathbb R^2))$
in the following sense: the distributional pairing of~\eqref{e:loc-mapp}
with any $\varphi\in \CIc(\mathbb R^2)$
lies in~$\mathscr O((0,1)\pm i[0,\infty))$.
\end{lemm}
\begin{proof}
We consider the case of $\Im\omega\geq 0$, with the case $\Im\omega\leq 0$ handled similarly.
Fix $\varphi\in\CIc(\mathbb R^2)$.

\noindent 1. We will prove the following limiting statement: for each $\lambda\in (0,1)$
\begin{equation}
  \label{e:loc-convergor}
\int_{\mathbb R^2}E_{\omega_j}(x)\varphi(x)\,dx\to\int_{\mathbb R^2}E_{\lambda+i0}(x)\varphi(x)\,dx
\quad\text{for all}\quad \omega_j\to\lambda,\quad \Im\omega_j>0.
\end{equation}
We write $\omega_j=\lambda_j+i\varepsilon_j$ where $\lambda_j\to\lambda$ and $\varepsilon_j\to 0+$.
We first show a bound on $E_{\omega_j}(x)=c_{\omega_j}\log A(x,\omega_j)$ which is uniform in~$j$. 
Taking the Taylor expansion of $\ell^+(x,\lambda+i\varepsilon)$
in $\varepsilon$, we get
$$
\ell^+(x,\omega_j)=\ell^+(x,\lambda_j)+i\varepsilon_j\partial_\lambda\ell^+(x,\lambda_j)+
\mathcal O(\varepsilon_j^2|x|),
$$
where the constant in $\mathcal O(\bullet)$ is independent of~$j$.
Since $\partial_\lambda\ell^+(x,\lambda_j)$ is real, we bound
\begin{equation}
\label{e:log-convergir-1}
|\ell^+(x,\omega_j)|\geq \tfrac{1}{2}\big(|\ell^+(x,\lambda_j)|+\varepsilon_j|\partial_\lambda\ell^+(x,\lambda_j)|\big)
-C\varepsilon_j^2|x|.
\end{equation}
As $\ell^+(x,\lambda_j),\partial_\lambda\ell^+(x,\lambda_j)$ are linearly independent linear forms in $x$
(see~\eqref{e:l-pm-differential}), we have
\begin{equation}
\label{e:log-convergir-2}
|x|\leq C\big(|\ell^+(x,\lambda_j)|+|\partial_\lambda \ell^+(x,\lambda_j)|\big).
\end{equation}
Together~\eqref{e:log-convergir-1} and~\eqref{e:log-convergir-2} show that for $j$ large enough
\begin{equation}
  \label{e:limit-E-bdor-0}
|\ell^+(x,\omega_j)|\geq \tfrac{1}{3}|\ell^+(x,\lambda_j)|.
\end{equation}
The same bound holds for $\ell^-$. Since $A(x,\omega_j)=\ell^+(x,\omega_j)\ell^-(x,\omega_j)$, we then have
$$
C^{-1}|A(x,\lambda_j)|\leq |A(x,\omega_j)|\leq C|x|^2
$$
which implies the
following bound for $j$ large enough, some $j$-independent constant $C$, and all~$x\in\mathbb R^2$:
\begin{equation}
  \label{e:limit-E-bdor}
|E_{\omega_j}(x)|\leq C(|E_{\lambda_j+i0}(x)|+\log(2+|x|)).
\end{equation}

\noindent 2. To pin the zero set of $A(x,\lambda_j)$, which depends on~$\lambda_j$,
we introduce the linear isomorphism $\Phi_\lambda:\mathbb R^2_y\to\mathbb R^2_x$
such that $\Phi_\lambda^{-1}(x)=(\ell^+(x,\lambda),\ell^-(x,\lambda))$. Then
$A(\Phi_\lambda(y),\lambda)=y_1y_2$, so
the pullback of $E_{\lambda_j+i0}$ by $\Phi_{\lambda_j}$ is given by
\begin{equation}
  \label{e:limit-E-Phior}
\Phi_{\lambda_j}^* E_{\lambda_j+i0}(y)=c_{\lambda_j}\log(y_1y_2+i0)
\end{equation}
which is a locally integrable function on $\mathbb R^2$.

We can now show~\eqref{e:loc-convergor}. For each $y\in\mathbb R^2$, we have
$A(\Phi_{\lambda_j}(y),\omega_j)\to A(\Phi_\lambda(y),\lambda)=y_1y_2$
and $\varphi(\Phi_{\lambda_j}(y))\to \varphi(\Phi_\lambda(y))$.
Using~\eqref{e:Im-A-sign} we then get the pointwise limit
$$
\Phi_{\lambda_j}^* (E_{\omega_j}\varphi)(y)\to\Phi_\lambda^*(E_{\lambda+i0}\varphi)(y)\quad\text{for all}\quad
y\in\mathbb R^2,\ y_1y_2\neq 0.
$$
Now~\eqref{e:loc-convergor} follows from the Dominated Convergence Theorem applied to the sequence of functions $\Phi_{\lambda_j}^*(E_{\omega_j}\varphi)$, where the dominant is given by the locally integrable function $C(1+|\log(y_1y_2+i0)|)$
as follows from the bound~\eqref{e:limit-E-bdor} and the identity~\eqref{e:limit-E-Phior}.

\noindent 3. Denote by $F_\varphi(\omega)$ the pairing of~\eqref{e:loc-mapp}
with $\varphi$.
Since $A(x,\omega)$ is a quadratic form depending holomorphically on~$\omega\in (0,1)+i(0,\infty)$
which has a positive definite imaginary part by~\eqref{e:Im-A-sign},
we see that $F_\varphi$ is holomorphic on $(0,1)+i(0,\infty)$.
Moreover, the restriction of~$F_\varphi$ to $(0,1)$ is smooth, as can be seen by writing
$$
F_\varphi(\lambda)=\int_{\mathbb R^2}E_{\lambda+i0}(x)\varphi(x)\,dx=c_\lambda |\det\Phi_\lambda|\int_{\mathbb R^2}
\log(y_1y_2+i0) \varphi(\Phi_\lambda(y))\,dy,\quad
\lambda\in (0,1)
$$
and using that the function $(y,\lambda)\mapsto\varphi(\Phi_\lambda(y))$ is smooth in $(y,\lambda)$.
By~\eqref{e:loc-convergor} $F_\varphi$ is continuous at the boundary interval $(0,1)$.
Since $F_\varphi$ is holomorphic, it is harmonic, so by boundary regularity
for the Dirichlet problem for the Laplacian (see the references in the proof of Lemma~\ref{l:elliptic-solved} below) we see
that $F_\varphi\in C^\infty((0,1)+i[0,\infty))$.
\end{proof}
Passing to the limit in~\eqref{e:fund-sol-eqn} we see that
\begin{equation}
  \label{e:fund-sol-eqn-2}
P(\lambda)E_{\lambda\pm i0}=\delta_0\quad\text{for all}\quad \lambda\in (0,1).
\end{equation}
Note that $E_{\lambda\pm i0}(x)$ is smooth except
on the union of the two lines $\{\ell^+(x,\lambda)=0\}$
and $\{\ell^-(x,\lambda)=0\}$. We remark that $ E_{\lambda \pm i 0 } $ are the
Feynman propagators in dimension one, see \cite[(6.2.1) and page 141]{Hormander1} for the
formula in all dimensions.

\subsection{Reduction to the boundary}
\label{s:redbone}

We now let $\Omega\subset\mathbb R^2$ be a bounded open set with
$C^\infty$ boundary and consider the elliptic boundary value problem
\begin{equation}
  \label{eq:ellip}
P ( \omega   ) u  = f, \quad u |_{\partial \Omega } = 0  , \quad
\Re \omega \in ( 0 , 1 ), \quad \Im \omega \neq 0 .
\end{equation}
\begin{lemm}
  \label{l:elliptic-solved}
For each $f\in\CIc(\Omega)$, the problem~\eqref{eq:ellip} has a unique
solution $u\in C^\infty(\overline\Omega)$.  
\end{lemm}
\Remark The proof shows that if $f$ is fixed, then $u\in C^\infty(\overline\Omega)$ depends holomorphically on~$\omega\in (0,1)\pm i(0,\infty)$.
\begin{proof}
1. We first show that for each $\mu\in \mathbb C\setminus [1,\infty)$ and $s\geq 2$, the map
\begin{equation}
  \label{e:els-1}
\overline H^{s} ( \Omega ) \ni u\ \mapsto\ \big( (\Delta-\mu\partial_{x_2}^2) u, u|_{\partial \Omega } \big) 
\in \overline H^{s-2} ( \Omega ) \oplus H^{s-\frac12} ( \partial \Omega )  
\end{equation}
is a Fredholm operator. (Here $\overline H^s(\Omega)$ denotes
the space of distributions on $\Omega$ which extend to $H^s$
distributions on $\mathbb R^2$.) We apply~\cite[Theorem~20.1.2]{Hormander3}.
The operator $\Delta-\mu\partial_{x_2}^2$ is elliptic, so it remains
to verify that the Shapiro--Lopatinski condition~\cite[Definition~20.1.1(ii)]{Hormander3}
holds for any domain $\Omega$. (An example of an operator for which this
condition fails is $(\partial_{x_1}+i\partial_{x_2})^2$.) In our specific case the Shapiro--Lopatinski condition can
be reformulated as follows: for each basis $(\xi,\eta)$ of $\mathbb R^2$,
if we denote by $\mathcal M$ the space of all bounded solutions on $[0,\infty)$ to the ODE
$$
p(\xi-i\eta\partial_t)u(t)=0,\quad
p(\xi):=\xi_1^2+(1-\mu)\xi_2^2
$$
then the map $u\in\mathcal M\mapsto u(0)$ is an isomorphism. This is equivalent to the requirement that
the quadratic equation $p(\xi+z\eta)=0$ have two roots,
one with $\Im z>0$ and one with $\Im z<0$. To see that the latter condition
holds, we argue by continuity: since $\Delta-\mu\partial_{x_2}^2$
is elliptic, the equation $p(\xi+z\eta)=0$ cannot have any real roots~$z$,
so the condition either holds for all $\mu,\xi,\eta$ or fails for all~$\mu,\xi,\eta$.
However, it is straightforward to check that the condition holds
when $\mu=0$, $\xi=(1,0)$, $\eta=(0,1)$, as the roots
are~$\pm i$. 

\noindent 2.
We next claim that the Fredholm operator~\eqref{e:els-1} is invertible.
We first show that it has index~0, arguing by continuity: since
the operator~\eqref{e:els-1} is continuous in $\mu$ in the operator norm topology,
its index should be independent of~$\mu$. However, for $\mu=0$ we get
the Dirichlet problem for the Laplacian, where \eqref{e:els-1}
is invertible.  

To show that~\eqref{e:els-1} is invertible it remains to prove injectivity,
namely
\begin{equation}
  \label{e:els-2}
u\in H^2(\Omega),\quad
(\Delta-\mu\partial_{x_2}^2)u=0,\quad
u|_{\partial\Omega}=0\quad\Longrightarrow\quad
u=0.
\end{equation}
Multiplying the equation $(\Delta-\mu\partial_{x_2}^2)u=0$ by $\overline u$
and integrating by parts over~$\Omega$, we get
$\|\nabla u\|_{L^2(\Omega)}^2=\mu \|\partial_{x_2}u\|_{L^2(\Omega)}^2$.
Since $0\leq \|\partial_{x_2}u\|^2_{L^2(\Omega)}\leq \|\nabla u\|^2_{L^2(\Omega)}$
and $\mu\not\in [1,\infty)$, we see that $\|\nabla u\|_{L^2(\Omega)}=0$,
which implies that $u=0$, giving~\eqref{e:els-2}.

\noindent 3. Writing
$$
P(\omega)=\partial_{x_2}^2-\omega^2\Delta=-\omega^2(\Delta-\mu\partial_{x_2}^2),\qquad
\mu:=\omega^{-2}\in\mathbb C\setminus [1,\infty)
$$
and using the invertibility of~\eqref{e:els-1}, we see that for each $s\geq 2$ and $f\in \overline H^{s-2}(\Omega)$, 
 the problem~\eqref{eq:ellip}
has a unique solution $u\in\overline H^s(\Omega)$.
When $f\in \CIc(\Omega)$, we may take arbitrary $s$ which
gives that $u\in C^\infty(\overline\Omega)$.
\end{proof}
We will next express the solution to~\eqref{eq:ellip} in terms of boundary data and single layer potentials.
Let us first define the operators used below. Let $T^*\partial\Omega$
be the cotangent bundle of the boundary $\partial\Omega$. Sections of this bundle are differential
1-forms on~$\partial\Omega$ (where we use the positive orientation on $\partial\Omega$);
they can be identified with functions on $\partial\Omega$ by fixing a coordinate $\theta$.
Define the operator
$\mathcal I:\mathcal D'(\partial\Omega;T^*\partial\Omega)\to\mathcal E'(\mathbb R^2)$ as follows:
for $v\in \mathcal D'(\partial\Omega;T^*\partial\Omega)$ and $\varphi\in C^\infty(\mathbb R^2)$,
\begin{equation}
  \label{e:I-def}
\int_{\mathbb R^2}\mathcal I v(x)\varphi(x)\,dx
:=\int_{\partial\Omega} \varphi v.
\end{equation}
Note that $\supp(\mathcal Iv)\subset\partial\Omega$ and we can think
of $\mathcal Iv$ as multiplying $v$ by the delta function on $\partial\Omega$.
Next, let $ E_\omega   $ be the fundamental solution constructed in~\eqref{e:fund-sol-imag} and define
the convolution operator
\begin{equation}
\label{eq:Rlag}
R_\omega  : \mathcal E' ( {\mathbb R}^2 ) \to \mathcal D' ( {\mathbb R}^2 ) ,\quad
R_\omega g := E_\omega * g.
\end{equation}
Using the limiting fundamental solutions $E_{\lambda\pm i0}$ constructed in~\eqref{eq:Elapm}, 
we similarly define the operators $R_{\lambda\pm i0}$ for $\lambda\in (0,1)$ which will be used later.
Finally, for $\omega\in (0,1)+i\mathbb R$, define the `Neumann data' operator
\begin{equation}
  \label{e:N-omega-def}
\mathcal N_\omega:C^\infty(\overline\Omega)\to C^\infty(\partial\Omega;T^*\partial\Omega),\quad
\mathcal N_\omega u:=-2\omega\sqrt{1-\omega^2}\,\mathbf j^*(L^+_\omega u\,d\ell^+(\bullet,\omega)),
\end{equation}
where $\mathbf j:\partial\Omega\to\overline\Omega$ is the embedding map
and $\mathbf j^*$ is the pullback on 1-forms.
We can now reduce the problem~\eqref{eq:ellip} to the boundary:
\begin{lemm}
  \label{l:reduction-to-boundary}
Assume that $u\in C^\infty(\overline\Omega)$ is the solution to~\eqref{eq:ellip}
for some $f\in\CIc(\Omega)$.
Put $U:=\indic_\Omega u\in\mathcal E'(\mathbb R^2)$ and $v:=\mathcal N_\omega u$.
Then
\begin{align}
 \label{eq:Plauf} 
P(\omega)U &= f - \mathcal I v,\\
  \label{eq:U2u}
U &= R_\omega f - R_\omega\mathcal Iv.
\end{align}
\end{lemm}
\Remark Note that we also have
$$
v=2\omega\sqrt{1-\omega^2}\,\mathbf j^*(L^-_\omega u\,d\ell^-(\bullet,\omega)).
$$
Indeed, $0=\mathbf j^*du=\mathbf j^* (L^+_\omega u\, d\ell^+
+L^-_\omega u\, d\ell^-)$ since $u|_{\partial\Omega}=0$ and by~\eqref{eq:L2l}.
\begin{proof}
Let $\varphi\in \CIc(\mathbb R^2)$, then
by~\eqref{e:L-pm-def}
$$
\begin{aligned}
\int_{\mathbb R^2}(P(\omega)U)\varphi\,dx &=
4\int_{\Omega}u L^+_\omega L^-_\omega \varphi\,dx =
-4\int_\Omega(L^+_\omega u)(L^-_\omega \varphi)\,dx \\ &=
\int_\Omega f\varphi\,dx
-4\int_\Omega L^-_\omega(\varphi L^+_\omega u)\,dx \\ & =
\int_\Omega f\varphi\,dx
+2\omega\sqrt{1-\omega^2}\int_{\partial\Omega} \varphi L^+_\omega u\, d\ell^+ \\ & =
\int_{\Omega} (f-\mathcal Iv)\varphi\,dx
\end{aligned}
$$
which gives~\eqref{eq:Plauf}.
The identity~\eqref{eq:U2u} follows from~\eqref{eq:Plauf}, the fundamental solution equation~\eqref{e:fund-sol-eqn},
and the fact that $U$ is a compactly supported distribution:
$
E_\omega * P(\omega)U=(P(\omega)E_\omega) * U=U
$.
\end{proof}
In the notation of Lemma~\ref{l:reduction-to-boundary},
define $S_\omega v:=(R_\omega\mathcal I v)|_{\Omega}\in\mathcal D'(\Omega)$.
Then~\eqref{eq:U2u} implies that
\begin{equation}
  \label{e:u-lambda-form}
u=(R_\omega f)|_{\Omega}-S_\omega v.
\end{equation}
Since $R_\omega f\in C^\infty(\mathbb R^2)$, we have $S_\omega v\in C^\infty(\overline{\Omega})$.
Denote by $\mathcal C_\omega v:=(S_\omega v)|_{\partial\Omega}$ its boundary trace, then
the boundary condition $u|_{\partial\Omega}=0$ gives the following equation on~$v$:
\begin{equation}
  \label{e:C-lambda-equator}
\mathcal C_\omega v=(R_\omega f)|_{\partial\Omega}.
\end{equation}
This motivates the study of the operator $S_\omega$ in~\S\ref{s:slp}
and of the operator $\mathcal C_\omega$ in~\S\ref{s:restricted-slp}.

\subsection{Single layer potentials}
\label{s:slp}

We now introduce single layer potentials. 
For $\omega\in\mathbb C$ with $0<\Re\omega<1$
and $\Im\omega\neq 0$
the single layer potential is the operator $S_\omega:\mathcal D'(\partial\Omega;T^*\partial\Omega)\to \mathcal D'(\Omega)$ given by
\begin{equation}
  \label{e:s-lambda-def}
S_\omega v:=(E_\omega*\mathcal Iv)|_\Omega,\quad
v\in \mathcal D'(\partial\Omega;T^*\partial\Omega).
\end{equation}
Here $E_\omega\in\mathcal D'(\mathbb R^2)$ is the fundamental solution defined in~\eqref{e:fund-sol-imag}
and the operator $\mathcal I:\mathcal D'(\partial\Omega;T^*\partial\Omega)\to\mathcal E'(\mathbb R^2)$
is defined in~\eqref{e:I-def}.
Similarly, if $\lambda\in (0,1)$ and $\Omega$ is $\lambda$-simple (see Definition~\ref{d:1}) then we can define
operators
\begin{equation}
  \label{e:s-lambda-def-real}
S_{\lambda\pm i0}:\mathcal D'(\partial\Omega;T^*\partial\Omega)\to \mathcal D'(\Omega)
\end{equation}
by the formula~\eqref{e:s-lambda-def}, using the limiting distributions $E_{\lambda\pm i0}$ defined
in~\eqref{eq:Elapm}.

If we fix a positively oriented coordinate $\theta$ on $\partial\Omega$
and use it to identify $\mathcal D'(\partial\Omega;T^*\partial\Omega)$
with $\mathcal D'(\partial\Omega)$, then the action of $S_\omega$ on smooth functions
is given by
\begin{equation}
  \label{e:s-lambda-integrated}
S_\omega (f\,d\theta)(x)=\int_{\partial\Omega}E_\omega(x-y)f(y)\,d\theta(y),\quad
f\in C^\infty(\partial\Omega),\quad
x\in\Omega
\end{equation}
and similarly for $S_{\lambda\pm i0}$.

We now discuss the mapping properties of $S_\omega$,
in particular showing that $S_\omega v,S_{\lambda\pm i0}v\in C^\infty(\overline\Omega)$
when $v\in C^\infty(\partial\Omega;T^*\partial\Omega)$.
We break the latter into two cases:

\subsubsection{The non-real case}

We first consider the case $\Im\omega\neq 0$. We use the following standard result,
which is a version of the Sochocki--Plemelj theorem:
\begin{lemm}
  \label{l:sokhotski}
Assume that $\Omega_0\subset\mathbb C$ is a bounded open set with $C^\infty$
boundary (oriented in the positive direction). For $f\in C^\infty(\partial\Omega_0)$, define
$u\in C^\infty(\Omega_0)$ by
$$
u(z)=\int_{\partial\Omega_0} {f(w)\,dw\over w-z},\quad
z\in\Omega_0.
$$
Then $u$ extends smoothly to the boundary and the operator $f\mapsto u$
is continuous $C^\infty(\partial\Omega_0)\to C^\infty(\overline{\Omega}_0)$.
\end{lemm}
\Remark In the (unbounded) model case $\Omega_0=\{\Im z>0\}$,
we have for each $f\in\CIc(\mathbb R)$
$$
u(x+iy)=\int_{\mathbb R} {f(t)\,dt\over t-x-iy},
\quad
y>0.
$$
We see in particular that the function $x\mapsto\lim_{y\to 0+}\partial_y^k u(x+iy)$ is given by
the convolution of $f$ with $(-1)^{k+1}i^kk!(x+i0)^{-k-1}$.
\begin{proof}
Let $\tilde f\in\CIc(\mathbb C)$ be an almost analytic extension of $f$:
that is, $\tilde f|_{\partial\Omega_0}=f$ and $\partial_{\bar z}\tilde f$
vanishes to infinite order on $\partial\Omega_0$. (See for example~\cite[Lemma~4.30]{DZ-Book}
for the existence of such an extension.)
Denote by $dm$ the Lebesgue measure on $\mathbb C$. By the
Cauchy--Green formula (see for instance~\cite[(3.1.11)]{Hormander1}), we have
\[
u ( z ) =  2   \pi i \tilde f ( z ) + 2i 
\int_{\Omega_0 } \frac{ \partial_{\bar w}  \tilde f ( w )}{ w-z} \,dm ( w),\quad
z\in \Omega_0
\]
and this extends smoothly to $ z\in\mathbb C$: indeed,
the second term on the right-hand side is the convolution of
the distribution $-2iz^{-1}\in L^1_{\loc}(\mathbb C)$ with $\indic_{\Omega_0}\partial_{\bar z}\tilde f\in \CIc(\mathbb C)$.
\end{proof}
We now come back to the mapping properties of single layer potentials:
\begin{lemm}
  \label{l:singl-1}
Assume that $0<\Re\omega<1$ and $\Im\omega\neq 0$. Then
$S_\omega$ is a continuous operator from~$C^\infty(\partial\Omega;T^*\partial\Omega)$
to~$C^\infty(\overline\Omega)$.
\end{lemm}
\Remark With more work, it is possible to show that $S_\omega$ is actually continuous $C^\infty(\partial\Omega;T^*\partial\Omega)\to C^\infty(\overline\Omega)$
\emph{uniformly} as $\Im\omega\to 0$, with limits being the operators $S_{\lambda\pm i0}$,
$\lambda=\Re\omega$.
However, our proof of Lemma~\ref{l:singl-1} only shows the mapping property for any fixed
non-real $\omega$. This is enough for our purposes since we have weak convergence of
$S_{\lambda\pm i\varepsilon}$ to $S_{\lambda\pm i0}$ (Lemma~\ref{l:loc}, see also Lemmas~\ref{l:C-converges}
and~\ref{l:C-lambda-strong-limit} below)
and in~\S\ref{s:restricted-slp} we analyse the behaviour of the \emph{restricted} single layer potentials
uniformly as $\Im\omega\to 0$.
\begin{proof}
Let $v\in C^\infty(\partial\Omega;T^*\partial\Omega)$. Since $E_\omega$ is smooth on~$\mathbb R^2\setminus\{0\}$
and $\mathcal Iv$ is supported on~$\partial\Omega$,
we have $S_\omega v\in C^\infty(\Omega)$. It remains to show that $S_\omega v$
is smooth up to the boundary, and for this it is enough
to verify the smoothness of the derivatives $L^\pm_\omega S_\omega v$
where $L^\pm_\omega$ are defined in~\eqref{e:L-pm-def}. By~\eqref{eq:LlogA} we have (suppressing the dependence of $\ell^\pm$ on~$\omega$ in the notation)
$$
L^\pm_\omega S_\omega v(x)=c_\omega\int_{\partial\Omega}{v(y)\over\ell^\pm(x-y)},\quad
x\in\Omega.
$$
Since $\Im\omega\neq 0$, the maps $x\mapsto\ell^\pm(x)$ are linear isomorphisms
from $\mathbb R^2$ onto $\mathbb C$ (considered as a real vector space).
Using this we write
\begin{equation}
  \label{e:singl-1-1}
L^\pm_\omega S_\omega v(x)=\pm \sgn(\Im\omega)c_\omega \int_{\partial\Omega_\pm}
{f_\pm(w)\,dw\over z-w},\qquad
z:=\ell^\pm(x)\in \Omega_\pm
\end{equation}
where we put $\Omega_\pm:=\ell^\pm(\Omega)\subset\mathbb C$
and define the functions $f_\pm\in C^\infty(\partial\Omega_\pm)$ by the equality
of differential forms $v(y)=f_\pm(\ell^\pm(y))\,d\ell^\pm(y)$
on $\partial\Omega$. Here $\partial\Omega_\pm$
are positively oriented and the sign factor $\pm\sgn(\Im\omega)$ accounts for
the orientation of the map $\ell^\pm$, see Lemma~\ref{l:which-rotate}.

Now $L^\pm_\omega S_\omega v$ extends smoothly to the boundary
by Lemma~\ref{l:sokhotski}.
\end{proof}

\subsubsection{The real case}

We now consider the case $\lambda\in (0,1)$:
\begin{lemm}
  \label{l:singl-2}
Assume that $\lambda\in (0,1)$ and $\Omega$ is $\lambda$-simple
(see Definition~\ref{d:1}). Then
$S_{\lambda\pm i0}$ are continuous operators from $C^\infty(\partial\Omega;T^*\partial\Omega)$
to $C^\infty(\overline\Omega)$.
\end{lemm}
\begin{proof}
1. We focus on the operator $S_{\lambda+i0}$,
noting that $S_{\lambda-i0}$ is related to it by the identity
$$
S_{\lambda-i0} \overline v=\overline{S_{\lambda+i0}v}\quad\text{for all}\quad
v\in \mathcal D'(\partial\Omega;T^*\partial\Omega).
$$
We again suppress the dependence on $\lambda$ in the notation,
writing simply $\ell^\pm(x)$ and~$A(x)$.
Denoting by $H(x)=\indic_{(0,\infty) }(x) $ the Heaviside function, we can rewrite~\eqref{eq:Elapm}
as
$$
\log(A(x)+i0)=\log |\ell^+(x)|+\log|\ell^-(x)|+i\pi H(-A(x)).
$$
We then decompose
\begin{equation}
  \label{e:decker}
S_{\lambda+i0}=c_\lambda(S^+_\lambda+S^-_\lambda+i\pi S^0_\lambda)
\end{equation}
where for all $x\in\Omega$ and $v\in C^\infty(\partial\Omega;T^*\partial\Omega)$
$$
\begin{aligned}
S_\lambda^\pm v(x)&=\int_{\partial\Omega}\log|\ell^\pm(x-y)|v(y),\\
S_\lambda^0 v(x)&=\int_{\partial\Omega}H(-A(x-y))v(y).
\end{aligned}
$$

\noindent 2. Let $v\in C^\infty(\partial\Omega;T^*\partial\Omega)$.
Fix a positively oriented coordinate $\theta$ on $\partial\Omega$ and
write $v=f\,d\theta$ for some $f\in C^\infty(\partial\Omega)$.
We first analyse $S_\lambda^\pm v$, writing it as
$$
S_\lambda^\pm v(x)=g_\pm(\ell^\pm(x)),\quad
g_\pm(t):=\int_{\mathbb R} (\Pi_\lambda^\pm f)(s)\log |t-s|\,ds 
$$
where $\Pi_\lambda^\pm f\in\mathcal E'(\mathbb R)$ are the pushforwards of $f$ by the maps
$\ell^\pm$ defined in~\eqref{e:Pi-pm-def}.
Let $\ell^\pm_{\min}<\ell^\pm_{\max}$ be defined in~\eqref{e:ell-pm-min-max}.
By part~1 of Lemma~\ref{l:pushforwards},
$\Pi_\lambda^\pm f$ is supported in $[\ell^\pm_{\min},\ell^\pm_{\max}]$ and
$$
\sqrt{(s-\ell^\pm_{\min})(\ell^\pm_{\max}-s)}\,\Pi_\lambda^\pm f(s)\ \in\ C^\infty([\ell^\pm_{\min},\ell^\pm_{\max}]).
$$
Using Lemma~\ref{l:log-conv}, we then get
$$
g_\pm\in C^\infty([\ell^\pm_{\min},\ell^\pm_{\max}])
$$
which implies that $S^\pm_\lambda v\in C^\infty(\overline\Omega)$.

\noindent 3. It remains to show that $S^0_\lambda v\in C^\infty(\overline\Omega)$.
We may assume that $v=dF$ for some $F\in C^\infty(\partial\Omega)$,
that is $\int_{\partial\Omega}v=0$.
Indeed, if we are studying $S^0_\lambda v$ near some point $x_0\in\overline\Omega$
then we may take $y_0\in\partial\Omega$ such that $A(x_0-y_0)>0$
and change $v$ in a small neighborhood of $y_0$ so that $S^0_\lambda v(x)$
does not change for $x$ near $x_0$ and $v$ integrates to~0. 

For $s\in (\ell^\pm_{\min},\ell^\pm_{\max})$,
define $x^\pm_{(1)}(s),x^\pm_{(2)}(s)\in\partial\Omega$ by
$$
\ell^\pm(x^\pm_{(1)}(s))=\ell^\pm(x^\pm_{(2)}(s))=s,\quad
\ell^\mp(x^\pm_{(1)}(s))<\ell^\mp(x^\pm_{(2)}(s)).
$$
\begin{figure}
\includegraphics{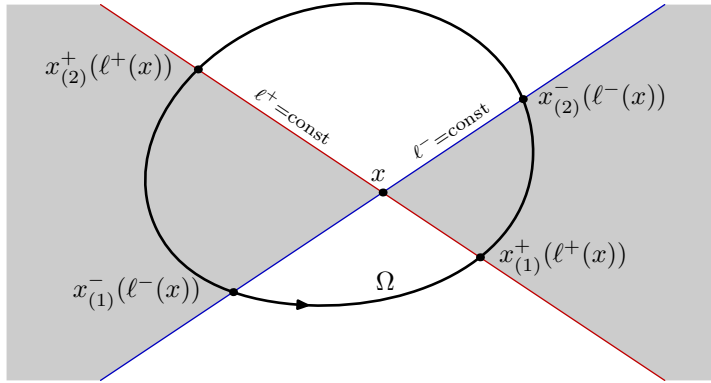}
\caption{A point $x\in\Omega$ and the corresponding projections
$x^\pm_{(j)}(\ell^\pm(x))\in\partial\Omega$.
The shaded region is the set of $y\in\mathbb R^2$ such that
$A(x-y)< 0$.}
\label{f:heaviside}
\end{figure}%
Then for any $x\in\Omega$, the set of $y\in \partial\Omega$ such that $A(x-y)<0$
consists of two intervals of the circle $\partial\Omega$, from $x_{(1)}^+(\ell^+(x))$
to $x_{(2)}^-(\ell^-(x))$ (with respect to the positive orientation on $\partial\Omega$) and from $x_{(2)}^+(\ell^+(x))$
to $x_{(1)}^-(\ell^-(x))$~-- see Figure~\ref{f:heaviside}.
Since $v=dF$, we compute for $x\in\Omega$
$$
S^0_\lambda v(x)=F_-(\ell^-(x))-F_+(\ell^+(x)),\quad
F_\pm (s):=F(x_{(1)}^\pm(s))+F(x_{(2)}^\pm(s)).
$$
By part~2 of Lemma~\ref{l:pushforwards}, we have
$F_\pm=\Upsilon^\pm_\lambda F\in C^\infty([\ell^\pm_{\min},\ell^\pm_{\max}])$.
Thus $S^0_\lambda v\in C^\infty(\overline\Omega)$ as needed.
\end{proof}

\subsubsection{Conormal singularities}

We now study the action of $S_{\lambda +i0}$ on conormal distributions (see~\S\ref{s:con}):
\begin{lemm}
\label{l:con}
Assume that $\lambda\in (0,1)$ and $\Omega$ is $\lambda$-simple. Fix
$y_0\in \partial\Omega\setminus \mathscr C_\lambda$, where the characteristic set $ \mathscr C_\lambda $
was defined in~\eqref{e:no-characteristic}. Then for each $v\in\mathcal D'(\partial\Omega;T^*\partial\Omega)$
we have
$$
v\in I^s(\partial\Omega,N^*_\pm\{y_0\})\quad\Longrightarrow\quad
S_{\lambda+i0}v\in I^{s-{5\over 4}}(\overline\Omega,N_\pm^* \Gamma^\pm_\lambda(y_0)).
$$
Here the positive/negative halves of the conormal bundle $N^*_\pm\{y_0\}\subset T^*\partial\Omega$ are defined using the positive orientation on~$\partial\Omega$; the line segments $\Gamma^\pm_\lambda(y_0)\subset \Omega$ are  defined in~\eqref{e:Gamma-def}
and transverse to the boundary $\partial\Omega$; and $N_\pm^*\Gamma^\pm_\lambda(y_0)$ are defined in~\eqref{e:n-pm-gamma-def}.
\end{lemm}
\begin{proof}
\noindent 1. By Lemma~\ref{l:singl-2} and since $v$ is smooth away from $y_0$, we may assume that
\begin{equation}
  \label{e:con-supp-v}
\supp v\subset U:=\{y\in\partial\Omega\mid \nu^+(y)=\nu^+(y_0),\ \nu^-(y)=\nu^-(y_0)\}
\end{equation}
where  $\nu^\pm (y) = \sgn\partial_\theta \ell^\pm(y)$, see~\eqref{e:omega-pm-def}. 
We denote $\nu^+:=\nu^+(y_0)$, $\nu^-:=\nu^-(y_0)$.

We claim that for all $y\in U$ and $x\in\Omega\setminus \Gamma_\lambda(y)$, where
$\Gamma_\lambda(y):=\Gamma^+_\lambda(y)\cup \Gamma^-_\lambda(y)$,
\begin{equation}
  \label{e:loggor}
\begin{gathered}
\log(A(x-y)+i0)=
\log(\ell^+(x-y)+i\nu^+0)+\log(\ell^-(x-y)-i\nu^-0)
+c_0,\\
c_0=\begin{cases}
2\pi i,&\text{if}\quad
\nu^+=-1\quad\text{and}\quad
\nu^-=1;\\
0,&\text{otherwise.}
\end{cases}
\end{gathered}
\end{equation}
(Here as always we use the branch of 
$ \log $ real on the positive real axis.)
Indeed, fix $x$ and $y$. By Lemma~\ref{l:identitor-2}, we have
$$
\nu^+\ell^-(x-y)>0\quad\text{or}\quad
\nu^-\ell^+(x-y)<0\quad\text{(or both).}
$$
Then there exist $\alpha^+,\alpha^->0$ such that
$\alpha^+\nu^+\ell^-(x-y)-\alpha^-\nu^-\ell^+(x-y)=1$. This implies that
for all $\varepsilon>0$
$$
\begin{aligned}
A(x-y)+i\varepsilon &=
\ell^+(x-y)\ell^-(x-y)+i\varepsilon \\
&= (\ell^+(x-y)+i\alpha^+\nu^+\varepsilon)(\ell^-(x-y)-i\alpha^-\nu^-\varepsilon)+\mathcal O(\varepsilon^2).
\end{aligned}
$$
Letting $\varepsilon\to 0+$, we obtain~\eqref{e:loggor}.

\noindent 2. Fix a coordinate $\theta$ on~$\partial\Omega$ and write $v=f(\theta)\,d\theta$.
Similarly to Step~2 in the proof of Lemma~\ref{l:singl-2} we get from~\eqref{e:loggor}
$$
S_{\lambda+i0}v(x)=c_\lambda\bigg(g_+(\ell^+(x))+g_-(\ell^-(x)) + c_0\int_{\partial\Omega}f(\theta)\,d\theta \bigg)
$$
where, denoting $\log_+ x:=\log(x+i0)$, $\log_- x:=\log(x-i0)$ and using \eqref{e:Pi-pm-def},
\begin{equation}
\label{eq:defgpm} 
g_+:=(\Pi_\lambda^+f)*\log_{\nu^+},\quad
g_-:=(\Pi_\lambda^-f)*\log_{-\nu^-}.
\end{equation}
(Here, $  \pm \nu^\bullet $ is meant as $ \pm $ if $ \nu^\bullet = 1 $ and $ \mp $ 
when $ \nu^\bullet = -1 $.)

By~\eqref{e:con-supp-v}, we have $\supp v\subset U$
where $ \ell^\pm : U \to \mathbb R $ are diffeomorphisms onto their ranges.
Since $v\in I^s(\partial\Omega,N^*_\pm\{y_0\})$ and recalling~\eqref{e:pushf-1}, 
we then have
$$
\Pi_\lambda^+ f\in I^s(\mathbb R,N^*_{\pm\nu^+}\{\ell^+(y_0)\}),\quad
\Pi_\lambda^- f\in I^s(\mathbb R,N^*_{\pm\nu^-}\{\ell^-(y_0)\}).
$$
By Lemma~\ref{l:log-conormal}, we see that
$$
g_\pm\in I^{s-1}(\mathbb R,N^*_{\pm\nu^\pm}\{\ell^\pm(y_0)\}),\quad
g_\mp\in C^\infty.
$$
Using the Fourier characterization of conormal distributions reviewed in~\S\ref{s:con},
we see that $S_{\lambda+i0}f\in I^{s-{5\over 4}}(\overline\Omega,N_\pm^*\Gamma^\pm_\lambda(y_0))$
as needed.
\end{proof}
\Remark In \S \ref{s:liap} we will apply this result to elements of 
$
I^{s} (\partial\Omega, N_+^* \Sigma^-_\lambda \sqcup N^*_-\Sigma^+_\lambda )$,
defined in~\eqref{e:I-s-pm}, where $\Sigma^\pm_\lambda$ are defined in~\eqref{e:Sigma-pm-def}. Lemma~\ref{l:con} gives
\begin{equation}
\label{eq:con-new} 
\begin{gathered}
S_{\lambda+i0} :  I^s( \partial\Omega,N_+^* \Sigma^-_\lambda \sqcup N^*_-\Sigma^+_\lambda ) 
\to I^{s-{5\over 4}}(\overline\Omega,\Lambda^-( \lambda ) )
\end{gathered}
\end{equation}
where $\Lambda^-(\lambda)=N_-^*\Gamma^+_\lambda(\Sigma^+_\lambda)\sqcup N_+^*\Gamma^-_\lambda(\Sigma^-_\lambda)
=N_+^* \Gamma^+_\lambda ( \Sigma^-_\lambda ) \sqcup 
N_-^* \Gamma^-_\lambda ( \Sigma^+_\lambda )$ is defined in~\eqref{eq:defLa}.
Here we define the conormal spaces on the right-hand side similarly to~\eqref{e:I-s-pm}:
$$
I^{s}(\overline\Omega,\Lambda^-( \lambda )):=
I^{s}(\overline\Omega,N_+^* \Gamma^+_\lambda ( \Sigma^-_\lambda ) )
+
I^{s}(\overline\Omega,N_-^* \Gamma^-_\lambda ( \Sigma^+_\lambda )).
$$

\subsection{The restricted single layer potentials}
  \label{s:restricted-slp}
  
We now study the restricted operators
\begin{equation}
 \label{eq:defsingle}
\mathcal C_\omega: C^\infty(\partial\Omega; T^* \partial \Omega )\to C^\infty(\partial\Omega),\quad
\mathcal C_\omega v :=(S_\omega v )|_{\partial\Omega},
\end{equation}
given by the boundary trace of $S_\omega v\in C^\infty(\overline\Omega)$, see Lemma~\ref{l:singl-1}.
When $\lambda$ is real and $ \Omega $ is $ \lambda$-simple (see Definition \ref{d:1})
we have two operators $\mathcal C_{\lambda\pm i0}$ obtained by restricting $S_{\lambda\pm i0}$, see Lemma~\ref{l:singl-2}.
From~\eqref{e:s-lambda-integrated} we have for $v\in C^\infty(\partial\Omega;T^*\partial\Omega)$
\begin{equation}
  \label{e:C-lambda-integrated}
\mathcal C_\omega v(x)=\int_{\partial\Omega}E_\omega(x-y)\,v(y),\quad
x\in\partial\Omega,
\end{equation}
with the integration in~$y$, and same is true for $\omega$ replaced with $\lambda\pm i0$.
Later in~\eqref{e:summa-bounder} we show that $\mathcal C_\omega$ and $\mathcal C_{\lambda\pm i0}$
extend to continuous operators $\mathcal D'(\partial\Omega;T^*\partial\Omega)\to \mathcal D'(\partial\Omega)$.

Composing $\mathcal C_\omega$ with the differential $d:C^\infty(\partial\Omega)\to C^\infty(\partial\Omega;T^*\partial\Omega)$
we get the operator
\[    
d \mathcal C_\omega  :
C^\infty ( \partial \Omega; T^* \partial \Omega )  \to C^\infty ( \partial \Omega; T^* \partial \Omega ).
\]
In this section we assume that
\begin{equation}
  \label{e:restrict-lambdas}
\omega=\lambda+i\varepsilon,\quad
\varepsilon>0,
\end{equation}
where $\lambda\in (0,1)$ is chosen so that $\Omega$ is $\lambda$-simple.
Our main result here is a microlocal description of $d\mathcal C_\omega$
uniformly as $\varepsilon\to 0+$, see Proposition~\ref{p:summa} below.
(This description is also locally uniform in~$\lambda$, see Remark~1 after Proposition~\ref{p:summa}.)

For convenience, we fix a positively oriented coordinate $\theta\in\mathbb S^1$ on~$\partial\Omega$
and identify 1-forms on $\partial\Omega$ with functions on~$\mathbb S^1$ by writing $v=f(\theta)\,d\theta$.
Let $\mathbf x:\mathbb S^1\to \partial\Omega$ be the corresponding parametrization map. Let
\begin{equation}
  \label{e:gamma-omega-def}
\gamma_\lambda^\pm:\mathbb S^1\to\mathbb S^1,\quad
\gamma^\pm(\mathbf x(\theta),\lambda)=\mathbf x(\gamma_\lambda^\pm(\theta))
\end{equation}
be the orientation reversing involutions on $\mathbb S^1$ induced by the maps $\gamma^\pm(\bullet,\lambda)$
defined in~\eqref{e:gamma-pm-def}.

\subsubsection{A weak convergence statement}

Before starting the microlocal analysis of $d\mathcal C_\omega$, we show 
that $\mathcal C_{\lambda+i\varepsilon}\to \mathcal C_{\lambda+i0}$ as $\varepsilon\to 0+$
in a weak sense in $x,y$ but uniformly with all derivatives in $\lambda$.
A stronger convergence will be shown later in Lemma~\ref{l:C-lambda-strong-limit}.
We use the letter $\mathscr O(\mathcal J+i[0,\infty))$ for spaces of holomorphic
functions that are smooth up to the boundary interval
$\mathcal J$, introduced in~\eqref{e:O-analytic}
and in the statement of Lemma~\ref{l:loc}.
\begin{lemm}
  \label{l:C-converges}
Let $\mathcal J\subset (0,1)$ be an open interval such that $\Omega$ is $\lambda$-simple
for all $\lambda\in\mathcal J$.
Then the Schwartz kernel of the operator
\begin{equation}
  \label{e:C-converges-family}
\omega \in \mathcal J+i[0,\infty)\ \mapsto\ \begin{cases}
\mathcal C_{\omega},&\Im\omega>0,\\
\mathcal C_{\lambda+i0},&\omega=\lambda\in\mathcal J
\end{cases}
\end{equation}
lies in $\mathscr O(\mathcal J+i[0,\infty);\mathcal D'(\partial\Omega\times\partial\Omega))$.
\end{lemm}
\begin{proof}
1. The holomorphy of~\eqref{e:C-converges-family} when $\Im\omega>0$ follows by differentiating~\eqref{e:C-lambda-integrated}
(one can cut away from the singularity at $x=y$ and represent the pairing of~\eqref{e:C-converges-family}
with any element of $C^\infty(\partial\Omega\times\partial\Omega)$
as the locally uniform limit of a sequence of holomorphic functions).
The smoothness of the restriction of~\eqref{e:C-converges-family}
to~$\mathcal J$
can be shown using the decomposition~\eqref{e:decker} and the $\lambda$-dependent local coordinates
introduced in Step~2 of the present proof.
Arguing similarly to Step~3 in the proof of Lemma~\ref{l:loc}
and recalling~\eqref{e:C-lambda-integrated}, we then see that it suffices
to show the following convergence statement for all $\varphi\in C^\infty(\partial\Omega\times\partial\Omega)$:
$$
\begin{gathered}
\int\limits_{\partial\Omega\times\partial\Omega}E_{\omega_j}(x-y)\varphi(x,y)\,d\theta(x)d\theta(y)\to
\int\limits_{\partial\Omega\times\partial\Omega}E_{\lambda+i0}(x-y)\varphi(x,y)\,d\theta(x) d\theta(y)\\
\quad\text{for all}\quad
\omega_j\to\lambda\in\mathcal J,\quad
\Im\omega_j>0.
\end{gathered}
$$
Similarly to~\eqref{e:decker} we decompose
$$
E_{\omega}=E_{\omega}^++E_{\omega}^-+E_{\omega}^0,\quad
E_{\omega}^\pm(x):=c_{\omega}\log|\ell^\pm(x,\omega)|,\quad
E_{\omega}^0:=ic_{\omega}\Im\log A(x,\omega)
$$
and similarly for $E_{\lambda+i0}$.
It suffices to show that for $\bullet=+,-,0$ we have
\begin{equation}
  \label{e:C-decomcon}
\int\limits_{\partial\Omega\times\partial\Omega}E_{\omega_j}^\bullet(x-y)\varphi(x,y)\,d\theta(x)d\theta(y)\to
\int\limits_{\partial\Omega\times\partial\Omega}E_{\lambda+i0}^\bullet(x-y)\varphi(x,y)\,d\theta(x) d\theta(y).
\end{equation}

\noindent 2.
We have $E_{\omega_j}^\bullet(x-y)\to E_{\lambda+i0}^\bullet(x-y)$ for almost every $(x,y)\in\partial\Omega\times\partial\Omega$,
more specifically for all $(x,y)$ such that $y\not\in\{x,\gamma^+(x,\lambda),\gamma^-(x,\lambda)\}$.
This gives~\eqref{e:C-decomcon} for $\bullet=0$ by the Dominated Convergence Theorem since
$|\Im\log A(x,\omega_j)|\leq \pi$.

To see~\eqref{e:C-decomcon} for $\bullet=+$ (a similar argument works for $\bullet=-$),
we follow Step~2 of the proof of Lemma~\ref{l:loc}. Instead of the family of linear isomorphisms $\Phi_{\lambda}$
used there we choose a specific local coordinate $\theta_j$ on $\partial\Omega$ which depends
on $\lambda_j=\Re\omega_j$.
More precisely, using a partition of unity we see that it suffices to show
that each $(x_0,y_0)\in \partial\Omega\times\partial\Omega$ has a neighborhood~$U$
such that~\eqref{e:C-decomcon} holds for all $\varphi\in\CIc(U)$.
Now we consider four cases (corresponding to~\S\S\ref{s:slp-away}--\ref{s:slp-char} below):
\begin{itemize}
\item $\ell^+(x_0,\lambda)\neq\ell^+(y_0,\lambda)$: we can use the Dominated Convergence Theorem
since $E_{\omega_j}^+(x-y)$ is bounded
uniformly in~$j$ and in $(x,y)\in U$ by~\eqref{e:limit-E-bdor-0}.
\item $y_0=x_0\neq \gamma^+(x_0,\lambda)$: we choose the coordinate $\theta_j=\ell^+(x,\lambda_j)$
near $x_0$. Then the argument in the proof of Lemma~\ref{l:loc}
goes through, using that $\log|\theta-\theta'|$ is a locally integrable function
of $(\theta,\theta')\in\mathbb R^2$.
\item $y_0=\gamma^+(x_0,\lambda)\neq x_0$: we again choose the coordinate
$\theta_j=\ell^+(x,\lambda_j)$ near $x_0$ and near $y_0$, and the argument goes through as in the previous case.
\item $x_0=y_0=\gamma^+(x_0,\lambda)$: assume that $x_0=x^+_{\min}(\lambda)$ is the minimum point of $\ell^+(\bullet,\lambda)$ on~$\partial\Omega$ (the case when $x_0$ is the maximum point is handled similarly).
We choose the coordinate $\theta_j$ near $x_0$ given by~\eqref{e:factorize}:
$$
\ell^+(x,\lambda_j)=\ell^+(x^+_{\min}(\lambda_j),\lambda_j)+\theta_j(x)^2.
$$
Then the argument in the proof of Lemma~\ref{l:loc}
goes through, using that $\log|\theta^2-(\theta')^2|$
is a locally integrable function of $(\theta,\theta')\in \mathbb R^2$.
\end{itemize}
\end{proof}

\subsubsection{Decomposition into $T^\pm_\omega$}

Since the linear functions $ \ell^\pm ( x,\omega )$ are dual to the vector fields $ L^\pm_\omega $ (see~\eqref{eq:L2l}), we have
\begin{equation}
  \label{e:C-lambda-sum}
d\mathcal C_\omega =T^+_\omega+T^-_\omega
\end{equation}
where the operators $T^\pm_\omega:C^\infty(\partial\Omega;T^*\partial\Omega)\to C^\infty(\partial\Omega;T^*\partial\Omega)$
are given by (with $j$ the embedding map)
\begin{equation}
  \label{e:T-pm-def}
T^\pm_\omega v=j^*\big((L^\pm_\omega S_\omega v)d\ell^\pm\big),\qquad
j:\partial\Omega\to\overline\Omega.
\end{equation}
Let $ K_\omega^\pm ( \theta, \theta' )\in\mathcal D'(\mathbb S^1\times\mathbb S^1)$ be the Schwartz kernel of $ T_\omega^\pm $,
that is
\begin{equation}
  \label{e:schwartzer}
T^\pm_\omega v(\theta)=\big(\partial_\theta\ell^\pm(\mathbf x(\theta),\omega)\big){L^\pm_\omega S_\omega v(\mathbf x(\theta))}\,d\theta=\bigg(\int_{\mathbb S^1}K^\pm_\omega(\theta,\theta')f(\theta')\,d\theta'\bigg)d\theta, 
\end{equation}
where we put $v=f(\theta)\,d\theta$.
Recalling the integral definition~\eqref{e:s-lambda-integrated} of $S_\omega$, the formula~\eqref{e:fund-sol-imag}
for $E_\omega$ (which in particular shows that $E_\omega$ is smooth on $\mathbb R^2\setminus \{0\}$), and the identity~\eqref{eq:LlogA}, we see that
$K_\omega^\pm$ is smooth on $(\mathbb S^1\times\mathbb S^1)\setminus \{\theta\neq\theta'\}$ and
\begin{equation}
  \label{e:K-lambda-away-diagonal}
K_\omega^\pm ( \theta, \theta' ) = c_\omega
 \frac{ \partial_\theta  \ell^\pm ( \mathbf x ( \theta ) , \omega )  } {\ell^\pm  ( \mathbf x(\theta)  - \mathbf x(\theta')   , \omega )} , \quad \theta  \neq   \theta'    .
\end{equation}
 
\subsubsection{Away from the singularities}
\label{s:slp-away}

Define the sets
\begin{equation}
\label{e:sing-sets-def}
\begin{aligned}
\Diag&:=\{(\theta,\theta)\mid \theta\in\mathbb S^1\},\\
\Refl^\pm_\lambda&:=\{(\theta,\gamma_\lambda^\pm(\theta))\mid \theta\in\mathbb S^1\}.
\end{aligned}
\end{equation}
Note that the intersection
\begin{equation}
  \label{e:diag-ref-int}
\Diag\cap\Refl^\pm_\lambda=\{(\theta,\theta)\mid \theta\in\mathbb S^1,\ \partial_\theta \ell^\pm(\mathbf x(\theta),\lambda)=0\}
\end{equation}
corresponds to the critical points $x^\pm_{\min}(\lambda),x^\pm_{\max}(\lambda)$ of $\ell^\pm(\bullet,\lambda)$
on $\partial\Omega$ (see Definition~\ref{d:1}). At these points the operator $P(\lambda)$ is characteristic with respect to~$\partial\Omega$.

We start the analysis of the uniform behaviour of~$K_\omega^\pm$ as $\varepsilon=\Im\omega\to 0$
by showing that the singularities are contained in $\Diag\cup\Refl^\pm_\lambda$:
\begin{lemm}
\label{l:smooth}
We have
$$
K^\pm_\omega|_{(\mathbb S^1\times\mathbb S^1)\setminus (\Diag\cup\Refl^\pm_\lambda)}
\ \in\ C^\infty\big((\mathbb S^1\times\mathbb S^1)\setminus (\Diag\cup\Refl^\pm_\lambda)\big)
$$
smoothly in $\varepsilon$ up to $\varepsilon=0$.
\end{lemm}
\begin{proof}
This follows immediately from~\eqref{e:K-lambda-away-diagonal}. Indeed,
for $(\theta,\theta')\notin\Diag\cup\Refl^\pm_\lambda$, we have $\ell^\pm(\mathbf x(\theta),\lambda)\neq
\ell^\pm(\mathbf x(\theta'),\lambda)$ and thus the denominator in~\eqref{e:K-lambda-away-diagonal}
is nonvanishing when $\varepsilon=0$.
\end{proof}

\subsubsection{Noncharacteristic diagonal}

We next consider the singularities of $ K_\omega^\pm ( \theta , \theta' ) $ near the diagonal but
away from the characteristic set $\Diag\cap\Refl^\pm_\lambda$.
In that case the structure of the kernel is similar to the model case
\eqref{eq:modelC}:
\begin{lemm}
\label{l:diag}
Take $\theta_0\in\mathbb S^1$ such that $\gamma^\pm_\lambda(\theta_0)\neq\theta_0$.
Then for $ \theta, \theta'  $ in some neighbourhood $U$ of~$\theta_0 $
and $\varepsilon=\Im\omega>0$ small enough, we have
\begin{equation}
\label{eq:diag}
K_\omega^\pm ( \theta, \theta' ) = c_\omega
( \theta - \theta' \pm i 0 )^{-1} 
+\mathscr K_\omega^\pm(\theta,\theta')
, 
\end{equation}
where $ \mathscr K_\omega^\pm \in C^\infty ( U\times U ) $ is smooth in $\theta,\theta',\varepsilon$
up to $\varepsilon=0$.
\end{lemm} 
\begin{proof}
1.
Fix some smooth vector field $\mathbf v(\theta)$ on $\partial\Omega$ which points inwards.
We have for all $v=f(\theta)\,d\theta\in C^\infty(\partial\Omega;T^*\partial\Omega)$,
$$
\begin{aligned}
\int_{\mathbb S^1}K^\pm_\omega(\theta,\theta')f(\theta')\,d\theta'
&=\big(\partial_\theta\ell^\pm(\mathbf x(\theta),\omega)\big)\lim_{\delta\to 0+}
L^\pm_\omega S_\omega v(\mathbf x(\theta)+\delta\mathbf v(\theta))\\
&=c_\omega\big(\partial_\theta\ell^\pm(\mathbf x(\theta),\omega)\big)\lim_{\delta\to 0+}
\int_{\mathbb S^1} {f(\theta')\over \ell^\pm(\mathbf x(\theta)-\mathbf x(\theta')+\delta\mathbf v(\theta),\omega)}\,d\theta'
\end{aligned}
$$
where the limit is in $C^\infty(\mathbb S^1)$.
Here in the first equality we use the definition~\eqref{e:schwartzer}
of~$K_\omega^\pm$ (recalling that $S_\omega v\in C^\infty(\overline\Omega)$ by
Lemma~\ref{l:singl-1}).
In the second equality we use the definition~\eqref{e:s-lambda-integrated} of $S_\omega$, the formula~\eqref{e:fund-sol-imag} for $E_\omega$, and the identity~\eqref{eq:LlogA}.
   
Since $\partial_\theta\ell^\pm(\mathbf x(\theta),\lambda)\neq 0$
at $\theta=\theta_0$, we factorize for $\theta,\theta'$ in some neighborhood $U$ of~$\theta_0$
and $\varepsilon=\Im\omega$ small enough
$$
\ell^\pm(\mathbf x(\theta)-\mathbf x(\theta'),\omega)=
G^\pm_\omega(\theta,\theta')(\theta-\theta')
$$
where $G^\pm_\omega(\theta,\theta')$ is a nonvanishing smooth function of $\theta,\theta',\varepsilon$
up to $\varepsilon=0$ and
\begin{equation}
  \label{e:diag-G-1}
G_\omega^\pm(\theta,\theta)=\partial_\theta\ell^\pm(\mathbf x(\theta),\omega),\quad
\theta\in U.
\end{equation}
Therefore, for $(\theta,\theta')\in U\times U$ we have
\begin{equation}
  \label{e:diagor-1}
K^\pm_\omega(\theta,\theta')={c_\omega\partial_\theta\ell^\pm(\mathbf x(\theta),\omega)\over G^\pm_\omega(\theta,\theta')}\lim_{\delta\to 0+}
\bigg(\theta-\theta'+\delta{\ell^\pm(\mathbf v(\theta),\omega)\over G^\pm_\omega(\theta,\theta')}\bigg)^{-1}
\end{equation}
with the limit in $\mathcal D'(U\times U)$.

\noindent 2.
We next claim that if $U$ is a small enough
neighborhood of $\theta_0$, then for all $(\theta,\theta')\in U\times U$ and $\Im\omega=\varepsilon>0$ small enough
\begin{equation}
  \label{e:diagor-2}
\pm\Im{\ell^\pm(\mathbf v(\theta),\omega)\over G^\pm_\omega(\theta,\theta')}>0.
\end{equation}
When $\omega=\lambda$ is real, the expression~\eqref{e:diagor-2} is equal to~0. Thus it suffices
to check that for all $(\theta,\theta')\in U\times U$
\begin{equation}
  \label{e:diagor-3}
\pm\partial_\varepsilon|_{\varepsilon=0}\Im{\ell^\pm(\mathbf v(\theta),\lambda+i\varepsilon)\over
G^\pm_{\lambda+i\varepsilon}(\theta,\theta')}>0.
\end{equation}
It is enough to consider the case $\theta=\theta'=\theta_0$, in which case the left-hand side
of~\eqref{e:diagor-3} equals
$$
\pm\partial_\varepsilon|_{\varepsilon=0}\Im{\ell^\pm(\mathbf v(\theta_0),\lambda+i\varepsilon)\over
\ell^\pm(\partial_\theta\mathbf x(\theta_0),\lambda+i\varepsilon)}.
$$
By~\eqref{e:l-pm-differential} and since $\ell^\pm$ is holomorphic in $\omega$ it then suffices to check that
\begin{equation}
  \label{e:diagor-4}
\pm\big(\ell^\mp ( \mathbf v(\theta_0) , \lambda ) \ell^\pm ( \partial_\theta \mathbf x(\theta_0) , \lambda ) - \ell^\pm ( \mathbf v(\theta_0) , 
\lambda ) \ell^\mp ( \partial_\theta \mathbf  x(\theta_0)  , \lambda )\big)>0.
\end{equation}
The inequality~\eqref{e:diagor-4} follows from the fact that $x\mapsto (\ell^+(x,\lambda),\ell^-(x,\lambda))$
is an orientation preserving linear map on $\mathbb R^2$ and $\partial_\theta\mathbf x(\theta_0),\mathbf v(\theta_0)$
form a positively oriented basis of $\mathbb R^2$ since the parametrization $\mathbf x(\theta)$ is positively
oriented and $\mathbf v(\theta)$ points inside $\Omega$. This finishes the proof of~\eqref{e:diagor-2}.

\noindent 3.
By Lemma~\ref{l:inverse-converger} (see also~\eqref{e:prepare-remark}),
with $\delta$ taking the role of~$\varepsilon$, the distributional limit on the right-hand side of~\eqref{e:diagor-1}
is equal to $(\theta-\theta'\pm i0)^{-1}$. Therefore 
\begin{equation}
  \label{e:kopper}
K^\pm_\omega(\theta,\theta')={c_\omega\partial_\theta\ell^\pm(\mathbf x(\theta),\omega)\over G^\pm_\omega(\theta,\theta')}(\theta-\theta'\pm i0)^{-1}.
\end{equation}
By~\eqref{e:diag-G-1} we can write for some $\mathscr K^\pm_\omega(\theta,\theta')$ which is smooth in $\theta,\theta',\varepsilon$ up to~$\varepsilon=0$, 
$$
{c_\omega\partial_\theta\ell^\pm(\mathbf x(\theta),\omega)\over G^\pm_\omega(\theta,\theta')}
=c_\omega+\mathscr K^\pm_\omega(\theta,\theta')(\theta-\theta')
$$
which gives~\eqref{eq:diag} since $(\theta-\theta')(\theta-\theta'\pm i0)^{-1}=1$.
\end{proof}

\subsubsection{Noncharacteristic reflection}
  \label{s:slp-ref}

We now move to the singularities on the reflection sets $\Refl^\pm_\lambda$, again
staying away from the characteristic set $\Diag\cap\Refl^\pm_\lambda$:
\begin{lemm}
\label{l:sing}
Take $\theta_0\in\mathbb S^1$ such that $\gamma^\pm_\lambda(\theta_0)\neq\theta_0$.
Then there exists neighborhoods $U,U'=\gamma^\pm_\lambda(U)$ of $\gamma^\pm_\lambda(\theta_0),\theta_0$ such that
for $(\theta,\theta')\in U\times U'$ and $\varepsilon=\Im\omega>0$ small enough, we have 
\begin{equation}
\label{eq:sing}
\begin{gathered}  K_\omega^\pm ( \theta, \theta' ) = \tilde c^\pm_\omega(\theta') \big( \gamma^\pm_\lambda ( \theta)  - \theta'   \pm  i \varepsilon 
z^\pm_\omega(\theta') \big)^{-1} 
+\mathscr K_\omega^\pm(\theta,\theta'),
\end{gathered}
\end{equation}
where $\mathscr K^\pm_\omega\in C^\infty(U\times U')$ is smooth in~$\theta,\theta',\varepsilon$ up to $\varepsilon=0$,
the functions $\tilde c^\pm_\omega(\theta')$ and $z^\pm_\omega(\theta')$ are smooth in $\theta',\varepsilon$
up to~$\varepsilon=0$,
and
\begin{equation}
\label{eq:Gpm} \tilde c^\pm_\omega(\theta') = { c_\omega \over
\partial_{\theta'} \gamma^\pm_\lambda ( \theta' )} + \mathcal O ( \varepsilon ),\quad
\Re z^\pm_\omega ( \theta' ) \geq c>0
\end{equation}
where $c$ is independent of~$\varepsilon,\theta'$.
\end{lemm}
\begin{proof}
1. Recall that $\omega=\lambda+i\varepsilon$. We take Taylor expansions of $\ell^\pm(x,\omega)$ at
$\varepsilon=0$, using its holomorphy in $\omega$:
\begin{equation}
  \label{e:ell-pm-expanded}
\ell^\pm(x,\omega)=\ell^\pm(x, \lambda)+i\varepsilon \ell^\pm_1(x, \lambda)+\varepsilon^2\ell^\pm_2(x,\lambda, \varepsilon),\quad
\ell_1^\pm ( x, \lambda ) := \partial_\lambda \ell^\pm ( x, \lambda ) 
\end{equation}
where the coefficients of the linear maps $ x \mapsto \ell^\pm_2 ( x, \lambda, \varepsilon) $ are smooth in~$\varepsilon$
up to~$\varepsilon=0$.
Since $\partial_\theta \ell^\pm(\mathbf x(\theta),\lambda)\neq 0$ at $\theta=\theta_0$, we
factorize for $\theta,\theta'$ in some neighborhoods $U,U'=\gamma^\pm_\lambda(U)$ of $\gamma^\pm_\lambda(\theta_0),\theta_0$
$$
\ell^\pm(\mathbf x(\theta)-\mathbf x(\theta'),\lambda)
=\ell^\pm(\mathbf x(\gamma^\pm_\lambda(\theta))-\mathbf x(\theta'),\lambda)
=G^\pm_\lambda(\theta,\theta')(\gamma^\pm_\lambda(\theta)-\theta')
$$
where $G^\pm_\lambda\in C^\infty(U\times U';\mathbb R)$ is nonvanishing and
\begin{equation}
  \label{e:refor-1}
G^\pm_\lambda(\gamma^\pm_\lambda(\theta'),\theta')=\partial_{\theta'}\ell^\pm(\mathbf x(\theta'),\lambda).
\end{equation}
Hence for $(\theta,\theta')\in U\times U'$
$$
\begin{aligned}
\ell^\pm(\mathbf x(\theta)-\mathbf x(\theta'),\omega)
&=G^\pm_\lambda(\theta,\theta')(\gamma^\pm_\lambda(\theta)-\theta'\pm i\varepsilon \psi^\pm_\omega(\theta,\theta')),\\
\psi^\pm_\omega(\theta,\theta')&:=\pm{\ell_1^\pm(\mathbf x(\theta)-\mathbf x(\theta'),\lambda)
-i\varepsilon\ell_2^\pm(\mathbf x(\theta)-\mathbf x(\theta'),\lambda,\varepsilon)\over G^\pm_\lambda(\theta,\theta')}.
\end{aligned}
$$
By~\eqref{e:K-lambda-away-diagonal} we have for $(\theta,\theta')\in U\times U'$
$$
K^\pm_\omega(\theta,\theta')=F^\pm_\omega(\theta,\theta')\big(\gamma^\pm_\lambda(\theta)-\theta'\pm i\varepsilon \psi^\pm_\omega(\theta,\theta')\big)^{-1},\quad
F^\pm_\omega(\theta,\theta'):=c_\omega{\partial_\theta\ell^\pm(\mathbf x(\theta),\omega)\over G^\pm_\lambda(\theta,\theta')}.
$$
Note that $\psi^\pm_\omega(\theta,\theta')$ and $F^\pm_\omega(\theta,\theta')$ are 
smooth in $\theta,\theta',\varepsilon$ up to $\varepsilon=0$.

\noindent 2. We next claim that $\Re \psi^\pm_\omega(\theta,\theta')\geq c>0$ for $\varepsilon$ small enough
and $(\theta,\theta')\in U\times U'$, if $U,U'$ are sufficiently small neighborhoods of $\gamma^\pm_\lambda(\theta_0),\theta_0$.
For that it suffices to show that
\begin{equation}
  \label{e:refor-2}
\pm {\ell_1^\pm(\mathbf x(\gamma^\pm_\lambda(\theta_0))-\mathbf x(\theta_0),\lambda)\over G^\pm_\lambda(\gamma^\pm_\lambda(\theta_0),\theta_0)}>0.
\end{equation}
By~\eqref{e:l-pm-differential} and~\eqref{e:refor-1}, and since
$\ell^\pm(\mathbf x(\gamma^\pm_\lambda(\theta_0))-\mathbf x(\theta_0),\lambda)=0$,
the left-hand side of~\eqref{e:refor-2} has the same sign as
$$
\pm {\ell^\mp(\mathbf x(\gamma^\pm_\lambda(\theta_0))-\mathbf x(\theta_0),\lambda)\over \partial_{\theta}\ell^\pm(\mathbf x(\theta),\lambda)|_{\theta=\theta_0}}
$$
which is positive by~\eqref{e:sign-identity} with $x:=\mathbf x(\theta_0)$.

\noindent 3.
Now~\eqref{eq:sing} and the second part of~\eqref{eq:Gpm} follow from Lemma~\ref{l:prepare}, see also the remark
following Lemma~\ref{l:prepare-unique}
where we replace $\theta$ with $\gamma^\pm_\lambda(\theta)$.
Finally, by~\eqref{e:refor-1} and differentiating the identity $\ell^\pm(\mathbf x(\gamma^\pm_\lambda(\theta')),\lambda)=\ell^\pm(\mathbf x(\theta'),\lambda)$ in~$\theta'$ we compute
$$
F^\pm_\omega(\gamma^\pm_\lambda(\theta'),\theta')=c_\omega{\partial_\theta\ell^\pm(\mathbf x(\theta),\lambda)|_{\theta=\gamma^\pm_\lambda(\theta')}\over \partial_{\theta'}\ell^\pm(\mathbf x(\theta'),\lambda)}+\mathcal O(\varepsilon)=
{c_\omega\over\partial_{\theta'}\gamma^\pm_\lambda(\theta')}+\mathcal O(\varepsilon)
$$
which gives the first part of~\eqref{eq:Gpm}.
\end{proof}

\subsubsection{Characteristic points}
  \label{s:slp-char}

We finally study the singularities of $K^\pm_\omega$ near the characteristic set
$\Diag\cap\Refl^\pm_\lambda$. Recalling~\eqref{e:diag-ref-int}, we see that
this set consists of two points $(\theta_{\min}^\pm,\theta_{\min}^\pm)$ and $(\theta_{\max}^\pm,\theta_{\max}^\pm)$
where $\mathbf x(\theta_{\min}^\pm)=x_{\min}^\pm(\lambda)$, $\mathbf x(\theta_{\max}^\pm)=x_{\max}^\pm(\lambda)$
are the critical points of $\ell^\pm(\bullet,\lambda)$ (see Definition~\ref{d:1}).
\begin{lemm}
\label{l:char}
Assume that $\theta_0\in\{\theta^\pm_{\min},\theta^\pm_{\max}\}$.
Then there exists a neighborhood $U=\gamma^\pm_\lambda(U)$ of~$\theta_0$ such that 
for $(\theta,\theta')\in U\times U$ and $\varepsilon=\Im\omega>0$ small enough, we have
\begin{equation}
\label{eq:char}   K^\pm_\omega ( \theta, \theta' ) = 
c_\omega ( \theta - \theta' \pm i 0)^{-1} +
\tilde c^\pm_\omega(\theta') \big( \gamma_\lambda^\pm (\theta ) - \theta' \pm 
i \varepsilon z_\omega^\pm (  \theta ' ) \big)^{-1} 
+\mathscr K^\pm_\omega(\theta,\theta')
\end{equation}
where $\mathscr K^\pm_\omega\in C^\infty(U\times U)$
is smooth in~$\theta,\theta',\varepsilon$ up to $\varepsilon=0$,
$\tilde c^\pm_\omega(\theta')$ and $z^\pm_\omega(\theta')$
are smooth in $\theta',\varepsilon$ up to~$\varepsilon=0$, and~\eqref{eq:Gpm} holds.
\end{lemm}
\Remarks 1. Note that Lemma~\ref{l:char} implies Lemmas~\ref{l:diag} and~\ref{l:sing}
in a neighborhood of the characteristic set, since the first term on the right-hand side
of~\eqref{eq:char} is smooth away from the diagonal $\Diag$ and
the second term is smooth (uniformly in~$\varepsilon$) away from the reflection set $\Refl^\pm_\lambda$.

\noindent 2. Since keeping track of the signs is frustrating we present a model situation:
$ \ell^+ ( x ) = x_1 + i \varepsilon x_2 $, $ \ell^- ( x ) = x_2 + i \varepsilon x_1 $
(which is compatible with Lemma~\ref{l:which-rotate}) and 
$ \partial \Omega $ which near~$ ( 0 , 0 ) $ is given by
\[
x_1 =  q(x_2),\quad
  q(0)=q'(0)=0,\quad
  q''(0)<0.
\]
This corresponds to the point $\theta^+_{\max}$, since
when $\varepsilon=0$ the function $\ell^+(x)=x_1$ has a nondegenerate maximum
on~$\partial\Omega$.

We can use $ \theta = x_2 $ as a positively oriented parametrization of~$\partial\Omega$ near $ (0, 0 ) $.   In that
case the involution $ \gamma^+ ( \theta ) $ is given by
$$
q(\gamma^+(\theta))=q(\theta),\quad
\gamma^+(\theta)=-\theta+\mathcal O(\theta^2).
$$
This gives
\[
q ( \theta ) - q ( \theta' ) = Q ( \theta, \theta' ) ( \theta - \theta' ) ( \gamma^+ ( \theta ) - 
\theta' ) ,\quad  Q ( 0 , 0 ) = -  {q''(0)\over 2}>0.
\]
The Schwartz kernel of the model restricted single layer potential $ \mathcal C $ is given by (with $ Q = Q ( \theta, \theta' ) $ and neglecting the overall constant $ c_\omega $ in~\eqref{e:fund-sol-imag})
\[
\begin{split} K ( \theta, \theta' ) & =   \log \big( \ell^+ ( \mathbf x ( \theta ) - \mathbf x ( \theta' ) ) \ell^- ( \mathbf x ( \theta ) - 
\mathbf x ( \theta' ) )\big) \\
& =  \log \Big( \big( q (\theta ) - q ( \theta' )  + i \varepsilon ( \theta - \theta ' ) \big) \big( 
\theta - \theta' + i \varepsilon ( q ( \theta ) - q ( \theta' )) \big) \Big) \\
& =   \log \big( (\theta - \theta' )^2 ( Q  ( \gamma^+ (\theta ) - \theta' ) + i\varepsilon ) ( 
1 + i \varepsilon  Q  ( \gamma^+ ( \theta ) - \theta' ) ) \big)
\\ 
&  = 2   \log | \theta - \theta' | +  \log (  \gamma^+ ( \theta ) - \theta'  + i  
\varepsilon Q^{-1} ) \\
& \quad +  \log ( 1 + i \varepsilon  Q  ( \gamma^+ ( \theta ) - \theta' )) +  \log  Q . 
\end{split}
\]
Hence (see \S \ref{s:mot}) the Schwartz kernel of $ \partial_\theta \mathcal C $ is 
\[  \partial_\theta K ( \theta, \theta' ) =  \sum_{ \pm } ( \theta - \theta' \pm i 0 )^{-1} +
  {\partial_\theta \gamma^+(\theta)+i\varepsilon\partial_\theta Q^{-1}(\theta,\theta')\over \gamma^+ (\theta)  - \theta' + i\varepsilon Q^{-1}( \theta , \theta')}  + 
\mathscr K ( \theta, \theta' ) \]
where  $ Q ( 0,0) > 0$ and $ \mathscr K \in C^\infty $ uniformly in~$\varepsilon$.
This is consistent with~\eqref{eq:char} and~\eqref{eq:Gpm}, where we use Lemma~\ref{l:prepare} and
recall that by~\eqref{e:C-lambda-sum} we have $\partial_\theta K=K^++K^-$.
\begin{proof}[Proof of Lemma~\ref{l:char}]
1. Recall that $ \omega = \lambda + i \varepsilon $, $ \varepsilon > 0 $
and consider the expansion~\eqref{e:ell-pm-expanded}:
$$
\ell^\pm(x,\omega)=\ell^\pm(x,\lambda)+i\varepsilon\ell^\pm_1(x,\lambda)+\varepsilon^2\ell^\pm_2(x,\lambda,\varepsilon).
$$
We have for $\theta,\theta'$ in a sufficiently small neighborhood $U$ of~$\theta_0\in\{\theta^\pm_{\min},\theta^\pm_{\max}\}$
\begin{equation}
\label{eq:defGG}
\begin{aligned}
\ell^\pm(\mathbf x(\theta) -\mathbf x( \theta' ) , \lambda )&=G_0(\theta, \theta' )(\gamma^\pm_\lambda(\theta)-\theta')(\theta-\theta'),\\
\ell^\pm_1(\mathbf x(\theta) -\mathbf x(\theta' ) , \lambda )&=G_1(\theta,\theta' )(\theta-\theta'),\\
\ell_2^\pm(\mathbf x(\theta) -\mathbf x ( \theta' ), \lambda,\varepsilon )&= G_2(\theta,\theta',\varepsilon)(\theta-\theta'), 
\end{aligned}
\end{equation}
where $G_0,G_1,G_2$
are smooth in~$\theta,\theta',\varepsilon$ up to~$\varepsilon=0$,
and $G_0,G_1$ are real-valued and nonvanishing.
Indeed, the first decomposition follows from~\eqref{e:factorize} and
the second one, from~\eqref{e:l-pm-differential} and the fact that $\partial_\theta\ell^\mp(\mathbf x(\theta),\lambda)\neq 0$
at $\theta=\theta_0$. We have now (with $G_j=G_j(\theta,\theta')$)
\begin{equation}
\label{eq:ellexp}
\ell^\pm(\mathbf x(\theta) -\mathbf x( \theta' ) ,\omega)=(\theta-\theta')
\big(G_0(\gamma^\pm_\lambda(\theta)-\theta')+i\varepsilon G_1+\varepsilon^2 G_2\big).
\end{equation}

\noindent 2.
The argument in the proof of Lemma~\ref{l:diag} (see~\eqref{e:kopper}) shows that
for any fixed small $\varepsilon>0$
\begin{equation}
\label{eq:Kpmla}
K^\pm_\omega (\theta,\theta')={c_{\omega}
\partial_\theta \ell^\pm ( \mathbf x ( \theta ) , \omega )
\over G_0 (\gamma^\pm_\lambda(\theta)-\theta')+i\varepsilon G_1+\varepsilon^2 G_2}(\theta-\theta'\pm i0)^{-1}.
\end{equation}
To apply this argument we need to check the condition~\eqref{e:diagor-2}, which we rewrite as
\begin{equation}
  \label{e:diagor-char}
\pm\Im{G_0 (\gamma^\pm_\lambda(\theta)-\theta')+i\varepsilon G_1+\varepsilon^2 G_2\over \ell^\pm(\mathbf v(\theta),\omega)}<0
\end{equation}
for $\theta,\theta'$ near~$\theta_0$, $\varepsilon=\Im\omega>0$ small enough, and $\mathbf v(\theta)$ an inward pointing vector field
on $\partial\Omega$. Here the denominator is separated away from zero since
$\ell^\pm(\mathbf v(\theta_0),\lambda)\neq 0$.
   
For $\varepsilon=0$, the expression~\eqref{e:diagor-char} is equal to~0.
Thus is suffices to check the sign of its derivative in $\varepsilon$ at $\varepsilon=0$
and $\theta=\theta'=\theta_0$, that is,  show that (where we use~\eqref{e:l-pm-differential})
\begin{equation}
  \label{e:signs-again}
\pm \ell^\pm(\mathbf v(\theta_0),\lambda)\ell^\mp(\partial_\theta \mathbf x(\theta_0),\lambda)<0.
\end{equation}
The latter follows from the fact that $\ell^\pm(\partial_\theta\mathbf x(\theta_0),\lambda)=0$,
$x\mapsto (\ell^+(x,\lambda),\ell^-(x,\lambda))$ is an orientation
preserving linear map on $\mathbb R^2$, and $\partial_\theta\mathbf x(\theta_0),\mathbf v(\theta_0)$
form a positively oriented basis of~$\mathbb R^2$.

\noindent 3. Differentiating~\eqref{eq:ellexp} in $\theta$ to get a formula for $\partial_\theta\ell^\pm(\mathbf x(\theta),\omega)$
and substituting into~\eqref{eq:Kpmla} we get the following identity for $\theta,\theta'\in U$:
\begin{equation}
  \label{e:gorilla}
K_\omega^\pm(\theta,\theta')=c_\omega (\theta-\theta'\pm i0)^{-1}+{c_\omega \partial_\theta\big(G_0 (\gamma^\pm_\lambda(\theta)-\theta')+i\varepsilon G_1+\varepsilon^2 G_2\big)\over G_0 (\gamma^\pm_\lambda(\theta)-\theta')+i\varepsilon G_1+\varepsilon^2 G_2}
\end{equation}
where as before, $G_j=G_j(\theta,\theta')$. Dividing the numerator and denominator of the last term on the right-hand side
by~$G_0$, we see that the second term on the right-hand side of~\eqref{e:gorilla} is
equal to
$F^\pm_\omega(\theta,\theta')(\gamma^\pm_\lambda(\theta)-\theta'\pm i\varepsilon\psi^\pm_\omega(\theta,\theta'))^{-1}$
where the functions
$$
\begin{aligned}
\psi^\pm_\omega(\theta,\theta'):=\,&\pm{G_1(\theta,\theta')-i\varepsilon G_2(\theta,\theta',\varepsilon)\over G_0(\theta,\theta')},
\\
F^\pm_\omega(\theta,\theta'):=\,&{c_\omega\partial_\theta\big(G_0(\theta,\theta')(\gamma^\pm_\lambda(\theta)-\theta')
+i\varepsilon G_1(\theta,\theta')+\varepsilon^2G_2(\theta,\theta',\varepsilon)\big)\over G_0(\theta,\theta')}.
\end{aligned}
$$
are smooth in $\theta,\theta',\varepsilon$ up to $\varepsilon=0$
and $\psi^\pm_\omega$ is real and nonzero when $\varepsilon=0$.

To get~\eqref{eq:char} we can now use Lemma~\ref{l:prepare}
(and the remark
following Lemma~\ref{l:prepare-unique}) similarly to Step~3 in the proof of Lemma~\ref{l:sing}.
Here the sign condition $\Re\psi^\pm_\omega\geq c>0$ and~\eqref{eq:Gpm} can be verified
by a direct computation using~\eqref{e:l-pm-differential}, definitions \eqref{eq:defGG} and~\eqref{e:signs-again};
note that for the sign condition it suffices to check the sign of $G_1/G_0$ at $\theta=\theta'=\theta_0$.
\end{proof}

\subsubsection{Summary}

We summarize the findings of this section in microlocal terms. Consider
the pullback operator by $\gamma^\pm_\lambda$ on 1-forms on~$\mathbb S^1$,
$$
(\gamma^\pm_\lambda)^*:C^\infty(\mathbb S^1;T^*\mathbb S^1)\to C^\infty(\mathbb S^1;T^*\mathbb S^1).
$$
In terms of the identification of functions with 1-forms, $f\mapsto f\,d\theta$,
we have
\begin{equation}
  \label{e:1-forms-pulling}
(\gamma^\pm_\lambda)^*(f\,d\theta)=\big((f\circ\gamma^\pm_\lambda)\partial_\theta\gamma^\pm_\lambda)\,d\theta.
\end{equation}
\begin{prop}
\label{p:summa}
Assume that $\omega=\lambda+i\varepsilon$ where $\lambda\in (0,1)$,
 $\varepsilon\geq 0$, and
$\Omega$ is $\lambda$-simple in the sense
of Definition~\ref{d:1}.
Let $\mathcal C_\omega$ be the operator defined in~\eqref{eq:defsingle},
where for $\varepsilon=0$ we understand it as the operator $\mathcal C_{\lambda+i0}$.
Using the coordinate~$\theta$, we treat $d\mathcal C_\omega$
as an operator on $C^\infty(\mathbb S^1;T^*\mathbb S^1)$.
Then for all $\varepsilon$ small enough, we can write
\begin{equation}
  \label{eq:summa}
\mathcal E_\omega d \mathcal C_\omega=I+(\gamma^+_\lambda)^*A^+_\omega+(\gamma^-_\lambda)^*A^-_\omega
\end{equation}
where $\mathcal E_\omega,A^\pm_\omega$ are pseudodifferential operators in~$\Psi^0(\mathbb S^1;T^*\mathbb S^1)$ bounded uniformly in~$\varepsilon$ and such that,
uniformly in~$\varepsilon$ (see~\eqref{eq:Opasu})
$$
\sigma(\mathcal E_\omega)(\theta,\xi)={i\sgn\xi\over 2\pi c_\omega},\quad
\WF(A^\pm_\omega)\subset \{\pm\xi >0\},\quad
\sigma(A^\pm_\omega)(\theta,\xi)=a^\pm_\omega(\theta)H(\pm\xi)e^{-\varepsilon z^\pm_\omega(\theta)|\xi|}
$$
where $H(\xi)$ denotes the Heaviside function,
$a^\pm_\omega$ and $z^\pm_\omega$ are smooth in $\theta,\varepsilon$ up to $\varepsilon=0$,
$\Re z^\pm_\omega(\theta)\geq c>0$, and $a^\pm_\omega(\theta)=-1+\mathcal O(\varepsilon)$.
\end{prop}
\Remarks 1. Proposition~\ref{p:summa} is stated for a fixed value of~$\lambda=\Re\omega$. However, its proof
still works when $\lambda$ varies in some open interval $\mathcal J\subset (0,1)$ such that
$\Omega$ is $\lambda$-simple for all $\lambda\in \mathcal J$. The conclusions of Proposition~\ref{p:summa}
hold locally uniformly in~$\lambda\in \mathcal J$ and the functions $a^\pm_\omega(\theta)$,
$z^\pm_\omega(\theta)$ can be chosen depending smoothly on $\theta\in \mathbb S^1$, $\lambda\in \mathcal J$, and $\varepsilon=\Im\omega\geq 0$. Moreover, the operators $A^\pm_\omega$ and $\mathcal E_\omega$ depend smoothly on~$\lambda$
and all their $\lambda$-derivatives are in $\Psi^0$ uniformly in~$\varepsilon$;
same is true for the pseudodifferential operators featured in the decomposition~\eqref{e:C-lambda-supersplit} below.

\noindent 2. One can formulate a version of~\eqref{eq:summa} directly
on $\partial\Omega$ which does not depend on the choice of the (positively oriented) coordinate~$\theta$, using the fact that the principal symbol~\eqref{eq:Psim}
is invariantly defined.
\begin{proof}
1. Recall from~\eqref{e:C-lambda-sum} that $d\mathcal C_\omega=T^+_\omega+T^-_\omega$,
where $T^\pm_\omega$ are defined in~\eqref{e:T-pm-def}.
As with $d\mathcal C_\omega$, we use the coordinate~$\theta$ to think
of $T^\pm_\omega$ as operators on $C^\infty(\mathbb S^1;T^*\mathbb S^1)$.
We will write $T^\pm_\omega$ as a sum of a pseudodifferential
operator and a composition of a pseudodifferential operator with $(\gamma^\pm_\lambda)^*$,
see~\eqref{e:C-lambda-supersplit} below.
The singular supports
of the Schwartz kernels of these two operators will lie in the sets
$\Diag$ and $\Refl^\pm_\lambda$ defined in~\eqref{e:sing-sets-def}.

Fix a cutoff $\chi_{\Diag}\in C^\infty(\mathbb S^1\times\mathbb S^1)$
supported in a small neighborhood of the diagonal $\Diag$ and equal to~1
on a smaller neighborhood of $\Diag$. Define the ($\omega$-dependent) operator
$$
T^\pm_{\Diag}:C^\infty(\mathbb S^1;T^*\mathbb S^1)
\to C^\infty(\mathbb S^1;T^*\mathbb S^1)
$$
with the Schwartz kernel $c_\omega\chi_{\Diag}(\theta,\theta')(\theta-\theta'\pm i0)^{-1}$.
Here Schwartz kernels are defined in~\eqref{e:schwartzer}.
By Lemma~\ref{l:inversor-psi0} we have
\begin{equation}
  \label{e:T-diag-props}
T^\pm_{\Diag}\in \Psi^0(\mathbb S^1;T^*\mathbb S^1),\quad
\sigma(T^\pm_{\Diag})(\theta,\xi)=\mp 2\pi i c_\omega H(\pm\xi).
\end{equation}

\noindent 2. Next, define the reflected operators
$$
T^\pm_{\Refl}:=T^\pm_{\omega}-T^\pm_{\Diag},\quad
\widehat T^\pm_{\Refl}:=(\gamma^\pm_\lambda)^*T^\pm_{\Refl}.
$$
Denote by $K^\pm_{\Refl},\widehat K^\pm_{\Refl}$ the corresponding Schwartz kernels.
Combining Lemmas~\ref{l:smooth}, \ref{l:diag}, \ref{l:sing}, and~\ref{l:char}
we see that, putting $\chi_{\Refl}^\pm(\theta,\theta'):=\chi_{\Diag}(\gamma^\pm_\lambda(\theta),\theta')$,
$$
K^\pm_{\Refl}(\theta,\theta')=\chi_{\Refl}^\pm(\theta,\theta')\tilde c_\omega^\pm(\theta')\big(\gamma^\pm_\lambda(\theta)-\theta'\pm i\varepsilon z^\pm_\omega(\theta')\big)^{-1}+\mathscr K_\omega^\pm(\theta,\theta'),\quad
0<\varepsilon<\varepsilon_0
$$
where $\mathscr K^\pm_\omega$ is smooth in $\theta,\theta',\varepsilon$ up to~$\varepsilon=0$,
$\tilde c^\pm_\omega(\theta')$ and $z^\pm_\omega(\theta')$ are smooth in $\theta',\varepsilon$ up to $\varepsilon=0$, $\Re z^\pm_\omega(\theta')\geq c>0$
for some constant~$c$, and $\tilde c^\pm_\omega(\theta')=c_\omega/\partial_{\theta'}\gamma^\pm_\lambda(\theta')+\mathcal O(\varepsilon)$. Here we use a partition of unity and Lemma~\ref{l:prepare-unique} to patch together different local representations from Lemmas~\ref{l:sing} and~\ref{l:char} and get globally defined $\tilde c_\omega^\pm,z^\pm_\omega$.
Recalling~\eqref{e:1-forms-pulling}, we have
$$
\widehat K^\pm_{\Refl}(\theta,\theta')=(\partial_\theta\gamma^\pm_\lambda(\theta))K^\pm_{\Refl}(\gamma^\pm_\lambda(\theta),\theta').
$$
Thus by Lemma~\ref{l:inversor-psi0}
the operator $\widehat T^\pm_{\Refl}$ is pseudodifferential:
we have uniformly in $\varepsilon>0$
\begin{equation}
  \label{e:T-pm-prop}
\begin{gathered}
\widehat T^\pm_{\Refl}\in \Psi^0(\mathbb S^1;T^*\mathbb S^1),\quad
\WF(\widehat T^\pm_{\Refl})\subset \{\pm \xi>0\},\\
\sigma(\widehat T^\pm_{\Refl})(\theta,\xi)=\mp 2\pi i\tilde c^\pm_\omega(\theta) (\partial_\theta\gamma^\pm_\lambda(\theta))e^{-\varepsilon z^\pm_\omega(\theta) |\xi|}H(\pm \xi).
\end{gathered}
\end{equation}

\noindent 3. We now have the decomposition for $\varepsilon>0$
\begin{equation}
  \label{e:C-lambda-supersplit}
d\mathcal C_\omega=T^+_{\Diag}+T^-_{\Diag}+(\gamma_\lambda^+)^* \widehat T^+_{\Refl}
+(\gamma_\lambda^-)^* \widehat T^-_{\Refl}.
\end{equation}
Taking the limit as $\varepsilon\to 0+$ and using Lemma~\ref{l:inverse-converger}
(see also~\eqref{e:prepare-remark})
and Lemma~\ref{l:C-converges} we see that the same decomposition holds for $\varepsilon=0$,
where we have
$$
K^\pm_{\Refl}(\theta,\theta')=\chi_{\Refl}^\pm(\theta,\theta')\tilde c_\omega^\pm(\theta')(\gamma^\pm_\lambda(\theta)-\theta'\pm i0)^{-1}+\mathscr K_\omega^\pm(\theta,\theta')\quad\text{when}\quad \varepsilon=0
$$ 
and by Lemma~\ref{l:inversor-psi0} the properties~\eqref{e:T-pm-prop} hold for $\varepsilon=0$.

The operator $T_{\Diag}:=T^+_{\Diag}+T^-_{\Diag}$ lies in $\Psi^0(\mathbb S^1;T^*\mathbb S^1)$
and has principal symbol $-2\pi i c_\omega \sgn \xi$ (away from $\xi=0$), which is elliptic.
Let $\mathcal E_\omega$ be the elliptic parametrix of~$T_{\Diag}$,
so that $\mathcal E_\omega T_{\Diag}=I+\Psi^{-\infty}$
(see~\cite[Theorem~18.1.9]{Hormander3}).
We have $\sigma(\mathcal E_\omega)=1/\sigma(T_{\Diag})=i\sgn\xi/(2\pi c_\omega)$.
Multiplying~\eqref{e:C-lambda-supersplit} on the left by $\mathcal E_\omega$ we get~\eqref{eq:summa}
where the operators $A^\pm_\omega$ have the following form:
$$
A^\pm_\omega=(\gamma^\pm_\lambda)^*\mathcal E_\omega(\gamma^\pm_\lambda)^* \widehat T^\pm_{\Refl}.
$$
By~\cite[Theorem~18.1.17]{Hormander3}, $(\gamma^\pm_\lambda)^*\mathcal E_\omega(\gamma^\pm_\lambda)^*\in\Psi^0(\mathbb S^1;T^*\mathbb S^1)$
has the principal symbol $-i\sgn\xi/(2\pi c_\omega)$ (as $\gamma^\pm_\lambda$ is orientation reversing), so
from~\eqref{e:T-pm-prop}
we get the needed properties of~$A^\pm_\omega$, with
$$
a^\pm_\omega(\theta)=-{\tilde c^\pm_\omega(\theta)\over c_\omega}\partial_\theta \gamma^\pm_\lambda(\theta)=-1+\mathcal O(\varepsilon).
$$
\end{proof}

\subsubsection{A strong convergence statement}

A corollary of Lemma~\ref{l:C-converges} and Proposition~\ref{p:summa} is the following limiting statement:
\begin{lemm}
  \label{l:C-lambda-strong-limit}
Assume that $\lambda\in (0,1)$, $\Omega$ is $\lambda$-simple, $k\in\mathbb N_0$, and $s+1>t$. Then
\begin{equation}
  \label{e:C-lambda-strong-limit}
\|\partial_\omega^k\mathcal C_{\omega_j}-\partial_\lambda^k\mathcal C_{\lambda+i0}\|_{H^{s+k}(\partial\Omega;T^*\partial\Omega)\to H^t(\partial\Omega)}\to 0\quad\text{for all}\quad \omega_j\to\lambda,\quad
\Im\omega_j>0.
\end{equation}
\end{lemm}
\begin{proof}
1. Fix~$k$. We first show the following uniform bound: for each $s$ there exists $C_s$ such that for all
large $j$,
\begin{equation}
  \label{e:summa-bounder}
\|\partial_\omega^k\mathcal C_{\omega_j}\|_{H^{s+k}(\partial\Omega;T^*\partial\Omega)\to H^{s+1}(\partial\Omega)}\leq C_s.
\end{equation}
Indeed, Proposition~\ref{p:summa} (more precisely, \eqref{e:C-lambda-supersplit}) and Remark 1 after it
imply that 
\begin{equation}
  \label{e:summa-bounder-1}
\|d\partial_\omega^k \mathcal C_{\omega_j}\|_{H^{s+k}(\partial\Omega;T^*\partial\Omega)\to H^{s}(\partial\Omega;T^*\partial\Omega)}\leq C_s\quad\text{for all }s,
\end{equation}
where the loss of $k$ derivatives comes from differentiating the pullback operators $\gamma^\pm_\lambda$
in $\lambda=\Re\omega$.
On the other hand Lemma~\ref{l:C-converges} shows that
for each $\varphi\in C^\infty(\partial\Omega\times\partial\Omega)$, and denoting
by $\partial_\omega^k\mathcal C_\omega(x,y)$ the Schwartz kernel
of the operator $\partial^k_\omega\mathcal C_\omega$, the sequence
$$
\int_{\partial\Omega\times\partial\Omega} \partial_\omega^k \mathcal C_{\omega_j}(x,y) \varphi(x,y)\,d\theta(x)d\theta(y)
$$
converges (to the same integral
for $\partial^k_\lambda\mathcal C_{\lambda+i0}$) and thus in particular is bounded. By the Banach--Steinhaus Theorem
in the Fr\'echet space $C^\infty(\partial\Omega\times\partial\Omega)$, we see that
there exists $N_k$ such that
\begin{equation}
  \label{e:summa-bounder-2}
\|\partial^k_\omega\mathcal C_{\omega_j}\|_{H^s(\partial\Omega;T^*\partial\Omega)\to H^{t}(\partial\Omega)}\leq C_{s,t}\quad\text{for all }s\geq N_k,\
t\leq -N_k.
\end{equation}
(Another way to show~\eqref{e:summa-bounder-2}, avoiding Banach--Steinhaus, would be to carefully
examine the proof of Lemma~\ref{l:C-converges}.)

Together~\eqref{e:summa-bounder-1}, \eqref{e:summa-bounder-2}, and the elliptic estimate for $\partial_\theta$
imply that~\eqref{e:summa-bounder} holds for all~$s\geq N_k-k$. The operator
$\partial_\omega^k\mathcal C_{\omega_j}$ is its own transpose under the natural bilinear pairing
on $C^\infty(\partial\Omega;T^*\partial\Omega)\times C^\infty(\partial\Omega)$. 
Since $H^{-s}$ is dual to~$H^s$ under this pairing, \eqref{e:summa-bounder} holds for all $s\leq -N_k-1$. Then~\eqref{e:summa-bounder-2} holds for all $s,t$ such that $t\leq\min(s+1-k,-N_k)$.
Together with~\eqref{e:summa-bounder-1} and the elliptic estimate for~$\partial_\theta$ this implies that~\eqref{e:summa-bounder} holds in general. Same bound holds for
the operator $\partial_\lambda^k\mathcal C_{\lambda+i0}$.

\noindent 2. We now show that
\begin{equation}
  \label{e:summa-bounder-3}
\partial_\omega^k\mathcal C_{\omega_j} v\to\partial_\lambda^k\mathcal C_{\lambda+i0}v\quad\text{in }C^\infty(\partial\Omega)\quad\text{for all }
v\in C^\infty(\partial\Omega;T^*\partial\Omega).
\end{equation}
Indeed, by~\eqref{e:summa-bounder} the sequence $\partial_\omega^k\mathcal C_{\omega_j}v$ is precompact
in $H^s$ for every~$s$, and any convergent subsequence has to converge
to $\partial_\lambda^k\mathcal C_{\lambda+i0}v$ since $\partial_\omega^k\mathcal C_{\omega_j} v\to\partial_\lambda^k\mathcal C_{\lambda+i0}v$
in~$\mathcal D'$ by Lemma~\ref{l:C-converges}.

Since $C^\infty$ is dense in $H^{s+k}$, we get from~\eqref{e:summa-bounder}, \eqref{e:summa-bounder-3},
and a standard argument in functional analysis 
the strong-operator convergence
\begin{equation}
  \label{e:summa-bounder-4}
\partial_\omega^k\mathcal C_{\omega_j} v\to\partial_\lambda^k\mathcal C_{\lambda+i0}v\quad\text{in }H^{s+1}(\partial\Omega)\quad\text{for all }
v\in H^{s+k}(\partial\Omega;T^*\partial\Omega).  
\end{equation}
We are now ready to prove~\eqref{e:C-lambda-strong-limit}. Let $s+1> t$. Assume that~\eqref{e:C-lambda-strong-limit}
fails, then by passing to a subsequence we may assume that there exists some $c>0$ and a sequence
$$
v_j\in H^{s+k}(\partial\Omega;T^*\partial\Omega),\quad
\|v_j\|_{H^{s+k}}=1,\quad
\|(\partial_\omega^k\mathcal C_{\omega_j}-\partial_\lambda^k\mathcal C_{\lambda+i0})v_j\|_{H^t}\geq c.
$$
Since $H^{s+k}$ embeds compactly into $H^{t-1+k}$, passing to a subsequence we may assume
that $v_j\to v_0$ in $H^{t-1+k}$. But then
$$
\|(\partial_\omega^k\mathcal C_{\omega_j}-\partial_\lambda^k\mathcal C_{\lambda+i0})v_j\|_{H^t}\leq
\|(\partial_\omega^k\mathcal C_{\omega_j}-\partial_\lambda^k\mathcal C_{\lambda+i0})(v_j-v_0)\|_{H^t}+
\|(\partial_\omega^k\mathcal C_{\omega_j}-\partial_\lambda^k\mathcal C_{\lambda+i0})v_0\|_{H^t}.
$$
Now the first term on the right-hand side goes to~0 as $j\to \infty$ by~\eqref{e:summa-bounder},
and the second term goes to~0 by~\eqref{e:summa-bounder-4}, giving a contradiction.
\end{proof}

\subsubsection{Action on conormal distributions}
  \label{s:C-conormal}

We finish this section by showing  that $\mathcal C_\omega$
is bounded uniformly as $\Im\omega\to 0+$
on conormal spaces $I^s(\partial\Omega;N^*_+\Sigma^-_\lambda\sqcup N^*_-\Sigma^+_\lambda)$
defined in~\eqref{e:I-s-pm},
where $\Sigma^\pm_\lambda$ are the attractive/repulsive sets of the chess billiard~$b(\bullet,\lambda)$
defined in~\eqref{e:Sigma-pm-def}
and $\lambda=\Re\omega$ -- see Lemma~\ref{l:C-conormal} below. Moreover, we get similar estimates on all the derivatives $\partial^k_\omega\mathcal C_\omega$.
This is used in the proof of Proposition~\ref{p:boundvder} below.

Since the conormal spaces above depend on~$\lambda$, we introduce a $\lambda$-dependent system
of coordinates which maps $\Sigma^\pm_\lambda$ to $\lambda$-independent sets.
Assume that $\mathcal J\subset (0,1)$ is an open interval such that
the Morse--Smale conditions hold for each $\lambda\in \mathcal J$ (see Definition~\ref{d:2}).
Recall from Lemma~\ref{l:perturb-MS} that the points in the sets $\Sigma^\pm_\lambda$ depend smoothly on~$\lambda\in\mathcal J$.
Fix any finite set $\widetilde\Sigma\subset\mathbb S^1$
with the same number of points as $\Sigma_\lambda=\Sigma^+_\lambda\sqcup\Sigma^-_\lambda$ and a family of orientation preserving diffeomorphisms
depending smoothly on~$\lambda$
$$
\Theta_\lambda:\mathbb S^1\to\partial\Omega,\quad \lambda\in\mathcal J,\qquad
\Theta_\lambda(\widetilde\Sigma)=\Sigma_\lambda.
$$
We may decompose $\widetilde\Sigma=\widetilde\Sigma^+\sqcup\widetilde\Sigma^-$ where
$\widetilde\Sigma^\pm$ are $\lambda$-independent sets and
\begin{equation}
  \label{e:Theta-mapper}
\Theta_\lambda(\widetilde\Sigma^\pm)=\Sigma^\pm_\lambda\quad\text{for all}\quad\lambda\in\mathcal J.
\end{equation}
Note that for any fixed $\lambda\in\mathcal J$ the pullback $\Theta_\lambda^*$ gives an isomorphism
$$
\Theta_\lambda^*:I^s(\partial\Omega,N^*_+\Sigma^-_\lambda\sqcup N^*_-\Sigma^+_\lambda)
\to I^s(\mathbb S^1,N^*_+\widetilde\Sigma^-\sqcup N^*_-\widetilde\Sigma^+)
$$
and the space on the right-hand side is independent of~$\lambda$.

For $\omega\in \mathcal J+i(0,\infty)$ define the conjugated operator
(here $\Theta_\lambda^{-*}$ is the pullback by~$\Theta_\lambda^{-1}$)
\begin{equation}
  \label{e:tilde-C-omega}
\widetilde{\mathcal C}_\omega:=\Theta_\lambda^* \mathcal C_\omega \Theta_\lambda^{-*}:
C^\infty(\mathbb S^1;T^*\mathbb S^1)\to C^\infty(\mathbb S^1)\quad\text{where}\quad \lambda:=\Re\omega.
\end{equation}
We write $\omega=\lambda+i\varepsilon$ and define $\partial_\lambda^k\widetilde{\mathcal C}_\omega$ 
by differentiating in~$\lambda$ with $\varepsilon$ fixed.
(Note that $\mathcal C_\omega$ is holomorphic in~$\omega$ by Lemma~\ref{l:C-converges}
but $\widetilde{\mathcal C}_\omega$ is not holomorphic.)

We say that a sequence of operators
$$
T_j:I^{s+}(\mathbb S^1,N^*_+\widetilde\Sigma^-\sqcup N^*_-\widetilde\Sigma^+)\to
I^{t+}(\mathbb S^1,N^*_+\widetilde\Sigma^-\sqcup N^*_-\widetilde\Sigma^+)
$$
is bounded uniformly in~$j$
if for each sequence $\widetilde v_j\in I^{s+}(\mathbb S^1,N^*_+\widetilde\Sigma^-\sqcup N^*_-\widetilde\Sigma^+)$
with every seminorm~\eqref{e:I-s-pm-seminorms} bounded uniformly in~$j$,
the sequence $T_j\widetilde v_j$ also has all the seminorms~\eqref{e:I-s-pm-seminorms}
bounded uniformly in~$j$. 
Similarly we consider operators acting on differential forms on $\mathbb S^1$, which are identified
with functions using the canonical coordinate $\theta$.
\begin{lemm}
  \label{l:C-conormal}
Assume that $\lambda\in\mathcal J$ and $\omega_j\to\lambda$, $\Im\omega_j>0$.
Then for each $k$ and~$s$, the sequence of operators
$$
\partial^k_\lambda \widetilde{\mathcal C}_{\omega_j}:
I^{s+}(\mathbb S^1,N^*_+\widetilde\Sigma^-\sqcup N^*_-\widetilde\Sigma^+)
\to I^{s-1+}(\mathbb S^1,N^*_+\widetilde\Sigma^-\sqcup N^*_-\widetilde\Sigma^+)
$$
is bounded uniformly in~$j$.
\end{lemm}
\begin{proof}
1.
From~\eqref{e:summa-bounder} we see that for each~$r$,
$\partial^k_\lambda \widetilde{\mathcal C}_{\omega_j}$ is bounded
$H^r\to H^{r-k+1}$ uniformly in~$j$.
By elliptic regularity, it then suffices to show that the sequence of operators
$$
d\partial^k_\lambda \widetilde{\mathcal C}_{\omega_j}:
I^{s+}(\mathbb S^1,N^*_+\widetilde\Sigma^-\sqcup N^*_-\widetilde\Sigma^+)
\to I^{s+}(\mathbb S^1,N^*_+\widetilde\Sigma^-\sqcup N^*_-\widetilde\Sigma^+)
$$
is bounded uniformly in~$j$.
Using the decomposition~\eqref{e:C-lambda-supersplit}, we write
\begin{equation}
  \label{e:C-lambda-conjee}
d\widetilde{\mathcal C}_{\omega}=\widetilde T_{\Diag,\omega}
+(\widetilde\gamma^+_\lambda)^*\widetilde T_{\Refl,\omega}^+
+(\widetilde\gamma^-_\lambda)^*\widetilde T_{\Refl,\omega}^-
\end{equation}
where $\omega=\lambda+i\varepsilon$,
$$
\widetilde T_{\Diag,\omega}:=\Theta_\lambda^* (T_{\Diag}^++T_{\Diag}^-)\Theta_\lambda^{-*},\quad
\widetilde T_{\Refl,\omega}^\pm:=\Theta_\lambda^* \widehat T_{\Refl}^\pm \Theta_\lambda^{-*}
$$
are families of pseudodifferential operators in $\Psi^0(\mathbb S^1;T^*\mathbb S^1)$
smooth in $\lambda\in\mathcal J$ uniformly in~$\varepsilon$ (see Remark~1 following Proposition~\ref{p:summa}),
and
$$
\widetilde\gamma^\pm_\lambda:=\Theta_\lambda^{-1}\circ\gamma^\pm(\bullet,\lambda)\circ \Theta_\lambda
$$
is a family of orientation reversing involutive diffeomorphisms of $\mathbb S^1$ depending smoothly on~$\lambda\in\mathcal J$ and
such that by~\eqref{e:b-inversor} and~\eqref{e:Theta-mapper}
\begin{equation}
  \label{e:tilde-gamma-saved}
\widetilde\gamma^\pm_\lambda(\widetilde \Sigma^+)=\widetilde \Sigma^-,\quad
\widetilde\gamma^\pm_\lambda(\widetilde \Sigma^-)=\widetilde \Sigma^+.
\end{equation}

\noindent 2. Differentiating~\eqref{e:C-lambda-conjee} in $\lambda=\Re\omega$, we see that it suffices to show
that for each~$k$ and~$s$ the sequences of operators
$$
\partial^k_\lambda \widetilde T_{\Diag,\omega_j},\
\partial^k_\lambda \widetilde T_{\Refl,\omega_j}^\pm,\
\partial^k_\lambda (\widetilde\gamma^\pm_{\lambda_j})^*:
I^{s+}(\mathbb S^1,N^*_+\widetilde\Sigma^-\sqcup N^*_-\widetilde\Sigma^+)
\to I^{s+}(\mathbb S^1,N^*_+\widetilde\Sigma^-\sqcup N^*_-\widetilde\Sigma^+)
$$
are bounded uniformly in~$j$. The operators $\partial^k_\lambda \widetilde T_{\Diag,\omega_j}$
and $\partial^k_\lambda \widetilde T_{\Refl,\omega_j}^\pm$ are bounded in $\Psi^0$ uniformly in~$j$
and thus bounded on any space of conormal distributions~\cite[Theorem~18.2.7]{Hormander3},
so it remains to show the boundedness of $\partial^k_\lambda (\widetilde\gamma^\pm_{\lambda_j})^*$.

Instead of pullback on 1-forms we study pullback on functions, since the two differ by a multiplication
operator which can be put into $\widetilde T_{\Refl,\omega}^\pm$.
We then have for all $\lambda\in\mathcal J$
$$
\partial_\lambda(\widetilde\gamma^\pm_{\lambda})^*=X^\pm_\lambda (\widetilde\gamma^\pm_\lambda)^*
$$
where $X^\pm_\lambda$ is the vector field on $\mathbb S^1$ given by
$$
X^\pm_\lambda(\theta)={\partial_\lambda\widetilde\gamma^\pm_\lambda(\theta)\over
\partial_\theta\widetilde\gamma^\pm_\lambda(\theta)}\,\partial_\theta.
$$
We note that $X^\pm_\lambda$ vanishes on $\widetilde\Sigma$ by~\eqref{e:tilde-gamma-saved}.

It follows that $\partial_\lambda^k (\widetilde\gamma^\pm_\lambda)^*$ is a linear combination with constant coefficients
of operators of the form
$$
(\partial_\lambda^{k_1} X^\pm_\lambda) \cdots (\partial_\lambda^{k_\ell} X^\pm_\lambda) (\widetilde\gamma^\pm_\lambda)^*,\quad
k_1+\dots+k_\ell+\ell=k.
$$
Thus it remains to show that for all $k$ the operators
\begin{equation}
  \label{e:finita-la}
\partial^k_\lambda X^\pm_{\lambda_j},\ 
(\widetilde\gamma^\pm_{\lambda_j})^*:I^{s+}(\mathbb S^1,N^*_+\widetilde\Sigma^-\sqcup N^*_-\widetilde\Sigma^+)
\to I^{s+}(\mathbb S^1,N^*_+\widetilde\Sigma^-\sqcup N^*_-\widetilde\Sigma^+)
\end{equation}
are bounded uniformly in~$j$. 

Each $\partial^k_\lambda X^\pm_{\lambda}$
is a vector field which vanishes on~$\widetilde\Sigma$ and thus can be written in the form
$a\rho\partial_\theta$ for some $a\in C^\infty(\mathbb S^1)$ depending smoothly on~$\lambda$ and $\rho$ which is a defining function
of $\widetilde\Sigma$ (see the discussion preceding~\eqref{e:I-s-pm-seminorms}).
Thus $\partial^k_\lambda X^\pm_{\lambda_j}$ is bounded on the spaces~\eqref{e:finita-la} uniformly in~$j$.
Finally, $(\widetilde\gamma^\pm_{\lambda_j})^*$ is bounded on these spaces uniformly in~$j$ by the mapping property~\eqref{e:tilde-gamma-saved} and since $\widetilde\gamma^\pm_{\lambda_j}$ is orientation reversing, thus its symplectic lift
maps $N^*_+\widetilde\Sigma^-$ and $N^*_-\widetilde\Sigma^+$ to each other.
\end{proof}

\section{High frequency analysis on the boundary}
\label{s:abo}

In this section, we take
$$
\omega=\lambda+i\varepsilon,\quad
0<\varepsilon\ll 1,
$$
where $\lambda\in(0,1)$ satisfies the Morse--Smale conditions on $\Omega$ (see Definition~\ref{d:2}), and consider
the elliptic boundary value problem~\eqref{eq:ellip}:
$$
P(\omega)u_\omega=f,\quad
u_\omega|_{\partial\Omega}=0.
$$
Here $f\in \CIc(\Omega)$ is fixed and the solution $u_\omega$ lies in $C^\infty(\overline\Omega)$
(see Lemma~\ref{l:elliptic-solved}). Our goal is to prove high frequency
estimates on $u_\omega$ which are uniform in the limit $\varepsilon\to 0+$, when
the operator $P(\omega)$ becomes hyperbolic. To do this we
combine the detailed analysis of~\S\ref{s:restricted-slp} with
the dynamical properties following from the Morse--Smale conditions.

\subsection{Splitting into positive and negative frequencies}
\label{s:abo-prepare}

Fix a positively oriented coordinate $\theta:\partial\Omega\to \mathbb S^1$
to identify $\partial\Omega$ with $\mathbb S^1$.
Recall from~\eqref{e:C-lambda-equator} that
\begin{equation}
  \label{e:C-lambda-equator-2}
\mathcal C_\omega v_\omega=G_\omega:=(R_\omega f)|_{\partial\Omega}.
\end{equation}
Here the 1-form $v_\omega:=\mathcal N_\omega u_\omega\in C^\infty(\mathbb S^1;T^*\mathbb S^1)$ is the `Neumann data' of $u_\omega$ defined using~\eqref{e:N-omega-def}; however, we do not
have uniform bounds on $v_\omega$ in $C^\infty$ as $\varepsilon=\Im\omega\to 0+$. The function~$G_\omega$ lies in $C^\infty(\mathbb S^1)$
uniformly in $\varepsilon$ since $f\in \CIc$ and
$R_\omega$ is the convolution operator with the fundamental solution~$E_\omega$, which has
a distributional limit as $\varepsilon\to 0+$ by Lemma~\ref{l:loc}.

Let $\gamma^\pm_\lambda$ be defined in~\eqref{e:gamma-omega-def}.
By Proposition~\ref{p:summa} we have
$$
\mathcal E_\omega dG_\omega=\mathcal E_\omega d\mathcal C_\omega v_\omega=v_\omega+(\gamma_\lambda^+)^*A^+_\omega v_\omega+
(\gamma_\lambda^-)^* A^-_\omega v_\omega.
$$
We rewrite this equation as
\begin{equation}
  \label{e:Cle-1}
v_\omega=-\mathcal A_\omega v_\omega+\mathcal E_\omega dG_\omega,\quad
\mathcal A_\omega:=(\gamma_\lambda^+)^*A^+_\omega+(\gamma_\lambda^-)^*A^-_\omega. 
\end{equation}
The operator $\mathcal A_\omega$ exchanges positive and negative frequencies, since $A^\pm_\omega$
are pseudodifferential and the maps $\gamma_\lambda^\pm$ are orientation reversing.
We thus study the square of $\mathcal A_\omega$, which maps positive and negative frequencies to themselves.
It is expressed in terms of the pullback of the chess billiard map $b=\gamma^+\circ\gamma^-$
to $\mathbb S^1$:
\begin{equation}
  \label{e:b-omega-def}
b_\lambda:=\gamma^+_\lambda\circ\gamma^-_\lambda,\qquad
b_\lambda^{-1}=\gamma^-_\lambda\circ\gamma^+_\lambda,
\end{equation}
which is an orientation preserving diffeomorphism of $\mathbb S^1$.
Denote the pullback operators by $b_\lambda$ and $b^{-1}_\lambda$ on 1-forms by 
$$
b_\lambda^*,b_\lambda^{-*}:C^\infty(\mathbb S^1;T^*\mathbb S^1)\to C^\infty(\mathbb S^1;T^*\mathbb S^1).
$$
\begin{lemm}
  \label{l:A-squared}
We have
\begin{equation}
  \label{e:A-squared}
\mathcal A_\omega^2=B^+_\omega b_\lambda^*+B^-_\omega b_\lambda^{-*}
\end{equation}
where $B^\pm_\omega$ are pseudodifferential,
more precisely we have uniformly in~$\varepsilon\geq 0$
(see~\eqref{eq:Opasu})
\begin{equation}
  \label{e:B-props}
\begin{gathered}
B^\pm_\omega\in \Psi^0(\mathbb S^1;T^*\mathbb S^1),\quad
\WF(B^\pm_\omega)\subset \{\pm\xi>0\},\\
\sigma(B^\pm_\omega)(\theta,\xi)=\tilde a^\pm_\omega(\theta)H(\pm\xi)e^{-\varepsilon\tilde z^\pm_\omega(\theta)|\xi|}
\end{gathered}
\end{equation}
where $H$ denotes the Heaviside function, the functions $\tilde a^\pm_\omega(\theta),\tilde z^\pm_\omega(\theta)$ are smooth
in $\theta\in\mathbb S^1$ and $\varepsilon\geq 0$, $\Re \tilde z^\pm_\omega\geq c>0$, and
$\tilde a^\pm_\omega(\theta)=1+\mathcal O(\varepsilon)$.
\end{lemm}
\Remark From Remark~1 after Proposition~\ref{p:summa} we see that
$B^\pm_\omega$ are smooth in~$\lambda$
(where $\omega=\lambda+i\varepsilon$), with $\lambda$-derivatives of all orders
lying in~$\Psi^0$ uniformly in~$\varepsilon$.
\begin{proof}
From Proposition~\ref{p:summa} and the change of variables
formula for pseudodifferential operators~\cite[Theorem~18.1.17]{Hormander3} we see that
$(\gamma^\pm_\lambda)^*A^\pm_\omega(\gamma^\pm_\lambda)^*$
lies in $\Psi^0(\mathbb S^1;T^*\mathbb S^1)$ and has wavefront set inside $\{\mp\xi>0\}$
uniformly in~$\varepsilon$. Since products of pseudodifferential operators
with nonintersecting wavefront sets are smoothing, we see that
$$
\big((\gamma^\pm_\lambda)^*A^\pm_\omega\big)^2\in \Psi^{-\infty}(\mathbb S^1;T^*\mathbb S^1)\quad\text{uniformly in}\quad\varepsilon\geq 0.
$$
Recalling~\eqref{e:Cle-1} we see that (with $\Psi^{-\infty}$ denoting smoothing operators uniformly in~$\varepsilon$)
$$
\mathcal A_\omega^2=(\gamma^+_\lambda)^*A^+_\omega(\gamma^-_\lambda)^*A^-_\omega
+(\gamma^-_\lambda)^*A^-_\omega(\gamma^+_\lambda)^*A^+_\omega+\Psi^{-\infty}.
$$
This gives the decomposition~\eqref{e:A-squared} with
$$
B^\pm_\omega=\big((\gamma^\mp_\lambda)^*A^\mp_\omega (\gamma^\mp_\lambda)^*\big)
\big((b^{\pm 1}_\lambda)^* A^\pm_\omega (b^{\mp 1}_\lambda)^*\big)+\Psi^{-\infty}.
$$
Using the properties of $A^\pm_\omega$ in Proposition~\ref{p:summa} together
with the product formula and the change of variables formula for pseudodifferential operators,
we see that $B^\pm_\omega\in\Psi^0(\mathbb S^1;T^*\mathbb S^1)$ and
$\WF(B^\pm_\omega)\subset \{\pm \xi>0\}$ uniformly in~$\varepsilon$.
This also gives
$$
\sigma(B^\pm_\omega)(\theta,\xi)=
\sigma(A^\mp_\omega)\bigg(\gamma^\mp_\lambda(\theta),{\xi\over\partial_\theta\gamma^\mp_\lambda(\theta)}\bigg)
\sigma(A^\pm_\omega)\bigg(b^{\pm 1}_\lambda(\theta),{\xi\over\partial_\theta b^{\pm 1}_\lambda(\theta)}\bigg)
$$
in the sense of~\eqref{eq:Opasu}, 
which implies the formula for the principal symbol in~\eqref{e:B-props} with
$$
\tilde a^\pm_\omega(\theta)=a^\mp_\omega(\gamma^\mp_\lambda(\theta))a^\pm_\omega(b^{\pm 1}_\lambda(\theta)),\quad
\tilde z^\pm_\omega(\theta)={z^\mp_\omega(\gamma^\mp_\lambda(\theta))\over |\partial_\theta \gamma^\mp_\lambda(\theta)|}
+{z^\pm_\omega(b^{\pm 1}_\lambda(\theta))\over \partial_\theta b^{\pm 1}_\lambda(\theta)}
$$
where $a^\pm_\omega$, $z^\pm_\omega$ are given in Proposition~\ref{p:summa}.
\end{proof}
Applying~\eqref{e:Cle-1} twice, we get the equation
\begin{equation}
  \label{e:Cle-2}
v_\omega=B^+_\omega b_\lambda^*v_\omega+B^-_\omega b_\lambda^{-*}v_\omega+g_\omega
\end{equation}
where
$$
g_\omega:=(I-\mathcal A_\omega)\mathcal E_\omega dG_\omega
$$
is in $C^\infty(\mathbb S^1;T^*\mathbb S^1)$ uniformly in~$\varepsilon>0$.

We now split $v_\omega$ into positive and negative frequencies. Consider a pseudodifferential
partition of unity 
\begin{equation}
  \label{e:Pi-pm}
\begin{aligned}
I=\Pi^++\Pi^-,\quad
&\Pi^\pm\in\Psi^0(\mathbb S^1,T^*\mathbb S^1),\\
\WF(\Pi^\pm)\subset \{\pm \xi>0\},\quad
&\sigma(\Pi^\pm)(\theta,\xi)=H(\pm\xi).
\end{aligned}
\end{equation}
Put
\begin{equation}
\label{eq:defvla}
v^\pm_\omega:=\Pi^\pm v_\omega,\quad
g^\pm_\omega:=\Pi^\pm g_\omega,
\end{equation}
with $g^\pm_\omega$ in $C^\infty(\mathbb S^1;T^*\mathbb S^1)$ uniformly in~$\varepsilon$,
and apply $\Pi^\pm$ to~\eqref{e:Cle-2} to get
\begin{equation}
  \label{e:Cle-3}
v^\pm_\omega=B^\pm_\omega(b^{\pm 1}_\lambda)^*v^\pm_\omega+\mathscr R^\pm_\omega v_\omega +g^\pm_\omega
\end{equation}
where the operator
$$
\mathscr R^\pm_\omega:=\big([\Pi^\pm,B^\pm_\omega]+B^\pm_\omega(\Pi^\pm-(b^{\pm 1}_\lambda)^*\Pi^\pm (b^{\mp 1}_\lambda)^*)\big)
(b^{\pm 1}_\lambda)^*+\Pi^\pm B^\mp_\omega(b^{\mp1}_\lambda)^*
$$
is in $\Psi^{-\infty}(\mathbb S^1;T^*\mathbb S^1)$ uniformly in~$\varepsilon$,
as follows from~\eqref{e:Pi-pm} and the fact that $\WF(B^\pm_\omega)\subset \{\pm \xi>0\}$.

\subsection{Microlocal Lasota--Yorke inequalities}

We now show that 
$B^\pm_\omega (b^{\pm 1}_\lambda)^*$ featured in the equation~\eqref{e:Cle-3} are contractions at high frequencies
on appropriately chosen inhomogeneous Sobolev spaces,
and use this to prove a high frequency estimate on~$v_\omega$,
see Proposition~\ref{p:dC-soft} below.
 This is reminiscent of Lasota–Yorke inequalities (see
\cite{balad} and references given there)
and could be considered a simple version of radial estimates
(see \cite[\S E.4.3]{DZ-Book} and references given there) for Fourier integral operators.
It is also related to microlocal weights used by Faure--Roy--Sj\"ostrand \cite{FRS}. 

Unlike applications to volume preserving Anosov maps in~\cite{balad,FRS}, where critical regularity
is given by $L^2$, for us the critical regularity space is $H^{-{1\over 2}}$. This can be
informally explained as follows: if we have $v^\pm_\omega=df^\pm$ for some functions $f^\pm$
then the {\em flux}  $\Im\int_{\mathbb S^1} \overline{f^\pm}\, df^\pm$, is invariant under replacing $f^\pm$ with the pullback $(b^{\pm 1}_\lambda)^* f^\pm$ and is well defined for 
$ f^\pm \in H^{\frac12} $. When $ \WF ( f^\pm ) \subset \{ \pm \xi > 0 \}$, the flux is
related to $\|f^\pm\|_{H^{1\over 2}}^2\sim \|v^\pm_\omega\|_{H^{-{1\over 2}}}^2$.

To simplify notation, we only study in detail the case of the `$+$' sign. The case of the `$-$' sign 
is handled similarly by replacing $b_\lambda$ with $b_\lambda^{-1}$, switching $\Sigma^+_\lambda$
with $\Sigma^-_\lambda$, and using the escape function in Lemma~\ref{l:escape-function}
(rather than in the remark following it).

We identify $\partial\Omega$ with $\mathbb S^1$ using the adapted coordinate $\theta$ constructed in Lemma~\ref{l:adapted-theta},
which satisfies for $\delta>0$ small enough
\begin{equation}
\label{eq:logdb} 
\mp \log \partial_\theta b_\lambda > 0\quad\text{on}\quad \overline{\Sigma^\pm_\lambda(\delta)}
\end{equation}
where $ \Sigma^\pm_\lambda\subset\mathbb S^1$ are the attractive ($+$) and repulsive ($-$) periodic points of 
$ b_\lambda$ defined in~\eqref{e:Sigma-pm-def}
and $\Sigma^\pm_\lambda(\delta)$ are their open $\delta$-neighborhoods. 

Take arbitrary $ \alpha_-  < \alpha_+ $ and small $\delta>0$ (in particular, so that~\eqref{eq:logdb} holds).
Let $\mathbf g\in C^\infty(\mathbb S^1;\mathbb R)$ be the escape function defined in the remark 
following Lemma~\ref{l:escape-function}. 
We have
\begin{equation}
  \label{e:N-0-def}
\alpha_-\leq \mathbf g(\theta)\leq N_0\quad\text{for some}\quad N_0.
\end{equation}
Define the symbol
\begin{equation}
\label{eq:defG}    G ( \theta, \xi ) := \mathbf g ( \theta ) (1-\chi_0( \xi )) \log  |\xi|  , \quad ( \theta, \xi ) \in T^*\mathbb S^1 
\end{equation}
where $\chi_0\in \CIc((-1,1))$ is equal to~1 near~0.
We use Lemma~\ref{l:expwe1} to construct
\begin{equation}
\begin{aligned}
\label{eq:EtG}  E_G := \Op ( e^G )\in \Psi^{N_0}_{0+}(\mathbb S^1;T^*\mathbb S^1) ,& \quad \widetilde E_{-G} := \Op ( e^{-G } ( 1 + r_G ) )
\in \Psi^{-\alpha_-}_{0+}(\mathbb S^1;T^*\mathbb S^1) ,\\
 r_G \in S^{-1+},& \quad
\widetilde E_{-G} E_G - I,
E_G\widetilde E_{-G}-I \in \Psi^{-\infty } .
\end{aligned}
\end{equation}
By property~(4) in the remark following Lemma~\ref{l:escape-function} we have
$\mathbf g\geq \alpha_+$ on $\mathbb S^1\setminus \Sigma^-_\lambda(\delta)$. Therefore by~\eqref{eq:wecomp}
\begin{equation}
  \label{e:EtG-extra}
\chi \widetilde E_{-G}\in\Psi^{-\alpha_+}_{0+}(\mathbb S^1;T^*\mathbb S^1)\quad\text{for all}\quad
\chi\in C^\infty(\mathbb S^1),\quad
\supp \chi\cap \overline{\Sigma^-_\lambda(\delta)}=\emptyset.
\end{equation}
We now apply $E_G$ to~\eqref{e:Cle-3} (with the `$+$' sign) to get
\begin{equation}
\label{eq:v2vG2}
\begin{gathered}
v_G = T_G  v_G + \mathscr R_G v_\omega+g_G \qquad\text{where}\quad
v_G := E_G v^+_\omega,\quad
g_G := E_G g^+_\omega,\\
T_G:=E_G B^+_\omega b_\lambda^* \widetilde E_{-G},\quad
\mathscr R_G:= E_GB^+_\omega b_\lambda^* (I-\widetilde E_{-G}E_G)\Pi^++E_G\mathscr R^+_\omega.
\end{gathered}  
\end{equation}
Here $g_G\in C^\infty(\mathbb S^1;T^*\mathbb S^1)$
and $\mathscr R_G\in\Psi^{-\infty}(\mathbb S^1;T^*\mathbb S^1)$,
both uniformly in~$\varepsilon$. The function $v_G$ lies in $C^\infty(\mathbb S^1;T^*\mathbb S^1)$
for $\varepsilon>0$, but it is not bounded in this space uniformly in~$\varepsilon$.
We also have the following bounds for each~$N$, which follow from~\eqref{eq:EtG}
and~\eqref{e:EtG-extra}
(writing $v^+_\omega=\widetilde E_{-G}v_G+(I-\widetilde E_{-G}E_G)v^+_\omega$):
\begin{align}
  \label{e:anisotroper-1}
\|v^+_\omega\|_{H^{\alpha_-}}&\leq C\|v_G\|_{L^2}+C_N\|v_\omega\|_{H^{-N}},\\
  \label{e:anisotroper-2}
\|\chi v^+_\omega\|_{H^{\alpha_+}}&\leq   C\|v_G\|_{L^2}+C_N\|v_\omega\|_{H^{-N}}\quad\text{if}\quad
\supp \chi\cap \overline{\Sigma^-_\lambda(\delta)}=\emptyset,\\
  \label{e:anisotroper-3}
\|g_G\|_{L^2}&\leq C\|g_\omega\|_{H^{N_0}}.
\end{align}

The key result in this section is the following lemma. The point is that for
$\alpha_-<-{1\over 2}<\alpha_+$,
we can obtain a contraction property of the microlocally conjugated operator~$ T_G $:
\begin{lemm}
\label{l:TG}
Suppose that $ G $ is given by \eqref{eq:defG} (using a coordinate~$\theta$ in which 
\eqref{eq:logdb} holds) with $ \mathbf g $ defined with parameters
$ \alpha_- < \alpha_+ $, $ \delta > 0 $, and that $ T_G $ is defined in~\eqref{eq:v2vG2}. Define the norm on $L^2(\mathbb S^1;T^*\mathbb S^1)$
using the coordinate~$\theta$.
Then for any $ N$ and $\nu>0$ there exists~$ C_N $ such that for all
small $\varepsilon=\Im\omega>0$ and all
 $ w \in C^\infty ( \mathbb S^1; T^* \mathbb S^1 ) $, 
\begin{equation}
\label{eq:TG}     \|  T_G w \|_{ L^2 } \leq   \big( \max_{\pm} \sup_{\Sigma^\pm_\lambda(\delta)} (\partial_\theta b_\lambda) ^{\frac12 + \alpha_\pm }
 + \nu \big) 
\| w\|_{ L^2} + C_N \| w \|_{ H^{- N }} .
\end{equation}
\end{lemm}
\begin{proof}
1. Recalling the formula~\eqref{e:1-forms-pulling} for pullback operators on 1-forms, we see
that the operator
$$
(b_\lambda^{-*})(\partial_\theta b_\lambda)^{1\over 2}:L^2(\mathbb S^1;T^*\mathbb S^1)\to L^2(\mathbb S^1;T^*\mathbb S^1)
$$
is unitary. Multiplying $T_G$ by this operator on the right, we see that it suffices to show that
\begin{equation}
  \label{e:TG-2}
\begin{gathered}
\|\widetilde T_G w\|_{L^2}\leq   \big( \max_{\pm} \sup_{\Sigma^\pm_\lambda(\delta)} (\partial_\theta b_\lambda) ^{\frac12 + \alpha_\pm }
 + \nu \big) 
\| w\|_{ L^2} + C_N \| w \|_{ H^{- N }}\\
\text{where}\quad
\widetilde T_G:=E_G B^+_\omega b_\lambda^* \widetilde E_{-G}b_\lambda^{-*}(\partial_\theta b_\lambda)^{1\over 2}.
\end{gathered}
\end{equation}
By~\eqref{eq:chava} we have
$b_\lambda^* \widetilde E_{-G}b_\lambda^{-*}=\Op(e^{-G_b}(1+r))$ for
$G_b(\theta,\xi):=G(b_\lambda(\theta),\xi/\partial_\theta b_\lambda(\theta))$
and some $r\in S^{-1+}$.
Recalling the definition~\eqref{eq:defG} of~$G$, we compute for $|\xi|$ large enough
\begin{equation}
  \label{e:G-subtracted}
G(\theta,\xi)-G_b(\theta,\xi)=\big(\mathbf g(\theta)-\mathbf g(b_\lambda(\theta))\big)\log|\xi|+\mathbf g(b_\lambda(\theta))\log\partial_\theta b_\lambda(\theta).
\end{equation}
Since $\mathbf g(\theta)-\mathbf g(b_\lambda(\theta))\leq 0$ by property~(1) in the remark following Lemma~\ref{l:escape-function},
we see that $G-G_b$ is bounded above by some constant.
By~\eqref{eq:wecomp} and Lemma~\ref{l:A-squared}
we then see that $\widetilde T_G\in\Psi^0_{0+}(\mathbb S^1;T^*\mathbb S^1)$ uniformly in~$\varepsilon$
and its principal symbol is (in the sense of~\eqref{eq:Opasu})
$$
\sigma(\widetilde T_G)(\theta,\xi)=\tilde a^+_\omega(\theta)H(\xi)e^{-\varepsilon \tilde z^+_\omega(\theta)\xi}(\partial_\theta b_\lambda(\theta))^{1\over 2}e^{G(\theta,\xi)-G_b(\theta,\xi)},\quad
|\xi|\geq 1.
$$
Thus~\eqref{e:TG-2} follows from Lemma~\ref{l:psi0-norm} once we
show that there exists $C_1>0$ such that for all $\xi\geq C_1$
\begin{equation}
  \label{e:TG-3}
|\tilde a^+_\omega(\theta)|e^{-\varepsilon \Re\tilde z^+_\omega(\theta)\xi}(\partial_\theta b_\lambda(\theta))^{1\over 2}e^{G(\theta,\xi)-G_b(\theta,\xi)}\leq
\max_{\pm} \sup_{\Sigma^\pm_\lambda(\delta)} (\partial_\theta b_\lambda) ^{\frac12 + \alpha_\pm }.
\end{equation}

\noindent 2. Since $\tilde a^+_\omega(\theta)=1+\mathcal O(\varepsilon)$ and $\Re \tilde z^+_\omega(\theta)\geq c>0$,
for $\xi\geq C_1$ and $C_1$ large enough we have
$|\tilde a^+_\omega(\theta)|e^{-\varepsilon \Re\tilde z^+_\omega(\theta)\xi}\leq 1$.
Thus \eqref{e:TG-3} reduces to showing that for all $\xi\geq C_1$
\begin{equation}
  \label{e:TG-4}
\begin{gathered}
\widetilde G(\theta,\xi)\leq \max_{\pm} \sup_{\Sigma^\pm_\lambda(\delta)}\big(\tfrac12 + \alpha_\pm)\log \partial_\theta b_\lambda
\\
\text{where}\quad \widetilde G(\theta,\xi):=\big(\mathbf g(\theta)-\mathbf g(b_\lambda(\theta))\big)\log\xi+\big(\tfrac12+\mathbf g(b_\lambda(\theta))\big) \log\partial_\theta b_\lambda(\theta).
\end{gathered}
\end{equation}
This in turn is proved if we show that there exists $c_0>0$ such that
for $\xi$ large enough
\begin{equation}
  \label{e:TG-5}
\widetilde G(\theta,\xi)\leq\begin{cases}
-c_0\log\xi,& \theta\in \mathbb S^1\setminus (\Sigma^-_\lambda(\delta)\cup \Sigma^+_\lambda(\delta)),\\
(\tfrac 12+\alpha_+)\log\partial_\theta b_\lambda(\theta),& \theta\in\Sigma^+_\lambda(\delta),\\
(\frac 12+\alpha_-)\log\partial_\theta b_\lambda(\theta),&  \theta\in\Sigma^-_\lambda(\delta).
\end{cases}
\end{equation}
We now prove~\eqref{e:TG-5} using properties~(1)--(6) in Lemma~\ref{l:escape-function} (or rather the 
remark which follows it).
The first inequality follows from property~(2), since $\mathbf g(\theta)-\mathbf g(b_\lambda(\theta))\leq -2c_0$
for some $c_0>0$.
The second inequality follows from properties~(1) and~(4)
together with~\eqref{eq:logdb}.
Finally, the third inequality follows from property~(6) with
$M:=(\log\xi)/(\log\partial_\theta b_\lambda(\theta))\gg 1$, 
where we again use~\eqref{eq:logdb}.
\end{proof}
With Lemma~\ref{l:TG} in place we give a basic high frequency estimate
on solutions to~\eqref{e:Cle-2} which is uniform as $\Im\omega\to 0$.
An upgraded version of this estimate (Proposition~\ref{p:dC}) is used in the proof of Limiting Absorption Principle
in~\S\ref{s:liap} below.
\begin{prop}
\label{p:dC-soft}
Fix $\beta>0$, $N$, and some functions $\chi^\pm\in C^\infty(\mathbb S^1)$
such that $\supp\chi^\pm\cap \Sigma^\mp_\lambda=\emptyset$.
Then there exist $N_0$ and $C$ such that
for all small $\varepsilon=\Im\omega>0$ and each solution $v_\omega\in C^\infty(\mathbb S^1;T^*\mathbb S^1)$ 
to~\eqref{e:Cle-2} we have
\begin{align}
  \label{e:dC-soft}
\|v_\omega\|_{H^{-{1\over 2}-\beta}}&\leq C\big(\|g_\omega\|_{H^{N_0}}+\|v_\omega\|_{H^{-N}}\big),\\
  \label{e:dC-soft-2}
\|\chi^\pm \Pi^\pm v_\omega\|_{H^{N}}&\leq C\big(\|g_\omega\|_{H^{N_0}}+\|v_\omega\|_{H^{-N}}\big).
\end{align}
\end{prop}
\Remarks 1. The a priori assumption that $v_\omega$ is smooth (without any uniformity
as $\varepsilon\to 0+$) is important in the argument because
it ensures that the norm $\|v_G\|_{L^2}$ is finite.

\noindent 2. Using the notation~\eqref{eq:N2pm}, we see that~\eqref{e:dC-soft-2} implies that,
assuming that the right-hand side of this inequality is bounded uniformly in~$\varepsilon$ for each $N_0$
and some~$N$, we have
$\WF(v_\omega)\subset N_+^*\Sigma^-_\lambda\sqcup N_-^*\Sigma^+_\lambda$
uniformly in~$\varepsilon$.
\begin{proof}
1. Fix $\alpha_\pm$ satisfying
$$
-\tfrac12-\beta\leq \alpha_-<-\tfrac12<\alpha_+,\quad
\alpha_+\geq N.
$$
Next, fix $\delta>0$ in the construction of the escape function $\mathbf g$ small enough so that~\eqref{eq:logdb}
holds and $\supp\chi^+\cap \overline{\Sigma^-_\lambda(\delta)}=\emptyset$.
By~\eqref{eq:logdb} and since $\alpha_-<-\frac 12<\alpha_+$ we may choose $\tau$ such that
$$
\max_{\pm} \sup_{\Sigma^\pm_\lambda(\delta)} (\partial_\theta b_\lambda) ^{\frac12 + \alpha_\pm }
<\tau<1.
$$
Take $N_0$ so that~\eqref{e:N-0-def} holds.
We use the equation~\eqref{eq:v2vG2} and~\eqref{e:anisotroper-3} to get
$$
\|v_G\|_{L^2}\leq \|T_G v_G\|_{L^2}+C\big(\|g_\omega\|_{H^{N_0}}+\|v_\omega\|_{H^{-N}}\big).
$$
Applying Lemma~\ref{l:TG} to $w:=v_G$, we see that
$$
\|T_G v_G\|_{L^2}\leq\tau\|v_G\|_{L^2}+C \|v_\omega\|_{H^{-N}}.
$$
Since $\tau<1$, together these two inequalities give
\begin{equation}
  \label{e:v-G-bounded}
\|v_G\|_{L^2}\leq C\big(\|g_\omega\|_{H^{N_0}}+\|v_\omega\|_{H^{-N}}\big).
\end{equation}

\noindent 2. From~\eqref{e:v-G-bounded} and~\eqref{e:anisotroper-1}
we have
\begin{equation}
  \label{e:dcS-1}
\|v^+_\omega\|_{H^{-{1\over 2}-\beta}}\leq C\big(\|g_\omega\|_{H^{N_0}}+\|v_\omega\|_{H^{-N}}\big).
\end{equation}
The bound~\eqref{e:dC-soft-2} for the `$+$' sign follows from~\eqref{e:v-G-bounded} and~\eqref{e:anisotroper-2}.
Similar analysis (replacing $b_\lambda$ with $b_\lambda^{-1}$, switching the roles of $\Sigma^+_\lambda$
and~$\Sigma^-_\lambda$, and using Lemma~\ref{l:escape-function} instead of the remark
that follows it) shows that~\eqref{e:dcS-1} holds for $v^-_\omega$
and~\eqref{e:dC-soft-2} holds for the `$-$' sign.
Since $v_\omega=v^+_\omega+v^-_\omega$, we obtain~\eqref{e:dC-soft}.
\end{proof}

\subsection{Conormal regularity}

We now upgrade Proposition \ref{p:dC-soft} to obtain iterated conormal regularity
uniformly as 
$ \Im \omega \to 0 + $. 
We also relax the assumptions on the right-hand side $g_\omega$:
instead of being smooth uniformly in~$\varepsilon$ it only needs
to be bounded in a certain conormal space uniformly in~$\varepsilon$.
This is the high frequency estimate
used in the proof of Lemma~\ref{l:cvg} below.

As before
we identify $\partial\Omega$ with $\mathbb S^1$ and 1-forms on $\mathbb S^1$ with functions using the coordinate $\theta$ constructed in Lemma~\ref{l:adapted-theta},
which in particular makes it possible to define the operator $\partial_\theta$ on 1-forms.
Fix some defining function $\rho$ of $\Sigma_\lambda=\Sigma^+_\lambda\sqcup\Sigma^-_\lambda$
and an operator $A_{\Sigma_\lambda}\in \Psi^0(\mathbb S^1;T^*\mathbb S^1)$ such that
$\WF(A_{\Sigma_\lambda})\cap (N^*_+\Sigma^-_\lambda\sqcup N^*_-\Sigma^+_\lambda)=\emptyset$
and $A_{\Sigma_\lambda}$ is elliptic on $N^*_-\Sigma^-_\lambda\sqcup N^*_+\Sigma^+_\lambda$.
The estimate~\eqref{eq:ugl} below features the seminorms~\eqref{e:I-s-pm-seminorms}
for the space $I^{\frac 14+}(\mathbb S^1,N^*_+\Sigma^-_\lambda\sqcup N^*_-\Sigma^+_\lambda)$
defined in~\eqref{e:I-s-pm}. The proposition below applies to any $v_\omega,g_\omega\in C^\infty$
solving~\eqref{e:Cle-2}, not just to $v_\omega$ discussed in~\S\ref{s:abo-prepare}.
\begin{prop}
\label{p:dC}
Fix $\beta>0$, $k\in\mathbb N_0$, and $N$.
Then there exist $N_0=N_0(\beta,k)$ and $C=C(\beta,k,N)$ such that
for all small $\varepsilon=\Im\omega>0$ and any solution $v_\omega\in C^\infty(\mathbb S^1;T^*\mathbb S^1)$ 
to~\eqref{e:Cle-2} we have
\begin{equation}
\label{eq:ugl}
\begin{aligned}
\| ( \rho\partial_\theta )^k v_\omega \|_{ H^{-\frac12 - \beta }}
+ \|A_{\Sigma_\lambda} v_\omega\|_{H^k} \leq 
C \big( &\max_{0\leq \ell\leq N_0}\|(\rho\partial_\theta)^\ell g_\omega\|_{H^{-\frac12-\beta}}
\\&  + \|A_{\Sigma_\lambda} g_\omega\|_{H^{N_0}} + \| v_\omega \|_{ H^{-N} } \big) .
\end{aligned}
\end{equation}
\end{prop}
\Remark From the Remark at the end of~\S\ref{s:escape-functions} we see that 
the statement of Proposition~\ref{p:dC} holds locally uniformly in $\lambda$.
More precisely, assume that $\mathcal J\subset (0,1)$ is an open interval such that
each $\lambda\in\mathcal J$ satisfies the Morse--Smale conditions of Definition~\ref{d:2}.
We may choose the coordinate $\theta$, the defining function~$\rho$ of $\Sigma_\lambda$, and the operator $A_{\Sigma_\lambda}$
depending smoothly on~$\lambda$. Then for each compact set $\mathcal K\subset \mathcal J$
we may choose constants $N_0$ and $C$ so that~\eqref{eq:ugl} holds
for all $\lambda=\Re\omega\in\mathcal K$. This can be seen from the remark following Lemma~\ref{l:A-squared}
and the fact that the escape function $\mathbf g$ can be chosen to depend smoothly on~$\lambda$.
\begin{proof}
1. By the discussion following~\eqref{e:I-s-pm-seminorms}
it suffices to show~\eqref{eq:ugl} for one specific choice of~$\rho$.
We choose $\rho$ such that
\begin{equation}
  \label{eq:proprho}
\rho^{-1}(0)=\Sigma_\lambda,\quad
|\partial_\theta\rho|=1\quad\text{on }\Sigma_\lambda.
\end{equation}
Recalling the formula~\eqref{e:1-forms-pulling} for pullback on 1-forms we have the commutation identity of operators on $C^\infty(\mathbb S^1;T^*\mathbb S^1)$
\begin{equation}
  \label{e:udon-comm}
\rho\partial_\theta b^*_\lambda=\varphi b_\lambda^* \rho\partial_\theta
+\psi b_\lambda^*,\quad
\varphi(\theta)={\rho(\theta)\partial_\theta b_\lambda(\theta)\over \rho(b_\lambda(\theta))},\quad
\psi(\theta)={\rho(\theta)\partial^2_\theta b_\lambda(\theta)\over \partial_\theta b_\lambda(\theta)}.
\end{equation}
By~\eqref{eq:proprho} and since $b_\lambda(\Sigma_\lambda)=\Sigma_\lambda$ we have $|\varphi|=1$ on~$\Sigma_\lambda$.

As in~\eqref{eq:defvla}, let $v^\pm_\omega:=\Pi^\pm v_\omega$ and $g^\pm_\omega:=\Pi^\pm g_\omega$.
Since $\WF(A_{\Sigma_\lambda})\cap N^*_\pm\Sigma^\mp_\lambda=\emptyset$, we may fix $\chi^\pm\in C^\infty(\mathbb S^1)$ such that
$$
\supp\chi^\pm\cap \Sigma^\mp_\lambda=\emptyset,\quad
\chi^\pm=1\quad\text{near}\quad \{\theta\in\mathbb S^1\mid (\theta,\pm 1)\in\WF(A_{\Sigma_\lambda})\}.
$$
We will show that there exist $N_0=N_0(\beta,k)$ and $\widetilde\chi^\pm\in C^\infty(\mathbb S^1)$ such that
$\supp\widetilde\chi^\pm\cap \Sigma^\mp_\lambda=\emptyset$ and
\begin{equation}
  \label{e:udon-1}
\begin{aligned}
\|(\rho\partial_\theta)^kv^\pm_\omega\|_{H^{-\frac12-\beta}}
+ \|\chi^\pm v^\pm_\omega\|_{H^k}
\leq C \big( & \max_{0\leq\ell\leq k}\|(\rho\partial_\theta)^\ell g_\omega^\pm\|_{H^{-\frac 12-\beta}}
\\&
+\|\widetilde\chi^\pm g_\omega^\pm\|_{H^{N_0}} + \| v_\omega \|_{ H^{-N} } \big).
\end{aligned}
\end{equation}
Adding these together and using that $v_\omega=v^+_\omega+v^-_\omega$, we get
$$
\begin{gathered}
\|(\rho\partial_\theta)^kv_\omega\|_{H^{-\frac12-\beta}}
+ \|\chi^+v^+_\omega+\chi^-v^-_\omega\|_{H^k}\\
\leq C \big(  \max_{0\leq\ell\leq k}\|(\rho\partial_\theta)^\ell g_\omega\|_{H^{-\frac 12-\beta}}
+\|\widetilde\chi^+g^+_\omega\|_{H^{N_0}} + \|\widetilde\chi^- g^-_\omega\|_{H^{N_0}} + \| v_\omega \|_{ H^{-N} } \big).
\end{gathered}
$$
Since $\chi^+\Pi^++\chi^-\Pi^-$ is elliptic on~$\WF(A_{\Sigma_\lambda})$,
we may estimate $\|A_{\Sigma_\lambda}v_\omega\|_{H^k}$ in terms of $\|\chi^+v^+_\omega+\chi^-v^-_\omega\|_{H^k}$.
Since $\WF(\widetilde\chi^\pm \Pi^\pm)\cap (N^*_+\Sigma^-_\lambda\sqcup N^*_-\Sigma^+_\lambda)=\emptyset$,
we may estimate $\|\widetilde\chi^\pm g^\pm_\omega\|_{H^{N_0}}$ by~\eqref{e:I-s-pm-sameor}. Thus~\eqref{e:udon-1} implies~\eqref{eq:ugl} (possibly with a larger value of~$N_0$).

\noindent 2. It remains to show~\eqref{e:udon-1}. We show an estimate on $v^+_\omega$, with the case
of $v^-_\omega$ handled similarly. We start with the case $k=0$.
Let $E_G$, $\widetilde E_{-G}$ be constructed in~\eqref{eq:EtG}
where the escape function $\mathbf g$ is constructed using parameters
$\alpha_-<\alpha_+$, $\delta>0$ such that
\begin{align}
  \label{e:gatorade-1}
\alpha_-=-\tfrac 12-\beta,&\quad
\supp\chi^+\cap\overline{\Sigma^-_\lambda(\delta)}=\emptyset,\\
  \label{e:gatorade-2}
\alpha_+\geq 0,&\quad
\max_{\pm} \sup_{\Sigma^\pm_\lambda(\delta)} (\partial_\theta b_\lambda) ^{\frac12 + \alpha_\pm }
<1.
\end{align}
Using the equation~\eqref{eq:v2vG2} and Lemma~\ref{l:TG} similarly to the proof of Proposition~\ref{p:dC-soft}, we get the inequality
\begin{equation}
  \label{e:vivalaG-1}
\|v_G\|_{L^2}\leq C \big(\|g_G\|_{L^2}+\|v_\omega\|_{H^{-N}}\big)
\end{equation}
where $v_G:=E_G v^+_\omega$, $g_G:=E_G g^+_\omega$.
By~\eqref{e:anisotroper-1} and~\eqref{e:anisotroper-2} we have
\begin{equation}
  \label{e:vivalaG-2}
\|v^+_\omega\|_{H^{-\frac12-\beta}}+\|\chi^+v^+_\omega\|_{L^2}\leq C\big(\|v_G\|_{L^2}+\|v_{\omega}\|_{H^{-N}}\big).
\end{equation}
By  property~(5) in the remark following Lemma~\ref{l:escape-function} we have $\mathbf g=\alpha_-$ on some neighborhood of $\Sigma^-_\lambda$. Thus we can choose $\widetilde\chi^+\in C^\infty(\mathbb S^1)$ such that
$\supp\widetilde\chi^+\cap \Sigma^-_\lambda=\emptyset$ and
$$
\mathbf g=\alpha_-\quad\text{near}\quad\supp(1-\widetilde\chi^+).
$$
Then $E_G(1-\widetilde\chi^+)\in\Psi^{\alpha_-}_{0+}$ by~\eqref{eq:wecomp}.
Fix $N_0$ such that~\eqref{e:N-0-def} holds, so that
$E_G\in\Psi^{N_0}_{0+}$.
Writing $g_G=E_G(1-\widetilde\chi^+)g^+_\omega+E_G\widetilde\chi^+ g^+_\omega$,
we get
\begin{equation}
  \label{e:vivalaG-3}
\|g_G\|_{L^2}\leq C\big(\|g^+_\omega\|_{H^{-\frac12-\beta}}+\|\widetilde\chi^+ g^+_\omega\|_{H^{N_0}}\big).
\end{equation}
Putting together~\eqref{e:vivalaG-1}--\eqref{e:vivalaG-3}, we get~\eqref{e:udon-1} for~$k=0$.

\noindent 3. We next show~\eqref{e:udon-1} for $k=1$. 
Put for $j\in\mathbb N_0$
$$
v^j:=(\rho\partial_\theta)^j v^+_\omega\ \in\ C^\infty(\mathbb S^1;T^*\mathbb S^1),\quad
v_G^j:=E_G v^j,\quad
g_G^j:=E_G(\rho\partial_\theta)^j g^+_\omega.
$$
We apply $\rho\partial_\theta$ to~\eqref{e:Cle-3} and use~\eqref{e:udon-comm} to get a similar equation on $v^1=\rho\partial_\theta v^+_\omega$ which also involves $v^0=v^+_\omega$:
\begin{equation}
  \label{e:udon-2}
\begin{aligned}
v^1&=B^+_\omega \varphi b_\lambda^* v^1+Q^+_\omega b_\lambda^* v^0+\rho\partial_\theta(\mathscr R^+_\omega v_\omega
+g^+_\omega),\\
Q^+_\omega &=[\rho\partial_\theta,B^+_\omega]+B^+_\omega\psi\ \in\ \Psi^0(\mathbb S^1;T^*\mathbb S^1)\quad\text{uniformly in }\varepsilon.
\end{aligned}
\end{equation}
Applying $E_G$ to~\eqref{e:udon-2}, we get similarly to~\eqref{eq:v2vG2}
\begin{equation}
  \label{e:udon-3}
v_G^1=T_G^1v_G^1+Q_G v_G^0+\mathscr R_G^1 v_\omega+g_G^1
\end{equation}
where $\mathscr R_G^1\in \Psi^{-\infty}$ uniformly in~$\varepsilon$ and
$$
T_G^1:=E_GB^+_\omega\varphi b^*_\lambda \widetilde E_{-G},\qquad
Q_G:=E_GQ^+_\omega b_\lambda^* \widetilde E_{-G}.
$$
We fix the parameters $\alpha_\pm,\delta$ in the construction of the escape function $\mathbf g$
such that we have~\eqref{e:gatorade-1} and the following strengthening of~\eqref{e:gatorade-2}:
\begin{equation}
  \label{e:gatorade-3}
\alpha_+\geq 1,\quad
\max_{\pm} \sup_{\Sigma^\pm_\lambda(\delta)} \max(1,|\varphi|)(\partial_\theta b_\lambda) ^{\frac12 + \alpha_\pm }
<1.
\end{equation}
This is possible by~\eqref{eq:logdb} and since $|\varphi|=1$ on $\Sigma^\pm_\lambda$.

Arguing similarly to the proof of Lemma~\ref{l:TG}, we get the bounds for some $\tau<1$
$$
\|T_G^1 v_G^1\|_{L^2}\leq \tau\|v_G^1\|_{L^2}+C\|v_\omega\|_{H^{-N}},\quad
\|Q_Gv_G^0\|_{L^2}\leq C\|v_G^0\|_{L^2}.
$$
Combining these with~\eqref{e:udon-3} and recalling~\eqref{e:vivalaG-1} we get
\begin{equation}
  \label{e:vivalaG-4}
\begin{aligned}
\|v_G^1\|_{L^2}&\leq C\big(\|v_G^0\|_{L^2}+\|g_G^1\|_{L^2}+\|v_\omega\|_{H^{-N}}\big),\\
\|v_G^0\|_{L^2}&\leq C\big(\|g_G^0\|_{L^2}+\|v_\omega\|_{H^{-N}}\big).
\end{aligned}
\end{equation}
Similarly to~\eqref{e:vivalaG-2}--\eqref{e:vivalaG-3} we have for $j=0,1$
\begin{equation}
  \label{e:vivalaG-5}
\begin{aligned}
\|v^j\|_{H^{-\frac12-\beta}}+\|\chi^+v^j\|_{H^1}&\leq C\big(\|v_G^j\|_{L^2}+\|v_{\omega}\|_{H^{-N}}\big),\\
\|g_G^j\|_{L^2}&\leq C\big(\|(\rho\partial_\theta)^jg^+_\omega\|_{H^{-\frac12-\beta}}+\|\widetilde\chi^+ g^+_\omega\|_{H^{N_0}}
\big).
\end{aligned}
\end{equation}
Together~\eqref{e:vivalaG-4}--\eqref{e:vivalaG-5} give~\eqref{e:udon-1} for $k=1$.

\noindent 4. The case of general $k$ is handled similarly to $k=1$.
We write similarly to~\eqref{e:udon-2}
$$
v^k=B^+_\omega\varphi^kb^*_\lambda v^k+\sum_{j=0}^{k-1}Q^+_{\omega,k,j}b_\lambda^* v^j+(\rho\partial_\theta)^k(\mathscr R^+_\omega v_\omega+g^+_\omega).
$$
Here $Q^+_{\omega,k,j}\in \Psi^0$ uniformly in~$\varepsilon$, is defined inductively as follows:
$$
\begin{aligned}
Q^+_{\omega,k,j}:=\,&
\big([\rho\partial_\theta,B^+_\omega\varphi^{k-1}]+B^+_\omega \varphi^{k-1}\psi\big)\delta_{j,k-1}\\
&
+Q^+_{\omega,k-1,j-1}\varphi+[\rho\partial_\theta,Q^+_{\omega,k-1,j}]+Q^+_{\omega,k-1,j}\psi
\end{aligned}
$$
and we use the notation $\delta_{a,b}=1$ if $a=b$ and $0$ otherwise, and $Q^+_{\omega,k-1,j}=0$ when $j\in \{-1,k-1\}$.
The argument in Step~3 of this proof now goes through, replacing~\eqref{e:gatorade-3} with
\begin{equation}
  \label{e:gatorade-4}
\alpha_+\geq k,\quad
\max_{\pm} \sup_{\Sigma^\pm_\lambda(\delta)} \max(1,|\varphi|^k)(\partial_\theta b_\lambda) ^{\frac12 + \alpha_\pm }
<1\end{equation}
and gives~\eqref{e:udon-1} for any value of~$k$.
\end{proof}
We will also need a refinement concerning Lagrangian regularity.
Let $B^\pm_{\lambda+i0}$ be the operators $B^\pm_\omega$ from Lemma~\ref{l:A-squared} with
$\varepsilon:=0$.
\begin{lemm}
\label{p:020}
Suppose that $ v\in \mathcal D'(\mathbb S^1;T^*\mathbb S^1)$ satisfies~\eqref{e:Cle-2} with $\varepsilon=0:$
\begin{equation}
  \label{e:Cle-2-limiting}
v=B^+_{\lambda+i0}b_\lambda^*v+B^-_{\lambda+i0}b_\lambda^{-*}v+g\quad\text{for some}\quad
g\in C^\infty(\mathbb S^1;T^*\mathbb S^1).
\end{equation}
Similarly to~\eqref{eq:defvla} define $v^\pm:=\Pi^\pm v$.
Then
\begin{equation}
\label{eq:lagra}
v^\pm \in I^{\frac14+} ( \mathbb S^1 , N^* \Sigma^\mp_\lambda ) 
\quad \Longrightarrow \quad v^\pm \in I^{\frac14} ( \mathbb S^1 , N^* \Sigma^\mp_\lambda ) .
\end{equation}
\end{lemm}
\begin{proof}
1. Let us consider $v^+$, with $v^-$ handled similarly.
Similarly to~\eqref{e:Cle-3} we have from~\eqref{e:Cle-2-limiting}
$$
v^+=B^+_{\lambda+i0}b_\lambda^* v^++g_1\quad\text{where}\quad g_1\in C^\infty(\mathbb S^1;T^*\mathbb S^1).
$$
Iterating this $n$ times, where $n$ is the period of the closed trajectories of $b_\lambda$, we see that
\begin{equation}
  \label{e:lagra-1}
v^+=B f^* v^++g_2\quad\text{where}\quad
f:=b_\lambda^n:\mathbb S^1\to\mathbb S^1,
\quad g_2\in C^\infty(\mathbb S^1;T^*\mathbb S^1),
\end{equation}
and the pseudodifferential operator
$$
B:=B^+_{\lambda+i0}\big(b_\lambda^*B^+_{\lambda+i0}b_\lambda^{-*}\big)\cdots
\big((b_\lambda^{n-1})^*B^+_{\lambda+i0}(b_\lambda^{n-1})^{-*}\big)\ \in\ \Psi^0(\mathbb S^1;T^*\mathbb S^1)
$$
satisfies $\sigma(B)=H(\xi)$ by Lemma~\ref{l:A-squared}.

Take arbitrary $x_0\in \Sigma^-_\lambda$ and assume that the coordinate $\theta$ is chosen so that
$\theta(x_0)=0$. Note that $f(0)=0$.
Fix $\chi\in C^\infty(\mathbb R)$ supported on a small neighborhood
of~$0$ which does not contain any other point in $\Sigma^-_\lambda$, and such that $\chi=1$ near $0$.
We write
$$
\chi v^+=u(\theta)\,d\theta\quad\text{for some}\quad u\in \mathcal E'(\mathbb R).
$$
Then $u\in I^{\frac 14+}(\mathbb R,N^*\{0\})$ and by~\eqref{e:lagra-1} we have
\begin{equation}
  \label{e:lagra-2}
u=\widetilde B f^*u+g_3\quad\text{where}\quad 
g_3\in \CIc(\mathbb R)
\end{equation}
and $\widetilde B$ is a compactly supported operator in $\Psi^0(\mathbb R)$ such that
$\sigma(\widetilde B)(0,\xi)=f'(0)H(\xi)$.
Here in~\eqref{e:lagra-1} the operator $f^*$ is the pullback on 1-forms,
and in~\eqref{e:lagra-2} the same symbol denotes the pullback on functions, with the two
related by the formula~\eqref{e:1-forms-pulling}. We can take arbitrary $\widetilde B$ which is
equal to~$Bf'$ near $\theta=0$, since $u$ is smooth away from~0.

It suffices to show that $u\in I^{\frac 14}(\mathbb R,N^*\{0\})$, which
(recalling~\eqref{eq:uosc}) is equivalent to $\widehat u\in S^0(\mathbb R)$.
Note that $\widehat u(\xi)$ is rapidly decaying as $\xi\to-\infty$
since $\WF(v^+)\subset \{\xi>0\}$, so it suffices to study what happens for $\xi>1$.

\noindent 2. We now use the invariance of the principal symbol of $u$ coming from~\eqref{e:lagra-2}.
More precisely, by Lemma~\ref{l:conormal-symbol}, and since the Fourier transform $\widehat g_3$ is rapidly decaying,
the equation~\eqref{e:lagra-2} implies for $\xi>1$
$$
\widehat u(\xi)=\widehat u(\xi/R)+q(\xi)\quad\text{where}\quad
R:=f'(0)>1,\quad
q\in S^{-1+}(\mathbb R).
$$
Iterating this, we see that for any $k\in\mathbb N_0$ and $\eta\geq 1$
\begin{equation}
  \label{e:lagra-3}
\widehat u(R^k \eta)=\widehat u(\eta)+\sum_{\ell=1}^{k} q(R^\ell\eta).
\end{equation}
We now estimate (using for simplicity that $q\in S^{-\frac 12}$ rather than $q\in S^{-1+}$)
$$
\begin{aligned}
\sup_{\xi\geq 1}|\widehat u(\xi)|&=\sup_{k\in\mathbb N_0}\sup_{1\leq \eta\leq R}|\widehat u(R^k\eta)|\\
&\leq \sup_{1\leq\eta\leq R}|\widehat u(\eta)|+C\sum_{\ell=1}^\infty R^{-\frac\ell 2}<\infty.
\end{aligned}
$$
Differentiating~\eqref{e:lagra-3} $m$ times in~$\eta$, we similarly see that $\sup_{\xi\geq 1} \xi^m |\partial_\xi^m\widehat u(\xi)|<\infty$.
This gives $\widehat u\in S^0(\mathbb R)$ and finishes the proof.
\end{proof} 

\section{Microlocal properties of Morse--Smale maps}
\label{s:microp}

Here we prove properties of distributions invariant under
Morse--Smale maps (see Definition \ref{d:2}). We start with a stand alone local result
about distributions invariant under contracting maps. The quantum flux 
defined below \eqref{eq:qufl} is reminiscent of similar quantities appearing 
in scattering theory -- see \cite[(3.6.17)]{DZ-Book}. The wave front condition 
\eqref{eq:WFc} is an analogue of the outgoing condition in scattering theory -- see
\cite[Theorem~3.37]{DZ-Book}. Although technically very different,
Lemma \ref{l:loca} and Proposition \ref{p:globe} are analogous to 
\cite[Lemma 2.3]{DZ-zazi} and play the role  of that lemma in showing
the absence of embedded eigenvalues -- see  \cite[\S 3.2]{DZ-FLOP}.

\subsection{Local analysis}
\label{ss:loca}

In this section 
we assume that $ f: [-1,1] \to (-1,1)$ is a $C^\infty$ map such that
\begin{equation}
\label{eq:propf}
 f ( 0 ) = 0  ,   \ \ \  0 <  f' ( x ) < 1 .
\end{equation}
We also assume that 
\begin{equation}
\label{eq:assu} 
u \in \mathcal D' ( (-1,1) ),\quad
\singsupp u\subset \{0\},\quad
f^* u = u  \text{ on }(-1,1) .
\end{equation}
For $ \chi \in \CIc ( f( -1, 1 )  ) $, $ \chi = 1 $ near $ 0 $
we then define the {\em flux} of $ u $
(understood as an integral of a differential 1-form):
\begin{equation}
\label{eq:qufl}
\mathbf F ( u ) :=   i \int_{(-1,1)} ( f^* \chi  - \chi  ) \bar u\, du . 
\end{equation}
The integral is well defined since $ u $ is smooth on
$ \supp(f^*\chi - \chi) \subset (-1,1)\setminus \{0\}$.

We note that $ \mathbf F ( u ) $ is independent of
$ \chi $. In fact, if $ \chi_j \in C^\infty_{\rm{c} } ( f( - 1, 1 ) ) $, 
$ \chi_1 = \chi_2 $ near $ 0 $,  then the difference 
of the fluxes defined using $ \chi_j $'s in place of $ \chi$, is given by \eqref{eq:qufl} with $ \widetilde\chi = \chi_1 - \chi_2
\in \CIc ( f( -1 , 1) \setminus \{ 0 \} ) $, in place of $\chi $. Since $\widetilde\chi$ is supported
away from~0 we can split the integral:
\[ 
\int_{(-1,1)} (f^* \widetilde\chi - \widetilde\chi ) \bar u\, du = \int (f^*\widetilde\chi) \bar u\,du-\int f^*(\widetilde\chi\bar u\,du)
= \int (f^*\widetilde\chi)(\bar u\,du-f^*(\bar u\,du)) = 0 . 
\]
Here in the first equality we made a change of variables by $f:(-1,1)\to f(-1,1)$ and
in the last equality we used~\eqref{eq:assu}. In fact, this argument shows
that we could take $\chi$ in~\eqref{eq:qufl} to be the indicator function of some interval
$f(a_-,a_+)$ with $-1<a_-<0<a_+<1$, obtaining
\begin{equation}
  \label{e:F-interval}
\mathbf F(u)=i\int_{[a_-,f(a_-)]\sqcup [f(a_+),a_+]}\bar u\,du.
\end{equation}
Similarly we see that $ \mathbf F ( u ) $ is real. For that we take
$ \chi $ real valued so that
\[ \begin{split} 2  \Im  \mathbf F ( u ) & =2 \int_{(-1,1)} (  f^* \chi - \chi ) \Re ( \bar u\, du ) = 
 \int (  f^* \chi - \chi ) d ( |u|^2 ) \\
& =\int |u|^2 \, d (\chi - f^*\chi   )  = 
\int |u|^2 d \chi - \int |u|^2  f^* d \chi = 0,
\end{split}  \]
where in the last line we used~\eqref{eq:assu} and the fact that $ \chi' = 0 $ near $ 0 $. 

The key local result is given in 
\begin{lemm}
\label{l:loca}
Suppose that \eqref{eq:propf} and \eqref{eq:assu} hold. Then
\begin{equation}
\label{eq:WFcc}   \WF ( u ) \subset \{ 0 \} \times \mathbb R_+ , \ \ 
\mathbf F ( u ) \geq 0 \ \ \Longrightarrow \ \ u =\const.
\end{equation}
\end{lemm}
\Remark The wavefront set restriction to positive frequencies is crucial:
for example, if $u$ is the Heaviside function, then~\eqref{eq:assu} holds
and $\mathbf F(u)=0$. A nontrivial example when~\eqref{eq:propf}, \eqref{eq:assu},
and the wavefront set condition in~\eqref{eq:WFcc} hold is
$f(x)=e^{-2\pi}x$, $u(x)=(x+i0)^{ik}$, $k\in\mathbb Z\setminus \{0\}$,
where $\mathbf F(u)=2\pi k(e^{-2\pi k}-1)<0$.

To prove Lemma \ref{l:loca}  we use a standard one dimensional linearization result \cite{SS}. 
For the reader's convenience we present a variant of the 
proof from \cite[Appendice 4]{yocco}.
\begin{lemm}
\label{l:linear}
Assume that $f$ satisfies~\eqref{eq:propf}. Then
there exists a unique $ C^\infty $ diffeomorphism $ h : [ -1,1] \to h ( [ -1 , 1] ) \subset \mathbb R $ 
such that for all $x\in [-1,1]$
\begin{equation}
\label{eq:Schroeder}
h ( f ( x ) ) = f'(0)  h ( x ) , \ \  \  h ( 0 ) = 0 , \ \ \ h' ( 0 ) = 1. 
\end{equation}
\end{lemm}
\begin{proof}
1. We first note that any $ C^1 $ diffeomorphism satisfying \eqref{eq:Schroeder} is
unique. In fact, suppose that $ h_j $, $ j = 1, 2 $ are two such 
diffeomorphisms. With $ a = f' (0)\in(0,1) $,  $ a h_j=h_j \circ f   $ we have
$ a h_1 \circ h_2^{-1} ( x) = h_1 \circ f h_2^{-1} ( x ) = h_1 \circ h_2^{-1} ( ax ) $ for all~$x\in h_2([-1,1])$, so that
\[ 
 h_1\circ h_2^{-1} ( x ) = a^{-n} h_1 \circ h_2^{-1} ( a^n x ) = \lim_{ n \to \infty }
a^{-n} h_1 \circ h_2^{-1} ( a^n x ) = ( h_1 \circ h_2^{-1} )' ( 0 ) x = x. 
\]
Hence it is enough to show that for every $ n $ there exists a $ C^n $ diffeomorphism
satisfying~\eqref{eq:Schroeder}. 

Using the fact that $ a= f' ( 0 ) \in(0, 1) $ we can 
construct a formal power series such that \eqref{eq:Schroeder} holds
for the Taylor series of $ f $ as an asymptotic expansion. 
Using Borel's Lemma \cite[Theorem  1.2.6]{Hormander1} we 
can then construct a diffeomorphism $ h_0 $ of $ [ - 1, 1 ] $ onto itself with that
formal series as Taylor series at $ 0$.
 Then $  h_0 \circ f \circ h_0^{-1} = ax ( 1 + g ( x ) ) $
where $ g \in C^\infty $ vanishes to infinite order at $ 0 $. Hence we can 
assume that
$$
f ( x ) = ax ( 1 + g ( x )).
$$
We might no longer have
$f'<1$ but $f$ is still eventually contracting: there exists $m>0$ such that
the $m$-th iterate $f^m$ satisfies
\begin{equation}
\label{e:eventually-contract}
\partial_x(f^m(x))<1\quad\text{for all }x\in [-1,1].
\end{equation}

\noindent 2. We are now looking for $ h ( x ) = x ( 1 + \varphi ( x ) ) $, $ \varphi ( 0 ) = 0 $ such
that $ h ( a x ( 1 + g ( x ) ) ) = a h  ( x ) $, that is 
$  a x ( 1 + g ( x ) ) ( 1 + \varphi ( f ( x ) ) ) = ax ( 1 + \varphi ( x ) ) $, or
\[     ( 1 + g ( x ) ) ( 1 + \varphi ( f ( x ) ) ) = 1 + \varphi ( x ) . \]
A formal solution is then given by 
$ 1+   \varphi ( x ) = \prod_{\ell=0}^\infty ( 1 + g ( f^\ell ( x ) )) $.
Rather 
than analyse 
this expression, we follow \cite[Appendice 4]{yocco}
and use the contraction mapping principle for Banach
spaces, $ B_n $, 
of $ C^n $ functions on $ [ - \delta, \delta ] $ vanishing to order $ n \geq 2 $ at $ 0 $: 
 we look for $ \varphi  \in B_n $ such that 
\begin{equation}
\label{eq:fixed}    g ( x ) + ( 1 + g ( x ) )  \varphi ( f(x) ) = \varphi ( x ) ,\quad
x\in [-\delta,\delta].
\end{equation}
We claim that 
for $ \delta>0 $ small enough, 
\[ \varphi ( x ) \mapsto ( T \varphi) ( x ) :=  ( 1 + g ( x ) )  \varphi ( f(x) )  \] 
is a contraction on $ B_n $. The norm on $ B_n $ is given by 
\begin{equation}
\label{eq:Bn2D}  \| \varphi \|_{ B_n } := \sup_{ |x| \leq \delta } | \partial^n \varphi ( x ) | , \ \ \ 
 \sup_{ |x| \leq \delta } | \partial^j \varphi ( x ) | \leq C_n \delta^{n-j} \| \varphi \|_{B_n} , 
 \ \ \  \varphi \in B_n ,\ j\leq n,
 \end{equation}
 where the last inequality follows from Taylor's formula.
Since $ f ( x ) = a x ( 1 + g ( x ) ) $, we have 
$  f' ( x)  =  a + \mathcal O ( x^\infty) $ and
$  f ^{(j)} ( x)  =  \mathcal O ( x^\infty )$ for $ j > 1$.
Hence, we obtain for $|x|\leq\delta$, using \eqref{eq:Bn2D} and with homogenous polynomials $ Q_j $, 
\[
\begin{split}  \partial^n [ \varphi ( f ( x ) ) ] & =  \partial^n \varphi ( f ( x ) ) (\partial f ( x ) )^n+ 
\sum_{j=1}^{ n-1 } \partial^j \varphi ( f ( x ) ) Q_j ( \partial f ( x ) , \dots , \partial^{n-j + 1} f ( x ) )\\
&  =  \partial^n \varphi ( f ( x ) ) a^n ( 1 + \mathcal O_n ( \delta ) ) + 
\sum_{j=1}^{n-1}  \mathcal O_n ( \delta^{n-j} ) \| \varphi \|_{ B_n } . 
\end{split} 
\]
It follows that 
$   \| T \varphi \|_{B_n }  \leq \left( a^n  + \mathcal O_n ( \delta )  \right) \| \varphi\|_{ B_n } $, 
which for $ \delta $ small enough (depending on~$ n $) shows that $ T $ is a contraction
on~$ B_n $. That gives a solution $ \varphi $ to \eqref{eq:fixed}. 
 Consequently, we have shown that for every $ n $ there exist $ \delta > 0 $  and
$ \varphi \in C^n ( [ - \delta, \delta ] ) $ such that for
$ h ( x ) = x ( 1 + \varphi ( x ) )$,
\[   h( f ( x ) ) = a h ( x ) , \ \    |x| \leq \delta, \ \ h  \in C^n ( [-\delta, \delta ] ) .
\] 
By~\eqref{e:eventually-contract}, there exists $N>0$ such that $f^N([-1,1])\subset [-\delta,\delta]$.
We extend $h$ to $[-1,1]$ by putting $h(x):=a^{-N}h(f^N(x))$, to obtain
a $C^n$ diffeomorphism $h:[-1,1]\to h([-1,1])$ satisfying~\eqref{eq:Schroeder}.
\end{proof}
%
\begin{proof}[Proof of Lemma \ref{l:loca}]
1. We first note that if $u\in C^\infty((-1,1))$ then $u$ is constant as follows
from~\eqref{eq:assu}: for each $x\in (-1,1)$ we have $u(x)=u(f^N(x))\to u(0)$
as $N\to\infty$. Since we assumed that $\singsupp u\subset \{0\}$
it suffices to show that $u$ is smooth in a neighborhood of~0. 

Making the change of variable given by Lemma~\ref{l:linear}, we may assume
that $f(x)=ax$ for small~$x$, where $a:=f'(0)\in(0,1)$.
Restricting to a neighborhood of~0, rescaling,
and using~\eqref{e:F-interval} we reduce to the following statement: if
\begin{gather}
  \label{e:assumor-1}
  u\in \mathcal D'((-a^{-1},a^{-1})),\quad
\WF(u)\subset \{0\}\times\mathbb R_+,\quad
u(ax)=u(x),\ |x|<a^{-1},\\
  \label{e:assumor-2}
\mathbf F(u):=i\int_{[-1,-a]\sqcup [a,1]} \overline u\,du\geq 0
\end{gather}
then $u\in C^\infty((-1,1))$.

\noindent 2. We next extend $u$ to a distribution on the entire $\mathbb R$.
Fix
\[
\psi \in C^\infty_{\rm{c}} ( (-a^{-1},a^{-1})\setminus [-a,a] ),\quad
    \sum_{ k \in \mathbb Z } \psi ( a^{-k} x ) = 1, \quad x \neq 0 .
\]
Then $\psi u\in\mathcal \CIc(\mathbb R\setminus\{0\})$. Define
\begin{equation}
\label{eq:uu2vv}  v ( x ) := \sum_{ k \in \mathbb Z } (\psi u ) ( a^{-k} x ) \in C^\infty(\mathbb R\setminus \{0\}) .
\end{equation}
Since $u(ax)=u(x)$ for $|x|<a^{-1}$,
we have $u=v$ on $(-a^{-1},a^{-1})\setminus \{0\}$. Thus we may extend $v$
to an element of~$\mathcal D'(\mathbb R)$ so that $u=v|_{(-a^{-1},a^{-1})}$. We note that
\begin{equation}
\label{eq:propv}
\begin{aligned}
v\in\mathscr S'(\mathbb R),&\qquad
v(ax)=v(x),\ x\in\mathbb R,\\
\WF(v)\subset \{0\}\times\mathbb R_+,&\qquad
\mathbf F(v)=\mathbf F(u)\geq 0.
\end{aligned}
\end{equation}
It remains to show that $v\in C^\infty$; in fact, we will show that $v$ is constant.

\noindent 3.
Fix $\chi\in\CIc(\mathbb R)$ such that $\chi=1$ near~$0$ and write
$$
v=v_1+v_2,\quad
v_1:=\chi v,\quad
v_2:=(1-\chi)v.
$$
From~\eqref{eq:uu2vv} we obtain uniformly in~$x\neq 0$,
\[
\partial_x^\ell v ( x ) 
= x^{-\ell } \sum_{k \in \mathbb Z  } ((\bullet)^\ell  (\psi u )^{(\ell)} ) ( a^{-k } x )
= \mathcal O ( x^{-\ell } ) ,
\]
since $ (\bullet)^\ell  (\psi u )^{(\ell)}:x\mapsto x^\ell(\psi u)^{(\ell)}(x) \in C^\infty_{\rm{c}} ( \mathbb R\setminus \{ 0 \} ) $ and the sum is locally finite with a uniformly bounded number of terms.
Hence
$  \partial_x^\ell v_2(x)  = \mathcal 
O ( \langle x \rangle^{-\ell } ) $
which implies that $\widehat v_2(\xi)$ (and thus $\widehat v(\xi)$) is smooth when $\xi\in \mathbb R\setminus \{0\}$
and
\[
\widehat v_2  ( \xi ) = \mathcal O ( \langle \xi \rangle ^{-\infty} ) , \ \ |\xi| \to \infty .
\]
On the other hand the assumption on $ \WF ( v ) $ and~\cite[Proposition~8.1.3]{Hormander1} shows that
$ \widehat v_1 ( \xi ) = \mathcal O ( \langle \xi \rangle^{-\infty} ) $, as $  \xi \to - \infty . $
From \eqref{eq:propv}  we obtain for $ \xi < 0 $ and $ k \in \mathbb N $, 
\begin{equation}
\label{eq:double}  \widehat v ( \xi ) = a^{-1}   \widehat v ( a^{-1}  \xi ) = a^{-k} \widehat v ( a^{-k} \xi ) = 
\mathcal O_\xi ( a^k )  \ \ \Longrightarrow \ \ \widehat v |_{ \mathbb R_-} \equiv 0 . 
\end{equation}

\noindent 4. It follows from~\eqref{eq:double} that the distributional pairing 
\begin{equation}
  \label{e:VV-def}
   V ( z ) := \widehat v( e^{ i z \bullet } ) / 2 \pi , \quad \Im z > 0 
\end{equation}
is well defined and holomorphic in $\{\Im z>0\}$ and $ | V ( z ) | \leq C \langle z \rangle^N / (\Im z)^M $, $  \Im z > 0 $ (with more precise bounds possible).
We also have $v(x)=V(x+i0)$ for $x\in\mathbb R\setminus \{0\}$,
and $V(az)=V(z)$ when $\Im z>0$ which follows from~\eqref{e:VV-def}.
We will now use $ V $ to calculate $ \mathbf F ( v ) $.
We have
\[
    \mathbf F ( v )  =  i \int_{ \gamma_0 } \overline { V ( z ) } \partial_z V( z )\, dz , \quad  \gamma_0 := [-1,-a] \cup [ a, 1 ] , 
\]
where the curve $ \gamma_0 $ is positively oriented. Let $ \gamma_\alpha $, $ \alpha > 0 $, be the
half circle $ |z|=\alpha $, $ \Im z > 0 $ oriented counterclockwise.  Since $ V ( a z ) = V ( z ) $,
\[  
   \int_{ \gamma_1 } \overline { V ( z ) } \partial_z V( z ) dz =
 \int_{\gamma_1 } \overline { V ( a z ) }(  \partial_z V ) ( a z) d ( az ) = 
 \int_{\gamma_a } \overline { V ( z ) } \partial_z V( z ) dz .
\]
\begin{figure}%
\includegraphics{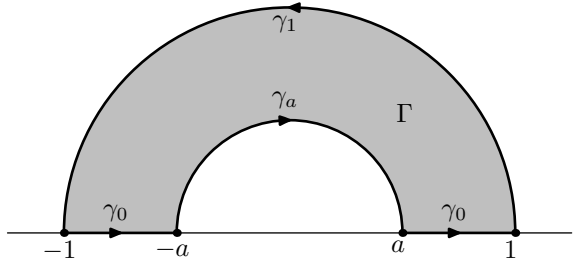}%
\caption{The domain $\Gamma$ used in the proof of Lemma~\ref{l:loca}.}%
\label{f:semi-annulus}%
\end{figure}%
If $ \Gamma $ is the semi-annulus bounded by $ \partial \Gamma : = \gamma_0 + \gamma_1 - \gamma_a $
(see Figure~\ref{f:semi-annulus}) it follows
from the Cauchy--Pompeiu formula \cite[(3.1.9)]{Hormander1} that (with $ z = x + i y $)
\[
 \mathbf F ( v ) =  i \oint_{\partial \Gamma } \overline{ V ( z ) } \partial_z V ( z )\, dz = 
- 2 \int_{\Gamma } \partial_{\bar z}  ( \overline{ V ( z ) } \partial_z V ( z )) \,dx dy  = 
- 2 \int_{\Gamma } | \partial_z V ( z ) |^2 \,dx dy . 
\]
Since we assumed 
$ \mathbf F ( v ) \geq 0 $ it follows that $V$ is constant on $\Gamma$
and thus on the entire upper half-plane,
which implies that $v$ is constant on $\mathbb R\setminus \{0\}$. 
Since functions supported at~0 are linear combinations of derivatives of the delta function
and cannot solve the equation $v(ax)=v(x)$, we see that $v$ is constant on $\mathbb R$,
which finishes the proof.
\end{proof} 

\subsection{A global result}
\label{ss:globe}

We now use the local result in Lemma \ref{l:loca} to obtain a global result for
Morse--Smale diffeomorphisms of the circle.
\begin{prop}
\label{p:globe}
Let $ b : \mathbb \partial \Omega \to \mathbb \partial \Omega $ be a Morse--Smale
diffeomorphism  (see Definition~\ref{d:2}). Denote by $\Sigma^+,\Sigma^-\subset \partial\Omega$
the sets of attractive, respectively repulsive, periodic points of~$b$,
and define $N^*_\pm\Sigma^\pm\subset T^*\partial\Omega$ by~\eqref{eq:N2pm}.
Suppose 
that $ u \in \mathcal D' ( \partial \Omega ) $
satisfies 
\begin{equation}
\label{eq:WFc}
b^*  u = u, \ \ \ \WF ( u ) \subset N^*_+ \Sigma^+ \sqcup N^*_- \Sigma^- .
\end{equation}
Then $ u $ is constant.
\end{prop} 
\Remark The same conclusion holds when the wavefront set condition in~\eqref{eq:WFc}
is replaced by $\WF ( u ) \subset N^*_+ \Sigma^- \sqcup N^*_- \Sigma^+ $,
as can be seen by applying Proposition~\ref{p:globe} to the complex conjugate $\overline u$.
\begin{proof}
We introduce fluxes associated to $ g := b^n $, where $ n $ is the minimal period of periodic points of $ b $.
For that we take two arbitrary cutoff functions
$$
\chi_\pm \in C^\infty(\partial\Omega),\quad
\supp(1-\chi_\pm)\cap \Sigma^\pm=\emptyset,\quad
\supp\chi_\pm\cap \Sigma^\mp=\emptyset.
$$
Assume that $u$ satisfies~\eqref{eq:WFc}
and define the fluxes (where we again use positive orientation
on $\partial\Omega$ to define the integrals of 1-forms):
$$
\begin{aligned}
\mathbf F_+(u)&\,:=i \int_{\partial\Omega}(g^* \chi_+-\chi_+ )\,
\overline{u}\, du,\\
\mathbf F_-(u)&\,:= i \int_{\partial\Omega}(( g^{-1} )^*\chi_-- \chi_- )\,
\overline{u}\, du.
\end{aligned}
$$
The integrals above are well-defined since
$g^* \chi_+-\chi_+$ and $ (g^{-1} )^* \chi_-- \chi_- $
are supported in $\partial\Omega\setminus(\Sigma^+\sqcup\Sigma^-)$,
where $u$ is smooth.
Moreover as in the case of $ \mathbf F ( u)$ defined in \eqref{eq:qufl}, 
 $\mathbf F_\pm(u)$ are real and do not depend on the choice of~$\chi_\pm$.
 We also note that (by taking $ \chi_\pm $ real valued) 
 \begin{equation}
 \label{eq:u2bu}
 \mathbf F_\pm ( \bar u ) = - \overline{ \mathbf F_\pm  ( u ) } = - \mathbf F_\pm ( u ) . 
 \end{equation}
 Since $\mathbf F_\pm(u)$ are independent of~$\chi_\pm$, we may choose
$\chi_+ :=1- (g^{-1})^* \chi_- $ to get the identity
\begin{equation}
\label{e:flux-equal}
  \mathbf F_+(u)=\mathbf F_-(u).
\end{equation}
Let $\Sigma^+=\{x_1^+,\dots,x_m^+\}$.
By taking $\chi_+=\chi_+^1+\dots+\chi_+^m$ where
each $\chi_+^j$ is supported near $x_j^+$, we can write $\mathbf F_+(u)
=\mathbf F_{+}^1(u)+\dots+\mathbf F_{+}^m(u)$.
We may apply Lemma~\ref{l:loca} with $ f $ defined by $ g $ in local coordinates near 
$ x_j^+ \simeq 0 $ to see that $\mathbf F_{+}^j(u)\leq 0$ with equality only
if $u$ is constant near $x_j^+$. Adding these together, we see that
$\mathbf F_+(u)\leq 0$ with equality only if $u$ is locally constant near~$\Sigma^+$.

Arguing similarly near $\Sigma^-$, using $f:=g^{-1}$ and replacing $u$ by $\overline u$,
with $\WF(\overline u)=\{(x,-\xi)\mid (x,\xi)\in\WF(u)\}$,
we see that $\mathbf F_-(u)\geq 0$ with equality only if $u$ is locally constant near~$\Sigma^-$.
By~\eqref{e:flux-equal} we then see that $u$ is locally constant near $\Sigma^+\sqcup\Sigma^-$
and hence $ u \in C^\infty 
( \partial \Omega ) $. 
Since for $ x \in \partial \Omega \setminus \Sigma^- $, $ g^n ( x ) \to x_0 $, for some
$ x_0 $ in $ \Sigma^+ $, we conclude that $ u \in C^\infty $ takes finitely many values and hence
is constant.
\end{proof}

\section{Limiting absorption principle}
\label{s:liap} 

In this section we consider operators
\begin{equation}
\label{eq:defP2}
\begin{gathered}   P := \partial_{x_2}^2\Delta_\Omega^{-1}  : H^{-1} ( \Omega )
 \to H^{-1} ( \Omega ) , \ \\
P ( \omega ) := \partial_{x_2}^2 - \omega^2 \Delta =
( P - \omega^2 ) \Delta_\Omega : H^1_0 ( \Omega ) \to H^{-1} ( \Omega ) .
\end{gathered}
\end{equation}
We prove the limiting absorption principle for $ P$ in the form presented in
Theorem~\ref{t:2}.
To do this we follow~\S\ref{s:redbone} to reduce the equation
$P(\omega)u_\omega=f$ to the boundary $\partial\Omega$. We next analyse the resulting
`Neumann data' $v_\omega=\mathcal N_\omega u_\omega$ (see~\eqref{e:N-omega-def}) uniformly as $\varepsilon=\Im\omega\to 0+$,
using the high frequency estimates of~\S\ref{s:abo} and the absense
of embedded spectrum following from the results of~\S\ref{s:microp}.
This is slightly non-standard since the boundary has characteristic points and 
the problem
changes from elliptic to hyperbolic as $ \Im \omega \to 0+ $.

\subsection{Poincar\'e spectral problem}
\label{s:sap}

We recall (see for instance~\cite[Chapter~6]{Davies-Book}) that $ \Delta = \partial_{x_1}^2 + \partial_{x_2}^2 $ with the domain 
$ H^2 ( \Omega ) \cap H^1_0 ( \Omega ) $, ($ H^1_0 ( \Omega ) $ is the closure of
$ C_{\rm{c}}^\infty ( \Omega ) $ with respect to the norm $ \| \bullet \|_{ H^1_0 ( \Omega ) } $
below) is a negative definite unbounded self-adjoint operator on $ L^2 ( \Omega ) $.
Its inverse is an isometry,
\[  \Delta_\Omega^{-1} : H^{-1} ( \Omega ) \to H_0^1 ( \Omega )  ,\]
with inner products on these Hilbert spaces given by 
\[ \langle u , w \rangle_{ H_0^1 ( \Omega ) } := \int_\Omega  \nabla u  \cdot \overline{\nabla w} \,  dx, 
\ \ \ 
\langle U, W \rangle_{ H^{-1} ( \Omega ) } := \langle \Delta_\Omega^{-1} U , \Delta_\Omega^{-1} W 
\rangle_{H^1_0 ( \Omega )} . \]
Since $ \partial_{x_2}^2 : H^1_0 ( \Omega ) \to H^{-1} ( \Omega ) $ the operator 
$ P $ in \eqref{eq:defP2} is indeed bounded on $ H^{-1} ( \Omega ) $. 

Let $ \{ e_\alpha \}_{\alpha\in A} $ be an 
$L^2 (\Omega)$-orthonormal basis of eigenfunctions of $ - \Delta_\Omega $:
\[ - \Delta_\Omega e_\alpha = \mu_\alpha^2 e_\alpha,  \ \ e_\alpha |_{\partial \Omega } = 0 , \ \ 
\langle e_\alpha, e_\beta \rangle_{L^2 ( \Omega) } = \delta_{\alpha, \beta} .\]
Then  
$ \{ \mu_\alpha e_\alpha \}_{\alpha \in A } $ is an 
orthonormal basis of the Hilbert space $ H^{-1} ( \Omega ) $.  
 The matrix elements of $ P $ in this basis are given by 
\[  \begin{split}  \langle P \mu_\alpha e_\alpha, \mu_\beta e_\beta \rangle_{ H^{-1} } &=
 \langle \Delta_\Omega^{-1} \partial_{x_2}^2 \mu_\alpha^{-1} e_\alpha, \mu_\beta^{-1} e_\beta \rangle_{ H^1_0 } = -\mu_\alpha^{-1} \mu_\beta^{-1} \langle \partial_{x_2}^2 e_\alpha, e_\beta
\rangle_{L^2} \\
& =  \mu_\alpha^{-1} \mu_{\beta}^{-1} \langle \partial_{x_2} e_\alpha, \partial_{x_2} e_\beta \rangle_{L^2 ( \Omega ) } , \end{split} \]
where the last integration by parts is justified as $ e_\beta|_{\partial \Omega } = 0 $.
This shows that $ P $ is a bounded self-adjoint operator on~$H^{-1}(\Omega)$. This representation is particularly useful in numerical 
calculations needed to produce Figure~\ref{f:tilted}. Testing $ P$ against
$\Delta^2( \psi ( x) e^{ i \langle n , x \rangle }) $, $ \psi \in C_{\rm{c}}^\infty ( \Omega ) $,  
$ n \in \mathbb Z^2 $, shows that 
\[ \Spec ( P ) = [ 0 , 1] , \]
see \cite[Theorem 2]{RalstonR}. In particular, for  $ \omega^2 \in \mathbb C \setminus 
[ 0 , 1 ] $, 
\begin{equation}
\label{eq:resolve}
\| P ( \omega )^{-1}\|_{H^{-1}(\Omega)\to H^1_0(\Omega)} 
=\|(P - \omega^2)^{-1}\|_{H^{-1}(\Omega)\to H^{-1}(\Omega)}
=  {1\over d ( \omega^2 , [0,1])}.
\end{equation}
Limiting absorption principle in its most basic form means showing we have limiting operators
acting on smaller spaces with values in larger spaces: for $\lambda\in (0,1) $
satisfying the Morse--Smale conditions
\begin{equation}
\label{eq:lap}
(P - \lambda^2 -i0 )^{-1} : C^\infty_{\rm{c}}  ( \Omega ) \to H^{-\frac32-} ( \Omega ) , \ \ \ 
P ( \lambda + i 0  )^{-1} : C^\infty_{\rm{c}}  ( \Omega ) \to H^{\frac12-} ( \Omega ) .
\end{equation}

\subsection{Regularity of limits as $\varepsilon\to0+$}
\label{s:anal-conv}

In this section we use the results of~\S\ref{s:abo} to get a conormal regularity statement for weak limits of boundary data.
In~\S\ref{s:anal-boundary} below we apply this to the Neumann data $v_\omega=\mathcal N_{\omega}u_\omega$,
$P(\omega)u_\omega=f$.

Since the conormal spaces used below depend on~$\lambda=\Re\omega$, we need to define what it means
for a sequence of distributions to be bounded in these spaces uniformly in~$\lambda$. 
Assume that $\mathcal J\subset (0,1)$
is an open interval such that each $\lambda\in\mathcal J$ satisfies the Morse--Smale conditions of Definition~\ref{d:2}.
Let $\Sigma^\pm_\lambda$ be defined in~\eqref{e:Sigma-pm-def} and $\Sigma_\lambda=\Sigma^+_\lambda\sqcup\Sigma^-_\lambda$.
Fix a defining function $\rho_\lambda\in C^\infty(\partial\Omega;\mathbb R)$ of $\Sigma_\lambda$
and a pseudodifferential operator $A_{\Sigma_\lambda}\in\Psi^0(\partial\Omega)$ such that
$\WF(A_{\Sigma_\lambda})\cap (N^*_+\Sigma^-_\lambda\sqcup N^*_-\Sigma^+_\lambda)=\emptyset$
and $A_{\Sigma_\lambda}$ is elliptic on $N^*_-\Sigma^-_\lambda\sqcup N^*_+\Sigma^+_\lambda$.
We choose both $\rho_\lambda$ and $A_{\Sigma_\lambda}$ depending smoothly on~$\lambda\in\mathcal J$.

Given two sequences $\lambda_j\to\lambda\in\mathcal J$ and
$v_j\in C^\infty(\partial\Omega)$, we say
that
$$
v_j\quad\text{is bounded in}\quad I^{s+}(\partial\Omega,N^*_+\Sigma^-_{\lambda_j}\sqcup N^*_-\Sigma^+_{\lambda_j})\quad\text{uniformly in~$j$}
$$
if each of the seminorms~\eqref{e:I-s-pm-seminorms} is bounded uniformly in~$j$.
We can similarly talk about uniform boundedness of sequences of 1-forms
$v_j\in C^\infty(\partial\Omega;T^*\partial\Omega)$, identifying these with scalar distributions
using a coordinate $\theta$.
\begin{lemm}
  \label{l:cvg}
Assume that $\omega_j\to\lambda\in\mathcal J$, $\Im\omega_j>0$, and the sequence $v_j\in C^\infty(\partial\Omega;T^*\partial\Omega)$
has the following properties:
\begin{gather}
  \label{e:cvg-1}
v_j\to v_0\quad\text{in}\quad H^{-N}\quad\text{for some }N,\\
  \label{e:cvg-2}
\mathcal C_{\omega_j}v_j\quad\text{is bounded in}\quad I^{-\frac34+}(\partial\Omega,N^*_+\Sigma^-_{\lambda_j}\sqcup N^*_-\Sigma^+_{\lambda_j})
\quad\text{uniformly in $j$},
\end{gather}
where $\lambda_j=\Re\omega_j$ and $\mathcal C_{\omega_j}$ was defined in~\S\ref{s:restricted-slp}.
Then we have
\begin{gather}
  \label{e:cvg-3}
v_j\to v_0\quad\text{in}\quad H^{-\frac12-\beta}\quad\text{for all}\quad \beta>0,\\
  \label{e:cvg-4}
v_0\in I^{\frac14+}(\partial\Omega,N^*_+\Sigma^-_{\lambda}\sqcup N^*_-\Sigma^+_{\lambda}).
\end{gather}
\end{lemm}
\Remark In fact we have $v_j\to v_0$ in $I^{\frac14+}(\partial\Omega,N^*_+\Sigma^-_{\lambda_j}\sqcup N^*_-\Sigma^+_{\lambda_j})$
where convergence is defined using the seminorms~\eqref{e:I-s-pm-seminorms}~--
see the last paragraph of the proof below.
\begin{proof}
The function $v_j$ satisfies the equation~\eqref{e:Cle-2}:
\begin{equation}
  \label{e:boundyc-2}
v_j=B^+_{\omega_j}b_{\lambda_j}^*v_j+B^-_{\omega_j}b_{\lambda_j}^{-*}v_j+g_j,\quad
g_j=(I-\mathcal A_{\omega_j})\mathcal E_{\omega_j}d\mathcal C_{\omega_j}v_j.
\end{equation}
Here the operator $\mathcal A_{\omega_j}=(\gamma^+_{\lambda_j})^*A^+_{\omega_j}+(\gamma^-_{\lambda_j})^*A^-_{\omega_j}$ is defined in~\eqref{e:Cle-1}.

Applying Proposition~\ref{p:dC} to $v_\omega:=v_j$ we see that for each $k\in\mathbb N_0$ and $\beta>0$ there exist $N_0,C$ independent of~$j$ such that
\begin{equation}
  \label{e:bouncy-castle-1}
\begin{aligned}
\|(\rho_{\lambda_j}\partial_\theta)^kv_j\|_{H^{-\frac12-\beta}}+\|A_{\Sigma_{\lambda_j}}v_j\|_{H^k}
\leq C\big(&
\max_{0\leq \ell\leq N_0}\|(\rho_{\lambda_j}\partial_\theta)^\ell g_j\|_{H^{-\frac12-\beta}}\\
&+\|A_{\Sigma_{\lambda_j}}g_j\|_{H^{N_0}}+\|v_j\|_{H^{-N}}\big).
\end{aligned}
\end{equation}
The pseudodifferential operators $A^\pm_{\omega_j},\mathcal E_{\omega_j}$ are bounded on $ I^{\frac14+}(\partial\Omega,N^*_+\Sigma^-_{\lambda_j}\sqcup N^*_-\Sigma^+_{\lambda_j})$ uniformly in~$j$, as are the pullback operators $(\gamma^\pm_{\lambda_j})^*$
(see Remark~1 after Proposition~\ref{p:summa} and the end of the proof of Lemma~\ref{l:C-conormal}).
Thus by~\eqref{e:cvg-2} we see that $g_j$ is bounded uniformly in~$j$ in the space $I^{\frac14+}(\partial\Omega,N^*_+\Sigma^-_{\lambda_j}\sqcup N^*_-\Sigma^+_{\lambda_j})$. Moreover, by~\eqref{e:cvg-1} $\|v_j\|_{H^{-N}}$ is bounded uniformly in~$j$ as well.
It follows that the right-hand side of~\eqref{e:bouncy-castle-1}, and thus its left-hand side as well, is bounded uniformly in~$j$ for any choice of $k\in\mathbb N_0,\beta>0$.

Take arbitrary $0<\beta'<\beta$.
Then $\|v_j\|_{H^{-\frac12-\beta'}}$ is bounded in~$j$.
Using compactness of the embedding 
$ H^{ -\frac12 - \beta' } \hookrightarrow H^{-\frac12 - \beta } $ 
we see that each subsequence of $\{v_j\}$ has a subsequence converging in $H^{-\frac12-\beta}$;
the limit of this further subsequence has to be equal to $v_0$ by~\eqref{e:cvg-1}.
This implies~\eqref{e:cvg-3}.

A similar argument using again the boundedness of the left-hand side of~\eqref{e:bouncy-castle-1} shows
that $(\rho_{\lambda_j}\partial_\theta)^kv_j\to (\rho_\lambda \partial_\theta)^kv_0$ in $H^{-\frac 12-\beta}$
for all~$k\in\mathbb N_0,\beta>0$ and $A_{\Sigma_{\lambda_j}}v_j\to A_{\Sigma_\lambda}v_0$ in $C^\infty$.
In particular, this implies that $(\rho_\lambda \partial_\theta)^kv_0\in H^{-\frac 12-}$
and $A_{\Sigma_\lambda}v_0\in C^\infty$ which by~\eqref{e:I-s-pm-seminorms} gives~\eqref{e:cvg-4}.
\end{proof}

\subsection{Uniqueness for the limiting problem}
\label{s:anal-uniq}

We next use the analysis of~\S\ref{s:microp}
to show a uniqueness result for the restricted single layer potential operator $\mathcal C_{\lambda+i0}$
(see~\S\ref{s:restricted-slp})
in the space of distributions satisfying additional conditions.
This will give us the lack of embedded spectrum for the operator~$P$ in the Morse--Smale case.
To formulate this result, we recall the operators $R_{\lambda+i0}:g\mapsto E_{\lambda+i0}*g$ defined in~\eqref{eq:Rlag}
and $\mathcal I:\mathcal D'(\partial\Omega;T^*\partial\Omega)\to \mathcal E'(\mathbb R^2)$ defined in~\eqref{e:I-def}.
\begin{lemm}
  \label{l:unique-2}
Let $\lambda\in (0,1)$ satisfy the Morse--Smale conditions of Definition~\ref{d:2}.
Assume that $v\in \mathcal D'(\partial\Omega;T^*\partial\Omega)$
lies in $I^{s} ( \partial \Omega, N^*_+ \Sigma_\lambda^-  \sqcup
N^*_- \Sigma_\lambda^+)$ for some $s$ (see~\eqref{e:I-s-pm}),
where $\Sigma^\pm_\lambda$ are defined in~\eqref{e:Sigma-pm-def}. Then
\begin{equation}
  \label{e:unique-2}
\mathcal C_{\lambda+i0}v=0,\quad
\supp (R_{\lambda+i0}\mathcal Iv)\subset\overline\Omega
\quad\Longrightarrow\quad
v=0.
\end{equation}
\end{lemm}
\begin{proof}
1. Put $U:=R_{\lambda+i0}\mathcal Iv\in \mathcal D'(\mathbb R^2)$. Since $P(\lambda)E_{\lambda+i0}=\delta_0$ by~\eqref{e:fund-sol-eqn-2}, we have
\begin{equation}
  \label{e:uni2-1}
P(\lambda)U=\mathcal Iv.
\end{equation}
We first show that
\begin{equation}
  \label{e:uni2-2}
\supp U\subset\partial\Omega.
\end{equation}
By the second assumption in~\eqref{e:unique-2} we have $\supp U\subset \overline\Omega$, thus it suffices to show
that $u=0$ where $u:=U|_{\Omega}=S_{\lambda+i0}v$
and $S_{\lambda+i0}$ is the limiting single layer potential
defined in~\eqref{e:s-lambda-def-real}.

Since $\supp\mathcal Iv\subset\partial\Omega$, from~\eqref{e:uni2-1} we have $P(\lambda)u=0$.
As $ \lambda \in (0,1)$,
$ P ( \lambda ) $ is a constant coefficient hyperbolic operator. 
In view of~\eqref{e:L-pm-def} and~\eqref{eq:L2l} we then have,
denoting $\ell^\pm(x):=\ell^\pm(x,\lambda)$, $\ell^\pm_{\min}:=\ell^\pm ( x^\pm_{\min} )$,
$\ell^\pm_{\max}:=\ell^\pm ( x^\pm_{\max} )$
\begin{equation}
\label{eq:uxpm} 
 u ( x ) := u_+ ( \ell^+ ( x) )-u_-(\ell^-(x)), \quad  x \in \Omega, \quad
 u_\pm \in \mathcal D' ( (\ell^\pm_{\min},\ell^\pm_{\max} ) ).
\end{equation} 
From~\eqref{eq:con-new} we see that $u\in I^{s-\frac 54}(\overline\Omega,\Lambda^-(\lambda))$,
in particular by~\eqref{e:no-characteristic} $u$ is smooth up to the boundary near the
characteristic set $\mathscr C_\lambda$. It follows that
$u_\pm$ are smooth near the boundary points $\ell^\pm_{\min},\ell^\pm_{\max}$ up to the boundary.
Define the pullbacks of $u_\pm$ to $\partial\Omega$ by the maps $\ell^\pm$,
$$
w_\pm = u_\pm ( \ell^\pm ( x  ) ) |_{\partial \Omega }\in \mathcal D'(\partial\Omega),\quad
(\gamma^\pm)^*w_\pm=w_\pm.
$$
From the proof of conormal regularity of~$u$ in Lemma~\ref{l:con} we see that
$\WF(w_\pm)\subset N^*_+ \Sigma^-_\lambda  \sqcup 
N^*_- \Sigma^+_\lambda$. 

The restriction $u|_{\partial\Omega}$ is equal to both $w_+-w_-$ and
$\mathcal C_{\lambda+i0}v$. Thus by the first assumption in~\eqref{e:unique-2}
we have $w_+=w_-$. Denoting $w:=w_+=w_-$, we have
$$
( \gamma^\pm )^* w = w ,\quad  \WF ( w ) \subset N^*_+ \Sigma^-_\lambda  \sqcup 
N^*_- \Sigma^+_\lambda.
$$
This implies that $ b^* w = w $ and we can apply Proposition~\ref{p:globe} to see that $ w $ is constant. But then $ u_\pm $ are constant
and $ u=0 $, giving~\eqref{e:uni2-2}.

\noindent 2. We now show that $v=0$ away from the characteristic set
$\mathscr C_\lambda$ of $P(\lambda)$ on $\partial\Omega$
(see~\eqref{e:no-characteristic}). For each $x_0\in\partial\Omega\setminus\mathscr C_\lambda$ 
we can find a neighbourhood $ V \subset \mathbb R^2 $ of 
$ x_0 $ and coordinates
$(y_1,y_2)$ on~$V$
such that for some open interval $\mathscr I\subset\mathbb R$
\[ 
  \partial \Omega \cap V = \{ y_1 = 0,\ y_2\in\mathscr I\}, \quad P( \lambda )|_V = \sum_{ |\alpha| \leq 2} 
a_\alpha (y) \partial_y^\alpha , \quad  a_{2,0} \neq 0 . 
\]
(The non-characteristic property means that the conormal bundle of $\{ y_1 = 0 \} $ is 
disjoint from the set of zeros of $ \sum_{ |\alpha | = 2} a_\alpha \eta^\alpha $.)
Now, by~\cite[Theorem~2.3.5]{Hormander1} we see that~\eqref{e:uni2-2} implies
$ U|_{ V} = \sum_{ k \leq K } u_k (y_2 ) \delta^{(k)} ( y_1 ) $, $ u_k \in \mathcal D'  (\mathscr I ) $. Hence, 
for some $ \widetilde u_k \in \mathcal D' ( \mathscr I) $, 
\[
    P( \lambda ) U|_V = a_{2,0}(y) u_K ( y_2) \delta^{(K+2)} (y_1 ) + \sum_{ k \leq K +1 } \widetilde u_k (y_2 ) \delta^{ (k)} (y_1 )  . 
\]
By~\eqref{e:uni2-1} we have $ P ( \lambda ) U |_V = \mathcal I v |_V =  a(y_2) v ( y_2) \delta (y_1)$, $a\neq 0$.
Thus $ u_K = 0$.
(Here we use $y_2$ as a local coordinate on $\partial\Omega$ to identify $v|_{\partial\Omega\cap V}$
with a distribution on~$\mathscr I$.)
Iterating 
this argument shows that $ U |_ V =0 $ which means $ v |_{ V \cap \partial \Omega } = 0 $.

\noindent 3. We have shown that $\supp v$ is contained in the finite set $\mathscr C_\lambda$.
On the other hand, $v\in I^{s} ( \partial \Omega, N^*_+ \Sigma_\lambda^-  \sqcup
N^*_- \Sigma_\lambda^+ )$ is smooth away from
$\Sigma_\lambda$. Since $\Sigma_\lambda\cap\mathscr C_\lambda=\emptyset$ by~\eqref{e:no-characteristic}, we get $v=0$.
\end{proof}
\Remark The proof would be simpler if we knew that the limiting single layer potential operators $S_{\lambda+i0}$ were injective
acting on the conormal spaces~\eqref{eq:con-new}~-- this would imply the injectivity
of $\mathcal C_{\lambda+i0}$ on $I^s(\partial\Omega,N^*_+\Sigma^-_\lambda\sqcup N^*_-\Sigma^+_\lambda)$
without the support condition in~\eqref{e:unique-2}
as follows from Step~1 of the proof above. However, that is not clear. 
Under the dynamical assumptions made here, the proof of Proposition~\ref{p:dC} shows that 
$ \ker S_{\lambda+i0}\subset \ker \mathcal C_{\lambda+i0} $ is finite dimensional but injectivity seems to 
be a curious open problem.

\subsection{Boundary data analysis}
\label{s:anal-boundary}

Fix $f\in \CIc(\Omega)$ and let $\mathcal J\subset (0,1)$ be an open interval such that
each $\lambda\in\mathcal J$ satisfies the Morse--Smale conditions of Definition~\ref{d:2}.
Consider the solution to the boundary-value problem~\eqref{eq:ellip}:
$$
u_\omega\in C^\infty(\overline\Omega),\quad
P(\omega)u_\omega=f,\quad
u_\omega|_{\partial\Omega}=0,\quad
\omega\in \mathcal J+i(0,\infty).
$$
In this section, we combine the results of~\S\S\ref{s:anal-conv}--\ref{s:anal-uniq} to
study the behaviour as $\Im\omega\to 0+$ of the `Neumann data'
defined using~\eqref{e:N-omega-def}:
$$
v_\omega  := \mathcal N_\omega u_\omega\ \in\ C^\infty(\partial\Omega;T^*\partial\Omega).
$$
We first show the following convergence statement:
\begin{prop}
\label{l:boundv}
Assume that $\omega_j\to\lambda\in\mathcal J$, $\Im\omega_j>0$.
Then for all $\beta>0$
\begin{equation}
  \label{e:limitator}
v_{\omega_j}\to v_{\lambda+i0}\quad\text{in }H^{-\frac 12-\beta}(\partial\Omega;T^*\partial\Omega)\quad\text{as }j\to\infty
\end{equation}
where $v_{\lambda+i0}\in H^{-\frac 12-}(\partial\Omega;T^*\partial\Omega)$ is the unique distribution such that
\begin{equation}
  \label{e:limitator-characterization}
\begin{gathered}
v_{\lambda+i0}\in I^{\frac14 + } ( \partial\Omega, N^*_+ \Sigma_\lambda^-  \sqcup
N^*_- \Sigma_\lambda^+ ),\\
\mathcal C_{\lambda+i0}v_{\lambda+i0}=(R_{\lambda+i0}f)|_{\partial\Omega},\quad
\supp R_{\lambda+i0}(f-\mathcal Iv_{\lambda+i0})\subset\overline\Omega.
\end{gathered}
\end{equation}
Moreover, $v_{\lambda+i0}\in I^{\frac 14}(\partial\Omega, N^*_+ \Sigma_\lambda^-  \sqcup
N^*_- \Sigma_\lambda^+ )$.
\end{prop}
\begin{proof} 
1. We start with a few general observations.
Recall the equation~\eqref{e:C-lambda-equator-2}:
\begin{equation}
  \label{e:C-lambda-eqrev}
\mathcal C_\omega v_\omega=(R_\omega f)|_{\partial\Omega},\quad
R_\omega f=E_\omega * f\in C^\infty(\mathbb R^2).
\end{equation}
Moreover, by~\eqref{eq:U2u} we have
\begin{equation}
  \label{e:bv-U}
\indic_{\Omega}u_\omega=R_\omega(f-\mathcal Iv_{\omega}).
\end{equation}
By Lemma~\ref{l:loc}, $E_{\omega_j}\to E_{\lambda+i0}$ in $\mathcal D'(\mathbb R^2)$.
Passing to the limit in~\eqref{e:C-lambda-eqrev} we see that
\begin{equation}
  \label{e:bv-Comega}
\mathcal C_{\omega_j}v_{\omega_j}\to (R_{\lambda+i0}f)|_{\partial\Omega}\quad\text{in}\quad C^\infty(\partial\Omega).
\end{equation}

\noindent 2. We now show a boundedness statement: for each $\beta>0$ there exists a constant $C$ (depending on~$f$ and~$\beta$)
such that for all $j$
\begin{equation}
\label{eq:boundv}
\| v_{\omega_j} \|_{ H^{-\frac12-\beta }(\partial\Omega;T^*\partial\Omega) } \leq C .
\end{equation}
We proceed by contradiction. If~\eqref{eq:boundv} fails then we may pass to a subsequence
to make $\|v_{\omega_j}\|_{H^{-\frac 12-\beta}}\to \infty$.
We then put
\[
v_j := {v_{ \omega_j } / \| v_{ \omega_j } \|_{ H^{-\frac12 - \beta} } } ,\quad
u_j := {u_{ \omega_j } / \| v_{ \omega_j } \|_{ H^{-\frac12 - \beta} } }.
\]
By~\eqref{e:bv-Comega} we have
\begin{equation}
  \label{e:bv-1}
\mathcal C_{\omega_j}v_j\to 0\quad\text{in}\quad C^\infty(\partial\Omega).
\end{equation}
By compactness of the embedding $ H^{ -\frac12 - \beta } \hookrightarrow H^{-N } $,
where we fix $N>\frac12+\beta$, we may pass to a subsequence
to make
$$
v_j\to v_0\quad\text{in}\quad H^{-N}\quad\text{for some}\quad
v_0\in H^{-N}(\partial\Omega;T^*\partial\Omega).
$$
Now Lemma~\ref{l:cvg} applies and gives
\begin{equation}
  \label{e:bv-2}
v_j\to v_0\quad\text{in}\quad H^{-\frac 12-\beta}(\partial\Omega;T^*\partial\Omega),\quad
v_0\in I^{\frac 14+}(\partial\Omega,N^*_+\Sigma^-_\lambda\sqcup N^*_-\Sigma^+_\lambda).
\end{equation}
By Lemma~\ref{l:C-lambda-strong-limit} and passing to the limit in~\eqref{e:bv-U} using Lemma~\ref{l:loc} we get
$$
\mathcal C_{\omega_j}v_j\to \mathcal C_{\lambda+i0}v_0\quad\text{in}\quad\mathcal D'(\partial\Omega),\qquad
\indic_\Omega u_j\to -R_{\lambda+i0}\mathcal I v_0\quad\text{in}\quad \mathcal D'(\mathbb R^2).
$$
Thus by~\eqref{e:bv-1} and since $\supp(\indic_\Omega u_j)\subset\overline\Omega$ for all $j$
we have
$$
\mathcal C_{\lambda+i0}v_0=0,\qquad
\supp(R_{\lambda+i0}\mathcal Iv_0)\subset\overline\Omega.
$$
Now Lemma~\ref{l:unique-2} gives $v_0=0$.
On the other hand the first part of~\eqref{e:bv-2} and the fact that $\|v_j\|_{H^{-\frac12-\beta}}=1$
imply that $\|v_0\|_{H^{-\frac12-\beta}}=1$, which gives a contradiction.

\noindent 3. 
Fix $\beta>0$ and take an arbitrary subsequence $v_{\omega_{j_\ell}}$ which converges to some
$v$ in $H^{-\frac12-\beta}$. 
By Lemma~\ref{l:cvg} and~\eqref{e:bv-Comega} we have $v\in I^{\frac14+}(\partial\Omega,N^*_+\Sigma^-_\lambda\sqcup N^*_-\Sigma^+_\lambda)$.
By Lemma~\ref{l:C-lambda-strong-limit} and~\eqref{e:bv-Comega}
we have $\mathcal C_{\lambda+i0}v=(R_{\lambda+i0}f)|_{\partial\Omega}$. Finally, passing to the limit in~\eqref{e:bv-U} using Lemma~\ref{l:loc}
we have
$\supp R_{\lambda+i0}(f-\mathcal Iv)\subset\overline\Omega$.
Thus $v$ satisfies~\eqref{e:limitator-characterization}.
By Lemma~\ref{l:unique-2} there is at most one distribution which satisfies~\eqref{e:limitator-characterization}.
This implies that all the limits of convergent subsequences of $v_{\omega_j}$ in $H^{-\frac12-\beta}$ have to be the same.

On the other hand by~\eqref{eq:boundv} and compactness of the embedding
$ H^{ -\frac12 - \beta' } \hookrightarrow H^{-\frac12 - \beta } $ when $0<\beta'<\beta$
we see that the sequence $v_{\omega_j}$ is precompact in $H^{-\frac12-\beta}$.
Together with uniqueness of limit of subsequences this implies the convergence statement~\eqref{e:limitator}.

We finally show that $v_{\lambda+i0}\in I^{\frac 14}(\partial\Omega, N^*_+ \Sigma_\lambda^-  \sqcup
N^*_- \Sigma_\lambda^+ )$. From~\eqref{e:limitator-characterization} we get similarly to~\eqref{e:Cle-2}
$$
v_{\lambda+i0}=B^+_{\lambda+i0}b_\lambda^* v_{\lambda+i0}+B^-_{\lambda+i0}b_\lambda^{-*}v_{\lambda+i0}+g_{\lambda+i0}\quad\text{where}\quad
g_{\lambda+i0}\in C^\infty(\mathbb S^1;T^*\mathbb S^1).
$$
It remains to apply Lemma~\ref{p:020}.
\end{proof}
We now upgrade Proposition~\ref{l:boundv} to a convergence statement for all the derivatives~$\partial_\omega^k v_\omega$.
Here $v_{\omega}\in C^\infty(\partial\Omega;T^*\partial\Omega)$ is holomorphic in~$\omega\in \mathcal J+i(0,\infty)$:
indeed, $u_{\omega}\in C^\infty(\overline\Omega)$ is holomorphic by the Remark following Lemma~\ref{l:elliptic-solved}
and the operator $\mathcal N_\omega$ defined in~\eqref{e:N-omega-def} is holomorphic as well.

As in the proof of Proposition~\ref{l:boundv} we will use the spaces $I^s(\partial\Omega,N^*_+\Sigma^-_\lambda\sqcup N^*_-\Sigma^+_\lambda)$
which depend on $\lambda=\Re\omega$. We recall from~\S\ref{s:C-conormal} the family of diffeomorphisms
$$
\Theta_\lambda:\mathbb S^1\to \partial\Omega,\quad
\Theta_\lambda(\widetilde\Sigma^\pm)=\Sigma^\pm_\lambda,\quad
\lambda\in\mathcal J,
$$
with the pullback operator $\Theta_\lambda^*$ mapping $I^s(\partial\Omega,N^*_+\Sigma^-_\lambda\sqcup N^*_-\Sigma^+_\lambda)$
to the $\lambda$-independent space $I^s(\mathbb S^1,N^*_+\widetilde\Sigma^-\sqcup N^*_-\widetilde\Sigma^+)$. Denote
$$
\widetilde v_\omega:=\Theta_\lambda^* v_\omega\ \in\ C^\infty(\mathbb S^1;T^*\mathbb S^1),\quad
\omega\in \mathcal J+i(0,\infty),\quad
\lambda=\Re\omega.
$$
If $v_{\lambda+i0}$ is defined in~\eqref{e:limitator-characterization}, then we also put
$$
\widetilde v_{\lambda+i0}:=\Theta_\lambda^* v_{\lambda+i0}\ \in\ I^{\frac14+}(\mathbb S^1,N^*_+\widetilde\Sigma^-\sqcup N^*_-\widetilde\Sigma^+),\quad
\lambda\in\mathcal J.
$$ 
Writing $\omega=\lambda+i\varepsilon$, denote by $\partial_\lambda^\ell\widetilde v_\omega$ the $\ell$-th derivative
of $\widetilde v_\omega$ in $\lambda$ with $\varepsilon$ fixed. (Note that unlike $v_\omega$, the function $\widetilde v_\omega$
is not holomorphic in~$\omega$.)

We are now ready to give the main technical result of this section.
The proof is similar to that of Proposition~\ref{l:boundv} (which already contains the key ideas), using additionally Lemma~\ref{l:C-conormal}
which establishes regularity in~$\lambda$ of the operators $\mathcal C_\omega$ conjugated by $\Theta_\lambda$.
\begin{prop}
  \label{p:boundvder}
We have
\begin{equation}
  \label{e:boundvder-1}
\widetilde v_{\lambda+i0}\in C^\infty(\mathcal J;I^{\frac14+}(\mathbb S^1,N^*_+\widetilde\Sigma^-\sqcup N^*_-\widetilde\Sigma^+))
\end{equation}
where the topology on $I^{\frac14+}(\mathbb S^1,N^*_+\widetilde\Sigma^-\sqcup N^*_-\widetilde\Sigma^+)$ is
defined using the seminorms~\eqref{e:I-s-pm-seminorms}.
Moreover, for each $\lambda\in\mathcal J$ and $\ell$ we have as $\varepsilon\to 0+$
\begin{equation}
  \label{e:boundvder-2}
\partial^\ell_\lambda \widetilde v_{\lambda+i\varepsilon}\to \partial^\ell_\lambda \widetilde v_{\lambda+i0}\quad\text{in}\quad
I^{\frac14+}(\mathbb S^1,N^*_+\widetilde\Sigma^-\sqcup N^*_-\widetilde\Sigma^+),
\end{equation}
with convergence locally uniform in~$\lambda$.
\end{prop}
\Remarks 1. From~\eqref{e:boundvder-1} we get a regularity statement for $v_{\lambda+i0}=\Theta_{\lambda}^{-*}\widetilde v_{\lambda+i0}$:
$$
v_{\lambda+i0}\in C^\ell(\mathcal J;H^{-\frac12-\ell-}(\partial\Omega;T^*\partial\Omega))\quad\text{for all }\ell.
$$
Here the loss of~$\ell$ derivatives comes from differentiating $\Theta_{\lambda}^{-*}$ in~$\lambda$.

\noindent 2. The property~\eqref{e:boundvder-1} can be reformulated as follows: the distribution
$(\lambda,x)\mapsto v_{\lambda+i0}(x)$
lies in $I^{0+}(\mathcal J\times \partial\Omega,N^*_+\Sigma^-_{\mathcal J}\sqcup N^*_-\Sigma^+_{\mathcal J})$
where $\Sigma^\pm_{\mathcal J}:=\{(\lambda,x)\mid \lambda\in\mathcal J,\ x\in \Sigma^\pm_\lambda\}$
and we orient the conormal bundles $N^*\Sigma^\pm_{\mathcal J}$ using the positive orientation
on~$\partial\Omega$.
\begin{proof}
1. We start with a few identities on $\widetilde v_{\omega}$, $\omega\in \mathcal J+i(0,\infty)$.
Let $\widetilde {\mathcal C}_\omega=\Theta_\lambda^*\mathcal C_\omega\Theta_\lambda^{-*}$, $\lambda=\Re\omega$, be the conjugated
restricted single layer potential defined in~\eqref{e:tilde-C-omega}. Applying $\Theta_\lambda^*$ to~\eqref{e:C-lambda-eqrev} we get
\begin{equation}
  \label{e:bvderi-1}
\widetilde{\mathcal C}_\omega\widetilde v_\omega=\widetilde G_\omega\quad\text{where}\quad
\widetilde G_\omega:=\Theta_\lambda^*\big((R_\omega f)|_{\partial\Omega}\big).
\end{equation}
From~\eqref{e:bv-U} we have
\begin{equation}
  \label{e:bvderi-2}
\indic_\Omega u_\omega=R_\omega(f-\mathcal I \Theta_\lambda^{-*}\widetilde v_\omega).
\end{equation}
Differentiating these identities $\ell$ times in $\lambda=\Re\omega$, we get
\begin{align}
  \label{e:bvderi-3}
\widetilde{\mathcal C}_\omega \partial_\lambda^\ell \widetilde v_\omega&=
\partial_\lambda^\ell \widetilde G_\omega-\sum_{r=0}^{\ell-1}\binom{\ell}{r}(\partial_\lambda^{\ell-r}\widetilde{\mathcal C}_\omega)
(\partial_\lambda^r \widetilde v_\omega),\\
  \label{e:bvderi-4}
R_\omega \mathcal I\Theta_\lambda^{-*}\partial_\lambda^\ell\widetilde v_\omega&=
\partial_\omega^\ell R_\omega f-\indic_\Omega \partial_\omega^\ell u_\omega
-\sum_{r=0}^{\ell-1}\binom{\ell}{r}\big(\partial_\lambda^{\ell-r}(R_\omega\mathcal I\Theta_\lambda^{-*})\big)
(\partial_\lambda^r\widetilde v_\omega).
\end{align}

\noindent 2. Take an arbitrary sequence $\omega_j=\lambda_j+i\varepsilon_j\to\lambda\in\mathcal J$, $\Im\omega_j>0$. We show that for each $\ell\in\mathbb N_0$
\begin{equation}
  \label{e:bvder-1}
\partial^\ell_\lambda \widetilde v_{\omega_j}\quad\text{is bounded uniformly in~$j$}\quad\text{in }
I^{\frac14+}(\mathbb S^1,N^*_+\widetilde\Sigma^-\sqcup N^*_-\widetilde\Sigma^+).
\end{equation}
We use induction on~$\ell$, showing~\eqref{e:bvder-1} under the assumption
\begin{equation}
  \label{e:bvder-1.5}
\partial^r_\lambda \widetilde v_{\omega_j}\quad\text{is bounded uniformly in~$j$}\quad
\text{in }I^{\frac14+}(\mathbb S^1,N^*_+\widetilde\Sigma^-\sqcup N^*_-\widetilde\Sigma^+)\quad\text{for all }
r<\ell.
\end{equation}
Fix arbitrary $\beta>0$; the main task will be to show that
\begin{equation}
  \label{e:bvder-2}
\|\partial^\ell_\lambda \widetilde v_{\omega_j}\|_{H^{-\frac12-\beta}}\quad\text{is bounded in }j.
\end{equation}
We argue by contradiction: if~\eqref{e:bvder-2} does not hold then we can pass to a subsequence
to make $\|\partial^\ell_\lambda \widetilde v_{\omega_j}\|_{H^{-\frac12-\beta}}\to \infty$.
Define
\begin{equation}
  \label{e:bvder-3}
\widetilde v_j:=\partial^\ell_\lambda \widetilde v_{\omega_j}/\|\partial^\ell_\lambda \widetilde v_{\omega_j}\|_{H^{-\frac12-\beta}},\quad
\|\widetilde v_j\|_{H^{-\frac12-\beta}}=1.
\end{equation}
Since $H^{-\frac 12-\beta}$ embeds compactly into $H^{-N}$, where we fix $N>\frac 12+\beta$,
we may pass to a subsequence to get
\begin{equation}
  \label{e:bvder-4}
\widetilde v_j\to  \widetilde v_0\quad\text{in}\ H^{-N}\quad\text{for some}\quad
\widetilde v_0\in H^{-N}(\mathbb S^1;T^*\mathbb S^1).
\end{equation}
We now analyse the right-hand side of~\eqref{e:bvderi-3} for $\omega=\omega_j$. Since
$R_{\omega}f=E_{\omega}*f$, $\partial_\omega^r E_{\omega_j}\to \partial_\lambda^r E_{\lambda+i0}$ in $\mathcal D'(\mathbb R^2)$
by Lemma~\ref{l:loc}, and $f\in \CIc(\Omega)$ is independent of $j$, we see that
\begin{equation}
  \label{e:bvder-5}
\partial^\ell_\lambda\widetilde G_{\omega_j}\quad\text{is bounded uniformly in~$j$}\quad\text{in}\ C^\infty(\mathbb S^1).
\end{equation}
By Lemma~\ref{l:C-conormal} and~\eqref{e:bvder-1.5} we next have for all $r<\ell$
\begin{equation}
  \label{e:bvder-5.5}
(\partial_\lambda^{\ell-r}\widetilde{\mathcal C}_{\omega_j})
(\partial_\lambda^r\widetilde v_{\omega_j})\quad\text{is bounded uniformly in~$j$}\quad\text{in }I^{-\frac34+}(\mathbb S^1,N^*_+\widetilde\Sigma^-\sqcup N^*_-\widetilde\Sigma^+).
\end{equation}
Dividing~\eqref{e:bvderi-3} by $\|\partial^\ell_\lambda \widetilde v_{\omega_j}\|_{H^{-\frac12-\beta}}$, we then get
\begin{equation}
  \label{e:bvder-6}
\widetilde{\mathcal C}_{\omega_j}\widetilde v_j\to 0\quad\text{in}\quad I^{-\frac34+}(\mathbb S^1,N^*_+\widetilde\Sigma^-\sqcup N^*_-\widetilde\Sigma^+).
\end{equation}
We now apply Lemma~\ref{l:cvg} to
$$
v_j:=\Theta_{\lambda_j}^{-*}\widetilde v_j,\quad
v_0:=\Theta_{\lambda}^{-*}\widetilde v_0,\quad
\mathcal C_{\omega_j}v_j=\Theta_{\lambda_j}^{-*}\widetilde{\mathcal C}_{\omega_j}\widetilde v_j
$$
and get
\begin{equation}
  \label{e:bvder-7}
v_j\to v_0\quad\text{in}\quad H^{-\frac 12-\beta},\quad
v_0\in I^{\frac 14+}(\partial\Omega,N^*_+\Sigma^-_\lambda\sqcup N^*_-\Sigma^+_\lambda).
\end{equation}
By Lemma~\ref{l:C-lambda-strong-limit} and~\eqref{e:bvder-7} we have $\mathcal C_{\omega_j}v_j\to \mathcal C_{\lambda+i0}v_0$ in $\mathcal D'(\partial\Omega)$,
thus by~\eqref{e:bvder-6}
\begin{equation}
  \label{e:bvder-8}
\mathcal C_{\lambda+i0}v_0=0.
\end{equation}
We now obtain a support condition on $\mathcal R_{\lambda+i0}\mathcal Iv_0$ by analyzing the right-hand side of~\eqref{e:bvderi-4}
for $\omega=\omega_j$. Similarly to the proof of~\eqref{e:bvder-5} we have
$$
\partial^\ell_\omega R_{\omega_j}f\quad\text{is bounded uniformly in~$j$}\quad\text{in}\quad C^\infty(\mathbb R^2).
$$
By a similar argument using additionally~\eqref{e:bvder-1.5} we get for all $r<\ell$
$$
\big(\partial_\lambda^{\ell-r}(R_{\omega_j}\mathcal I\Theta_{\lambda_j}^{-*})\big)(\partial_\lambda^r \widetilde v_{\omega_j})\quad\text{is bounded
uniformly in~$j$}\quad\text{in }\mathcal D'(\mathbb R^2),
$$
where we denote $\partial_\lambda^k(R_{\omega_j}\mathcal I\Theta_{\lambda_j}^{-*}):=\partial_\lambda^k(R_{\omega}\mathcal I\Theta_{\lambda}^{-*})|_{\omega=\omega_j}$.

By Lemma~\ref{l:loc} and~\eqref{e:bvder-7} we get
$$
R_{\omega_j}\mathcal I\Theta^{-*}_{\lambda_j}\widetilde v_j\to R_{\lambda+i0}\mathcal I v_0\quad\text{in}\quad
\mathcal D'(\mathbb R^2).
$$
Now, dividing~\eqref{e:bvderi-4}
by $\|\partial^\ell_\lambda \widetilde v_{\omega_j}\|_{H^{-\frac12-\beta}}$ 
and using that $\supp(\indic_\Omega \partial^\ell_\omega u_{\omega_j})\subset\overline\Omega$ for all~$j$
we obtain
\begin{equation}
  \label{e:bvder-9}
\supp(R_{\lambda+i0}\mathcal Iv_0)\subset\overline\Omega.
\end{equation}
Applying Lemma~\ref{l:unique-2} and using~\eqref{e:bvder-7}--\eqref{e:bvder-9} we now see that $v_0=0$.
This gives a contradiction with~\eqref{e:bvder-7}, since $\|v_j\|_{H^{-\frac12-\beta}}$ is bounded away from~0 by~\eqref{e:bvder-3}.
This finishes the proof of~\eqref{e:bvder-2}.

The bound~\eqref{e:bvder-2} implies the stronger boundedness statement~\eqref{e:bvder-1}. Indeed, the proof of
Lemma~\ref{l:cvg} (more precisely, \eqref{e:bouncy-castle-1} and the paragraph following it) shows
that any seminorm of $\Theta_{\lambda_j}^{-*}\partial_\lambda^\ell \widetilde v_{\omega_j}$ in $I^{\frac 14+}(\partial\Omega,N^*_+\Sigma^-_{\lambda_j}\sqcup N^*_-\Sigma^+_{\lambda_j})$ (see~\eqref{e:I-s-pm-seminorms}) is bounded in terms
of $\|\partial_\lambda^\ell \widetilde v_{\omega_j}\|_{H^{-\frac12-\beta}}$ (for any choice of $\beta$) and
of some $I^{-\frac 34+}(\partial\Omega,N^*_+\Sigma^-_{\lambda_j}\sqcup N^*_-\Sigma^+_{\lambda_j})$-seminorm
of $\Theta_{\lambda_j}^{-*}\widetilde{\mathcal C}_{\omega_j}\partial_\lambda^\ell \widetilde v_{\omega_j}$.
The former is bounded in~$j$ by~\eqref{e:bvder-2} and the latter is bounded in~$j$
by~\eqref{e:bvderi-3}, \eqref{e:bvder-5}, and~\eqref{e:bvder-5.5}.

\noindent 3. From~\eqref{e:bvder-1} we see that (as before, using the seminorms~\eqref{e:I-s-pm-seminorms}),
the family of distributions $\widetilde v_{\lambda+i\varepsilon}$
is bounded uniformly in $\varepsilon\in (0,1]$ in the space $C^\infty(\mathcal J;I^{\frac 14+}(\mathbb S^1,N^*_+\widetilde\Sigma^-\sqcup N^*_-\widetilde\Sigma^+))$. 
By the Arzel\`a--Ascoli Theorem~\cite[Theorem~47.1]{Munkres-Topology} and since any sequence which is bounded
in $I^{\frac 14+}(\mathbb S^1,N^*_+\widetilde\Sigma^-\sqcup N^*_-\widetilde\Sigma^+)$ is also precompact
in this space (following from~\eqref{e:I-s-pm-seminorms} and the compactness of embedding $H^s\subset H^t$ for $s>t$) it
follows that $\widetilde v_{\lambda+i\varepsilon}$ is also precompact in the space
$C^\infty(\mathcal J;I^{\frac 14+}(\mathbb S^1,N^*_+\widetilde\Sigma^-\sqcup N^*_-\widetilde\Sigma^+))$.
Moreover, $\widetilde v_{\lambda+i\varepsilon}\to\widetilde v_{\lambda+i0}$
in the space $C^0(\mathcal J;H^{-\frac12-}(\mathbb S^1;T^*\mathbb S^1))$ by Proposition~\ref{l:boundv}.
Together these two statements imply that as $\varepsilon\to 0+$
$$
\widetilde v_{\lambda+i\varepsilon}\to \widetilde v_{\lambda+i0}\quad\text{in}\quad
C^\infty(\mathcal J;I^{\frac 14+}(\mathbb S^1,N^*_+\widetilde\Sigma^-\sqcup N^*_-\widetilde\Sigma^+)),
$$
giving~\eqref{e:boundvder-1} and~\eqref{e:boundvder-2}.
\end{proof}

\subsection{Proof of Theorem \ref{t:2}}
\label{s:prooft2}

Fix $f \in\CIc(\Omega)$, let $\omega=\lambda+i\varepsilon$ where $\lambda\in\mathcal J$ and $0<\varepsilon\ll 1$.
Without loss of generality we assume that $f$ is real-valued. It suffices to show existence
of the limit of $(P-(\lambda+i\varepsilon)^2)^{-1}f$,
since $(P-(\lambda-i\varepsilon)^2)^{-1}f$ is given by its complex conjugate.

Let $u_\omega\in C^\infty(\overline\Omega)$ be
the solution to the boundary-value problem~\eqref{eq:ellip}. Recalling~\eqref{eq:defP2} we see
that
$$
(P-\omega^2)^{-1}f=\Delta u_\omega\ \in\ C^\infty(\overline\Omega).
$$
Next, by~\eqref{e:u-lambda-form} we have
\begin{equation}
  \label{e:prooftor-1}
u_\omega=(R_\omega f)|_\Omega-S_\omega v_\omega
\end{equation}
where the `Neumann data' $v_\omega:=\mathcal N_\omega u_\omega\in C^\infty(\partial\Omega;T^*\partial\Omega)$ is defined using~\eqref{e:N-omega-def}.

By Proposition~\ref{l:boundv} we have
$$
v_{\lambda+i\varepsilon}\to v_{\lambda+i0}\quad\text{in }H^{-\frac12-}(\partial\Omega;T^*\partial\Omega)\quad\text{as }\varepsilon\to 0+,
$$
with convergence locally uniform in $\lambda\in\mathcal J$.
Using Lemma~\ref{l:loc} and recalling that $R_\omega f=E_\omega*f$ and $S_\omega v_\omega=(R_\omega \mathcal I v_\omega)|_{\Omega}$, we pass to the limit in~\eqref{e:prooftor-1} to get
$$
u_{\lambda+i\varepsilon}\to u_{\lambda+i0}:=(R_{\lambda+i0}f)|_{\Omega}-S_{\lambda+i0}v_{\lambda+i0}\quad\text{in }
\mathcal D'(\Omega)\quad\text{as }\varepsilon\to 0+,
$$
with convergence again locally uniform in $\lambda\in\mathcal J$.
This gives the convergence statement~\eqref{eq:Pla} with 
$$
(P-\lambda^2-i0)^{-1}f=\Delta u_{\lambda+i0}.
$$
Next, since $R_{\lambda+i0}f\in C^\infty(\mathbb R^2)$ and $v_{\lambda+i0} \in I^{\frac14  } ( \partial\Omega, N^*_+ \Sigma_\lambda^- \sqcup
N^*_- \Sigma_\lambda^+ )$, we apply the mapping property~\eqref{eq:con-new} to get
$u_{\lambda+i0}\in I^{-1}(\overline\Omega,\Lambda^-(\lambda))$
which implies
$$
(P-\lambda^2-i0)^{-1}f\in I^{1}(\overline\Omega,\Lambda^-(\lambda)).
$$
Since $ C^\infty_{\rm{c}} ( \Omega ) $ is dense in 
$ H^{-1} ( \Omega ) $ (see for instance \cite[Lemma 5]{RalstonR}), it is then standard 
(see for instance \cite[Proposition 4.1]{SiSchr}) that the spectrum of $ P $ in $\mathcal  J^2 $ is
purely absolutely continuous.
 \qed

\section{Large time asymptotic behaviour}
\label{s:asy}

We will now adapt the analysis of \cite[\S\S 5,6]{DZ-FLOP} and use \eqref{eq:sol} to describe 
asymptotic behaviour of solutions to~\eqref{eq:PDE}, giving the proof of Theorem~\ref{t:1}.
Assume that $\lambda\in (0,1)$ satisfies the Morse--Smale conditions of Definition~\ref{d:2}
and fix an open interval $\mathcal J\subset (0,1)$ containing~$\lambda$ such that
each $\omega\in \mathcal J$ satisfies the Morse--Smale conditions as well
(this is possible by Lemma~\ref{l:perturb-MS}).
We emphasize that in this section, in contrast with~\S\S\ref{s:blp}--\ref{s:liap},
we denote by $\lambda$ the fixed real frequency featured in the forcing term in~\eqref{eq:PDE}
and by $\omega$ an arbitrary real number (often lying in~$\mathcal J$).

\subsection{Reduction to the resolvent}
\label{s:asy-reduce}

Fix $ f \in C^\infty_{\rm{c}} ( \Omega ;\mathbb R ) $ and let $u$ be the solution to~\eqref{eq:PDE}.
We first split~\eqref{eq:sol} into two parts. Fix a cutoff function
\begin{equation}
  \label{e:varphi-def}
\varphi\in \CIc(\mathcal J;[0,1]),\quad
\varphi=1\quad\text{on}\quad [\lambda-\delta,\lambda+\delta]\quad\text{for some }\delta>0.
\end{equation}
By~\eqref{eq:sol} we can write
\begin{equation}
\label{eq:DeltaOmega}       u(t) = \Delta^{-1}_\Omega \Re\big(e^{i\lambda t}(w_1(t)+r_1(t))\big), 
\end{equation}
where, with $\mathbf W_{t,\lambda}$ defined in~\eqref{eq:sol},
\begin{equation}
\label{eq:b1} 
\begin{aligned}
w_1 (  t ) =
\varphi(\sqrt P)\mathbf W_{t,\lambda}(P)f,\quad
r_1 (  t ) = 
(I-\varphi(\sqrt P))\mathbf W_{t,\lambda}(P)f.
\end{aligned}
\end{equation}
The contribution of $ r_1 $ to $ u $ is bounded in $H^1(\Omega)$ uniformly as $t\to \infty$ as follows from
\begin{lemm}
\label{l:b1}
We have
\begin{equation}
\label{eq:lb1}
\big\| \Re \big(e^{i\lambda t}r_1(t)\big) \big\|_{ H^{-1}(\Omega) } \leq {2\over\lambda\delta}\|f\|_{H^{-1}(\Omega)} \quad\text{for all}\quad t\geq 0.
\end{equation}
\end{lemm}
\begin{proof}
We calculate $\Re(e^{i\lambda t}r_1(t)) = \mathbf R_{t,\lambda}(P)f$ where
\[ 
\mathbf R_{t,\lambda}(z)=
\Re\big(e^{i\lambda t}\mathbf W_{t,\lambda}(z)(1-\varphi(\sqrt z))\big)=
{( \cos (t \sqrt z)  - \cos (t \lambda )) ( 1 - \varphi ( \sqrt {z } ) )\over  \lambda^2-z } .
\]
Since $\varphi=1$ near $\lambda$, we see
that $\sup_{[0,1]}|\mathbf R_{t,\lambda}|\leq 2/(\lambda\delta)$.
Now~\eqref{eq:lb1} follows from the functional calculus for the self-adjoint operator $ P $ on $H^{-1}(\Omega)$. 
\end{proof} 
Define for $\omega\in\mathcal J$ the limits in~$\mathcal D'(\Omega)$
(which exist by Theorem~\ref{t:2}, see~\S\ref{s:prooft2})
\begin{equation}
  \label{e:u-pm-def}
u^\pm(\omega):=\Delta_\Omega^{-1}(P-\omega^2\pm i0)^{-1}f.
\end{equation}
Here $u^+(\omega)$ is the complex
conjugate of~$u^-(\omega)$ since $f$ is real valued.

By Stone's formula for the operator~$P$ (see for instance \cite[Theorem B.10]{DZ-Book})
and a change of variables in the spectral parameter we have
\begin{equation}
\label{eq:w2w1}
\begin{aligned}
\Delta_\Omega^{-1} w_1(t) &=
{1\over \pi i}\int_{\mathbb R} \varphi(\omega)\mathbf W_{t,\lambda}(\omega^2)(u^-(\omega)-u^+(\omega))\omega\,d\omega \\
&=
{1\over 2\pi}\int_0^t\int_{\mathbb R}\varphi(\omega)(e^{i(\omega-\lambda)s}-e^{-i(\omega+\lambda)s})(u^+(\omega)-u^-(\omega))\,d\omega ds.
\end{aligned}
\end{equation}

\subsection{Global geometry}
\label{s:gge}

The proof of Theorem~\ref{t:2} in~\S \ref{s:prooft2} shows that $ u^\pm (\omega )$ 
are smooth families of cononormal distributions associated to $ \omega $-dependent 
lines in $ \mathbb R^2 $, more precisely
\begin{equation}
  \label{e:u-pm-conormal}
u^\pm(\omega)\in I^{-1}(\overline\Omega,\Lambda^\pm(\omega))
\end{equation}
where $\Lambda^\pm(\omega)$ are defined in~\eqref{eq:defLa}.
To understand the behaviour of $ \Delta_\Omega^{-1} w_1(t) $ as $ t \to \infty $ we present an explicit version of~\eqref{e:u-pm-conormal},
relying on Proposition~\ref{p:boundvder}. The most confusing thing here are the signs defined 
in \eqref{e:n-pm-gamma-def}. Figures \ref{f:Lamp}, \ref{f:interlace}, and~\ref{f:ggeo} can be used for guidance
here. 

\begin{figure}
\qquad\includegraphics{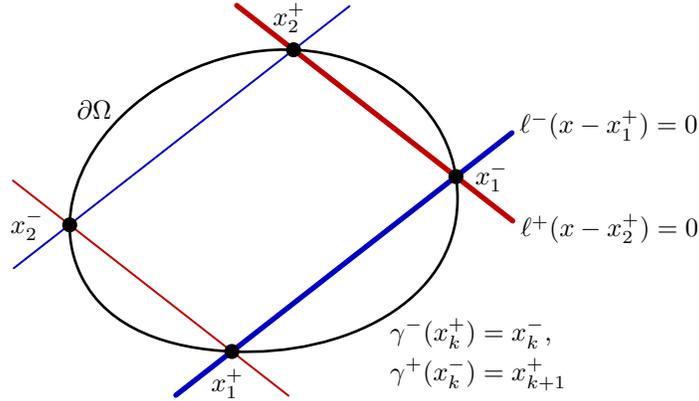}
\caption{An illustration of~\eqref{eq:Sigmapm} with $ \Sigma^\pm_\lambda = \{ x_1^\pm(\lambda), x_2^\pm(\lambda) \} $.}
\label{f:interlace}
\end{figure}

Let $\omega\in \mathcal J$ and $\Sigma^\pm_\omega\subset\partial\Omega$ be the attractive/repulsive sets of the chess billiard $b(\bullet,\omega)$
defined in~\eqref{e:Sigma-pm-def}. Recall that $b=\gamma^+\circ\gamma^-$ and  the involutions $\gamma^\pm(\bullet,\omega)$ map $\Sigma^+_\omega$ to~$\Sigma^-_\omega$.
Let $n$ be the minimal period of the periodic points of~$b$. To simplify notation,
we assume that each of the sets $\Sigma^\pm_\omega$ consists of exactly~$n$ points, that is, it
is a single periodic orbit of~$b$ (as opposed to a union of several periodic orbits),
but the analysis works in the same way in the general case. We write
(with the cyclic convention that $  x_{n+1}^\pm(\omega) = x_1^\pm ( \omega) $, $ x_0^\pm ( \omega ) = x_n^\pm ( \omega ) $)
\begin{equation}
\label{eq:Sigmapm}   \Sigma^\pm_ \omega  = \{ x_k^\pm ( \omega ) \}_{k=1}^n , \ \ \ 
\gamma^- ( x_k^+  ) = x_k^-  , \ \ \gamma^+ 
( x_k^- ) = x_{k + 1}^+, 
\end{equation}
and $  b^{\pm 1} ( x_k^\pm ( \omega ) , \omega ) = x_{k+1}^\pm$. By Lemma~\ref{l:perturb-MS},
we can make $x_k^\pm(\omega)$ depend smoothly on~$\omega\in\mathcal J$.

In the notation of \eqref{e:n-pm-gamma-def} and~\eqref{eq:defLa}, 
\begin{equation}
\label{eq:defLap}  \begin{gathered}  \Lambda^- ( \omega ) =
\bigsqcup_{ k = 1}^n N^*_+ \Gamma_\omega^- ( x_k^{-} ( \omega ) ) \sqcup
\bigsqcup_{ k = 1}^n N^*_- \Gamma_\omega^+ ( x_k^{+}  ( \omega ) ) , \\
N^*_+ \Gamma_\omega^- ( x^-_k ( \omega )) = 
\{  ( x,  \tau d \ell^-_\omega)   : \ell^-_\omega ( x - x_k^- ( \omega ) ) =0  ,  \ \
 \nu^-_{k}  \tau > 0 \}  , \\
N^*_- \Gamma_\omega^+ ( x^+_{k} ( \omega )) = 
\{  ( x,  \tau d \ell^+_\omega  ) : \ell^+_\omega  ( x - x_k^+( \omega )   )=0,  \ \
 \nu_k^+  \tau < 0 \}  ,\\
\nu^\pm_k  :=  \nu^\pm ( x_k^\pm ( \omega ) , \omega ) :=\sgn \partial_\theta \ell_\omega^\pm  ( x_k^\pm ( \omega )  ) 
 . \end{gathered} 
\end{equation}
where $ \ell^\pm_\omega  ( x ) := \ell^\pm ( x, \omega ) $. 
We note that $ \nu_k^\pm  $ are  independent of $ \omega\in\mathcal J$.
To obtain 
$ \Lambda^+ ( \omega ) $ we switch the sign of $ \tau $  in \eqref{eq:defLap} --
see Figure~\ref{f:Lamp}.

\begin{figure}
\includegraphics{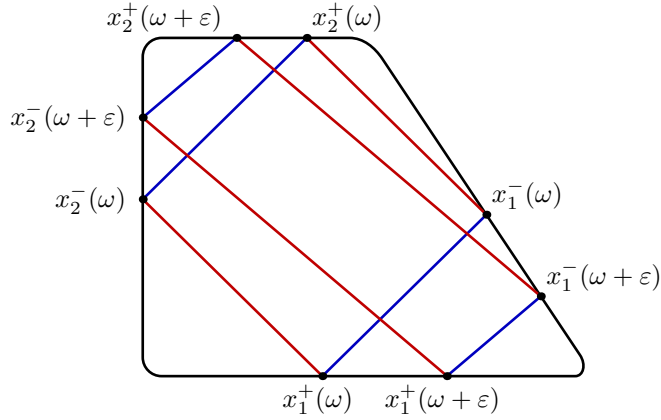}
\caption{An illustration of Lemma~\ref{l:ggeo}, showing
the periodic trajectory of the chess billiard for $\omega$
and for $\omega+\varepsilon$, $\varepsilon>0$. In this example
$\nu_1^+=\nu_1^-=1$, $\nu_2^+=\nu_2^-=-1$.
Lemma~\ref{l:ggeo} shows in which direction the blue and red segments move
as $\omega$ grows. For example,
the entire red segment $\{x\in\Omega\mid \ell^+_\omega(x-x^+_1(\omega))=0\}$,
which connects $x_1^+(\omega)$ to $x_2^-(\omega)$, lies
inside the half-plane $\{x\mid \ell^+_{\omega+\varepsilon}(x-x^+_1(\omega+\varepsilon))<0\}$,
which is consistent with~\eqref{eq:ggeo}.}
\label{f:ggeo}
\end{figure}

We need the following geometric result (see Figure~\ref{f:ggeo}):
\begin{lemm}
\label{l:ggeo}
With the notation above we have for all $\omega\in\mathcal J$
\begin{equation}
\label{eq:ggeo}
x \in \overline \Omega,  \ \ \ell^\pm_\omega ( x - x^\pm_k ( \omega ) ) = 0 \ 
\
\Longrightarrow \ \ \sgn \left[ \partial_\omega ( \ell_\omega^\pm ( x - x^\pm_k ( \omega ) ) )
\right] = \mp \nu_k^\pm
.
\end{equation}
\end{lemm}
\begin{proof}
1. We note that the definition~\eqref{e:gamma-pm-def}  of $ \gamma_\omega^\pm $  
and~\eqref{eq:Sigmapm} give
\begin{equation}
\label{eq:k2ell}  \ell^\pm_\omega ( x - x_k^\pm ( \omega ) ) =  \ell^\pm_\omega ( x - x_\ell^\mp  ( \omega )  ), \ \ \
\ell = \left\{ \begin{array}{ll} k-1, & + ,\\ \ k, & - .\end{array} \right.\end{equation}
We also note that \eqref{e:flip-sign} implies 
\begin{equation}
\label{eq:flips} 
\sgn \partial_\theta \ell^\pm_\omega ( x^\mp_\ell (\omega ) ) =:
\nu^\pm ( x_\ell^\mp ( \omega ) , \omega ) = - \nu^\pm ( x_k^\pm ( \omega ) , \omega ) = 
-\nu^\pm_k , 
\end{equation}
where $ \theta \mapsto \mathbf x ( \theta ) \in \partial \Omega $ is a positive parametrization of
$ \partial \Omega $ by $ \mathbb S^1_\theta $.
It is sufficient to establish \eqref{eq:ggeo} with $ x_k^\pm $ replaced by 
$ x_\ell^\mp $ where $ \ell $ is given in \eqref{eq:k2ell}:
\begin{equation}
\label{eq:gge}
x \in \overline \Omega,  \ \ \ell^\pm_\omega ( x - x^\mp_\ell ( \omega ) ) = 0 \ 
\
\Longrightarrow \ \ \sgn \left[ \partial_\omega ( \ell_\omega^\pm ( x - x^\mp_\ell ( \omega ) ) ) 
\right] = \mp \nu_k^\pm
.
\end{equation}

\noindent 2. 
Using \eqref{e:l-pm-differential} and the condition on $ x $ in \eqref{eq:gge} we see that
\begin{equation}
\label{eq:decomps} 
\partial_\omega \left[   \ell_\omega^\pm ( x - x_\ell^\mp ( \omega ) ) \right]  =
{\ell^\mp_\omega ( x - x_\ell^\mp ( \omega )  )\over  2 \omega ( 1 - \omega^2 )} - 
d_x \ell_\omega^\pm ( \partial_\omega x_\ell^\mp ( \omega ) ). \end{equation}
We start by considering the sign of the second term on the right hand side:
\begin{equation}
\label{eq:sgnsgn} 
-\sgn d_x \ell_\omega^\pm ( \partial_\omega x_\ell^\mp ( \omega ) ) = 
-\sgn  [ \partial_\theta \ell_\omega^\pm] ( x_\ell^\mp ( \omega ) ) \partial_\omega \theta ( 
x_\ell^\mp ( \omega ) ) =  \nu_k^\pm \sgn \partial_\omega [ \theta ( x_\ell^\mp ( \omega ) )] ,
\end{equation}
where we used \eqref{eq:flips}.

We now put $ f := \theta \circ b^n \circ \theta^{-1} $ with $ n $ the primitive period. Then \eqref{e:b-diff} and \eqref{e:Sigma-pm-def} give
\[
  f ( \theta ( x_\ell^\mp ( \omega ) ) , \omega ) =  \theta ( x_\ell^\mp  ( \omega ))  ,  \ \ 
  \partial_\omega  f ( x, \omega ) > 0 , \ \ 
\mp \big( 1 - [\partial_\theta f ] ( \theta (x^\mp_\ell ( \omega )),\omega) \big) > 0 .  \ \ 
\]
Differentiating the first equality in $ \omega $ gives
\[
\partial_\omega [ \theta ( x_\ell^\mp ( \omega ) )] = 
\partial_\omega  f (\theta(x) , \omega )|_{x= x_\ell^\mp ( \omega )} / \big( 1 - [\partial_\theta f ] ( \theta (x^\mp_\ell ( \omega )),\omega) \big) ,
\]
and hence $ \sgn \partial_\omega [ \theta ( x_\ell^\mp ( \omega ) )]  
= \mp 1 $. Returning to \eqref{eq:sgnsgn} we see that
\[
- \sgn d_x \ell_\omega^\pm ( \partial_\omega x_\ell^\mp ( \omega ) ) = \mp \nu_k^{\pm}.
\]

\noindent
3. We next claim that 
\begin{equation}
\label{eq:ggeo1}
x \in \overline \Omega,  \ \ \ell^\pm_\omega ( x - x^\mp_\ell ( \omega ) ) = 0 \ 
\
\Longrightarrow \ \  \sgn \ell_\omega^\mp ( x - x^\mp_\ell ( \omega ) ) \in\{ \mp \nu_k ^\mp ,0\}  . 
\end{equation}
Combined with \eqref{eq:decomps}  and the conclusion of Step 2, this will give \eqref{eq:gge}
and hence \eqref{eq:ggeo}.
Since the set on the left-hand side of \eqref{eq:ggeo1} is given by 
$ x = (1- t ) x_\ell^\mp ( \omega ) + t \gamma_\omega^\pm ( x_\ell^\mp ( \omega ) ) $, 
$ 0 \leq t \leq 1$, 
it suffices to establish the conclusion in \eqref{eq:ggeo1} for $ x = \gamma_\omega^\pm ( x_\ell^\mp ( \omega ) ) $.
For that we use \eqref{e:sign-identity} and \eqref{eq:flips} which give
\[
\sgn  \ell_\omega^\mp ( \gamma_\omega^\pm ( x_\ell^\mp  ( \omega ) ) - x_\ell^\mp  ( \omega ))  = \pm \nu^\pm ( x^\mp_\ell ( \omega ) , \omega ) = \mp \nu_k^\pm , 
\]
completing the proof.
\end{proof} 
In the notation of this section, Theorem \ref{t:2} is reformulated as follows.
Note that henceforth in this section, $\varepsilon$ denotes a sign (either $+$ or $-$)
in contrast with its use in the statement and proof of Theorem~\ref{t:2}.
\begin{lemm}
\label{l:conor}
In the notation of \eqref{e:u-pm-def}, \eqref{eq:defLap} and with 
$ \varepsilon \in\{ + , - \} $,
\begin{equation}
\label{eq:upm}
\begin{gathered}
u^\varepsilon (  x,\omega  ) = 
\sum_{ k = 1}^n \sum_{\pm} g^\varepsilon_{k, \pm} (x , \omega  )  + u^\varepsilon_{0} ( x, \omega )  , 
  \ \ \ u_0^\varepsilon \in C^\infty ( \overline{\Omega }\times {\mathcal J} ), \ \ \ 
g^\varepsilon_{k, \pm } \in \mathcal D' ( \mathbb R^2 ) , \\
g^\varepsilon_{k, \pm} ( x, \omega  ) = \frac{1}{2 \pi} \int_{\mathbb R} e^{   i \tau
  \ell^\pm_\omega (  x - x^\pm_k ( \omega ))} a^\varepsilon_{k, \pm} (  \tau,\omega 
) \,d\tau , \ \ \ 
( x, \omega)  \in \mathbb R^2 \times {\mathcal J}, 
  \end{gathered}
\end{equation}
where $ a^\varepsilon_{k, \pm } \in S^{-1+}( {\mathcal J}_\omega\times  \mathbb R_\tau ) $ is
supported in $\{\tau\colon\pm\varepsilon\nu_k^\pm\tau\geq 1\}$.
\end{lemm}
\begin{proof}
We consider the case of $\varepsilon=-$, with the case $\varepsilon=+$ following
since $u^+(\omega)=\overline{u^-(\omega)}$.
Recall from~\S\ref{s:prooft2} that
$$
u^-(\omega)=u_{\omega+i0}=(R_{\omega+i0}f)|_{\Omega}-S_{\omega+i0}v_{\omega+i0},
$$
where $R_{\omega+i0}f\in C^\infty(\mathbb R^2\times\mathcal J)$
by Lemma~\ref{l:loc} and
$$
v_{\omega+i0}\in C^\infty(\mathcal J;I^{\frac14+}(\partial\Omega,N^*_+\Sigma^-_\omega\sqcup N^*_-\Sigma^+_\omega)),
$$
with smoothness in $\omega$ understood in the sense of Proposition~\ref{p:boundvder}.
By the mapping property~\eqref{eq:con-new} we have
$$
S_{\omega+i0}v_{\omega+i0}\in C^\infty(\mathcal J;I^{-1+}(\overline\Omega,\Lambda^-(\omega))).
$$
Here smoothness in $\omega$ is obtained
by following the proof of Lemma~\ref{l:con}, which writes $S_{\omega+i0}v$ for $v\in C^\infty(\mathcal J;I^{\frac14+}(\partial\Omega,N^*_\pm\{x^\mp_\ell(\omega)\}))$ as a sum
of a function in $C^\infty(\overline\Omega\times\mathcal J)$ and the pullback by $\ell^\pm_\omega$ of a conormal distribution
to~$\ell^\pm_\omega(x_\ell^\mp(\omega))=\ell^\pm_\omega(x_k^\pm(\omega))$ with $k,\ell$ related by~\eqref{eq:k2ell}.
This gives the representation~\eqref{eq:upm}. Here
we can follow~\eqref{eq:defLa} and~\eqref{eq:defLap} to obtain an explicit parametrization of the conormal bundles
$ N_\mp^* \Gamma^\pm_\omega(x^\pm_k(\omega))=N_\pm^* \Gamma^\pm_\omega(x^\mp_\ell(\omega))$ and check that $a^\varepsilon_{k,\pm}$
can be written as a sum of a symbol supported in $\{\tau\colon\pm\varepsilon\nu_k^\pm\tau\geq 1\}$ and
a symbol which is rapidly decaying in $\tau$, with the contribution of the latter lying in $C^\infty(\overline\Omega\times\mathcal J)$.
\end{proof}
The next lemma disposes of the term $ u_0^\varepsilon $:
\begin{lemm}
\label{l:smou}
Suppose that $ u^\pm (x, \omega) \in C^\infty (  \overline \Omega \times {\mathcal J} ) $. If $ w_1 $ is
defined by \eqref{eq:w2w1} then for any $ k $, there exists $ C_k $ such that for all $t\geq 0$,
$ \| \Delta_\Omega^{-1}w_1 ( t) \|_{ C^k ( \overline \Omega ) }  \leq C_k $.
\end{lemm}
\begin{proof}
Recalling~\eqref{eq:w2w1}, we see that it suffices to prove that for any $u\in C^\infty(\overline\Omega\times\mathcal J)$
$$
\sup_{t\geq 0}\|w(t)\|_{C^k(\overline\Omega)}\leq C_k\quad\text{where}\quad
w(t):=\int_0^t\int_{\mathbb R}\varphi(\omega)(e^{i(\omega-\lambda)s}-e^{-i(\omega+\lambda)s})u(\omega)\,d\omega ds.
$$
Integrating by parts in~$\omega$, we get
$$
\begin{aligned}
w(x,t)&=\int_0^t\int_{\mathbb R } \varphi ( \omega ) u ( x,\omega) 
[ ( 1 + s^2 )^{-1} ( 1 + D_\omega^2 ) ] ( e^{ i ( \omega - \lambda ) s } - e^{ - i ( \omega +
\lambda ) s} )\, d \omega d s \\
&=\int_0^t \int_{\mathbb R } ( 1 + D_\omega^2  ) [ \varphi ( \omega ) u ( x,\omega) ]
( e^{ i ( \omega - \lambda ) s } - e^{ - i ( \omega +
\lambda ) s} ) ( 1 + s^2 )^{-1}\, d \omega d s
\end{aligned}
$$
which is bounded in $C^\infty(\overline\Omega)$ uniformly in $t\geq 0$.
\end{proof} 
Returning to \eqref{eq:w2w1} we see that we have to analyse the behaviour of 
\begin{equation}
\label{eq:defwep}
w_{k, \pm }^{\varepsilon,\varepsilon' }  (  x,t ) := 
\frac 1 { 2 \pi } \int_0^t \int_{\mathbb R} \varphi ( \omega ) 
g^\varepsilon_{k , \pm } ( x , \omega ) e^{ i s ( \varepsilon' \omega - \lambda ) } \,d \omega ds , \ \
\varepsilon, \varepsilon' \in \{ + , - \} ,
\end{equation}
as $ t \to  \infty $. More precisely, if the term $u_0^\varepsilon$ in the decomposition~\eqref{eq:upm}
were zero, then
\begin{equation}
  \label{e:w-1-con}
\Delta_\Omega^{-1}w_1(x,t)=\sum_{k=1}^n\sum_{\pm}\sum_{\varepsilon,\varepsilon'\in \{+,-\}}
\varepsilon\varepsilon'w_{k,\pm}^{\varepsilon,\varepsilon'}(x,t).
\end{equation}

\subsection{Asymptotic behaviour of \texorpdfstring{$ w_{k,\pm}^{\varepsilon, \varepsilon' } $}{w}}
\label{s:asw}

For $\tau\neq 0$, define
\[
   \begin{split}  A_{k, \pm }^{\varepsilon,\varepsilon' } ( x, t, \tau ) &  := 
\frac{1}{ 2\pi} \int_0^t \int_{\mathbb R }  e^{
i \tau  \ell_\omega^\pm ( x - x_k^\pm ( \omega ) ) + i ( \varepsilon' \omega - \lambda ) s}
\varphi ( \omega )a_{k, \pm }^{\varepsilon } ( \tau, \omega )\, d \omega d s \\
& = \frac{ \tau}{ 2\pi} \int_0^{t\over \tau}  \int_{\mathbb R }  e^{
i \tau  ( \ell_\omega^\pm ( x - x_k^\pm ( \omega ) ) +  ( \varepsilon' \omega - \lambda )r ) }
\varphi ( \omega )a_{k, \pm }^{\varepsilon } ( \tau,\omega )\, d \omega d r ,
\end{split}
\]
in the notation used for $ g_{k,\pm}^\varepsilon $ in Lemma \ref{l:conor},
where in the second line we made the change of variables $s=\tau r$. We then have
\begin{equation}
\label{eq:defwep1}
 w_{k, \pm }^{\varepsilon,\varepsilon' }  ( x,t ) =  \frac{1}{ 2\pi} \int_{\mathbb R } A_{k, \pm }^{\varepsilon,\varepsilon' } ( x,t, \tau ) \, d \tau ,
\end{equation}
in the sense of oscillatory integrals (since $ \partial_x  \ell_\omega^\pm ( x - x_k^\pm ( \omega ) ) = d \ell_\omega^\pm \neq 0 $ the phase is nondegenerate -- see
\cite[\S 7.8]{Hormander1}).
From the support condition in Lemma~\ref{l:conor} we get
\begin{equation}
  \label{e:A-supporter}
A_{k,\pm}^{\varepsilon,\varepsilon'}(x,t,\tau)\neq 0\quad\Longrightarrow\quad
\pm\varepsilon\nu_k^\pm\tau\geq 1.
\end{equation}
The lemma below shows that we only need to integrate over a compact interval in $ r $:
\begin{lemm}
\label{l:chis}
There exist $ \chi \in C^\infty_{\rm{c}} ( (0, \infty ) ) $ and $\varphi$ satisfying~\eqref{e:varphi-def}
such that for
\[
\widetilde A_{k, \pm}^{\varepsilon,\varepsilon' } ( x, t, \tau)  :=
 \frac{ \tau}{ 2\pi} \int_0^{t\over\tau} \int_{\mathbb R }e^{
i \tau ( \ell_\omega^\pm ( x - x_k^\pm ( \omega ) ) +  ( \varepsilon' \omega - \lambda ) r)} ( 1 - \chi ( \pm \varepsilon' \nu_k^\pm r ) ) \varphi ( \omega )a_{k, \pm }^{\varepsilon } ( \tau, \omega) 
 \,d \omega d r , \]
\[ \widetilde w_{k, \pm}^{\varepsilon,\varepsilon' } ( x,t ) := 
\frac{1}{ 2\pi} \int_{\mathbb R} \widetilde A_{k, \pm}^{\varepsilon,\varepsilon' } ( x,t,  \tau)\, d\tau ,
\]
we have  $\| \widetilde w_{k, \pm}^{\varepsilon,\varepsilon' } ( t) \|_{ C^k ( \overline \Omega )} 
\leq C_k $ for every $ k $ and uniformly as $ t \to \infty $.
\end{lemm}
\begin{proof}
1. Put 
$ F (x , \omega ) := \ell_\omega^\pm ( x - x_k^\pm ( \omega ) )  $. Lemma~\ref{l:ggeo} shows that
for all $\omega\in\mathcal J$
$$
x\in\overline\Omega,\
F(x,\omega)=0\quad\Longrightarrow\quad
\mp \nu_k^\pm\partial_\omega F(x,\omega)>0.
$$
Fix a cutoff function $\psi\in \CIc(\overline\Omega)$ such that
$\psi=1$ in a neighborhood of $\{x\in\overline\Omega\mid F(x,\lambda)=0\}$ and
$\mp \nu_k^\pm\partial_\omega F(x,\lambda)>0$ for all $x\in\supp\psi$.
Choosing $\varphi$ supported in a sufficiently small neighborhood of~$\lambda$, we 
see that there exists $\chi\in \CIc((0,\infty))$ such that for all $\omega\in\supp\varphi$
\begin{align}
  \label{e:chis-1}
x\in\overline\Omega\cap\supp(1-\psi)&\quad\Longrightarrow\quad
F(x,\omega)\neq 0;\\
  \label{e:chis-2}
x\in\overline\Omega\cap\supp\psi&\quad\Longrightarrow\quad
\mp \nu_k^\pm\partial_\omega F(x,\omega)\notin\supp (1-\chi).
\end{align}

\noindent 2.
Using the singular support property of conormal distributions (see \S \ref{s:con}) and~\eqref{e:chis-1}, we have 
\[
( 1 - \psi(x) )\varphi(\omega) g_{k,\pm}^{ \varepsilon } ( x, \omega ) \in C^\infty (\overline\Omega\times\mathbb R ) .
\]
The proof of Lemma \ref{l:smou} shows that $ \| ( 1 - \psi(x) ) \widetilde w_{k, \pm}^{\varepsilon,\varepsilon' } ( t ) \|_{C^k ( \overline \Omega )}\leq C_k $, uniformly as $t\to\infty $.

On the other hand, \eqref{e:chis-2} implies that for some constant $c>0$
\[
\begin{gathered}
\omega\in\supp\varphi,\quad 
x \in \overline\Omega\cap\supp \psi, \quad r \in \supp ( 1 - \chi ( \pm \varepsilon' \nu_k^\pm \bullet ) )
\\
\Longrightarrow\quad
| \partial_\omega F ( x, \omega )+ \varepsilon' r  | \geq c\langle r\rangle.
\end{gathered}
\]
Integration by parts in $ \omega $ shows that
\[
\partial_x^\alpha \left[ \psi ( x ) \widetilde A_{k, \pm}^{\varepsilon,\varepsilon' } (  x, t, \tau) \right] = \mathcal O ( \langle \tau \rangle^{-\infty } ) ,
\]
uniformly in $ t $. But that gives uniform smoothness of $ \psi(x) \widetilde w_{k,\pm}^{\varepsilon,\varepsilon' }(x,t) $,
finishing the proof.
\end{proof}
The lemma shows that in the study of \eqref{eq:defwep}
we can replace $ A $ in \eqref{eq:defwep1} by 
\[
\begin{gathered}
B_{k, \pm}^{\varepsilon,\varepsilon' } ( x,t, \tau)=
A_{k, \pm}^{\varepsilon,\varepsilon' } ( x,t, \tau)-
\widetilde A_{k, \pm}^{\varepsilon,\varepsilon' } ( x,t, \tau)
\\
=\frac{ \tau }{ 2\pi} \int_0^{t\over \tau}  \int_{\mathbb R }
e^{
i \tau  ( \ell_\omega^\pm ( x - x_k^\pm ( \omega ) ) +  ( \varepsilon' \omega - \lambda )r ) }
\chi ( \pm \varepsilon' \nu_k^\pm r )
\varphi ( \omega ) a_{k, \pm }^{\varepsilon } ( \tau, \omega ) 
\,d \omega d r.
\end{gathered}
\]
Define the limit $B_{k,\pm}^{\varepsilon,\varepsilon'}(x,\infty,\tau)$ by replacing the
integral $\int_0^{t/\tau}\,dr$ above by $\int_0^{(\sgn\tau)\infty}\,dr$,
which is well-defined thanks to the cutoff $\chi(\pm\varepsilon'\nu_k^\pm r)$,
where we recall that $ \chi \in \CIc ( ( 0, \infty ) ) $.
The next lemma describes the behaviour of this limit as $\tau\to\infty$:
\begin{lemm}
\label{l:Binf}
Denote $F(x,\omega):=\ell_\omega^\pm(x-x_k^\pm(\omega))$. Then 
$ e^{ -i \tau F(x,\lambda) } B_{k, \pm}^{\varepsilon,\varepsilon' } (  x, \infty,\tau)$ lies in the symbol class
$S^{-1+} ( \overline\Omega_x \times \mathbb R_\tau ) $ and 
\[  e^{ -i \tau F(x,\lambda)} B_{k,\pm}^{\varepsilon,\varepsilon'} (  x , \infty,\tau ) \in
\begin{cases}
\chi(\mp\nu_k^\pm\partial_\lambda F(x,\lambda)) a_{ k , \pm}^\varepsilon (  \tau, \lambda ) 
+ S^{-2+} (  \overline\Omega 
\times \mathbb R   )  , & \varepsilon =  \varepsilon' = +, \\
S^{-\infty} (  \overline\Omega  
\times \mathbb R)  , & \text{otherwise.}
\end{cases} 
\]
\end{lemm}
\begin{proof}
We first note that if  $\pm \varepsilon' \nu_k^\pm  \tau < 0 $ then $ B_{k,\pm}^{\varepsilon,\varepsilon'} ( x , \infty , \tau ) = 0 $. Hence we can assume that 
\begin{equation}
\label{eq:sgnt}
\sgn \tau = \pm \varepsilon' \nu_k^\pm.
\end{equation}
In that case we can replace limits of integration in $ r $ by $ (-\infty, \infty ) $, with $\tau$ replaced by $|\tau|$ in the prefactor
$\tau\over 2\pi $.
The method of stationary phase (see for instance \cite[Theorem~7.7.5]{Hormander1}) 
can be applied to the double integral $\int_{\mathbb R^2}\,d\omega dr$ and the critical point in given by 
$$
\omega = \varepsilon' \lambda,\quad r = - \varepsilon' \partial_\omega F ( x, \omega ).
$$
Since $ \omega = - \lambda $ lies outside of the support of $\varphi$,
if $\varepsilon'=-$ then (by the method of nonstationary phase) we have
$  B_{k,\pm}^{\varepsilon,\varepsilon'}(x,\infty,\tau)  \in S^{-\infty}(\overline\Omega\times\mathbb R)$. We thus assume that $\varepsilon'=+$,
which by~\eqref{eq:sgnt} gives $ \pm \nu_k^\pm \tau > 0 $.
If $\varepsilon=-$ then the support property of $a_{k,\pm}^\varepsilon$ in Lemma~\ref{l:conor}
shows that $B_{k,\pm}^{\varepsilon,\varepsilon'}(x,\infty,\tau)=0$. Thus we may assume that $\varepsilon=\varepsilon'=+$.
In the latter case the method of stationary phase gives the expansion for $B_{k,\pm}^{\varepsilon,\varepsilon'}(x,\infty,\tau)$.
\end{proof} 
We now analyse the remaining term given by 
\begin{equation}
\label{eq:defvk}    
v_{k, \pm}^{\varepsilon, \varepsilon'} ( x, t ) := 
\frac{1}{2 \pi } \int_{\mathbb R} C_{k,\pm}^{\varepsilon, \varepsilon'} ( x, t, \tau )\, d\tau , 
\end{equation}
where 
\[ 
 C_{k, \pm}^{\varepsilon,\varepsilon' } ( x, t, \tau)=\frac{ \tau  }{ 2\pi} \int_{t\over \tau}^{(\sgn\tau)\infty  }
\int_{\mathbb R }
e^{ i \tau  ( \ell_\omega^\pm ( x - x_k^\pm ( \omega ) ) +   ( \varepsilon' \omega - \lambda )r ) }
\chi_{k,\pm}^{\varepsilon'} ( r ) 
\varphi ( \omega )
a_{k, \pm }^{\varepsilon } ( \tau , \omega) 
\,d \omega d r ,
\] 
and $ \chi_{k,\pm}^{\varepsilon'} ( r )  := \chi ( \pm \varepsilon' \nu_k^\pm r ) \in 
C^\infty_{\rm{c}} ( \mathbb R \setminus \{ 0 \} ) $. 
The last lemma deals with this term:
\begin{lemm}
\label{l:last}
For $ v_{k,\pm}^{\varepsilon,\varepsilon'} $ given by \eqref{eq:defvk} we have, for every 
$ \beta > 0 $, 
\begin{equation}
\label{eq:lastl}
\|  v_{ k, \pm }^{\varepsilon,\varepsilon'}(t) \|_{ H^{\frac12- \beta } ( \Omega ) } \to 0 \quad\text{as}\ t \to  \infty. 
\end{equation}
\end{lemm}
\begin{proof}
1. We put $ \varepsilon' = + $ as the other case is similar and simpler.
To simplify notation we will often drop $ \varepsilon$ and $ k $.
Fix a cutoff function
\begin{equation}
  \label{e:endgame-psi}
\psi\in\CIc(\mathbb R),\quad
(\ell^\mp_\omega)^*\psi=1\quad\text{near }\overline\Omega\quad\text{for all}\quad
\omega\in\supp\varphi.
\end{equation}
For $x\in\mathbb R^2$, $t>0$, and $\omega\in\mathcal J$, define
\[ 
 U_\pm ( x, t,  \omega ) = \psi ( \ell^\mp_\omega( x) ) V_\pm ( \ell^\pm_\omega  ( x ), t, \omega ) , 
\]
where (in the sense of oscillatory integrals)
\begin{equation}
  \label{e:cobra}
\begin{gathered}
V_\pm (y, t, \omega ) :=  \int_{\mathbb R}
 {\tau\over 2\pi}   \int_{t\over \tau}^{(\sgn\tau)\infty}
e^{ i \tau  ( y - \ell_\omega^\pm ( x^\pm_k ( \omega ) ) +   ( \omega - \lambda )r ) }
\widetilde\chi ( r ) b (\tau, \omega ) 
\,  d r d \tau ,\\
b := {a_{k, \pm }^{\varepsilon }\over 2\pi} \in S^{-1+}(\mathcal J_\omega\times\mathbb R_\tau),\quad
\widetilde\chi := \chi_{k,\pm}^{+} \in \CIc ( \mathbb R \setminus 0 ).
\end{gathered}
\end{equation}
Then we have for $x\in\Omega$
$$
v_{k,\pm}^{\varepsilon,+}(x,t)=\int_{\mathbb R}\varphi(\omega) U_\pm(x,t,\omega)\,d\omega,
$$
which together with the Fourier characterization of the Sobolev space $H^{\frac12-\beta}(\mathbb R^2)$ implies the following bound,
where $\widehat U_\pm$ denotes the Fourier transform of $U_\pm$ in the $x$ variable:
\begin{equation}
\label{eq:U2FT}
  \| v_{k,\pm}^{\varepsilon, +} (t)  \|_{ H^{\frac12 - \beta }( \Omega ) }^2\leq  
\int_{\mathbb R^2} \langle \xi \rangle^{1-2 \beta} 
\left| \int_{\mathbb R} \varphi ( \omega  ) \widehat U_\pm ( \xi , t, \omega ) \,d\omega \right|^2 \,d \xi .
\end{equation}

\noindent 2. Thinking of $ L_\omega^\pm  $ (see \eqref{e:L-pm-def}) as elements of the dual $\mathbb R^2$ of $ ( \mathbb R^2 )^* $
we have by~\eqref{eq:L2l}
\[ (\mathbb R^2)^* \ni \xi =  L^+_\omega(\xi)  \ell^+_\omega + 
 L^-_\omega(\xi) \ell^-_\omega , \ \   \ell^\pm_\omega \in 
(\mathbb R^2)^* . \]
Hence, since $\det \partial ( x_1, x_2 )/\partial (\ell^+_\omega, \ell^-_\omega  ) = \frac12 \omega \sqrt { 1 - \omega^2 }$, 
\[
\begin{split}
\mathcal F_{ x \to \xi} \big( f ( \ell^+_\omega (  x ) ) g ( \ell^-_\omega ( x  ) ) \big) & = 
\int_{\mathbb R^2} e^{ - i L^+_\omega( \xi)  \ell^+_\omega (x) -i  L^-_\omega ( \xi)  \ell^-_\omega 
( x) } f ( \ell^+_\omega ( x ) ) g ( \ell^-_\omega ( x  ) )\, dx \\
& =
\tfrac 12 \omega \sqrt { 1 - \omega^2 } \hat f ( L^+_ \omega ( \xi)) \hat g (
  L^-_{ \omega } (\xi) ) .
\end{split} 
\]
Consequently,
\begin{equation} 
\label{eq:hatv} 
\widehat   U_\pm ( \xi , t, \omega ) = L^\pm_\omega( \xi) 
D (L^\pm_\omega( \xi)  , t,  \omega - \lambda )  \widehat \psi (  L^\mp_\omega ( \xi  ) )
e^{ - i   L^\pm_\omega ( \xi ) \ell^\pm_\omega  (x^\pm_k ( \omega ) )} b ( L^\pm_ \omega ( \xi ), \omega ) , \end{equation}
where we absorbed the Jacobian into $b $ and put
\[
D (\tau, t,  \rho ) :=  \int_{t\over\tau}^{ (\sgn\tau) \infty } \widetilde\chi ( r ) e^{ i \tau \rho r } dr ,\quad
\tau\neq 0.
\]
Since $\widetilde\chi\in\CIc(\mathbb R)$, we have
$ D ( t, \tau, \rho ) \to 0 $, for fixed $ \tau\neq 0 $ as $ t \to \infty$, uniformly in $ \rho$.
In view of the support condition in Lemma~\ref{l:conor} (which implies that $ |L^\pm_\omega ( \xi )|\geq 1$ 
on the support of~$ \widehat U_\pm $) we then get 
\begin{equation} 
\label{eq:limUhat}
\int_{\mathbb R} \varphi ( \omega ) \widehat U_\pm ( \xi, t, \omega )\, d\omega \to 0 \quad\text{as } t \to \infty , 
\end{equation}
for all $ \xi \in \mathbb R^2 $. Using the Dominated Convergence Theorem and~\eqref{eq:U2FT} we see
that to establish \eqref{eq:lastl} it is enough to show that the integrand on the right hand side of \eqref{eq:U2FT} is bounded by a
$t$-independent integrable function of~$\xi$. 

\noindent
3. We have
\begin{equation}
  \label{e:D-estimate}
| D ( \tau , t,  \rho ) | \leq C \langle \tau \rho \rangle^{-1}. 
\end{equation}
Indeed, since the support of $ \widetilde\chi $ is bounded, we have $ D = \mathcal O ( 1 ) $.
On the other hand, when $  | \tau \rho  | > 1 $ we can integrate by parts using that
$e^{i\tau\rho r}= (i \tau \rho)^{-1} \partial_r e^{ i \tau \rho r} $ which gives the estimate. 

Recalling~\eqref{eq:hatv}, \eqref{e:D-estimate} and using that $\widehat\psi\in\mathscr S(\mathbb R)$ by~\eqref{e:endgame-psi} and $ b(\tau,\omega)=\mathcal O(\langle\tau\rangle^{-1+\frac\beta2})$ by~\eqref{e:cobra}, we get
$$
|\widehat U_\pm(\xi,t,\omega)|\leq C\big\langle L^\mp_\omega(\xi)\big\rangle^{-10}
\big\langle L^\pm_\omega(\xi)\big\rangle^{\beta\over 2}\big\langle(\omega-\lambda)L^\pm_\omega(\xi)\big\rangle^{-1}.
$$
Thus is remains to show that
$$
\begin{gathered}
\bigg\|\int_{\mathbb R}\varphi(\omega)H(\xi,\omega)\,d\omega\bigg\|_{L^2(\mathbb R^2_\xi)}<\infty\\
\text{where}\quad
H(\xi,\omega):=\langle\xi\rangle^{\frac 12-\beta}\big\langle L^\mp_\omega(\xi)\big\rangle^{-10}
\big\langle L^\pm_\omega(\xi)\big\rangle^{\beta\over 2}\big\langle(\omega-\lambda)L^\pm_\omega(\xi)\big\rangle^{-1}.
\end{gathered}
$$
Using the integral version of the triangle inequality for $L^2(\mathbb R^2_\xi)$, this reduces to
\begin{equation}
  \label{e:majorant}
\int_{\mathbb R}\varphi(\omega)\|H(\xi,\omega)\|_{L^2(\mathbb R^2_\xi)}\,d\omega<\infty.  
\end{equation}
Fix $\omega\in\supp\varphi$ and make the linear change of variables
$\xi\mapsto\eta=(\eta_+,\eta_-)$, $\eta_\pm=L^\pm_\omega(\xi)$. Then we see that
$$
\|H(\xi,\omega)\|_{L^2(\mathbb R^2_\xi)}^2
\leq C\int_{\mathbb R^2}\langle\eta\rangle^{1-2\beta}\langle \eta_\mp\rangle^{-20}
\langle \eta_\pm\rangle^{\beta}\big\langle(\omega-\lambda)\eta_\pm\big\rangle^{-2}\,d\eta.
$$
Integrating out $\eta_\mp$ and making the change of variables
$\zeta:=(\omega-\lambda)\eta_\pm$, we get (for $\omega$ bounded and assuming $\beta<1$)
$$
\|H(\xi,\omega)\|_{L^2(\mathbb R^2_\xi)}^2
\leq C\int_{\mathbb R}\langle\eta_\pm\rangle^{1-\beta}
\big\langle(\omega-\lambda)\eta_\pm\big\rangle^{-2}\,d\eta_\pm
\leq C|\omega-\lambda|^{\beta-2}.
$$
Thus
$$
\int_{\mathbb R}\varphi(\omega)\|H(\xi,\omega)\|_{L^2(\mathbb R^2_\xi)}\,d\omega
\leq C\int_0^1 |\omega-\lambda|^{\frac\beta 2-1}\,d\omega <\infty,
$$
giving~\eqref{e:majorant} and finishing the proof.
\end{proof} 

\subsection{Proof of Theorem \ref{t:1}}

We now review how the pieces presented in~\S\S \ref{s:asy-reduce}--\ref{s:asw} fit together to give the proof of Theorem \ref{t:1}.

In view of \eqref{eq:DeltaOmega}, Lemma~\ref{l:b1}, and~\eqref{e:u-pm-conormal} it suffices to show that
\begin{equation}
\label{eq:w1t}  
\begin{gathered}
\Delta_\Omega^{-1}w_1 ( t ) = u^+ ( \lambda ) + r_2 ( t ) + \tilde e(t), \\
\| r_2 ( t ) \|_{H^1(\Omega) } = \mathcal O( 1 ) ,\quad \| \tilde e ( t ) \|_{H^{\frac12-} ( \Omega ) }
\to 0 \quad\text{as } t \to  \infty .
\end{gathered}
\end{equation}
We use the formula~\eqref{eq:w2w1} which expresses $\Delta_\Omega^{-1}w_1(t)$ as an integral
featuring the distributions $u^\varepsilon(x,\omega)$, $\varepsilon\in\{+,-\}$.
Lemma~\ref{l:conor} gives a decomposition of $ u^\varepsilon $ into 
the conormal components $ g_{k, \pm}^\varepsilon $ and the smooth component $u^\varepsilon_0$.
Lemma \ref{l:smou} then shows that the contribution of $u^\varepsilon_0$ to~$\Delta_\Omega^{-1}w_1(t)$ can
be absorbed into $ r_2 ( t) $. 

The contribution of conormal terms $g_{k,\pm}^\varepsilon$ to~$\Delta_\Omega^{-1}w_1(t)$ is then given by~\eqref{e:w-1-con}. Restricting integration in $ r $ using the cut-off 
$ 1 - \chi $ in Lemma \ref{l:chis}, produces other terms which can be absorbed into $ r_2 ( t ) $. The limit of the remaining terms as $ t\to+\infty $ is described in Lemma~\ref{l:Binf}: summing over $k$ and $\pm$ gives the leading term as $ u^+ ( \lambda ) $ (as seen by returning to Lemma~\ref{l:conor}, where the cutoff $\chi$ does not matter by~\eqref{e:chis-1} and~\eqref{e:chis-2}) and terms which again can be absorbed in $ r_2 ( t ) $. 

What is left is given by a sum of \eqref{eq:defvk}. Lemma \ref{l:last} shows that 
those terms all go to $ 0 $ in $ H^{\frac 12-} ( \Omega ) $ as $ t \to  \infty $ and their
sum constitutes $ \tilde e ( t ) $.

\medskip\noindent\textbf{Acknowledgements.}
We would like to thank Peter Hintz and Hart Smith for  helpful discussions concerning \S \ref{s:microp} and injectivity of $ \mathcal C_\lambda $ (see Remark in \S \ref{s:anal-uniq}), respectively, and
 for their interest in this project. Special thanks are also due to Leo Maas for his comments
 on an early version of the paper, help with the introduction and references and
experimental data in Figure \ref{f:psycho} (reproduced with permission from the
American Institute of Physics). We are grateful to the anonymous referees for
a careful reading of the paper and many comments
to improve the presentation. SD was partially supported by NSF CAREER grant DMS-1749858
and a Sloan Research Fellowship, while JW and MZ were partially supported by NSF grant DMS-1901462. 

\bibliographystyle{abbrv}
\bibliography{General,Dyatlov,Aqua}

\end{document}